\documentclass{amsart}
\usepackage{amssymb}
\usepackage{url}
\usepackage{float}
\usepackage{hyperref}

\usepackage{tikz, tkz-base} 
\usepackage{tikz-cd}
\usetikzlibrary{matrix, arrows, decorations.pathmorphing}

\usepackage[labelfont=rm]{caption}
\usepackage{chngcntr}
\counterwithin{figure}{section} 

\newtheorem{theorem}{Theorem}[section]
\newtheorem{lemma}[theorem]{Lemma}
\newtheorem{corollary}[theorem]{Corollary}
\newtheorem{proposition}[theorem]{Proposition}
\newtheorem{conjecture}[theorem]{Conjecture}

\theoremstyle{definition}
\newtheorem{definition}[theorem]{Definition}
\newtheorem{example}[theorem]{Example}

\theoremstyle{remark}
\newtheorem{remark}[theorem]{\bf Remark}

\numberwithin{equation}{section}

\newcommand{\C}{\mathbb C}
\newcommand{\PP}{\mathbb P}
\newcommand{\Q}{\mathbb Q}
\newcommand{\G}{\mathbb G}
\newcommand{\Z}{\mathbb Z}

\newcommand{\N}{\mathbb N}

\newcommand{\R}{\mathbb R}
\newcommand{\LL}{\mathbb L}
\newcommand{\TT}{\mathbb T}

\DeclareMathOperator{\alg}{alg}
\DeclareMathOperator{\gpast}{ast}
\DeclareMathOperator{\AST}{AST}
\DeclareMathOperator{\AH}{AH}
\DeclareMathOperator{\Aut}{Aut}
\DeclareMathOperator{\CH}{CH}
\DeclareMathOperator{\gpDH}{DH}
\DeclareMathOperator{\gpDL}{DL}
\DeclareMathOperator{\Fr}{fr}
\DeclareMathOperator{\End}{End}
\DeclareMathOperator{\Hom}{Hom}
\DeclareMathOperator{\et}{et}
\DeclareMathOperator{\Gal}{Gal}
\DeclareMathOperator{\GIso}{GIso}
\DeclareMathOperator{\GL}{GL}
\DeclareMathOperator{\SL}{SL}
\DeclareMathOperator{\GO}{GO}
\DeclareMathOperator{\GSp}{GSp}
\DeclareMathOperator{\gpH}{H}
\DeclareMathOperator{\id}{id}
\DeclareMathOperator{\Ind}{Ind}
\DeclareMathOperator{\Iso}{Iso}
\DeclareMathOperator{\Ker}{Ker}
\DeclareMathOperator{\gpL}{L}
\DeclareMathOperator{\gpO}{O}
\DeclareMathOperator{\MMT}{MMT}
\DeclareMathOperator{\MT}{MT}
\DeclareMathOperator{\MS}{MS}
\DeclareMathOperator{\Out}{Out}
\DeclareMathOperator{\USp}{USp}
\DeclareMathOperator{\Sp}{Sp}
\DeclareMathOperator{\SO}{SO}
\DeclareMathOperator{\SU}{SU}
\DeclareMathOperator{\ST}{ST}
\DeclareMathOperator{\Sym}{Sym}

\DeclareMathOperator{\gpU}{U}
\DeclareMathOperator{\Zar}{Zar}
\DeclareMathOperator{\gpP}{P}

\begin{document}

\title[Motivic Serre group and Sato--Tate conjecture]
{{Motivic Serre group and Sato--Tate conjecture}}

\author[Grzegorz Banaszak]{Grzegorz Banaszak}
\address{Department of Mathematics and Computer Science, Adam Mickiewicz University,
Pozna\'{n} 61-614, Poland} 
\email{banaszak@amu.edu.pl}

\author[Kiran S. Kedlaya]{Kiran S. Kedlaya}
\address{Department of Mathematics, University of California San Diego, La Jolla, CA 92093, USA}
\email{kedlaya@ucsd.edu}


\thanks{Thanks to Yves Andr\'e for answering our questions about motivated cycles, particularly concerning 
Remark~\ref{The key containment assumption in ahc and mot}.
Banaszak was supported by UC San Diego (Sept. 2014--June 2015), the grant 346300 for IMPAN from the Simons
Foundation and the matching 2015-2019 Polish MNiSW fund and Weizmann Institute of Science, Rehovot (Aug. 7--22, 2022). Kedlaya was supported by NSF (grants DMS-1101343, DMS-1501214, DMS-1802161, DMS-2053473), UC San Diego (Stefan E. Warschawski professorship), and the IAS School of Mathematics (Visiting Professorship 2018--2019), and benefitted from the hospitality of IM PAN (Simons Foundation grant 346300, Polish MNiSW fund 2015--2019) during September 2018, MSRI (NSF grants DMS-1440140, DMS-1928930) during May 2019 and January 2023,
and CIRM during February 2023.
} 
\begin{abstract}
{This paper concerns the Algebraic Sato--Tate and Sato--Tate conjectures, based on Serre's original motivic formulation, with an eye towards explicit computations of Sato--Tate groups.
We build on the algebraic framework for the Sato--Tate conjecture introduced in \cite{BK2}, which used Deligne's motivic category for absolute Hodge cycles and was restricted to motives of odd weight.
Here, we allow general weight and some other motivic categories, notably Andr{\' e}'s motivic category of motivated cycles; moreover, some results are also new in the odd weight case.
The paper consists of two parts; in the first part we work in the framework of strongly compatible families of $l$-adic representations associated with pure, rational, polarized Hodge structures,
while in the second part we use the language of motives.} 
\end{abstract}

\maketitle
\tableofcontents

\section{Introduction} 

This paper is a sequel to our previous papers \cite{BK1, BK2}. In those papers, we studied the algebraic underpinning of the generalized Sato--Tate conjecture, which gives a group-theoretic explanation of the distribution of normalized Euler factors of the $L$-function of a motive over a number field. This study provides a framework for rigorous
analysis of Sato--Tate groups associated to various classes of motives, which in turn can be used to explain and predict experimental results from the numerical computation of $L$-functions, as in the compilation of the LMFDB (L-Functions and Modular Forms Database) \cite{LMFDB}.

The paper \cite{BK1} focused on the case of motives of weight 1, with an eye towards the classification of Sato--Tate groups of abelian surfaces by Fit\'e--Kedlaya--Rotger--Sutherland \cite{FKRS12}. The paper \cite{BK2} was devoted to the Sato--Tate conjecture in the framework of Deligne's motivic category for absolute Hodge cycles, for motives of odd weight. In this paper, we extend the results of \cite{BK2} to other motivic categories and study motives of arbitrary weight; the case of even weight introduces some parity considerations that do not appear for odd weight. Extending to even weight is motivated in part by interest in the Sato--Tate conjecture for K3 surfaces \cite{EJ}.

As in \cite{BK2}, our approach is based on Serre's treatment of the generalized Sato--Tate conjecture \cite{Se94}
using a construction which we call the \emph{motivic Serre group}  (see \cite[Def. 9.5]{BK2}).
In \cite[p. 396]{Se94}, Serre gave the construction in terms of the  category $\mathcal{M}_{{\rm{num}}}$ of motives for numerical equivalence (assuming Grothendieck's standard conjectures and the Hodge conjecture). In \cite[Def. 11.3]{BK2}, 
we adapted Serre's construction for the category $\mathcal{M}_{{\rm{ahc}}}$ of motives for absolute Hodge cycles
to define an \emph{algebraic Sato--Tate (AST) group} associated to such a motive. In this paper, we consider
the analogous construction for the category of motives $\mathcal{M}_{{\sim}}$
for a more general equivalence relation ${{\sim}}$, defining an AST group for a homogeneous motive (i.e., a direct summand of a motive of the form $h^{r}_{{\sim}}(X) (m)$); this requires assuming certain conjectures, but for cetain motivic categories these are unconditional (see below).
\medskip

We now describe the structure of the paper in more detail.
The paper is divided into two parts.
The first part, \S \ref{Mumford--Tate groups of polarized Hodge structures}--\ref{application to the Sato--Tate conjecture}, is concerned with algebraic Sato--Tate groups and the Sato--Tate conjecture for families 
of $l$-adic representations associated with Hodge structures. 
\medskip

In \S \ref{Mumford--Tate groups of polarized Hodge structures}, for a polarized, rational, pure Hodge structure $(V, \psi)$, we recall the definition of 
the algebraic group $\gpDH(V, \psi)$ introduced in \cite{BK1, BK2} and its relation to 
the Hodge group $\gpH(V, \psi)$, the Mumford--Tate group $\MT(V, \psi)$, and the extended Mumford--Tate group
${\widetilde{\MT}}(V, \psi)$. For a ring $D \subseteq \End_{\mathbb{Q}} (V)$ preserving the Hodge decomposition and having discrete $G_K$-module structure (conditions {\bf{(D1)}} and {\bf{(D2)}}), 
we also recall the definition of the twisted decomposable Lefschetz group $\gpDL_{K}(V, \psi, D)$ introduced in \cite{BK1,BK2}. We obtain results concerning the structure and properties of $\gpDH(V, \psi)$ (Proposition \ref{Hodge realizing conn comp: even case}, Lemma \ref{for n odd DH = H}, Corollary \ref{DH = H if H = CDIso}). We define the Betti parity group $\gpP(V, \psi)$ and state its basic properties. We give explicit examples for which $\gpP(V, \psi)$ is nontrivial
(Examples \ref{An example of nontrivial Betti parity group} and 
\ref{Another example of nontrivial Betti parity group}).
\medskip

In \S \ref{de Rham structures associated with Hodge structures} we consider the bilinear space $(V_{_{\rm{DR}}}, \psi_{_{\rm{DR}}})$ over $K$, which has the property that 
$(V, \psi) \otimes_{\Q} \C \simeq (V_{_{\rm{DR}}}, \psi_{_{\rm{DR}}}) \otimes_{K} \C$.
We define the algebraic groups $\gpDH(V_{_{\rm{DR}}}, \psi_{_{\rm{DR}}})$ and 
$\gpH(V_{_{\rm{DR}}}, \psi_{_{\rm{DR}}})$ assuming that the ring $D$ satisfies additional conditions 
{\bf{(DR1)}} and {\bf{(DR2)}}. We obtain results concerning the structure and properties of the group
$\gpDH(V_{_{\rm{DR}}}, \psi_{_{\rm{DR}}})$ (Proposition \ref{De Rham realizing conn comp: even case}, 
Lemma \ref{for n odd DH-DR = H-DR}, Corollary \ref{DH-DR = H-DR if H-DR = CDIso-DR}).
We define the de Rham parity group $\gpP(V_{_{\rm{DR}}}, \psi_{_{\rm{DR}}})$ 
and state its basic properties. We give explicit examples, based on Example \ref{An example of nontrivial Betti parity group}, 
for which $\gpP(V_{_{\rm{DR}}}, \psi_{_{\rm{DR}}})$ is nontrivial (Example 
\ref{An example of De Rham nontrivial  parity group}).
\medskip

In \S \ref{families of l-adic representations associated with Hodge structures}, we work with a class of $l$-adic representations associated with Hodge structures from \S \ref{Mumford--Tate groups of polarized Hodge structures} and satisfying additional conditions {\bf{(R1)}}--{\bf{(R4)}}. Examples of such $l$-adic representations are given by {\' e}tale cohomology of proper and smooth schemes (Remarks \ref{Hodge--Tate representations in etale cohomology}--\ref{Action od D is GK equivariant}). We recall the basic properties of the group schemes ${G_{l, K}^{\alg}}$ and ${G_{l, K, 1}^{\alg}}$ and the relation of the latter group to the twisted decomposable Lefschetz group (Corollary \ref{Corollary GlL1alg subset DLLA}). The group scheme ${G_{l, K, 1}^{\alg}}$ is fundamental for the setup of the Sato--Tate conjecture. 
\medskip

In \S \ref{identity connected component of GlK1alg}, we investigate the structure of ${G_{l, K, 1}^{\alg}}$ in relation with
the structure of ${G_{l, K}^{\alg}}$. The basic results are 
Theorem \ref{L0realizing conn comp for GlK alg} and Corollary \ref{Corollary-GK1alg mod GK1algID with resp. to DL mod DLId}. There is a natural map between component groups:
$$
i_{CC}\colon \pi_{0} ({G_{l, K, 1}^{\alg}}) \rightarrow \pi_{0} ({G_{l, K}^{\alg}}).
$$ 
Based upon two technical results (Theorems \ref{L0realizing conn comp of GlKalg} and 
\ref{L0realizing conn comp}) we prove  that $i_{CC}$ is an isomorphism when $n$ is odd
(Theorem \ref{conn comp equal for for GlK1 alg and GlK alg}). For an extension $M_{0} / K$ such that
${G_{l, M_{0}}^{\alg}}$ is connected, we prove two technical results (Propositions \ref{L0realizing conn comp: even case} and
\ref{L0 realizing for rho l (GK): even case}) concerning the structure of ${G_{l, M_0, 1}^{\alg}}$ with relation to 
its identity component. Let $K_{0}/K$ be the smallest extension
such that ${G_{l, K_{0}}^{\alg}}$ is connected; we establish some results concerning 
${G_{l, K_{0}, 1}^{\alg}}$ and the map $i_{CC}$ (Theorem \ref{K0 almost realizing Glk1 alg: even case} 
and Corollary \ref{K0 realizing Glk1 alg: even case condition}). We also define the Serre $l$-adic parity group 
$P_{\rm{S}} (V_{l}, \psi_{l})$ and describe its basic properties. Again we give explicit examples, based on Example 
\ref{An example of nontrivial Betti parity group}, for which $P_{\rm{S}}(V_{l}, \psi_{l})$ is nontrivial 
(Example \ref{An example of nontrivial l-adic parity group}).
It follows from Theorem 
\ref{K0 almost realizing Glk1 alg: even case} that the map $i_{CC}$ is an isomorphism if and only if 
$P_{\rm{S}}(V_{l}, \psi_{l})$ is trivial. At the end of this section we give conditions for 
$\overline{\rho_{l} (G_K)_1} \, = \, {G_{l, K, 1}^{\alg}}$ (Theorem \ref{K giving the isom overline rho l (GK) the same as GlK1 alg: arbitrary weight} and Corollary \ref{K0 giving the isom overline rho l (GK) the same as GlK1 alg: odd case}). 
\medskip

In \S \ref{section-computation of the identity connected component}, following an approach of Serre, we define $\widetilde{G_{l, K}^{\alg}}$ and 
$\widetilde{G_{l, K, 1}^{\alg}}$. We prove that $\widetilde{G_{l, K, 1}^{\alg}} \subset G_{l, K, 1}^{\alg}$ with index at most 2 and $(\widetilde{G_{l, K, 1}^{\alg}})^{\circ} = (G_{l, K, 1}^{\alg})^{\circ}$. 
We also prove that for any weight $n$ the following map is an isomorphism:
$$
\widetilde{{i}}_{CC}\colon \pi_{0} (\widetilde{G_{l, K, 1}^{\alg}}) \,\,\, {\stackrel{\simeq}{\longrightarrow}} \,\,\, \pi_{0} (\widetilde{G_{l, K}^{\alg}}). 
$$
The identity connected component of $G_{l, K, 1}^{\alg}$ 
(Corollary \ref{Connected component GlKlK1alg in form of widetilde(GlK01alg)circ})
is computed as follows:
$$
(G_{l, K, 1}^{\alg})^{\circ} = \widetilde{G_{l, K_{0}, 1}^{\alg}}.
$$
We also prove that $G_{l, K, 1}^{\alg} = \widetilde{G_{l, K, 1}^{\alg}}$ if the parity group     
$\gpP(V_{l}, \psi_{l})$ is trivial.
\medskip

In \S \ref{algebraic ST conjecture for families of l-adic reps}, we formulate Algebraic Sato--Tate Conjecture \ref{general algebraic Sato Tate conj.} 
and Sato--Tate Conjecture \ref{general Sato Tate conj.} for $l$-adic representations with respect 
to the groups $G_{l, K, 1}^{\alg}$.
Under Algebraic Sato--Tate Conjecture \ref{general algebraic Sato Tate conj.}, we determine the relationship between the algebraic Sato--Tate and Sato--Tate groups for the fields $K$ and $K_{0}$.
We also formulate Algebraic Sato--Tate Conjecture \ref{general algebraic Sato Tate conj. Serre's approach} 
and Sato--Tate Conjecture \ref{general Sato Tate conj. tilde} with respect to the groups
$\widetilde{G_{l, K, 1}^{\alg}}$. We prove that Conjecture \ref{general algebraic Sato Tate conj. Serre's approach} implies Conjecture \ref{general algebraic Sato Tate conj.} and we discuss obstructions to the opposite 
implication. We also analyze the relationship between the algebraic Sato--Tate groups  
$\AST_{K} (V, \psi)$, $\widetilde{\AST_{K}} (V, \psi)$ and the Sato--Tate groups $\ST_{K} (V, \psi)$, 
$\widetilde{\ST_{K}} (V, \psi)$ (Remark \ref{Conjugacy classes of normalized Frobenius elements are in wide tilde AST}, Corollary \ref{relations between AST and wide tilde AST and between ST and wide tilde ST}).
Under Conjecture \ref{general algebraic Sato Tate conj. Serre's approach}, we prove (Theorem \ref{equidistribution property in ST implies equidistribution property in wide tilde ST})
 that if
$- {{\rm Id}}_{V} \notin \widetilde{\ST_{K}} (V, \psi)$ then Sato--Tate Conjecture
\ref{general Sato Tate conj.} does not hold, whereas if $- {{\rm Id}}_{V} \in \widetilde{\ST_{K}} (V, \psi)$
then Sato--Tate Conjectures \ref{general Sato Tate conj.}  and \ref{general Sato Tate conj. tilde}
are equivalent; this suggests that the latter is the more correct formulation.
\medskip

In \S \ref{AST and Tate for families} we state Conjectures \ref{Tate conjecture for families of l-adic representations} and 
\ref{Tate conjecture for families of l-adic representations tilde} which are $l$-adic analogues of the 
Tate conjecture for motives. For a special case of these conjectures see \cite{LP}. The difference between the  Tate conjecture for motives and these conjectures is the existence of a global object (motivic Galois group) in the Tate conjecture, whereas the existence of global objects 
$\mathcal{G}_{K}^{\alg} := \mathcal{G}_{K}^{\alg}(V, \psi)$ and 
$\widetilde{\mathcal{G}_{K}^{\alg}} := \widetilde{\mathcal{G}_{K}^{\alg}}(V, \psi)$ is assumed in Conjectures \ref{Tate conjecture for families of l-adic representations} and 
\ref{Tate conjecture for families of l-adic representations tilde}. Throughout \S \ref{AST and Tate for families}  we assume only the weaker form of Conjectures \ref{Tate conjecture for families of l-adic representations} and \ref{Tate conjecture for families of l-adic representations tilde}, namely Conjectures
\ref{Tate conjecture for families of l-adic representations}(a) and 
\ref{Tate conjecture for families of l-adic representations tilde}(a); all results in \S \ref{AST and Tate for families}  are 
under this assumption. This framework allows to define naturally Algebraic Sato--Tate groups 
$\AST_{K} (V, \psi) := \mathcal{G}_{K, 1}^{\alg}$ and $\widetilde{\AST_{K}} (V, \psi) := 
\widetilde{\mathcal{G}_{K, 1}^{\alg}}$; the results we obtain are very similar to the 
results of \S \ref{section-computation of the identity connected component} for the $l$-adic site and to the results of \S \ref{algebraic Sato--Tate group for motives} for the motivic site. In particular, among a number of other results, we prove that there is a natural isomorphism (Theorem \ref{pi 0 wide tilde math cal G K, 1 alg cong pi 0 wide tilde math cal G K alg}):
\begin{equation*}
\widetilde{i}_{CC}\colon \pi_{0} (\widetilde{\mathcal{G}_{K, 1}^{\alg}}) \,\,\, {\stackrel{\simeq}{\longrightarrow}} \,\,\, \pi_{0} (\widetilde{\mathcal{G}_{K}^{\alg}}).
\end{equation*}
Moreover we prove that 
all four Conjectures \ref{general algebraic Sato Tate conj.},
\ref{general algebraic Sato Tate conj. Serre's approach}, \ref{Tate conjecture for families of l-adic representations}, 
\ref{Tate conjecture for families of l-adic representations tilde} are equivalent.
The framework of this section provides a wider perspective on Algebraic Sato--Tate Conjectures \ref{general algebraic Sato Tate conj.} and \ref{general algebraic Sato Tate conj. Serre's approach}.
\medskip

In \S \ref{sec,category of polarized realizations} we present families $(\rho_{l})_{l}$ of $l$-adic representations which are associated with pure, polarized, rational Hodge structures and for which Conjectures \ref{Tate conjecture for families of l-adic representations}(a) and 
\ref{Tate conjecture for families of l-adic representations tilde}(a) are satisfied (see Proposition \ref{the family V l l of G K satisfies Conjectures 1 (a) and 2 (a)}). These families will be determined by certain objects in the category of polarized realizations $R_{K}^{\rm{p}}$ which we define; it is a full subcategory of the category of Jannsen realizations $R_{K}$ \cite{Ja90}. 
For the convenience of readers our computations concerning polarized realizations are very detailed.  
\medskip

\S \ref{remarks on equidistribution} is independent of the previous sections and somewhat different in character. In \S \ref{remarks on equidistribution} we consider an open subgroup $G_{0}$ of a compact group $G$ and the map 
$X(j)\colon X(G_{0}) \rightarrow X(G)$ of the conjugacy classes induced by natural embedding 
$j\colon G_{0} \rightarrow G$. Consider a sequence of elements $(x_n)$ in $X(G)$. Our first result, 
Proposition \ref{equidistr. of xv implies equidistr. of xnj}, shows that equidistribution of the sequence 
$(x_n)$ in $X(G)$ with respect to the Haar measure $\mu$ on $G$ implies, under a natural numerical condition
\eqref{limit condition}, the equidistribution with respect to the pushforward measure $j_{\ast} \mu_{G_{0}}$ 
 of the subsequence $(x_{n_k})$ of $(x_n)$ with elements in the image under the map $X(j)$. For the rest of
\S \ref{remarks on equidistribution}, we study the problem of inverting Proposition \ref{equidistr. of xv implies equidistr. of xnj}, i.e. to deduce an equidistribution statement in $X(G)$ from equidistribution statements in $X(G_{0})$ and $X(G/G_{0})$. As first results in this direction we prove Lemma \ref{reduction for equidistribution easy case} and Proposition
\ref{proposition about the reduction of equidistribution to G0}. To get stronger results concerning 
deduction of equidistribution on $G$ from equidistribution on $G_0$, we first extend Artin's Theorem on induced characters from finite groups to compact groups (Lemma \ref{Artin Clifford lemma}). We then prove Lemma \ref{reduction for equidistribution} and Theorem \ref{reduction for equidistribution2}; in the latter, $G_{0}$ is a connected Lie group satisfying some extra conditions on its factors.
\medskip

The first part of \S \ref{application to the Sato--Tate conjecture}, ending at the proof of Lemma \ref{unipotent elements in an algebraic group}, depends only on \S \ref{remarks on equidistribution}. In the first part of \S \ref{application to the Sato--Tate conjecture} we discuss equidistribution on a compact group $G$ via $L$-series associated with representations
of $G$ on finite dimensional $\C$-vector spaces, following the idea of Serre \cite[Appendix to Chapter I]{Se1}. Consider a sequence $(x_v)$ of elements of $X(G)$ indexed by prime ideals of $\mathcal{O}_{K}$ (see \textit{loc. cit.}).
Fix any finite extension $L/K$. We prove an auxiliary result (Lemma \ref{Equidistribution with respect to all primes and primes in S}) that $(x_v)$ is equidistributed on $X(G)$ with respect to the measure induced on $X(G)$ by the Haar measure $\mu$ on $G$ if and only if the subsequence $(x_v)_{S}$
of $(x_v)$, indexed by $v$'s that split completely in $\mathcal{O}_L$, is equidistributed on $X(G)$ with respect to the same measure. We include an auxiliary Lemma \ref{unipotent elements in an algebraic group} concerning unipotent elements 
in algebraic groups which is indispensable in further computations in this section. 
\medskip

In the second part of \S \ref{application to the Sato--Tate conjecture}, we assume Conjectures \ref{general algebraic Sato Tate conj.} and \ref{general algebraic Sato Tate conj. Serre's approach}, that is, we work under the assumptions of Theorems \ref{STK iff STK0} and \ref{tilde STK iff tilde STK0}. We investigate equidistribution of Frobenius elements in the conjugacy classes $X(G)$ of  $G := \ST_{K} (V, \psi)$ and its relationship with equidistribution of Frobenius elements in the conjugacy classes $X(G_{0})$ of $G_{0} := \ST_{K_{0}} (V, \psi)$. 
The main result in this section (Theorem \ref{Sato--Tate conjecture STK iff STK1})
states that the sequence of Frobenius elements in $X(G)$ is equidistributed if and only if the sequence of Frobenius elements in $X(G_1)$ is equidistributed for each subgroup $G_1$ of $G$ such that $G_0 \subset G_1$ and $G_1/G_0$ is cyclic. The proof of Theorem \ref{Sato--Tate conjecture STK iff STK1} is based on a number of technical results from previous sections and three technical results from this section: 
Lemmas \ref{primes of L splitting completely over  v0}, \ref{Lemma about bijections of double cosets for AST and ST}, and 
Corollary \ref{primes of L splitting completely over v0 correspond to conjugations of sv}.
 
Observe that under the assumption of Algebraic Sato--Tate Conjecture \ref{general algebraic Sato Tate conj.}, we have $\ST_{K_{0}} (V, \psi) = \ST_{K} (V, \psi)^{\circ}$ if and only if the Serre $l$-adic parity group $P_{\rm{S}} (V_{l}, \psi_{l})$ is trivial (cf. Proposition \ref{connected components iso} and Theorem \ref{connected component of ASTK}). Hence in the case where $\ST_{K_{0}} (V, \psi) = \ST_{K} (V, \psi)^{\circ}$, Theorem \ref{Sato--Tate conjecture STK iff STK1} states that it is enough to check the Sato--Tate Conjecture \ref{general Sato Tate conj.} on cyclic extensions of the identity component of the Sato--Tate group $\ST_{K} (V, \psi)$.
Because $\widetilde{\ST_{K_{0}}} (V, \psi) = \widetilde{\ST_{K}} (V, \psi)^{\circ}$ by 
Theorem \ref{connected component of tilde AST K} and Corollary
\ref{connected components of AST K (V, psi), ST K (V, psi) and widetilde AST K (V, psi), widetilde ST K (V, psi)}
then again Theorem \ref{Sato--Tate conjecture STK iff STK1} states that it is enough to check 
the Sato--Tate Conjecture \ref{general Sato Tate conj. tilde} on cyclic extensions of the identity component of the Sato--Tate group $\widetilde{\ST_{K}} (V, \psi)$.

\bigskip

The second part of the paper, \S \ref{relations among motivic categories}--\ref{equidistribution of Frobenii in l-adic realization of motives}, is concerned with the algebraic Sato--Tate and Sato--Tate conjectures for Hodge structures and $l$-adic representations arising from motives in various motivic categories.  
\medskip

In \S \ref{relations among motivic categories}, we give a very brief review of classical motivic categories and relations among them.
\medskip

In \S \ref{assumptions on Msim}, we describe the basic assumptions on the motivic category $\mathcal{M}_{{\sim}}$ that we need to impose in \S \ref{assumptions on Msim}--\ref{equidistribution of Frobenii in l-adic realization of motives}. In Assumptions 1--3, we require that $\mathcal{M}_{{\sim}}$ satisfies  the
Chow--K{\" u}nneth  conjecture (leading to motivic Poincare duality), is Tannakian semisimple over $\Q$, and has a motivic star operator. This gives the pure Hodge structure on the Betti realization of a homogeneous motive in $\mathcal{M}_{{\sim}}$. Note that Assumptions 1--3 are known to hold for $\mathcal{M}_{{{\rm{ahc}}}}$ and 
$\mathcal{M}_{{{\rm{mot}}}}$, the latter being the category of motives for \emph{motivated cycles}
in the sense of Andr\'e  \cite{An2}. We introduce Assumption 4 which gives better control on relations between categories $\mathcal{M}_{K}^{\circ} (D)$ and $\mathcal{M}_{{\sim}}(M)$ (the former one defined 
in \S \ref{motivic Galois group and motivic Serre group}); Assumption 4 again holds in $\mathcal{M}_{{\rm{ahc}}}$ and $\mathcal{M}_{{\rm{mot}}}$ (Remark  
\ref{The key containment assumption in ahc and mot}).
\medskip

In \S \ref{motivic Galois group and motivic Serre group}, for a homogeneous motive $M$ in $\mathcal{M}_{\sim}$ and its Betti realization 
$(V, \psi)$, we define the ring $D := D (M)$ (see \ref{definition of D(M)}) with $G_K$-discrete module structure 
and the  Artin motive $h^{\circ}(D)$ corresponding to $D$. We also introduce $\mathcal{M}_{K}^{\circ} (D)$, the smallest Tannakian subcategory of the category of Artin motives of $\mathcal{M}_{K}^{\circ}$ containing $h^{\circ}(D)$, and the corresponding motivic Galois group $G_{\mathcal{M}_{K}^{\circ} (D)}$. 
We recall the motivic Galois groups $G_{\mathcal{M}_{{\sim}}}$ and  
$G_{\mathcal{M}_{{\sim}}(M)}$. We finally recall the definition of the motivic Serre group 
$G_{\mathcal{M}_{\sim}(M), 1}$ and show the following formula (see \ref{GMKA1 subset of DLKA}):
$$
G_{\mathcal{M}_{\sim}(M), 1} \subset DL_{K}(V,\, \psi, D).
$$ 

In \S \ref{motivic Mumford--Tate group and motivic Serre group}, we investigate the structure of $G_{\mathcal{M}_{\sim}(M), 1}$ in relation with
the structure of $G_{\mathcal{M}_{{\sim}}(M)}$. There is a natural map between component groups:
$$
i_{M}\colon \pi_{0} (G_{\mathcal{M}_{{\sim}}(M), 1}) \,\, \rightarrow  \,\, \pi_{0} (G_{\mathcal{M}_{{\sim}}(M)}).
$$
The basic results concerning $i_{M}$ are Theorem \ref{equality of conn comp for GM and GM1}, 
Corollary \ref{equality of quotients concerning for GM and GM1}, and 
Corollary \ref{equality of quotients concerning for GMid and GM1id}. In Proposition
\ref{Motivic L0realizing conn comp: even case}, we analyze the relationship between $G_{\mathcal{M}_{{\sim}}(M), 1}$
and its identity component. It follows from Theorem \ref{equality of conn comp for GM and GM1} and 
Lemma \ref{for n odd G0M1 = GM10} that for $n$ odd, the map $i_{M}$ is an isomorphism. We recall Serre's Conjecture concerning the Motivic Mumford--Tate group and state Motivic Mumford--Tate and Motivic Sato--Tate conjectures. We define the Serre motivic parity group $\gpP_{\rm{S}} (M, \psi_{{\sim}})$ and prove its basic properties. 
Once more, we give explicit examples, 
based on Example \ref{An example of nontrivial Betti parity group}, for which $\gpP_{\rm{S}} (M, \psi_{{\sim}})$ is nontrivial  (Example \ref{An example of nontrivial motivic parity group}).
We immediately observe 
(see Theorem \ref{equality of conn comp for GM and GM1}) that $i_{M}$ is an isomorphism if and only if 
$\gpP_{\rm{S}} (M, \psi_{{\sim}})$ is trivial. At the end of \S \ref{motivic Mumford--Tate group and motivic Serre group} we discuss relations among the parity groups we have defined. The Betti, De Rham, and $l$-adic realizations of the motivic category $\mathcal{M}_{{\sim}}$ give 
natural homomorphisms among these groups.
It is immediate from definitions that every parity group is either trivial or isomorphic to $\{\pm 1\}$.
We collect from previous sections some general conditions for all parity groups to be trivial 
(Proposition \ref{general conditions for all parity groups to be trivial and nontrivial}(a)) and for all parity groups to be nontrivial (Proposition \ref{general conditions for all parity groups to be trivial and nontrivial}(b)). 
\medskip

In \S \ref{algebraic Sato--Tate group for motives}, following Serre  \cite[sec. 13, pp. 396--397]{Se94}  we define the \emph{algebraic Sato--Tate group} for a homogeneous motive $M$ as the motivic Serre group (Definition \ref{GMKA1 as AST group}):
$$
\AST_{K} (M) := \MS_{{\sim}} (M)
$$
and the Sato--Tate group $\ST_K(M)$ as a maximal compact subgroup of $\AST_{K} (M) (\C)$. 
Theorem \ref{General properties of AST} states that 
$\AST_{K} (M)$ is reductive and:
\begin{align*}
\AST_{K} (M) &\subset \gpDL_{K}(V, \psi, D), \\
\gpH(V, \psi) &\subset \AST_{K} (M)^{\circ}, \\
\pi_{0} (\AST_{K} (M)) & \simeq \pi_{0} ( \MMT_{{\sim}} (M)) \,\,\,\, \text{iff} \,\,\,
-\mathrm{Id}_{V} \in \AST_{K} (M)^{\circ}, \\
\pi_{0} (\AST_{K} (M)) &= \pi_{0} (\ST_{K} (M)), \\
G_{l, K, 1}^{\alg}  &\subset \AST_{K}(M)_{\Q_l}, \mbox{i.e., Conjecture \ref{general algebraic Sato Tate conj.}(a) holds for $M$}.
\end{align*}
Under the assumption that $\gpH(V, \psi) = C_{D} (\Iso_{(V, \psi)})$ (i.e., the Hodge group is explained by the endomorphisms in $D$) we determine the structure of $\AST_{K}(M)$ (Corollary \ref{The natural candidate for AST group example 1}):
\begin{align*}
\AST_{K}(M)^{\circ} &= \gpL(V, \psi, D), \\
\pi_{0}(\AST_{K}(M)) &=  \Gal(K_e / K), \\
\AST_{K}(M) &= \gpDL_{K} (V, \psi, D).
\end{align*}
In addition, if the Mumford--Tate conjecture holds for $M$, then the algebraic Sato--Tate conjecture holds
for $M$ and the structure of $G_{l, K, 1}^{\alg}$ is also determined (Corollary 
\ref{cor alg Sato--Tate for MT explained by endo}). 
\medskip

In \S \ref{computation of the identity connected component of ASTKM} for a polarized 
motive $M$ and the Tate motive $\TT$ we investigate
$G_{\mathcal{M}_{{\sim}}(M \oplus \TT)}$ and the motivic Serre group
$\MS_{{\sim}} (M \oplus \TT) := G_{\mathcal{M}_{{\sim}}(M \oplus \TT), 1} = {\rm{Ker}} \, ( G_{\mathcal{M}_{{\sim}}(M \oplus \TT)} \,\, {\stackrel{N}{\longrightarrow}}  \,\, \G_{m})$. We prove that there is a natural isomorphism
(Theorem \ref{proposition: i M oplus T is an isomorphism}):
$$
i_{M \oplus \TT}\colon \pi_{0} (G_{\mathcal{M}_{{\sim}}(M \oplus \TT), 1}) \,\,\, {\stackrel{\simeq}{\longrightarrow}}  \,\,\,  \pi_{0} (G_{\mathcal{M}_{{\sim}}(M \oplus \TT)}). 
$$
We define the \emph{algebraic Sato--Tate group}:
$$
\widetilde{\AST_{K}} (M) := \MS_{{\sim}} (M \oplus \TT). 
$$
Every maximal compact subgroup of $\widetilde{\AST_{K}} (M)(\C)$ will be called
a \emph{Sato--Tate group} associated with $M$ and denoted $\widetilde{\ST_{K}} (M)$. Comparing the groups
$\AST_{K}(M)$ and $\widetilde{\AST_{K}} (M)$ we obtain:
\begin{align*}
\AST_{K}(M)^{\circ} &= \widetilde{\AST_{K}} (M) \, \cap \, 
(G_{\mathcal{M}_{{\sim}}(M \oplus \TT})^{\circ}, \\
\AST_{K}(M) &= \widetilde{\AST_{K}} (M) \, \cup \, - {\rm{Id}}_{V} \, \widetilde{\AST_{K}} (M), \\
\ST_{K}(M) &= \widetilde{\ST_{K}} (M) \, \cup \, - {\rm{Id}}_{V} \, \widetilde{\ST_{K}} (M),
\end{align*}
up to conjugation in $\AST_{K}(M) (\C)$.

Further we prove that the algebraic Sato--Tate conjecture holds for $\AST_{K}(M)$ if and only if it holds for 
$\widetilde{\AST_{K}} (M)$ (Prop. \ref{AST conjecture for M is equivalent to wide tilde AST conjecture }). We prove 
that the Sato--Tate Conjecture for $\ST_{K}(M)$ implies the Sato--Tate Conjecture for $\widetilde{\ST_{K}} (M)$ (Prop. \ref{ST implies tilde ST}). In connection with \S \ref{sec,category of polarized realizations} we propose the Geometricity Conjecture \ref{Geometricity conjecture}. 
At the end of \S \ref{computation of the identity connected component of ASTKM} we discuss the Sato--Tate parity group $\gpP_{\rm{ST}} (M, \psi_{{\sim}})$ for motives $M$ and its impact on the Sato--Tate conjecture (Theorem \ref{equidistribution property in ST implies equidistribution property in wide tilde ST for 
PST of motives}). 
\medskip

In \S \ref{equidistribution of Frobenii in l-adic realization of motives} we consider the Sato--Tate conjecture  
for homogeneous motives $M$ (direct factors of $h_{{\sim}}^{r}(X)(m)$) of nonzero weight $n = 2m-r.$
We assume that the family $(\rho_{l}),$ associated with the $l$-adic realizations of $M,$ is strictly compatible.
Then under some additional technical assumptions, the main result in this section (Theorem \ref{ST with resp. to STK the same as with resp. to STK1}) states that
the Sato--Tate Conjecture holds for the representation $\rho_l \colon G_K \rightarrow 
\GIso_{(V_l, \psi_l)}(\Q_l)$ (resp. $\widetilde{\rho}_l\colon G_K \rightarrow \GL (V_{l} \oplus \Q_l(1))$)
with respect to $\ST_{K}(M)$ (resp. $\widetilde{\ST_{K}} (M)$) if and only if it holds for $\rho_l\colon G_{K_1} \rightarrow \GIso_{(V_l, \psi_l)}(\Q_l)$ (resp. $\widetilde{\rho}_l\colon G_{K_1} \rightarrow \GL (V_{l} \oplus \Q_l(1))$)  with respect to $\ST_{K_1} (M)$ (resp. $\widetilde{\ST_{K_{1}}} (M)$) for all subextensions $K_1$ of $K_0/K$ such that $K_{0} / K_{1}$ is cyclic (cf. Theorem~\ref{Sato--Tate conjecture STK iff STK1}).
\medskip

At the end of this section we discuss examples concerning Theorems \ref{reduction for equidistribution2} and 
\ref{ST with resp. to STK the same as with resp. to STK1}. The main observations are in Example \ref{example of ST for abelian surfaces} and Remarks \ref{Sato--Tate for abelian surfaces by Christian Johannson} and 
\ref{ST for abelian 3-folds} concerning abelian surfaces and abelian 3-folds.

\section{Mumford--Tate groups of polarized Hodge structures}
\label{Mumford--Tate groups of polarized Hodge structures}
\medskip

Let $(V, \psi)$ be a rational, polarized, pure Hodge structure of weight $n$ with polarization:
$$
\psi\colon V \times V \rightarrow \Q(-n),
$$
(cf. \cite[Definition 2.9]{PS}). Following \cite[Chapter 2]{BK2} define: 

{\small{
\begin{align}
\GIso_{(V, \psi)} & := \{g \in \GL_{V}:\, \exists \, \chi (g) \in \G_m \,\, s.t. \,\, \psi (gv, gw) = 
\chi (g) \psi (v, w) \,\,\,\, \forall \, v, w \in V\},
\label{def of GIso} \\
\Iso_{(V, \psi)} & := \{g \in \GL_{V}:\,\, \psi (gv, gw) =  
\psi (v, w) \,\,\,\,  \forall \, v, w \in V\}.
\label{def of Iso}
\end{align}}}
 
By \eqref{def of GIso} we observe that $\chi$ becomes a character of $\GIso_{(V, \psi)}$. We call
$\chi$ the \emph{character} of the polarization $\psi$. Observe that
$\G_{m} {\rm{Id}}_{V} \subseteq \GIso_{(V, \psi)}$ and
\begin{equation}
\GIso_{(V, \psi)} = \G_{m} {\rm{Id}}_{V} \cdot \Iso_{(V, \psi)}. 
\label{GIso = Gm Iso}
\end{equation}
Indeed, for every $\alpha \in \G_{m,\Q}$  we obtain
$$
\psi (\alpha \, {{\rm Id}}_{V} v,\,  \alpha \, {{\rm Id}}_{V} w) = \psi (\alpha v, \, \alpha w) =
\alpha^2 \, \psi(v, \, w).
$$
Hence
$$
\alpha \, {{\rm Id}}_{V} \in 
\GIso_{(V, \psi)}
\quad {{\rm and}} \quad \chi (\alpha \, {{\rm Id}}_{V}) = \alpha^2.
$$
By the definition of a polarized Hodge structure, $\psi$ is $(-1)^n$-symmetric, hence
\begin{align*}
\GIso_{(V, \psi)} \,\, &= \,\,
\left\{
\begin{array}{lll}
\GO_{(V, \psi)}&{\rm if}&n \,\,\, {\rm even,}\\
\GSp_{(V, \psi)}&{\rm if}&n \,\,\, {\rm odd;}\\
\end{array}
\right.
\\
\Iso_{(V, \psi)} \,\, &= \,\,
\left\{
\begin{array}{lll}
\gpO_{(V, \psi)}&{\rm if}&n \,\,\, {\rm even,}\\
\Sp_{(V, \psi)}&{\rm if}&n \,\,\, {\rm odd.}\\
\end{array}
\right.
\end{align*}

Let $\MT(V, \psi)$ (resp. $\gpH(V, \psi)$) denote the Mumford--Tate (resp. Hodge) group for $(V, \psi)$.
Recall the following definition  \cite[Def. 2.6 (2)]{BK2}:
$$
\gpDH(V, \psi) := \MT(V, \psi) \cap \Iso_{(V, \psi)}.
$$ 
By the definition of $\Iso_{(V, \psi)}$ one observes that \cite[Def. 2.6 (3)]{BK2}
$$
\gpH(V, \psi) = \gpDH(V, \psi)^{\circ}.
$$

By the definition of a rational Hodge structure, $\G_{m} {\rm{Id}}_{V} \subseteq \MT(V, \psi)$.
We thus have  the following commutative diagram in which the horizontal arrows are closed immersions and the columns
are exact.

\begin{figure}[H]
\[
\begin{tikzcd}
1 \arrow{d}[swap]{} &  1 \arrow{d}[swap]{}\\
\gpDH(V, \psi) \arrow{d}[swap]{} \arrow{r}{} &
\Iso_{(V, \psi)} \arrow{d}[swap]{} \\ 
\MT(V, \psi) \arrow{d}[swap]{\chi} \arrow{r}{} & 
\GIso_{(V, \psi)} \arrow{d}[swap]{\chi} \\
\G_{m} \arrow{d}[swap]{} \arrow{r}{=} & \G_{m} \arrow{d}[swap]{} \\
1  & 1\\
\end{tikzcd}
\]
\\[-0.8cm]
\caption{}
\label{DH(V, psi) subset Iso(V, psi)} 
\end{figure}

\begin{proposition} \label{Hodge realizing conn comp: even case} 
The algebraic group $\gpDH(V, \psi)$ has the following properties:
\begin{itemize}
\item[(a)] $\G_{m} {\rm{Id}}_{V} \cdot \gpH(V, \psi) \, = \, 
\G_{m} {\rm{Id}}_{V} \cdot \gpDH(V, \psi) \, = 
\, \MT(V, \psi)$.
\item[(b)] $\gpDH(V, \psi) = \gpH(V, \psi)  \, \cup \, 
- {\rm{Id}}_{V} \cdot \gpH(V, \psi)$.
\item[(c)] $- {\rm{Id}}_{V} \in \gpH(V, \psi)$ iff 
$\gpH(V, \psi) = \gpDH(V, \psi)$.
\item[(d)] When $n$ is odd then $- {\rm{Id}}_{V} \, \in \gpH(V, \psi) = 
\gpDH(V, \psi)$.  
\end{itemize} 
\end{proposition} 
\begin{proof}
(a) Consider a coset $g \gpH(V, \psi)$ in $\gpDH(V, \psi)$. 
Applying \cite[Section 7.4, Prop. B(b)]{Hu} to the homomorphism
\begin{gather*}
\G_{m} {\rm{Id}}_{V} \, \times \, \gpH(V, \psi) \rightarrow 
\MT(V, \psi), \\
(g_1, g_2) \, \mapsto \, g_1 g_2,
\end{gather*}
we observe that  $\G_{m} {\rm{Id}}_{V} \cdot \gpH(V, \psi)$ and $\G_{m} {\rm{Id}}_{V} \cdot g \gpH(V, \psi)$ are closed in $\MT(V, \psi)$. They are also of the same dimension as $\gpH(V, \psi)$ because the left vertical column in Diagram \ref{DH(V, psi) subset Iso(V, psi)} is exact and $\gpDH(V, \psi)^{\circ} = \gpH(V, \psi)$.

Because $\MT(V, \psi)$ is an irreducible algebraic group and $\G_{m} {\rm{Id}}_{V} \cdot \gpH(V, \psi)$ is a closed subgroup of the same dimension, we must have 
\begin{equation}
\G_{m} {\rm{Id}}_{V} \cdot \gpH(V, \psi) \, = 
\,\G_{m} {\rm{Id}}_{V} \cdot  g\, \gpH(V, \psi) = \MT(V, \psi).
\label{equality of Gm IdV . H and MT}
\end{equation}

\noindent
(b) From \eqref{equality of Gm IdV . H and MT} there are $\alpha, \beta \in \G_{m}$ and $g_1, g_2 \in 
\gpH(V, \psi)$ such that:
\begin{equation}
\alpha \, {\rm{Id}}_{V} \cdot g_1 = \beta \, {\rm{Id}}_{V} \cdot g g_2. 
\label{equality of coset generators for Hodge groups}
\end{equation}
Applying $\chi$ to \eqref{equality of coset generators for Hodge groups} we obtain 
$\alpha^2 = \beta^2$. This implies that $\pm {\rm{Id}}_{V} \cdot g_1 =  g g_2$. Hence
$g \gpH(V, \psi) = \gpH(V, \psi)$ or 
$g \gpH(V, \psi) = - {\rm{Id}}_{V} \cdot \gpH(V, \psi)$. 
\medskip

\noindent
(c) This follows immediately from (b).
\medskip

\noindent
(d) This follows immediately from (c) and Lemma \ref{for n odd DH = H} below.
\end{proof}
 
\begin{lemma}\label{G connected implies G0 connected} Let $L$ be a subfield of $\C$. Let $V$ be a finite dimensional $L$-vector space.
Let $G \subset GL(V)$ be a connected algebraic group defined over $L$ and $\dim \, G > 0. $  For a 
nontrivial character $\pi\colon G \rightarrow \G_m$ over $L$, put $G_{0} := {\rm{Ker}} \, \pi$. Suppose that there exists a cocharacter 
$s\colon \G_{m} (\C) \rightarrow G (\C)$ that splits $\pi$ in the following exact sequence: 
\begin{equation}
1 \,\, {\stackrel{}{\longrightarrow}} \,\, G_{0} (\C) \,\, {\stackrel{}{\longrightarrow}} \,\,  
G (\C)\,\, {\stackrel{\pi}{\longrightarrow}} \,\, \G_m (\C)\,\, {\stackrel{}{\longrightarrow}} \,\, 1.
\label{exact sequence for G and G0}
\end{equation} 
Then $G_{0}$ is connected.
\end{lemma}
\begin{proof} 
Observe that $G (\C)$ is a connected, complex Lie group. Take any two points 
$g_0$ and $g_1$ in $G_{0} (\C)$. There is a path $\alpha (t) \in 
G (\C)$ connecting $g_0$ and $g_1$, i.e., $\alpha (0) = g_0$ and
$\alpha (1) = g_1$. Define a new path 
$$
\beta (t) := {s (\pi (\alpha (t)))}^{-1} \alpha (t) \in G_{0} (\C).
$$
Notice that $\beta (t) \in G_{0} (\C)$ because 
$$
\pi (\beta (t)) := \pi ({s (\pi (\alpha (t)))}^{-1}) \pi (\alpha (t)) = 
\pi (\alpha (t))^{-1} \pi (\alpha (t)) = 1.
$$ 
Also observe that $\beta (0) = g_0$ and $\beta (1) = g_1$. 
Hence $\beta (t)$ connects $g_0$ and $g_1$ in $G_{0} (\C)$.
It follows that $G_{0}(\C)$ is connected, and so $G_0$ is connected. 
\end{proof}

\begin{lemma} 
\label{for n odd DH = H}
For $n$  odd we have
$$
\gpDH(V, \psi) = \gpH(V, \psi).
$$ 
\end{lemma}
\begin{proof}
By the definition of the Mumford--Tate group, we have a cocharacter
$$
\mu_{\infty, V}\colon \G_m (\C) \rightarrow  \MT(V, \psi)(\C)
$$
such that for all $z \in \G_m(\C)$, $\mu_{\infty, V}(z)$ acts by multiplication by $z^{-p}$ on the subspace 
$V^{p, n-p}$. It follows from the definition of a polarization of a Hodge structure that 
$\chi \circ \mu_{\infty, V} (z) = z^{-n}$ for every $z \in \G_m (\C)$. Because 
$\G_{m} {\rm{Id}}_{V} \subset \MT(V, \psi)$ we have the diagonal cocharacter:
\begin{gather*}
w\colon \G_m (\C) \rightarrow  \MT(V, \psi) (\C),\\
w(z) = z \, \rm{Id}_{V_{\C}}.
\end{gather*}
Because $\chi (w(z)) = z^2$ for every $z \in \G_m (\C)$ (\cite[cf. Remark 2.4]{BK2}) and $n$ is odd the cocharacter
\begin{gather*}
s\colon \G_m (\C) \rightarrow \MT(V, \psi) (\C) \\
s(z) := \mu_{\infty, V} (z) \, w(z)^{-\frac{n-1}{2}}
\end{gather*}
is a splitting of $\chi$ in the following exact sequence:
$$
1 \,\, {\stackrel{}{\longrightarrow}} \,\, \gpDH(V, \psi)(\C) \,\, {\stackrel{}{\longrightarrow}} \,\,  
\MT(V, \psi) (\C) \,\, {\stackrel{\chi}{\longrightarrow}} \,\, \G_m (\C) \,\, {\stackrel{}{\longrightarrow}} \,\, 1.
$$
Observe that $\MT(V, \psi)(\C)$ is a connected Lie group because $\MT(V, \psi)$
is a connected algebraic group. By Lemma \ref{G connected implies G0 connected},
$\gpDH(V, \psi)$ is connected. Hence the claim of this lemma follows from
Proposition \ref{Hodge realizing conn comp: even case}(b).
\end{proof}

Consider the cocharacter:
\begin{gather*}
{\widetilde{\mu}_{\infty, V}}\colon \G_{m} (\C) \rightarrow \GIso_{(V, \psi)} (\C) \times \G_{m} (\C), \\
{\widetilde{\mu}_{\infty, V}} (z) := (\mu_{\infty, V} (z), z).
\end{gather*}
\begin{definition}
The extended Mumford--Tate group ${\widetilde{\MT}}(V, \psi)$ (cf. \cite[Definition 2.13]{PS}) is the smallest subgroup over $\Q$ of the group scheme $\GIso_{(V, \psi)} \times \G_{m}$ such that
$$
{\widetilde{\mu}_{\infty, V}} (\C^{\times}) \subset {\widetilde{\MT}}(V, \psi) (\C).
$$
There is a natural embedding
\begin{equation}
{\widetilde{\MT}}(V, \psi) \hookrightarrow \MT (V, \psi) \times {\G_m}.
\label{embedding of tilde MT into MT times Gm}
\end{equation} 
Applying projections on the first and second factor of
$\MT (V, \psi) \times {\G_m}$, we obtain natural morphisms  $\pi\colon {\widetilde{\MT}}(V, \psi) 
\rightarrow \MT (V, \psi)$  and  $N\colon {\widetilde{\MT}}(V, \psi) \rightarrow \G_{m}$. 
\end{definition}

\begin{lemma}
There is the following commutative diagram with exact rows and columns. The left vertical 
arrow is the natural one. 

\begin{figure}[H]
\[
\begin{tikzcd}
1 \arrow{r}{} & \,\, \gpH (V, \psi) \arrow{d}[swap]{\pi_1} \arrow{r}{}  & 
{\widetilde{\MT}} (V, \psi)  \arrow{d}[swap]{\pi} \arrow{r}{N} & \G_{m} 
\arrow{d}{x \mapsto x^{-n}}[swap]{} \arrow{r}{} & 1 \\ 
1 \arrow{r}{} & \gpDH (V, \psi)  \arrow{r}{} & 
\MT (V, \psi) \arrow{d}[swap]{} \arrow{r}{\chi} & \G_{m} \arrow{d}[swap]{} 
\arrow{r}{} & 1\\
& & 1 & 1 \\
\end{tikzcd}
\]
\\[-0.8cm]
\caption{}
\label{diagram compatibility of tilde(MT (V, psi)) with MT (V, psi)} 
\end{figure}
\label{commutativity of diagram compatibility of tilde(MT (V, psi)) with MT (V, psi)}
\end{lemma}

\begin{proof}
Considering Diagram \ref{diagram compatibility of tilde(MT (V, psi)) with MT (V, psi)} on $\C$-points, the morphisms $(x \mapsto x^{-n}) \circ N$ and $\chi \circ \pi$ 
are equal when restricted to the subgroup ${\widetilde{\mu}_{\infty, V}} (\G_{m} (\C))$ of 
${\widetilde{\MT}}(V, \psi) (\C)$ because of \cite[(2.9)]{BK2} and the equality $\pi \circ {\widetilde{\mu}_{\infty, V}} =  \mu_{\infty, V}$.
The morphisms in the right square of Diagram 
\ref{diagram compatibility of tilde(MT (V, psi)) with MT (V, psi)} are defined over $\mathbb{Q}$. Hence 
$$ 
\{g \in {\widetilde{\MT}}(V, \psi)\colon  (x \mapsto x^{-n}) \circ N (g) = 
\chi \circ \pi (g)\}
$$ 
is a closed subgroup of ${\widetilde{\MT}}(V, \psi)$ over $\Q$ whose $\C$-points contain 
${\widetilde{\mu}_{\infty, V}} (\G_{m} (\C))$. Hence 
$ \{g \in {\widetilde{\MT}}(V, \psi)\colon (x \mapsto x^{-n}) \circ N (g) = 
\chi \circ \pi (g)\} = {\widetilde{\MT}}(V, \psi)$ by definition of the group 
${\widetilde{\MT}}(V, \psi)$. Hence the right square of Diagram 
\ref{diagram compatibility of tilde(MT (V, psi)) with MT (V, psi)} commutes. The map $\pi_1$ in this diagram
is defined by $\pi$ and clearly makes the left square also commute.
\medskip

The image $\pi ({\widetilde{\MT}}(V, \psi))$ is a closed algebraic subgroup of $\MT (V, \psi)$ over $\Q$
\cite[Prop. B, section 7.4 cf. section 34.2]{Hu} containing $\mu_{\infty, V} (\G_m (\C))$. Hence, by the definition of $\MT (V, \psi)$, the morphism $\pi$ in Diagram \ref{diagram compatibility of tilde(MT (V, psi)) with MT (V, psi)} is an epimorphism. 
\medskip

Observe that the cocharacter ${\widetilde{\mu}_{\infty, V}}$ splits the morphism 
$N$ on $\C$-points. Hence by Lemma \ref{G connected implies G0 connected} the kernel of 
$N$ in Diagram \ref{diagram compatibility of tilde(MT (V, psi)) with MT (V, psi)} is connected. 
\medskip

By the definition of $N$ and \eqref{embedding of tilde MT into MT times Gm}, the kernel of $\pi$ in Diagram \ref{diagram compatibility of tilde(MT (V, psi)) with MT (V, psi)}  is contained in 
${\rm{Id}}_{V} \times \G_{m}$. Hence the restriction of $N$ to the kernel of $\pi$ is a monomorphism. By the commutativity of Diagram \ref{diagram compatibility of tilde(MT (V, psi)) with MT (V, psi)}, the kernel 
of $\pi$ injects into $\mu_{n}$. Hence $\pi_{1}$ is a monomorphism and 
$\dim {\widetilde{\MT}}(V, \psi) = \dim \MT (V, \psi)$. Consequently $\dim {\rm{Ker}} \, N = 
\dim \gpDH (V, \psi)$. This shows that
$$
{\rm{Ker}} \, N = (\gpDH (V, \psi))^{\circ} = \gpH (V, \psi)
$$
as desired.
\end{proof}

Let $K$ be a number field. Fix an algebraic closure $\overline{K}$ of $K$. 
Let $G_K := G(\overline{K}/K)$ denote the absolute Galois group of $K$. 
We assume from now through \S \ref{AST and Tate for families} that $(V, \psi)$ satisfies the following additional 
conditions:
\begin{itemize}
\item[{\bf{(D1)}}] There is a ring $D \subseteq \End_{\mathbb{Q}} (V)$  such that the action of $D$ on
$V \otimes \C$ preserves the Hodge decomposition, i.e. $D V^{p,q} \subseteq V^{p,q}$ for all $p, q$. 
\item[{\bf{(D2)}}] The ring $D$ has a structure of discrete $G_K$-module. 
\end{itemize}
Applying \eqref{GIso = Gm Iso} we obtain
\begin{equation}
C_{D} \GIso_{(V, \psi)} = \G_{m} {\rm{Id}}_{V} \cdot C_{D} \Iso_{(V, \psi)}. 
\label{CD GIso = Gm CD Iso}
\end{equation}

By \textbf{(D1)} $D$ commutes  with $\mu_{\infty, V}(\C^{\times})$ on $V_{\C}$ elementwise, hence the 
properties of the polarization (cf. \cite[p. 4]{BK2}) give:
\begin{equation}
\mu_{\infty, V}(\C^{\times}) \, \subseteq C_{D} \GIso_{(V,\, \psi)} (\C).
\label{muInfty subset CDh GIso}
\end{equation}
By \textbf{(D2)} we obtain the following representation
\begin{equation}
\rho_e\colon G_K \rightarrow \Aut_{\Q} (D).
\label{rho e representation}
\end{equation}
Define $K_{e} \, := \, \overline{K}^{{\rm{Ker}}\rho_e}$ (cf. \cite[Def. 3.1]{BK2}). 
For any $\tau \in \Gal(K_{e}/K)$, put
\begin{equation}
\GIso_{(V, \psi)}^{\tau} := 
\{g \in \GIso_{(V, \psi)}\colon g \beta g^{-1} = 
\rho_{e}(\tau)(\beta) \,\ \forall \beta \in D\}.
\label{def of GSptau}
\end{equation}
Observe that
$$
\bigsqcup_{\tau \in \Gal(K_{e}/K)} \, \GIso_{(V, \psi)}^{\tau} \subseteq \GIso_{(V, \psi)}.
$$

\begin{definition}
\label{Galois twist of the Lefschetz group}
For any $\tau \in \Gal(K_{e}/K)$ put (see \cite[Def. 3.3]{BK2}):
\begin{equation}
\gpDL_{K}^{\tau}(V,\, \psi, D) := \Iso_{(V, \psi)} \, \cap \, \GIso_{(V, \psi)}^{\tau}.
\label{decomposable twisted Lefschetz for fixed element}
\end{equation} 
\end{definition}
\begin{remark}
The group $\gpDL_{K}^{\tau}(V, \psi, D)$ is a closed subscheme of 
$\Iso_{(V,\, \psi)}$ for each $\tau  \in \Gal(K_{e}/K)$ and is defined over $\Q$
because $D \subseteq \End_{\Q} (V)$.
\end{remark}  

\begin{definition}
(\emph{Twisted decomposable algebraic Lefschetz group} \cite[Def. 3.4]{BK2}) 
\label{twisted decomposable Lefschetz group}
For the triple 
$(V,\, \psi, D)$ define the closed algebraic subgroup of $\Iso_{(V,\, \psi)}$ given by
\begin{equation}
\gpDL_{K}(V, \psi, D) \,\,\, := \bigsqcup_{\tau \in \Gal(K_{e}/K)} \,\, \gpDL_{K}^{\tau}(V, \psi, D).
\label{decomposition into twisted Lefschetz for fixed elements}
\end{equation} 
\end{definition}

From Definition \ref{twisted decomposable Lefschetz group}, there is the following natural monomorphism:
\begin{equation}
 \gpDL_{K} (V, \psi, D) / \gpDL_{K}^{\id}(V, \psi, D)  
\,\, \hookrightarrow \,\, \Gal(K_e/K).
\label{Twisted decomposable embeds into GKeK}
\end{equation}

By \eqref{muInfty subset CDh GIso}, \eqref{def of GSptau}, \eqref{decomposable twisted Lefschetz for fixed element}, and the definition of $\MT(V, \psi)$ we obtain (cf. \cite[(2.16), (2.17), (2.18)]{BK2}): 
\begin{gather}
\MT(V, \psi) \subseteq C_{D}(\GIso_{(V,\, \psi)}) = \GIso_{(V,\, \psi)}^{{\rm{id}}},
\label{MT subset CD GIso} \\
\gpDH(V, \psi) \subseteq C_{D}(\Iso_{(V,\, \psi)}) = \gpDL_{K}^{{\rm{id}}}(V, \psi, D) \subseteq 
\gpDL_{K}(V, \psi, D).
\label{DH and CD}
\end{gather}

\begin{corollary}\label{DH = H if H = CDIso}
Assume that $\gpH(V, \psi) = C_{D} (\Iso_{(V, \psi)})$. Then
$$
\gpDH(V, \psi) = \gpH(V, \psi).
\label{DH = H}
$$
\end{corollary}
\begin{proof} This follows from \eqref{DH and CD} and the definition of $\gpH(V, \psi)$.
\end{proof}

\begin{remark}
If $A$ is an abelian variety over $\overline{\Q}$ of dimension 4 considered by Mumford \cite{M} and $V = H^1(A(\C), \Q)$, then $- {\rm{Id}}_{V} \in \gpH(V, \psi)$ but the condition $\gpH(V, \psi) = C_{D} (\Iso_{(V, \psi)})$
fails in this case.
\end{remark}   

\begin{remark}
See \cite[Chapter 3]{BK2} for further properties of the twisted decomposable algebraic 
Lefschetz group.
\end{remark}

\begin{definition} The {\em{Betti parity group}}:
\begin{equation}
\gpP(V, \psi) := \gpDH(V, \psi) / 
\gpH(V, \psi),
\end{equation}
is the component group of $\gpDH(V, \psi)$.
\label{The Betti parity group}
\end{definition}

\begin{proposition}
\label{P(V, psi) trivial for n odd or H = CD}
If $n$ is odd or if $\gpH(V, \psi) = C_{D} (\Iso_{(V, \psi)})$ then
$\gpP(V, \psi)$ is trivial. 
\end{proposition}
\begin{proof} 
This follows from Lemma \ref{for n odd DH = H} and Corollary 
\ref{DH = H if H = CDIso}.
\end{proof}

\begin{proposition}
If $n$ is even and ${\rm dim}_{\Q}\, V$ is odd then
$\gpP(V, \psi)$ is nontrivial. 
\label{P(V, psi) nontrivial for n even and dim V odd}
\end{proposition}
\begin{proof} Because $n$ is even, $\Iso_{(V, \psi)} = \gpO_{(V, \psi)}$. Hence 
$\gpH(V, \psi) \subseteq \Iso_{(V, \psi)} \subseteq {\pm {\rm{Id}}_{V}} \SL_{V}$.  It follows that $- {\rm{Id}}_{V} \notin \gpH(V, \psi)$ because 
${\rm{det}} (- {\rm{Id}}_{V}) = (-1)^{{\rm{dim}}_{\Q}\, V} = -1$ and $\gpH(V, \psi)$ is connected.
\end{proof}

\begin{example} Let $A$ be an abelian threefold over $K$. Let
$V :=  H^{2} (A(\C), \, \Q) = \bigwedge^{2} \, H^{1} (A(\C), \, \Q)$. Then $V$ has weight 2 and 
${\rm dim}_{\Q} \, V = {\binom{6}{2}} = 15$. Hence $\gpP(V, \psi)$ is nontrivial. 
\label{An example of nontrivial Betti parity group}
\end{example}

\begin{example} Let $A$ be an elliptic curve over $K$. Let
$V := \Sym^{2} \, H^{1} (A(\C), \, \Q)$. Then $V$ has weight 2 and 
${\rm dim}_{\Q} \, V = 3$. Hence again $\gpP(V, \psi)$ is nontrivial. 
\label{Another example of nontrivial Betti parity group}
\end{example}

\section{De Rham structures associated with Hodge structures}
\label{de Rham structures associated with Hodge structures}

In the following discussion, let $\sigma\colon K \hookrightarrow \C$ denote an arbitrary field embedding.
Suppose that we are given a $K$-vector space $V_{_{\rm{DR}}}$, a bilinear, nondegenerate pairing
\begin{equation}
\psi_{_{\rm{DR}}}\colon V_{_{\rm{DR}}} \times V_{_{\rm{DR}}} \rightarrow K(-n),
\label{bilinear form psi DR}\end{equation}  
and a decreasing filtration $F^{i} V_{_{\rm{DR}}}$. Suppose in addition that for each $\sigma$ there are isomorphisms: 
\begin{equation}
V \otimes_{\Q} \C \simeq V_{_{\rm{DR}}} \otimes_{K, \sigma} \C. 
\label{comparison VC and VDRC}
\end{equation}
\begin{equation}
\, \psi \otimes_{\Q} \C \, \simeq \, \psi_{_{\rm{DR}}} \otimes_{K, \sigma} \C.
\label{comparison psi for VC and VDRC}\end{equation}
and the vector space $V_{_{\rm{DR}}} \otimes_{\sigma} \C$ has a decreasing filtration, induced 
from $V_{_{\rm{DR}}}$, compatible with the Hodge filtration on $V \otimes_{\Q} \C$ via 
\eqref{comparison VC and VDRC}. 
\medskip

Setting $G^{i}V_{_{\rm{DR}}} := F^{i} V_{_{\rm{DR}}} / F^{i+1} V_{_{\rm{DR}}}$, we obtain an isomorphism: 
\begin{equation}
V^{i, n-i} \,\, \simeq \,\,  G^{i}V_{_{\rm{DR}}} \otimes_{K, \sigma} \C
\label{comparison Vi(n-i) and Gi}
\end{equation}
and a Hodge decomposition:
\begin{equation}
V \otimes_{\Q} \C = \bigoplus_{i = 0}^{n} \, G^{i}V_{_{\rm{DR}}} \otimes_{K, \sigma} \C.
\label{Hodge-De Rham decomposition}
\end{equation}
For the space $(V_{_{\rm{DR}}}, \psi_{_{\rm{DR}}})$ we define: 
{\small{
\begin{align*}
\GIso_{(V_{_{\rm{DR}}}, \psi_{_{\rm{DR}}})} & := \{g \in \GL_{V_{_{\rm{DR}}}}\colon \exists \, \chi_{_{\rm{DR}}} (g) \in \G_m \,\, s.t. \,\, \psi_{_{\rm{DR}}} (gv, gw) = 
\chi_{_{\rm{DR}}} (g) \psi_{_{\rm{DR}}} (v, w) \,\,\, \forall \, v, w \in V_{_{\rm{DR}}}\},
\\
\Iso_{(V_{_{\rm{DR}}}, \psi_{_{\rm{DR}}})} & := \{g \in \GL_{V_{_{\rm{DR}}}}\colon \psi_{_{\rm{DR}}} (gv, gw) =  
\psi_{_{\rm{DR}}} (v, w) \,\,\,\,  \forall \, v, w \in V_{_{\rm{DR}}}\}.
\end{align*}}}
 
We observe that $\chi_{_{\rm{DR}}}$ becomes a character of $\GIso_{(V_{_{\rm{DR}}}, \psi_{_{\rm{DR}}})}$. 
Moreover $\G_{m, K} {\rm{Id}}_{V_{_{\rm{DR}}}}\subseteq \GIso_{(V_{_{\rm{DR}}}, \psi_{_{\rm{DR}}})}$ and
\begin{equation}
\GIso_{(V_{_{\rm{DR}}}, \psi_{_{\rm{DR}}})} = \G_{m} {\rm{Id}}_{V_{_{\rm{DR}}}} \cdot \Iso_{(V_{_{\rm{DR}}}, \psi_{_{\rm{DR}}})}. 
\label{GIsoDR = Gm IsoDR}
\end{equation}
For every $\alpha \in \G_{m, K}$  we obtain
$$
\psi_{_{\rm{DR}}} (\alpha \, {{\rm Id}}_{V} v,\,  \alpha \, {{\rm Id}}_{V} w) = \psi_{_{\rm{DR}}} (\alpha v, \, \alpha w) =
\alpha^2 \, \psi_{_{\rm{DR}}} (v, \, w).
$$

\noindent
Hence
$$
\alpha \, {{\rm Id}}_{V_{_{\rm{DR}}}} \in 
\GIso_{(V_{_{\rm{DR}}}, \psi_{_{\rm{DR}}})}
\quad {{\rm and}} \quad \chi_{_{\rm{DR}}} (\alpha \, {{\rm Id}}_{V_{_{\rm{DR}}}}) = \alpha^2.
$$
Define  
$\MT(V_{_{\rm{DR}}}, \psi_{_{\rm{DR}}})$ and $\gpH(V_{_{\rm{DR}}}, \psi_{_{\rm{DR}}})$ by analogy with the
corresponding Mumford--Tate and Hodge groups $\MT(V, \psi)$ and $\gpH(V, \psi)$. Namely, consider the 
cocharacter
\begin{equation}
\mu_{_{\rm{DR}}, V}\colon \mathbb{G}_{m, K}({\C})\rightarrow \GL(V_{_{\rm{DR}}} \otimes_{K, \sigma} \C)
\simeq \GL(V_{\C})
\label{conjugate cocharacter}
\end{equation}
such that for any $z\in {\C}^{\times}$, the automorphism $\mu_{_{\rm{DR}}, V}(z)$ acts as 
multiplication by ${z}^{-p}$ on $V^{p,q}$ for each $p+q = n$ (cf. \eqref{comparison Vi(n-i) and Gi}). 
Then we define 
$\MT(V_{_{\rm{DR}}}, \psi_{_{\rm{DR}}})$ to be the smallest algebraic subgroup of 
$\GL_{V_{_{\rm{DR}}}}$ over $K$ containing $\mu_{_{\rm{DR}}, V} (\C)$. Define:
\begin{align*}
\gpDH(V_{_{\rm{DR}}}, \psi_{_{\rm{DR}}}) &:= \MT(V_{_{\rm{DR}}}, \psi_{_{\rm{DR}}}) \cap 
\Iso_{(V_{_{\rm{DR}}}, \psi_{_{\rm{DR}}})}; \\
\gpH(V_{_{\rm{DR}}}, \psi_{_{\rm{DR}}}) &:= \gpDH(V_{_{\rm{DR}}}, \psi_{_{\rm{DR}}})^{\circ}.
\end{align*}

\begin{proposition} \label{De Rham realizing conn comp: even case} 
The algebraic group $\gpDH(V_{_{\rm{DR}}}, \psi_{_{\rm{DR}}})$ has the following properties:
\begin{itemize}
\item[(a)] $\G_{m} {\rm{Id}}_{V_{_{\rm{DR}}}} \cdot \gpH(V_{_{\rm{DR}}}, \psi_{_{\rm{DR}}}) \, = \, 
\G_{m} {\rm{Id}}_{V_{_{\rm{DR}}}} \cdot \gpDH(V_{_{\rm{DR}}}, \psi_{_{\rm{DR}}}) \, = 
\, \MT(V_{_{\rm{DR}}}, \psi_{_{\rm{DR}}})$.
\item[(b)] $\gpDH(V_{_{\rm{DR}}}, \psi_{_{\rm{DR}}}) = \gpH(V_{_{\rm{DR}}}, \psi_{_{\rm{DR}}})  \, \cup \, 
- {\rm{Id}}_{V_{_{\rm{DR}}}} \cdot \gpH(V_{_{\rm{DR}}}, \psi_{_{\rm{DR}}}) $.
\item[(c)] $- {\rm{Id}}_{V_{_{\rm{DR}}}} \in \gpH(V_{_{\rm{DR}}}, \psi_{_{\rm{DR}}})$ iff 
$\gpH(V_{_{\rm{DR}}}, \psi_{_{\rm{DR}}}) = \gpDH(V_{_{\rm{DR}}}, \psi_{_{\rm{DR}}}) $.
\item[(d)] When $n$ is odd then $- {\rm{Id}}_{V_{_{\rm{DR}}}} \, \in \gpH(V_{_{\rm{DR}}}, \psi_{_{\rm{DR}}}) = 
\gpDH(V_{_{\rm{DR}}}, \psi_{_{\rm{DR}}})$.  
\end{itemize} 
\end{proposition} 
\begin{proof}
Similar to the proof of Proposition \ref{Hodge realizing conn comp: even case}.
\end{proof}

\begin{lemma} 
\label{for n odd DH-DR = H-DR}
For $n$ odd,
$$
\gpDH(V_{_{\rm{DR}}}, \psi_{_{\rm{DR}}}) = \gpH(V_{_{\rm{DR}}}, \psi_{_{\rm{DR}}}).
$$ 
\end{lemma}
\begin{proof}
Similar to the proof of Lemma \ref{for n odd DH = H}.
\end{proof}

\begin{lemma} 
\label{comparison MT, DH and H for DR and Betti}
We have the following inclusions:
\begin{align}
\MT(V_{_{\rm{DR}}}, \psi_{_{\rm{DR}}}) &\subset \MT(V, \psi)_{K},
\label{comparison of MT for DR and Betti} \\
\gpDH(V_{_{\rm{DR}}}, \psi_{_{\rm{DR}}}) &\subset \gpDH(V, \psi)_{K},
\label{comparison of DH for DR and Betti} \\
\gpH(V_{_{\rm{DR}}}, \psi_{_{\rm{DR}}}) &\subset \gpH(V, \psi)_{K}.
\label{comparison of H for DR and Betti}
\end{align}
\end{lemma}

\begin{proof}
By the isomorphisms \eqref{comparison VC and VDRC}, \eqref{comparison psi for VC and VDRC}
and the definition of $\MT(V_{_{\rm{DR}}}, \psi_{_{\rm{DR}}})$, we obtain
\eqref{comparison of MT for DR and Betti}. Consequently we also get 
\eqref{comparison of H for DR and Betti}. Now comparing  Proposition \ref{Hodge realizing conn comp: even case}(b) 
with Proposition \ref{De Rham realizing conn comp: even case}(b) 
we obtain \eqref{comparison of DH for DR and Betti}.
\end{proof}

We assume from now through \S \ref{AST and Tate for families} that $(V_{_{\rm{DR}}}, \psi_{_{\rm{DR}}})$ satisfies the following additional conditions:
\medskip

\begin{itemize}
\item[{\bf{(DR1)}}] 
There is a canonical embedding $D \subseteq {{\rm End}}_{\overline{K}} (\overline{V}_{_{\rm{DR}}})$  
compatible with isomorphisms \eqref{comparison VC and VDRC} and 
\eqref{comparison psi for VC and VDRC}, where 
$\overline{V}_{_{\rm{DR}}} := V_{_{\rm{DR}}} \otimes_{K} \overline{K}$
with $\overline{K}$ being the integral closure of $K$ via the chosen embedding in $\C$.
\item[{\bf{(DR2)}}] The action of $D$ preserves the Hodge decomposition 
\eqref{Hodge-De Rham decomposition}. 
\end{itemize}
\medskip

By \textbf{(DR1)} and \textbf{(DR2)},
$D$ commutes  with $\mu_{\infty, V}(\C^{\times})$ on $V_{_{\rm{DR}}} \otimes_{K, \sigma} \C$ 
elementwise. By \eqref{comparison VC and VDRC}--\eqref{Hodge-De Rham decomposition} and 
\cite[(2.9)]{BK2} we obtain
\begin{equation}
\mu_{_{\rm{DR}}, V}(\C^{\times}) \, \subseteq C_{D} \GIso_{(V_{_{\rm{DR}}}, \psi_{_{\rm{DR}}})} (\C).
\label{muInfty subset CDh-DR GIso-DR}
\end{equation}

In the De Rham structures we also consider the following subring $D_e \subset D$:
\begin{equation}
D_e := \{d \in D\colon d \in \End_{K} (V_{_{\rm{DR}}})\}.
\label{Definition of Dr}
\end{equation}
\begin{remark} Observe that $D_{e} = D^{G(K_e/K)}$. Hence if $K = K_{e}$ then $D_{e} = D$.  
\label{condition for De = D}
\end{remark}
Obviously $D_e$ is compatible with the isomorphisms \eqref{comparison VC and VDRC} and 
\eqref{comparison psi for VC and VDRC}, and from \eqref{muInfty subset CDh-DR GIso-DR} we obtain:
\begin{equation}
\mu_{_{\rm{DR}}, V}(\C^{\times}) \, \subseteq C_{D_e} \GIso_{(V_{_{\rm{DR}}}, \psi_{_{\rm{DR}}})} (\C).
\label{muInfty subset CDeh-DR GIso-DR}
\end{equation}

\noindent
Hence
\begin{equation}
\MT(V_{_{\rm{DR}}}, \psi_{_{\rm{DR}}}) \subseteq C_{D_e}(\GIso_{(V_{_{\rm{DR}}},\, \psi_{_{\rm{DR}}})}) 
\label{MT-DR subset CD GIso-DR}
\end{equation}
\begin{equation}
\gpDH(V_{_{\rm{DR}}}, \psi_{_{\rm{DR}}}) \subseteq C_{D_e}(\Iso_{(V_{_{\rm{DR}}},\, \psi_{_{\rm{DR}}})}).
\label{DH-DR and CD-DR}\end{equation}

\begin{corollary}\label{DH-DR = H-DR if H-DR = CDIso-DR}
Assume that $\gpH(V_{_{\rm{DR}}}, \psi_{_{\rm{DR}}}) = C_{D_e} (\Iso_{(V_{_{\rm{DR}}}, \psi)})$. Then
\begin{equation}
\gpDH(V_{_{\rm{DR}}}, \psi_{_{\rm{DR}}}) = \gpH(V_{_{\rm{DR}}}, \psi_{_{\rm{DR}}}),
\label{DH-DR = H-DR}
\end{equation}
\end{corollary}
\begin{proof} This follows from \eqref{DH-DR and CD-DR} and the definition of 
$\gpH(V_{_{\rm{DR}}}, \psi_{_{\rm{DR}}})$.
\end{proof}

\begin{definition} The {\em{De Rham parity group}}:
\begin{equation}
\gpP(V_{_{\rm{DR}}}, \psi_{_{\rm{DR}}}) := \gpDH(V_{_{\rm{DR}}}, \psi_{_{\rm{DR}}}) / 
\gpH(V_{_{\rm{DR}}}, \psi_{_{\rm{DR}}}),
\end{equation}
is the component group of $\gpDH(V_{_{\rm{DR}}}, \psi_{_{\rm{DR}}})$.
\label{De Rham parity group}
\end{definition}

\begin{proposition}
If $n$ is odd or $\gpH(V_{_{\rm{DR}}}, \psi_{_{\rm{DR}}}) = 
C_{D_e} (\Iso_{(V_{_{\rm{DR}}}, \psi_{_{\rm{DR}}})})$, then
$\gpP(V_{_{\rm{DR}}}, \psi_{_{\rm{DR}}})$ is trivial. 
\label{P(V-DR, psi) trivial for n odd or H-DR = CD}
\end{proposition}
\begin{proof}
This follows from Lemma \ref{for n odd DH-DR = H-DR}
 and Corollary \ref{DH-DR = H-DR if H-DR = CDIso-DR}.
\end{proof}

\begin{corollary}
\label{Corollary-epimorphism between DR parity and Betti parity groups}
There is a natural epimorphism:
\begin{equation}
\gpP(V_{_{\rm{DR}}}, \psi_{_{\rm{DR}}}) \rightarrow \gpP(V, \psi).
\label{epimorphism between DR parity and Betti parity groups}
\end{equation}
\end{corollary}
\begin{proof}
It follows from Lemma \ref{comparison MT, DH and H for DR and Betti}. 
\end{proof}

\begin{corollary} If $n$ is even and ${\rm dim}_{\Q}\, V$ is odd then
$\gpP(V_{_{\rm{DR}}}, \psi_{_{\rm{DR}}})$ is nontrivial. 
\label{gpP(VDR, psiDR) nontrivial for n even and dim V odd}
\end{corollary}
\begin{proof} 
It follows from Proposition \ref{P(V, psi) nontrivial for n even and dim V odd} and 
Corollary \ref{Corollary-epimorphism between DR parity and Betti parity groups}.
\end{proof}

\begin{example} 
\label{An example of De Rham nontrivial  parity group}
Let $A$ be an abelian threefold over $K$. Put $V_{_{\rm{DR}}} :=  H^{2}_{_{\rm{DR}}} (A)$.
Then by Example \ref{An example of nontrivial Betti parity group} and Corollary
\ref{gpP(VDR, psiDR) nontrivial for n even and dim V odd}, the group
$\gpP(V_{_{\rm{DR}}}, \psi_{_{\rm{DR}}})$ is nontrivial in this case. 
\end{example}

\begin{example} Let $A$ be an elliptic curve over $K$. Put 
$V_{_{\rm{DR}}} := \Sym^{2} \, H^{1}_{_{\rm{DR}}} (A)$. 
Then by Example \ref{Another example of nontrivial Betti parity group} and Corollary
\ref{gpP(VDR, psiDR) nontrivial for n even and dim V odd}, the group
$\gpP(V_{_{\rm{DR}}}, \psi_{_{\rm{DR}}})$ is again nontrivial.
\label{Another example of nontrivial De Rham parity group}
\end{example}

\section{Families of \texorpdfstring{$l$}{l}-adic representations associated with Hodge structures}
\label{families of l-adic representations associated with Hodge structures}

\medskip 

For $(V, \, \psi)$ as in \S \ref{Mumford--Tate groups of polarized Hodge structures} and $l$ a prime number, put $V_l := V \otimes_{\mathbb{Q}} {\mathbb Q}_l$, \, $\psi_{l} := \psi \otimes_{\Q} \Q_l$, and 
$(V_l, \psi_{l}) := (V \otimes_{\Q} \Q_l, \psi \otimes_{\Q} \Q_l)$. Assume that 
the bilinear form $\psi_l\colon V_l \times V_l \rightarrow \Q_l(-n)$  is $G_K$-equivariant. Let 
\begin{equation}
\rho_l\colon G_K \rightarrow \GIso (V_l, \psi_{l})
\label{the family of l-adic representations}
\end{equation}
be the corresponding $l$-adic representation. Observe that:

\begin{align*}
\GIso_{(V_l, \psi_{l})} \,\, &= \,\, 
\left\{
\begin{array}{lll}
\GO_{(V_l, \psi_{l})}&\rm{if}&n \,\,\, \rm{even,}\\
\GSp_{(V_l, \psi_{l})}&\rm{if}&n \,\,\, \rm{odd;}\\
\end{array}
\right.
\\
\Iso_{(V, \psi)} \,\, &= \,\,
\left\{
\begin{array}{lll}
\gpO_{(V_l, \psi_{l})}&\rm{if}&n \,\,\, \rm{even,}\\
\Sp_{(V_l, \psi_{l})}&\rm{if}&n \,\,\, \rm{odd.}\\
\end{array}
\right.
\end{align*}

Let $\chi_{c}\colon G_K \rightarrow \Z_{l}^{\times}$ be the cyclotomic character. Then by the $G_K$-equivariance of 
$\psi_l$, we obtain: 
\begin{equation}
\chi \circ \rho_{l}  \, = \, \chi_{c}^{-n}.
\label{compatibility of chi with chi-cycl.}
\end{equation}

We assume from now through \S \ref{AST and Tate for families} that the family of $l$-adic representations \eqref{the family of l-adic representations} satisfy the following additional, natural conditions:

\begin{itemize}
\item[{\bf{(R1)}}] The family $(\rho_l)_{l}$ is strictly compatible in the sense of Serre \cite[Chapter I, \S 2.3]{Se1}, and $\rho_l$ is of Hodge--Tate type 
for every prime of $\mathcal{O}_K$ over $l$.

\item[{\bf{(R2)}}] 
For each $\rho_l$, for each prime $v$ of $\mathcal{O}_K$ outside of a finite set of primes 
containing all primes over $l$, the complex absolute values of the eigenvalues of a Frobenius at $v$ are 
$q_{v}^{- \frac{n}{2}}$.

\item[{\bf{(R3)}}] Let $\lambda$ be any prime in $\mathcal{O}_K$ over $l$ and put $\C_l := {\widehat{\overline{K}}}_{\lambda}$.
For the $K$-vector space $V_{_{\rm{DR}}}$ with corresponding structure
and the gradation $G^{i}V_{_{\rm{DR}}}$ as in \S \ref{Mumford--Tate groups of polarized Hodge structures}, there is a corresponding Hodge--Tate decomposition:
\begin{equation}
V_{l} \otimes_{{\Q}_l} \C_{l} = \bigoplus_{i= 0}^{n} \, G^{i}V_{_{\rm{DR}}}  \otimes_{K} \C_{l} (i-n).
\label{Hodge--Tate decomposition}
\end{equation}

\item[{\bf{(R4)}}] The induced action of $D$ on $V_l$ is $G_K$-equivariant, i.e., $\forall \beta \in D$, $\forall v_l \in V_l$, and $\forall \sigma \in G_K$,
\begin{equation}
\rho_{l} (\sigma) (\beta \, v_l) = \sigma(\beta) \, \rho_{l} (\sigma) (v_l).
\label{GK equivariant action of D on Vl}\end{equation}
Moreover, the induced action of $D$ on $V_l \otimes_{{\Q}_l} \C_{l}$ is compatible with  \eqref{Hodge--Tate decomposition}. 
\end{itemize}

\begin{remark}\label{Hodge--Tate representations in etale cohomology} 
Strictly compatible families of $l$-adic representations of Hodge--Tate type arise naturally from {\' e}tale cohomology.
Indeed, if $X/K$ is a proper scheme and ${\overline X} := X \otimes_{K} {\overline F}$,
then $V_{l, \et}^{n} := H^{n}_{\et} ({\overline X}, \, \Q_l)$ is potentially semistable as a
$G_{K_{\lambda}}$-representation for every $\lambda | l$ (see \cite[Cor.\ 2.2.3]{Ts1}, \cite{Ts2}). 
Hence the representation 
\begin{equation}
\rho_{l, \et}^{n}\colon G_K \rightarrow \GL (V_{l, \et}^{n})
\label{etale cohom. representation}
\end{equation}
is of Hodge--Tate type (cf.\ \cite[p.\ 603]{Su}). So condition {\bf{(R1)}} holds for the family
$(\rho_{l, \et}^{n})$.
\end{remark}

\begin{remark} In \cite[Page 18]{BK2}, when assuming that the complex absolute values of the eigenvalues of a Frobenius element at $v$ are $q_{v}^{\frac{n}{2}}$, we meant the geometric Frobenius. In {\bf{(R2)}} 
above we mean the arithmetic Frobenius. 
\label{Frobenius choice}
\end{remark}

\begin{remark}\label{Hodge and Hodge--Tate representations in etale cohomology} 
The assumption {\bf{(R3)}} holds for families $(\rho_{l, \et}^{n})$ from Remark
\ref{Hodge--Tate representations in etale cohomology} by the following considerations. Let $X/K$ be a smooth projective  variety. For a 
field extension $L/K$, let $X_L := X \otimes_{K} L$.  The $\Q$-vector space $V := 
H^{n} (X(\mathbb{C}), \, \Q)$ admits a pure, polarized Hodge structure of weight $n$: 
$$
V \otimes_{\Q} \, \C = \bigoplus_{i+j = n} \, V^{i,j}
$$ 
where  $V^{i, j} := H^{i}_{{\rm{an}}} (X({\mathbb{C}}), \Omega_{X ({\C})}^j) \simeq H^{i}_{{\rm{zar}}} 
(X_{\mathbb{C}}, \Omega_{X_{\C}}^j)$. There is the following spectral sequence (cf. \cite[p.17]{D1}):
\begin{equation}
E_{1}^{j,i} = H^{i}_{{\rm{zar}}} (X, \Omega_{X}^j) \Rightarrow H^{n}_{DR} (X).  
\label{De Rham cohomology spectral sequence}
\end{equation}

By \cite[Theorem 5.5]{D2} and \cite[III, Prop. 9.3]{Ha} we have the following natural isomorphisms:
\begin{align} 
H^{i} (X_{\mathbb{C}}, \Omega_{X_{\C}}^j) &\simeq H^{i}_{{\rm{zar}}} (X, \Omega_{X}^j) \otimes_{K} \C, 
\label{base change to C in Hodge-De Rham cohom. } \\
H^{i}_{{\rm{zar}}} (X_{K_{\lambda}}, \Omega_{X_{K_{\lambda}}}^j)  &\simeq H^{i}_{{\rm{zar}}} (X, \Omega_{X}^j) \otimes_{K} K_{\lambda}.
\label{base change to Kv in Hodge-De Rham cohom. }
\end{align}

By \cite{Fa} we have the following decomposition:
$$
H^{n}_{et} (\overline{{X_{K_{\lambda}}}}, \,\, \Q_l) \otimes_{\Q_l} \mathbb{C}_l \simeq \bigoplus_{i+j = n} \,
H^{i}_{{\rm{zar}}} (X_{K_{\lambda}}, \Omega_{X_{K_{\lambda}}}^j) \otimes_{K_{\lambda}} \mathbb{C}_l (- j).
$$
Observe that $H^{n}_{et} (\overline{{X_{K}}}, \,\, \Q_l) \simeq H^{n}_{et} (\overline{{X_{K_{\lambda}}}}, \,\, \Q_l)$ (\cite[Corollary 2.6, Chap. VI]{Mi1}) and
\begin{equation} 
h^{i,j} :=  \dim_{\C} \, V^{i,j} = \dim_{K} \, H^{i}_{{\rm{zar}}} (X, \Omega_{X}^j) = 
\dim_{K_{\lambda}} \, H^{i}_{{\rm{zar}}} (X_{K_{\lambda}}, \Omega_{X_{K_{\lambda}}}^j).
\label{hpq numbers}\end{equation}
\end{remark}

\begin{remark}\label{Action od D is GK equivariant} 
{\bf{(R4)}} yields a natural, $G_K$-equivariant embedding 
$D \otimes_{\Q} \, \Q_l \,\, \subseteq \End_{\mathbb{Q}_l} (V_l)$. 
\end{remark}

Recall the following notation from \cite[pp. 18--19]{BK2}.

\begin{definition} Let $L/K$ be a finite extension. Let
\begin{equation}
G_{l, L}^{\alg} := G_{l, L}^{\alg}(V, \psi) \subseteq  \GIso_{(V_{l}, \psi_{l})}
\label{Def of GlLalg}
\end{equation} 
be the Zariski closure of $\rho_{l} (G_L)$ in $\GIso_{(V_{l}, \psi_{l})}$. Put:
\begin{align}
\rho_{l} (G_L)_{1} &:= \rho_{l} (G_L) \cap \Iso_{(V_{l}, \psi_{l})},
\label{Def of rholGL1} \\
G_{l, L, 1}^{\alg} &:= G_{l, L, 1}^{\alg}(V, \psi) :=  G_{l, L}^{\alg} \cap
\Iso_{(V_{l}, \psi_{l})}.
\label{Def of GlL1alg}
\end{align}
\end{definition}

For a representation $\rho_l$ of Hodge--Tate type, the theorem of Bogomolov on homotheties 
(\cite[Th{\' e}or{\` e}me 1]{Bo}; cf.\ \cite[Prop.\ 2.8]{Su}) applies, meaning that $\rho_{l}(G_K) \, \cap\,  \Q_{l}^{\times} \, {\rm{Id}}_{V_l}$ is open in $\Q_{l}^{\times} \, {\rm{Id}}_{V_l}$. Therefore $G_{l, L}^{\alg}$ contains the 
homotheties of $\GL(V_l)$, and there is an exact sequence
\begin{equation}
1 \,\, {\stackrel{}{\longrightarrow}} \,\, G_{l, L, 1}^{\alg} \,\, {\stackrel{}{\longrightarrow}} \,\,  G_{l, L}^{\alg} \,\, {\stackrel{\chi}{\longrightarrow}} \,\, \G_m \,\, {\stackrel{}{\longrightarrow}} \,\, 1.
\label{The exact sequence for GlL1alg and GlLalg}
\end{equation}

\begin{remark}\label{properties of Hodge--Tate representations } 
By the theorem of Bogomolov, $\rho_{l} (G_K)$ is open 
in $G_{l, K}^{\alg}(\Q_l)$.
\end{remark}

\begin{remark} By \eqref{GK equivariant action of D on Vl} $\forall \beta \in D$, $\forall v_l \in V_l$, and $\forall \sigma \in G_K$:
\begin{equation}
\rho_{l} (\sigma) \beta \rho_{l} (\sigma^{-1}) (v_l) = \sigma (\beta) (v_l),
\end{equation}
hence
\begin{equation}
\rho_{l} (G_K)_1 \subseteq \gpDL_K(V, \psi, D) (\Q_l).
\end{equation}
\end{remark}

\begin{definition} 
\label{twisted form of GlK and GlK1}
Put: 
\begin{align*}
(G_{l, K}^{\alg})^{\tau} &:= \{g \in G_{l, K}^{\alg}\colon g \beta g^{-1} = \rho_e(\tau) (\beta) \,\,\,\,
\forall \,\beta \in D\}, \\
(G_{l, K, 1}^{\alg})^{\tau} &:= (G_{l, K}^{\alg})^{\tau} \cap G_{l, K, 1}^{\alg}.
\end{align*}
Observe that by \eqref{decomposable twisted Lefschetz for fixed element}, \eqref{Def of GlLalg},
and \eqref{Def of GlL1alg} we obtain
\begin{equation}
(G_{l, K, 1}^{\alg})^{\tau}  \subseteq \gpDL_{K}^{\tau}(V, \psi, D)_{\Q_l}.
\label{GlK1algtau subset DLKtauVpsiDQl}
\end{equation}
\end{definition}

\begin{lemma} There are natural equalities of group schemes:
\begin{align}
G_{l, K}^{\alg}
 &= \bigsqcup_{\tau \in \Gal(K_{e}/K)} \,\, (G_{l, K}^{\alg})^{\tau},
\label{decomposition of GlKalg into cosets of twisted algebraic closures} \\
G_{l, K, 1}^{\alg}
 &= \bigsqcup_{\tau \in \Gal(K_{e}/K)} \,\, (G_{l, K, 1}^{\alg})^{\tau}.
\label{decomposition of Glk1 into twists}
\end{align} 
\end{lemma}
\begin{proof}
Let $\tilde\tau \in G_K$ be a lift of $\tau \in \Gal(K_{e}/K)$. The coset
$\tilde\tau \, G_{K_{e}}$ does not depend on the lift. The Zariski closure of  
$\rho_{l} (\tilde\tau \, G_{K_{e}}) = \rho_{l} (\tilde\tau) \, \rho_{l} (G_{K_{e}})$ 
in $\GIso_{(V_l, \psi_l)}$ is  $\rho_{l} (\tilde\tau) \, G_{l, K_{e}}^{\alg}$.  
Since $\rho_{l} (\tilde\tau) \, \rho_{l} (G_{K_{e}}) \subseteq (G_{l, K}^{\alg})^{\tau}$ then 
$\rho_{l} (\tilde\tau) \, G_{l, K_{e}}^{\alg} \subseteq (G_{l, K}^{\alg})^{\tau}$. 
Because
\begin{equation}
\label{Im rhol GK decomposed cosets}
\rho_{l} (G_K)
 = \bigsqcup_{\tau \in \Gal(K_{e}/K)} \, \, 
 \rho_{l} (\tilde\tau) \, \rho_{l} (G_{K_{e}}),
\end{equation}
we then have
\begin{equation}
G_{l, K}^{\alg}
 = \bigsqcup_{\tau \in \Gal(K_{e}/K)} \,\, \rho_{l} (\tilde\tau) \, G_{l, K_{e}}^{\alg}.
\label{decomposition of GlKalg into cosets of algebraic closures}
\end{equation} 

\noindent
This implies the equalities 
\eqref{decomposition of GlKalg into cosets of twisted algebraic closures} and
\eqref{decomposition of Glk1 into twists}.
\end{proof}

\noindent
\begin{remark}
In the proof of \eqref{decomposition of Glk1 into twists} we observe that
$\rho_{l} (\tilde\tau) \, G_{l, K_{e}}^{\alg} = (G_{l, K}^{\alg})^{\tau}$  for all $\tau$. 
Hence we obtain $(G_{l, K, 1}^{\alg})^{\id} = G_{l, K_e, 1}^{\alg}$ 
and a natural isomorphism:
\begin{equation}
G_{l, K}^{\alg}/ (G_{l, K}^{\alg})^{\id} \simeq \Gal(K_e/K).
\label{GlK1alg mod GlK1algid = GLeK}
\end{equation}
\end{remark}

\begin{corollary} 
\label{Corollary GlL1alg subset DLLA}
We have
\begin{equation}
G_{l, K, 1}^{\alg} \subseteq \gpDL_K(V, \psi, D)_{\Q_l}.
\label{GlL1alg subset DLLA}
\end{equation}
\end{corollary}

\begin{proof}
This follows from \eqref{decomposition into twisted Lefschetz for fixed elements}, 
\eqref{GlK1algtau subset DLKtauVpsiDQl} and \eqref{decomposition of Glk1 into twists}.
\end{proof}

\noindent
Because $\gpDL_{K}^{\id}(V, \psi, D) = \gpDL_{K_{e}}(V, \psi, D)$ and 
$(G_{l, K, 1}^{\alg})^{\id} = G_{l, K_e, 1}^{\alg}$ we obtain
\begin{equation}
G_{l, K_{e}, 1}^{\alg} \subseteq \gpDL_{K_{e}}(V, \psi, D)_{\Q_l}.
\label{GlLe1alg subset DLLeA}
\end{equation} 

\begin{remark}
We are interested in triples $(V, \psi, D)$ for which equality in \eqref{GlL1alg subset DLLA} holds 
for each $l$. In such cases, Conjecture \ref{general algebraic Sato Tate conj.} below also holds, i.e., 
$\AST_{K}(V, \psi) = \gpDL_K(V, \psi, D)$. In Corollary \ref{connected components iso for DL}, we prove that equality in \eqref{GlL1alg subset DLLA}
is equivalent to equality in \eqref{GlLe1alg subset DLLeA}.
\end{remark}
\medskip

\section{The identity connected component of \texorpdfstring{$G_{l, K, 1}^{\alg}$}{GlK1alg}}
\label{identity connected component of GlK1alg}

In this section we extend our results from \cite[Chapters 2--8]{BK2} for odd weight to any weight, for families of 
strictly compatible $l$-adic representations of Hodge--Tate type associated with pure Hodge structures. 
In \cite[Section 4]{BK2} we proved the following theorem:

\begin{theorem} \label{L0realizing conn comp for GlK alg} Let $L/K$ be finite Galois. The following natural map
is an isomorphism of finite groups:
\begin{equation}
i_{L/K}\colon G_{l, K, 1}^{\alg} / G_{l, L, 1}^{\alg} \quad {\stackrel{\simeq}{\longrightarrow}} \quad 
G_{l, K}^{\alg} / {G_{l, L}^{\alg}}.
\label{L over K gives iso of GlK1 alg mod GlL1 alg with GlK alg mod GlL alg}
\end{equation} 
In particular there are the following isomorphisms:
\begin{equation}
(G_{l, K, 1}^{\alg})^{\circ} = (G_{l, L, 1}^{\alg})^{\circ}  \quad {\rm{and}} \quad 
(G_{l, K}^{\alg})^{\circ} =  ({G_{l, L}^{\alg}})^{\circ}
\label{GlK1 alg circ equals GlL1 alg circ and GlK alg circ equals GlL alg circ}
\end{equation} 
\end{theorem} 
\begin{proof} See \cite[Theorem 4.6]{BK2}.
\end{proof}

\begin{corollary} 
\label{Corollary-GK1alg mod GK1algID with resp. to DL mod DLId}
There are natural isomorphisms:
\begin{equation}
G_{l, K, 1}^{\alg}/ (G_{l, K, 1}^{\alg})^{\id} \, \simeq \,
\gpDL_{K}(V, \psi, D)/ \gpDL_{K}^{\id}(V, \psi, D) \, \simeq \, \Gal(K_e/K).
\label{GK1alg mod GK1algID with resp. to DL mod DLId}\end{equation}
\end{corollary}
\begin{proof}
This follows by \eqref{GlK1algtau subset DLKtauVpsiDQl},  \eqref{GlK1alg mod GlK1algid = GLeK}, and Theorem 
\ref{L0realizing conn comp for GlK alg}.
\end{proof}

\begin{corollary}\label{connected components iso for DL} The equality in \eqref{GlL1alg subset DLLA}
holds if and only if it holds in \eqref{GlLe1alg subset DLLeA}.
\end{corollary}
\begin{proof} This follows from \eqref{GK1alg mod GK1algID with resp. to DL mod DLId} and the equality $(G_{l, K, 1}^{\alg})^{\id} = G_{l, K_e, 1}^{\alg}$.
\end{proof}

\begin{theorem} \label{L0realizing conn comp of GlKalg} Let $L_{0}/K$ be a finite Galois extension such that
$G_{l, L_{0}}^{\alg}$ is connected. There is the following exact sequence: 
 \begin{equation}
1 \,\, {\stackrel{}{\longrightarrow}} \,\, G_{l, L_{0},1}^{\alg} / (G_{l, K, 1}^{\alg})^{\circ}
\,\, {\stackrel{}{\longrightarrow}} \,\,  \pi_{0} (G_{l, K, 1}^{\alg})
 \,\, {\stackrel{i_{CC}}{\longrightarrow}} \,\, \pi_{0} (G_{l, K}^{\alg}) \,\, {\stackrel{}{\longrightarrow}} \,\, 1.
\label{icc exact sequence}\end{equation} 
Moreover $i_{CC}$ is an isomorphism if and only if $G_{l, L_{0},1}^{\alg}$ is connected. 
\end{theorem} 
\begin{proof} The exactness of \eqref{icc exact sequence} follows from Theorem \ref{L0realizing conn comp for GlK alg} and the equalities
$G_{l, L_{0}}^{\alg} = (G_{l, K}^{\alg})^{\circ}$ and $G_{l, L_{0}, 1}^{\alg} = G_{l, K, 1}^{\alg} \cap
G_{l, L_{0}}^{\alg}$. Recall that the group of connected components of an algebraic group is finite. 
Hence it follows from \eqref{icc exact sequence} that $G_{l, L_{0},1}^{\alg}$ is connected iff 
$G_{l, L_{0},1}^{\alg} = (G_{l, K, 1}^{\alg})^{\circ}$, but this is equivalent to the statement that 
$i_{CC}$ is an isomorphism. 
\end{proof}

\begin{theorem} \label{L0realizing conn comp} Assume that the weight $n$ of  $(V, \psi)$ is odd. Then there is a 
finite Galois extension $L_{0}/K$ 
such that $G_{l, L_{0}}^{\alg} = (G_{l, K}^{\alg})^{\circ}$ and $G_{l, L_{0}, 1}^{\alg} = ({G_{l, K,
1}^{\alg}})^{\circ}$. 
\end{theorem} 

\begin{proof} For the proof of the first equality, see the proof of \cite[Prop. 4.7]{BK2}. We refine 
the proof of the second equality. 

Let $L_0$ be such that the first equality holds. Restrict the $l$-adic representation 
to the base field $L_0$. Using the Hodge--Tate property of $V_l$, after taking $\C_l$-points in the exact sequence \eqref{The exact sequence for GlL1alg and GlLalg} one can apply the homomorphism
\cite[p.\ 114]{Se12} defined by Serre:
$$
h_l\colon \G_m (\C_l) \rightarrow  G_{l, L_{0}}^{\alg} (\C_l)
$$
such that for all $x \in \G_m(\C_l)$, $h_l (x)$ acts by multiplication by $x^{-p}$ on the subspace: 
$$
G_{K_{\lambda}}^{p} \otimes_{K_{\lambda}} \C_{l} \, \simeq \, G_{K_{\lambda}}^{p} \otimes_{K_{\lambda}} 
\C_{l} (p-n) \, \subset \, V_{l} \otimes_{\Q_l} \C_{l}.
$$

Choose embeddings of $L_0$ into $\C_l$ and $\C$ and a compatible isomorphism
$\C_l \simeq \C$. Extending coefficients to $\C$ using this isomorphism 
and applying the property {\bf{(R3)}}, we obtain from $h_l$ a homomorphism 
$$
h\colon \G_m (\C) \rightarrow  G_{l, L_{0}}^{\alg} (\C)
$$
which is defined in the same way as $\mu_{\infty, V}$ of Lemma \ref{for n odd DH = H} via natural isomorphisms:
$$
G_{K_{\lambda}}^{p} \otimes_{K_{\lambda}} \C \, \simeq \, G_{K}^{p} \otimes_{K} \C \, \simeq \, 
V^{p, n-p}.
$$
Hence $\chi (h (z)) = z^{-n}$ for every $z \in \G_m (\C)$. Let 
\begin{gather*}
w\colon \G_m (\C) \rightarrow  G_{l, L_{0}}^{\alg} (\C),\\
w(z) := z \, \rm{Id}_{V_{\C}}
\end{gather*}
be the diagonal homomorphism; this is well-defined thanks to the comments before
Remark~\ref{properties of Hodge--Tate representations }.
We know (see \S \ref{Mumford--Tate groups of polarized Hodge structures}) that $\chi (w(z)) = z^2$ for every 
$z \in \G_m (\C)$. Hence the homomorphism
\begin{gather*}
s\colon \G_m (\C) \rightarrow  G_{l, L_{0}}^{\alg} (\C), \\
s(z) := h(z) w(z)^{-\frac{n-1}{2}}
\end{gather*}
is a splitting of $\chi$ in the following exact sequence:
$$
1 \,\, {\stackrel{}{\longrightarrow}} \,\, G_{l, L_{0}, 1}^{\alg} (\C) \,\, {\stackrel{}{\longrightarrow}} \,\,  G_{l, L_{0}}^{\alg} (\C) \,\, {\stackrel{\chi}{\longrightarrow}} \,\, \C^{\times} \,\, {\stackrel{}{\longrightarrow}} \,\, 1.
$$
Because $G_{l, L_{0}}^{\alg}$ is connected, Lemma \ref{G connected implies G0 connected}
shows that $G_{l, L_{0}, 1}^{\alg}$ is connected. 
It follows from the proof of Theorem \ref{L0realizing conn comp of GlKalg} that
 $G_{l, L_{0}, 1}^{\alg} = (G_{l, K, 1}^{\alg})^{\circ}$. 
\end{proof}

\begin{theorem} \label{conn comp equal for for GlK1 alg and GlK alg} Let $n$ be odd. 
The following natural map is an isomorphism of finite groups:
\begin{equation}
i_{CC}\colon \pi_{0} (G_{l, K, 1}^{\alg}) \quad {\stackrel{\simeq}{\longrightarrow}} \quad 
\pi_{0} (G_{l, K}^{\alg}).
\label{conn comp isomorphic for for GlK1 alg and GlK alg}
\end{equation} 
\end{theorem} 
\begin{proof} This follows from Theorems \ref{L0realizing conn comp of GlKalg}  and \ref{L0realizing conn comp}. 
\end{proof}

The main results of \cite{BK2} concern $n$ odd. In this paper we will present a number of results that hold also for 
$n$ even. The following proposition holds for $(V, \psi)$ of any weight but gives essentially new information 
for $(V, \psi)$ of even weight. 

\begin{proposition} \label{L0realizing conn comp: even case} 
Let $M_0/K$ be a field extension such that ${G_{l, M_0}^{\alg}}$ is connected. Then:
\begin{itemize}
\item[(a)] $\G_{m} {\rm{Id}}_{V_l} \cdot ({G_{l, M_0, 1}^{\alg}})^{\circ} \, = \, 
\G_{m} {\rm{Id}}_{V_l} \cdot {G_{l, M_0, 1}^{\alg}} \, = \, G_{l, M_0}^{\alg}$.
\item[(b)] ${G_{l, M_0, 1}^{\alg}} = ({G_{l, M_0, 1}^{\alg}})^{\circ} \, \cup \, 
- {\rm{Id}}_{V_l} \cdot ({G_{l, M_0, 1}^{\alg}})^{\circ}$.
\item[(c)] $- {\rm{Id}}_{V_l} \in ({G_{l, M_0, 1}^{\alg}})^{\circ}$ iff $({G_{l, M_0, 1}^{\alg}})^{\circ}
= {G_{l, M_0, 1}^{\alg}}$.
\item[(d)] When $n$ is odd then $- {\rm{Id}}_{V_l} \, \in \, ({G_{l, M_0, 1}^{\alg}})^{\circ} \,  = \,
{G_{l, M_0, 1}^{\alg}}$.  
\end{itemize} 
\end{proposition} 
\begin{proof}
(a)  Consider a coset $g ({G_{l, M_0, 1}^{\alg}})^{\circ}$ in $G_{l, M_0, 1}^{\alg}$. Applying \cite[Section 7.4, Prop. B(b)]{Hu} to the homomorphism
\begin{gather*}
\G_{m} {\rm{Id}}_{V_l} \, \times \, ({G_{l, M_0, 1}^{\alg}})^{\circ} \rightarrow G_{l, M_0}^{\alg}, \\
(g_1, g_2) \, \mapsto \, g_1 g_2
\end{gather*}
we observe that  $\G_{m} {\rm{Id}}_{V_l} \cdot ({G_{l, M_0, 1}^{\alg}})^{\circ}$ and $\G_{m} {\rm{Id}}_{V_l} \cdot g ({G_{l, M_0, 1}^{\alg}})^{\circ}$ are closed in $G_{l, M_0}^{\alg}$. They are also of the same dimension as $G_{l, M_0}^{\alg}$ because of the following exact sequence:
\begin{equation}
1 \,\, {\stackrel{}{\longrightarrow}} \,\, G_{l, M_0, 1}^{\alg} \,\, {\stackrel{}{\longrightarrow}} \,\,  G_{l, M_0}^{\alg} \,\, {\stackrel{\chi}{\longrightarrow}} \,\, \G_m \,\, {\stackrel{}{\longrightarrow}} \,\, 1.
\label{The exact sequence for GlL01alg and GlL0alg}
\end{equation}
Because  $G_{l, M_0}^{\alg}$ is irreducible by assumption, we obtain 
\begin{equation}
\G_{m} {\rm{Id}}_{V_l} \cdot ({G_{l, M_0, 1}^{\alg}})^{\circ} \, = \,\G_{m} {\rm{Id}}_{V_l} \cdot  g\, ({G_{l, M_0, 1}^{\alg}})^{\circ} = G_{l, M_0}^{\alg}.
\label{nonempty intersection of Gm GlK1 alg with Gm gGlK1 alg}
\end{equation}

\noindent
(b) From \eqref{nonempty intersection of Gm GlK1 alg with Gm gGlK1 alg} there are $\alpha, \beta \in \G_{m}$ and $g_1, g_2 \in ({G_{l, M_0, 1}^{\alg}})^{\circ}$ such that:
\begin{equation}
\alpha \, {\rm{Id}}_{V_l} \cdot g_1 = \beta \, {\rm{Id}}_{V_l} \cdot g g_2. 
\label{equality of coset generators}\end{equation}

\noindent
Applying $\chi$ to \eqref{equality of coset generators} we obtain $\alpha^2 = \beta^2$. This implies that
$\pm {\rm{Id}}_{V_l} \cdot g_1 =  g g_2$. Hence
$g ({G_{l, M_0, 1}^{\alg}})^{\circ} = ({G_{l, M_0, 1}^{\alg}})^{\circ}$ or $g ({G_{l, M_0, 1}^{\alg}})^{\circ} = 
- {\rm{Id}}_{V_l} ({G_{l, M_0, 1}^{\alg}})^{\circ}$. 
\medskip

\noindent
(c) This follows immediately from (b).
\medskip

\noindent
(d) For such an $M_0$, in the case of $n$ odd, the group ${G_{l, M_0, 1}^{\alg}}$  is also connected. 
The proof is the same as the proof of Theorem \ref{L0realizing conn comp}. 
Hence (d) follows immediately from (c).
\end{proof}

\begin{proposition} \label{L0 realizing for rho l (GK): even case} 
Let $M_0/K$ be a field extension such that ${G_{l, M_0}^{\alg}}$ is connected. Then we have the following 
statements.
\begin{itemize}
\item[(a)] ${G_{l, M_0}^{\alg}} \, = \, \G_{m} {\rm{Id}}_{V_l} \cdot \overline{\rho_{l} (G_{M_0})_{1}}$.
\item[(b)] $\overline{\rho_{l} (G_{M_0})_{1}} \, = \, ({G_{l, M_0, 1}^{\alg}})^{\circ}$ \, or \,  
$\overline{\rho_{l} (G_{M_0})_{1}} \, = \, {G_{l, M_0, 1}^{\alg}}$.
\item[(c)] $- {\rm{Id}}_{V_l} \in \overline{\rho_{l} (G_{M_0})_{1}}$ \,\, iff \,\, $\overline{\rho_{l} (G_{M_0})_{1}} 
= {G_{l, M_0, 1}^{\alg}}$.
\item[(d)] If $- {\rm{Id}}_{V_l} \notin \overline{\rho_{l} (G_{M_0})_{1}}$ \,\, then \,\, $\overline{\rho_{l} (G_{M_0})_{1}} = ({G_{l, M_0, 1}^{\alg}})^{\circ}$.
\item[(e)] When $n$ is odd then $- {\rm{Id}}_{V_l} \, \in \, \overline{\rho_{l} (G_{M_0})_{1}}  \, = \, 
({G_{l, M_0, 1}^{\alg}})^{\circ} \,  = \, {G_{l, M_0, 1}^{\alg}}$.  
\end{itemize} 
\end{proposition} 

\begin{proof}
(a) Let $H_l := \rho_{l} (G_{M_0}) \, \cap \,\,  \Q_{l}^{\times} {\rm{Id}}_{V_l}$ denote the subgroup of homotheties
of the group $\rho_{l} (G_{M_0})$. Because $\rho_{l} (G_{M_0})$ is compact, the restriction of $\chi$ to $\Q_{l}$-points
is continuous for the $l$-adic topology, and $\Z_{l}^{\times}$ is the maximal compact subgroup of $\Q_{l}^{\times}$, we have $\chi (\rho_{l} (G_{M_0})) \subset \Z_{l}^{\times}$.
Hence we have the following exact sequence of groups:
\begin{equation}
1 \,\, {\stackrel{}{\longrightarrow}} \,\,\, H_{l} \cdot \rho_{l} (G_{M_0})_{1} \, / \, H_{l} \,\,\, {\stackrel{}{\longrightarrow}} \,\,\,  
\rho_{l} (G_{M_0}) / \, H_{l} \,\,\, {\stackrel{\chi}{\longrightarrow}} \,\,\, \Z_{l}^{\times} / \chi(H_l).
\label{The exact sequence for rho (G) over Hl}
\end{equation}
It follows by Bogomolov's theorem on homotheties that the group $\Z_{l}^{\times} /  \chi(H_l)$ is finite. Hence
$[\rho_{l} (G_{M_0})\colon H_{l} \cdot \rho_{l} (G_{M_0})_{1}] \, < \, \infty$. In this way we obtain the 
equality
\begin{equation}
\rho_{l} (G_{M_0})  \, = \, \bigsqcup_{i=1}^{k} \,\,\, g_i \, (H_{l} \cdot \rho_{l} (G_{M_0})_{1})
\label{rho l GL0 decomposed into cosets of rho l GL0 1}\end{equation}
for some coset representatives $g_i \in \rho_{l} (G_{M_0})$. Recall that
$\overline{\rho_{l} (G_{M_0})_{1}}  \, \subset \,  G_{l, M_0, 1}^{\alg}$. Hence:
\begin{equation}
\rho_{l} (G_{M_0})  \subset \bigcup_{i=1}^{k} \, g_i \, ( \G_{m} {\rm{Id}}_{V_l} \cdot 
\overline{\rho_{l} (G_{M_0})_{1}} \, ) \, \subset \, \bigcup_{i=1}^{k} \, g_i \, ( \G_{m} {\rm{Id}}_{V_l} \cdot 
G_{l, M_0, 1}^{\alg}) \subset  G_{l, M_0}^{\alg}. 
\label{rho l GL0 contained in cosets of Gl L0 1 alg}
\end{equation}
Since the second and third terms from the left of \eqref{rho l GL0 contained in cosets of Gl L0 1 alg} are closed in 
$G_{l, M_0}^{\alg}$ (cf. the proof of Proposition \ref{L0realizing conn comp: even case}(a)), then taking the Zariski closure of the left term of \eqref{rho l GL0 contained in cosets of Gl L0 1 alg} in $G_{l, M_0}^{\alg}$ we obtain
\begin{equation}
G_{l, M_0}^{\alg}  =  \bigcup_{i=1}^{k} \, g_i \, ( \G_{m} {\rm{Id}}_{V_l} \cdot 
\overline{\rho_{l} (G_{M_0})_{1}} \, )  = \bigcup_{i=1}^{k} \, g_i \, ( \G_{m} {\rm{Id}}_{V_l} \cdot G_{l, M_0, 1}^{\alg}).
\label{rho l GL0 decomposed into cosets of Gl L0 1 alg}
\end{equation}
The same argument as in the proof of Proposition \ref{L0realizing conn comp: even case}(a) shows that 
all the sets $g_i \, ( \G_{m} {\rm{Id}}_{V_l} \cdot \overline{\rho_{l} (G_{M_0})_{1}})$ are closed in 
$G_{l, M_0, 1}^{\alg}$ and of the same dimension as $G_{l, M_0}^{\alg}$. Because 
$G_{l, M_0}^{\alg}$ is irreducible the equality \eqref{rho l GL0 decomposed into cosets of Gl L0 1 alg} implies 
\begin{equation}
{G_{l, M_0}^{\alg}} \, = \, \G_{m} {\rm{Id}}_{V_l} \cdot \overline{\rho_{l} (G_{M_0})_{1}} \, =
\, \G_{m} {\rm{Id}}_{V_l} \cdot G_{l, M_0, 1}^{\alg} \, .
\label{Gl L0alg = Gm times rho l G L0 1}
\end{equation}

\noindent
(b) Observe that $\G_{m} {\rm{Id}}_{V_l} \, \cap \, G_{l, M_0, 1}^{\alg} = \{\pm 1\}$. Hence it follows from 
\eqref{Gl L0alg = Gm times rho l G L0 1} that $\overline{\rho_{l} (G_{M_0})_{1}}$ has the 
same dimension as $G_{l, M_0, 1}^{\alg}$. Since $(G_{l, M_0, 1}^{\alg})^{\circ}$ is irreducible, the claim follows by 
Proposition \ref{L0realizing conn comp: even case}(b). 
\medskip

\noindent
(c) This follows by (b) and Proposition \ref{L0realizing conn comp: even case}(c).
\medskip

\noindent
(d) This follows by (b) and (c).
\medskip

\noindent
(e) This follows by (b) and (c) and Proposition \ref{L0realizing conn comp: even case}(d).
\end{proof}
\medskip

Consider the natural continuous homomorphism
\begin{equation}
\epsilon_{l, K}\colon G_K \rightarrow G_{l, K}^{\alg}(\Q_l). 
\label{epsilon}
\end{equation}
Since $(G_{l, K}^{\alg})^{\circ}$ is open in $G_{l, K}^{\alg}$, we get
\begin{equation}
\epsilon_{l, K}^{-1} ((G_{l, K}^{\alg})^{\circ}(\Q_l)) \,\, = \,\, G_{K_0}
\label{epsilon and tilde epsilon}
\end{equation} 
for some finite Galois extension $K_{0}/K$. 

\begin{remark}
We observe that $K_{0}/K$ is the minimal extension such that $G_{l, K_{0}}^{\alg} = (G_{l, K}^{\alg})^{\circ}$.
When $n$ is odd, it follows by Theorems \ref{L0realizing conn comp} and 
\ref{conn comp equal for for GlK1 alg and GlK alg} (equivalently \cite[Prop. 4.7, Th. 4.8]{BK2}) 
that $K_{0}/K$ is the minimal extension such that $G_{l, K_{0}}^{\alg} = (G_{l, K}^{\alg})^{\circ}$ and $G_{l, K_{0}, 1}^{\alg} = 
(G_{l, K, 1}^{\alg})^{\circ}$.  Obviously, Propositions \ref{L0realizing conn comp: even case} and \ref{L0 realizing for rho l (GK): even case} hold for $M_0 = K_0$. In principle, $K_0$ may depend on $l$.
\label{properties of K0}
\end{remark} 

The following theorem and corollary extend \cite[Prop. 4.7 and Theorem 4.8]{BK2} to 
the arbitrary weight case. 

\begin{theorem} \label{K0 almost realizing Glk1 alg: even case} Let $n$ be an arbitrary weight.
Then:
\begin{itemize}
\item[(a)]  $G_{l, K_0, 1}^{\alg} = ({G_{l, K, 1}^{\alg}})^{\circ} \cup  
\, - {\rm{Id}}_{V_l} ({G_{l, K, 1}^{\alg}})^{\circ}$.
 \item[(b)] The homomorphism 
$$
i_{CC}\colon \pi_{0} (G_{l, K, 1}^{\alg}) \quad {\stackrel{}{\longrightarrow}} \quad 
\pi_{0} (G_{l, K}^{\alg})
$$ 
is an epimorphism with kernel $\{({G_{l, K, 1}^{\alg}})^{\circ}, \,\,\, - {\rm{Id}}_{V_l} ({G_{l, K, 1}^{\alg}})^{\circ}\}$.
\end{itemize}
\end{theorem}

\begin{proof} (a) By the definition of $K_0$ we have $({G_{l, K}^{\alg}})^{\circ} = {G_{l, K_0}^{\alg}}$ (see Remark
 \ref{properties of K0}). It follows by Proposition \ref{L0realizing conn comp: even case}(b) 
and by \eqref{GlK1 alg circ equals GlL1 alg circ and GlK alg circ equals GlL alg circ} that 
\begin{equation}
{G_{l, K_0, 1}^{\alg}} = 
({G_{l, K_0, 1}^{\alg}})^{\circ} \cup \, - {\rm{Id}}_{V_l} ({G_{l, K_0, 1}^{\alg}})^{\circ} = 
({G_{l, K, 1}^{\alg}})^{\circ} \cup \, - {\rm{Id}}_{V_l} ({G_{l, K, 1}^{\alg}})^{\circ}. 
\label{-Id in GlK1 alg circ gives GlK1 alg circ the same as GlK01 alg circ the same as GlK1 alg}
\end{equation}

\noindent
(b) This follows by (a) and \eqref{L over K gives iso of GlK1 alg mod GlL1 alg with GlK alg mod GlL alg}.
\end{proof}

\begin{corollary} 
\label{K0 realizing Glk1 alg: even case condition}
Let $n$ be an arbitrary weight.  Then 
$- {\rm{Id}}_{V_l} \in ({G_{l, K, 1}^{\alg}})^{\circ}$ if and only if the following conditions hold:  
\begin{itemize}
\item[(1)] $G_{l, K_0, 1}^{\alg} = ({G_{l, K, 1}^{\alg}})^{\circ}$;
\item[(2)] $i_{CC}\colon \pi_{0} (G_{l, K, 1}^{\alg}) \quad {\stackrel{\simeq}{\longrightarrow}} \quad 
\pi_{0} (G_{l, K}^{\alg})$.
\end{itemize}
\end{corollary}
\begin{proof}
This follows immediately from Proposition \ref{L0realizing conn comp: even case} (c) and Theorem \ref{K0 almost realizing Glk1 alg: even case}.
\end{proof}
\medskip

Define 
$$
G_{l, K, 1}^{\alg, \, 0} \, := \, (G_{l, K}^{\alg})^{\circ} \, \cap \, {G_{l, K, 1}^{\alg}}.
$$
\begin{lemma} 
\label{Lemma G l K 0 alg = G l K 1 alg 0}
We have the following equality
\begin{equation}
G_{l, K, 1}^{\alg, \, 0} \, = \, G_{l, K_0, 1}^{\alg} 
\label{G l K 0 alg = G l K 1 alg 0}
\end{equation}
\end{lemma}
\begin{proof}
We have
\begin{align*}
G_{l, K, 1}^{\alg, \, 0} \, &= \, (G_{l, K}^{\alg})^{\circ} \cap {G_{l, K, 1}^{\alg}} \, = \, 
(G_{l, K}^{\alg})^{\circ} \, \cap \, \Iso_{(V_{l}, \psi_{l})}  \\
&= \,  G_{l, K_0}^{\alg} \, \cap \, \Iso_{(V_{l}, \psi_{l})} \, = \,  G_{l, K_0, 1}^{\alg}.  \qedhere
\end{align*}
\end{proof}

\begin{definition} The {\em{Serre $l$-adic parity group}}:
\begin{equation}
\gpP_{\rm{S}} (V_{l}, \psi_{l}) := G_{l, K, 1}^{\alg, \, 0} \, / \, ({G_{l, K, 1}^{\alg}})^{\circ},
\end{equation}
is the component group of $G_{l, K, 1}^{\alg, \, 0} = G_{l, K_0, 1}^{\alg}$.
\label{The l-adic parity group}
\end{definition}

\begin{proposition} 
If $n$ is odd or $({G_{l, K, 1}^{\alg}})^{\circ} = 
C_{D} (\Iso_{(V_{l}, \psi_{l})})$ then
$\gpP_{\rm{S}} (V_{l}, \psi_{l})$ is trivial. 
\label{P(Vl, psil) trivial for n odd or GlK01 = CDl}
\end{proposition}
\begin{proof}
This follows from 
Remark \ref{properties of K0}, \cite[Remark 6.3]{BK2} and the 
following observation:
\begin{gather*}
({G_{l, K, 1}^{\alg}})^{\circ} \subset G_{l, K_0, 1}^{\alg} \subset \gpDL_{K_0}(V, \psi, D)_{\Q_l} \subset
\gpDL_{K_e}(V, \psi, D)_{\Q_l} = \\
= \gpDL_{K}^{\id}(V, \psi, D)_{\Q_l}  = C_{D} (\Iso_{(V, \psi)})_{\Q_{l}} =
C_{D} (\Iso_{(V_{l}, \psi_{l})}). \qedhere
\end{gather*}
\end{proof}

\begin{proposition}
If $n$ is even and ${\rm dim}_{\Q}\, V$ is odd then
$\gpP_{\rm{S}} (V_{l}, \psi_{l})$ is nontrivial. 
\label{P(Vl, psil) nontrivial for n even and dim V odd}
\end{proposition}
\begin{proof} Because $n$ is even, $\Iso_{(V_{l}, \psi_{l})} = \gpO_{(V_{l}, \psi_{l})}$. Hence 
$({G_{l, K, 1}^{\alg}})^{\circ} \subset \Iso_{(V_{l}, \psi_{l})} \subset {\pm {\rm{Id}}_{V_{l}}} \SL_{V_{l}}$.  
It follows that $- {\rm{Id}}_{V_{l}} \notin ({G_{l, K, 1}^{\alg}})^{\circ}$ because ${\rm{det}} (- {\rm{Id}}_{V_{l}}) = 
(-1)^{{\rm{dim}}_{\Q_{l}}\, V_{l}} = -1$ and $({G_{l, K, 1}^{\alg}})^{\circ}$ is connected.
\end{proof}

\begin{example} 
Let $A$ be an abelian threefold over $K$. Let
$V_{l} :=  H_{et}^{2} ({\overline{A}}, \, \Q_{l}) = \bigwedge^{2} \, H_{et}^{1} ({\overline{A}}, \, \Q_{l})$. 
Then by Example \ref{An example of nontrivial Betti parity group} and Proposition 
\ref{P(Vl, psil) nontrivial for n even and dim V odd}, the group $\gpP_{\rm{S}} (V_{l}, \psi_{l})$ is nontrivial. 
\label{An example of nontrivial l-adic parity group}
\end{example}

\begin{example} 
Let $A$ be an elliptic curve over $K$. Let
$V_{l} := \Sym^{2} \, H_{et}^{1} ({\overline{A}}, \, \Q_{l})$. Then by Example
\ref{Another example of nontrivial Betti parity group}
and Proposition \ref{P(Vl, psil) nontrivial for n even and dim V odd},
the group $\gpP_{\rm{S}} (V_{l}, \psi_{l})$ is nontrivial. 
\label{Another example of nontrivial l-adic parity group}
\end{example}
\medskip

In \cite[Chapter 6]{BK2} we proved two technical results \cite[Lemma 6.8]{BK2} and \cite[Corollary 6.9]{BK2} 
that hold for any weight $n$. These results imply two theorems \cite[Theorems 6.10 and 6.11]{BK2} 
that also hold for any weight $n$. We state these theorems below. Put 
$$
\bar{l} \,\, = \,\,
\left\{
\begin{array}{lll}
l&\rm{if}&l > 2,\\
8&\rm{if}&l=2.\\
\end{array}\right.
$$

\begin{theorem} Assume that the following conditions hold:
\begin{itemize}
\item[(1)] $K_e \, \cap \, K(\mu_{\bar{l}}^{\otimes \, n}) \, = \, K$,
\item[(2)] $1 + l\Z_{l} \,\, {\rm{Id}}_{V_l}  \, \subset \, \rho_{l} (G_K)$.
\end{itemize}
Then all arrows in the following commutative diagram are isomorphisms:

\begin{figure}[H]
\[
\begin{tikzcd}[column sep=huge, row sep=large]
\rho_{l} (G_K) / \rho_{l} (G_{K_e})  \arrow{r}{\Zar_{K_e/K}}[swap]{\simeq}
& G_{l, K}^{\alg} / G_{l, K_e}^{\alg} \\ 
\rho_{l} (G_K)_1 / \rho_{l} (G_{K_e})_1  \arrow{r}{\Zar_{K_e / K, \, 1}}[swap]{\simeq} 
\arrow{u}{j_{K_e/K}}[swap]{\simeq} & G_{l, K, 1}^{\alg} / G_{l, K_{e}, 1}^{\alg} \arrow{u}{i_{K_e/K}}[swap]{\simeq} \\
\end{tikzcd}
\]
\\[-0.8cm]
\caption{}
\label{diagram comparing the rho (GK) with GLalg for Ke}
\end{figure}
\label{jKeK and Zar1 are isomorphisms}
\end{theorem}
\begin{proof} 
See \cite[Theorem 6.10]{BK2}.
\end{proof}

\begin{theorem} Assume that the following two conditions hold:
\begin{itemize}
\item[(1)] $K_0 \, \cap \, K(\mu_{\bar{l}}^{\otimes \, n}) \, = \, K$,
\item[(2)] $1 + l\Z_{l} \,\, {\rm{Id}}_{V_l}  \, \subset \, \rho_{l} (G_K)$.
\end{itemize}
Then all arrows in the following commutative diagram are isomorphisms:
\begin{figure}[H]
\[
\begin{tikzcd}[column sep=huge, row sep=large]
\rho_{l} (G_K) / \rho_{l} (G_{K_0})  \arrow{r}{\Zar_{K_0/K}}[swap]{\simeq} & G_{l, K}^{\alg} / G_{l, K_0}^{\alg} \\ 
\rho_{l} (G_K)_1 / \rho_{l} (G_{K_0})_1  \arrow{r}{\Zar_{K_0/K,\, 1}}[swap]{\simeq} \arrow{u}{j_{K_0/K}}[swap]{\simeq}&  G_{l, K, 1}^{\alg} / G_{l, K_{0}, 1}^{\alg} \arrow{u}{i_{K_0/K}}[swap]{\simeq}
\end{tikzcd}
\]
\caption{}
\label{diagram comparing the rho (GK) with GLalg for K0}
\end{figure}
Moreover each coset of $G_{K}/G_{K_0}$ has the form $\tilde{\sigma}_{1} \, G_{K_0}$ such that:
\begin{gather}
\rho_{l} (\tilde{\sigma}_{1}) \in \rho_{l} (G_K)_1,
\label{jK0K and Zar1 are isomorphisms eq1} \\
\rho_{l} (G_K)_1 \,\, = \,\, 
{\bigsqcup}_{\tilde{\sigma}_1  G_{K_0}}  \,\, \rho_{l} (\tilde{\sigma}_{1})\, \rho_{l} (G_{K_0})_1,
\label{jK0K and Zar1 are isomorphisms eq2} \\
G_{l, K, 1}^{\alg} \,\, = \,\, 
{\bigsqcup}_{\tilde{\sigma}_1  G_{K_0}}  \,\, \rho_{l} (\tilde{\sigma}_{1})\, G_{l, K_0, 1}^{\alg}.
\label{jK0K and Zar1 are isomorphisms eq3}
\end{gather} 
\label{jK0K and Zar1 are isomorphisms}
\end{theorem}
\begin{proof} 
See \cite[Theorem 6.11]{BK2}.
\end{proof}

\begin{theorem} \label{K giving the isom overline rho l (GK) the same as GlK1 alg: arbitrary weight} 
Let $n$ be an arbitrary weight. Assume that:
\begin{itemize}
\item[(1)] $K_0 \, \cap \, K(\mu_{\bar{l}}^{\otimes \, n}) \, = \, K$,
\item[(2)] $1 + l\Z_{l} \,\, {\rm{Id}}_{V_l}  \, \subset \, \rho_{l} (G_K)$,
\item[(3)] $- {\rm{Id}}_{V_l} \in ({G_{l, K, 1}^{\alg}})^{\circ}$. 
\end{itemize}
Then the following equality holds:
\begin{equation}
\overline{\rho_{l} (G_K)_1} \, = \, {G_{l, K, 1}^{\alg}}.
\label{overline rho l GK1 the same as GlK1 alg}\end{equation}
\end{theorem}

\begin{proof} By the assumption (3) and Theorem \ref{K0 almost realizing Glk1 alg: even case}, we have $G_{l,K_0,1}^{\alg} = (G_{l,K,1}^{\alg})^\circ$.
Hence by Proposition \ref{L0 realizing for rho l (GK): even case} we obtain
\begin{equation}
\overline{\rho_{l} (G_{K_0})_1}  = ({G_{l, K_0, 1}^{\alg}})^{\circ} = {G_{l, K_0, 1}^{\alg}}.
\label{-Id in GlK1 alg circ gives overline rho l GK1 the same as GlK01 alg circ the same as GlK1 alg}
\end{equation}  
By assumptions (1) and (2), we obtain conditions (2) and (3) in the conclusion of 
Theorem \ref{jK0K and Zar1 are isomorphisms}. By equation \eqref{-Id in GlK1 alg circ gives overline rho l GK1 the same as GlK01 alg circ the same as GlK1 alg} above, the right-hand side of \eqref{jK0K and Zar1 are isomorphisms eq3} is the Zariski closure of the right-hand side of \eqref{jK0K and Zar1 are isomorphisms eq2}. Hence the left-hand side 
of \eqref{jK0K and Zar1 are isomorphisms eq3} is the Zariski closure of the left-hand side of \eqref{jK0K and Zar1 are isomorphisms eq2}.
\end{proof}
 
\begin{corollary} 
\label{K0 giving the isom overline rho l (GK) the same as GlK1 alg: odd case} 
Let $n$ be odd. Assume that:
\begin{itemize}
\item[(1)] $K_0 \, \cap \, K(\mu_{\bar{l}}^{\otimes \, n}) \, = \, K$,
\item[(2)] $1 + l\Z_{l} \,\, {\rm{Id}}_{V_l}  \, \subset \, \rho_{l} (G_K)$.
\end{itemize}
Then the following equality holds:
\begin{equation}
\overline{\rho_{l} (G_K)_1} \, = \, {G_{l, K, 1}^{\alg}}.
\label{overline rho l GK1 the same as GlK1 alg odd weight}
\end{equation}
\end{corollary}

\begin{proof} Because $n$ is odd, the assumption (3) of Theorem 
\ref{K giving the isom overline rho l (GK) the same as GlK1 alg: arbitrary weight} follows by
Proposition \ref{L0realizing conn comp: even case}(d).
\end{proof}

\begin{remark}
When $A/K$ is an abelian variety over $K$ and $\overline{A} := A \otimes_{K} \overline{K}$, then the $\Q$-vector space $V_{A} = H^1 ( \overline{A} (\C), \Q)$ admits a rational, polarized, pure Hodge structure $(V_A, \psi_A)$
of weight 1 associated with a polarization of $A$.  
Hence Corollary \ref{K0 giving the isom overline rho l (GK) the same as GlK1 alg: odd case} can be applied to the $l$-adic representation of $G_K$ on $V_l = H^{1}_{et} ( \overline{A}, \Q_l) \simeq T_{l} (A)^{\vee}$.
\end{remark}

\begin{remark} In Theorems \ref{K0 almost realizing Glk1 alg: even case}, \ref{jKeK and Zar1 are isomorphisms}, \ref{jK0K and Zar1 are isomorphisms}, \ref{K giving the isom overline rho l (GK) the same as GlK1 alg: arbitrary weight} and Corollaries \ref{K0 realizing Glk1 alg: even case condition}, \ref{K0 giving the isom overline rho l (GK) the same as GlK1 alg: odd case}, we can replace $K$ with a finite extension $L$ and $K_{0}$ with the corresponding
$L_{0}$ (resp. $K_{e}$ with the corresponding $L_{e}$). If $K \subset L \subset K_{0}$, then $K_{0} = L_{0}$ and
$K_{e} = L_{e}$. 
\label{Base change from K to L in basic theorems} 
\end{remark}

\section{Computation of the identity connected component of \texorpdfstring{$G_{l, K, 1}^{\alg}$}{GlK1alg}}
\label{section-computation of the identity connected component}

In this section we compare our approach with the approach of Serre for the setup of 
the Sato--Tate conjecture. We apply this comparison to compute the identity connected component 
of $G_{l, K, 1}^{\alg}$.

\medskip
Compare the $l$-adic representation $\rho_{l}$ in \eqref{the family of l-adic representations}
and the $l$-adic representation $\widetilde{\rho}_l$ considered by Serre cf. \cite[p. 111--112]{Se12}, 
given by the diagonal action of $G_K$ on the $\Q_{l}$-vector space $V_l \oplus \Q_{l}(1)$:
\begin{equation}
\widetilde{\rho}_l\colon G_K \rightarrow \GL(V_l \oplus \Q_{l}(1)),
\label{Representation rho tilde of Serre}
\end{equation}
$$
\widetilde{\rho}_l (g) := (\rho_{l} (g), \chi_{c} (g)),
$$ 
where $\chi_{c}$ is the $l$-adic cyclotomic character.
\begin{remark}
In \emph{loc. cit.}, Serre has the summand $\Q_{l}(-1)$ because he works with geometric Frobenius.
\end{remark}

Put $W_l := V_l \oplus \Q_{l}(1)$. Let $\widetilde{G_{l, K}^{\alg}} := \widetilde{G_{l, K}^{\alg}(W_l)}$ 
denote the Zariski closure of 
$\widetilde{\rho}_l$ in $\GL_{W_{l}}$. Observe that 
\begin{equation}
\widetilde{G_{l, K}^{\alg}} \subset G_{l, K}^{\alg} \times \G_{m}. 
\label{GlKalg subset GlKalg times Gm}
\end{equation}
Applying the projections onto the first and second factor of
$G_{l, K}^{\alg} \times \G_{m}$ we obtain the  natural morphism  $\pi\colon \widetilde{G_{l, K}^{\alg}} 
\rightarrow G_{l, K}^{\alg}$  and the morphism $N\colon \widetilde{G_{l, K}^{\alg}}  \rightarrow \G_{m}$
defined by Serre \cite[p. 111--112]{Se12}. 
\medskip

Consider the following diagram. 

\begin{figure}[H]
\[
\begin{tikzcd}
1 \arrow{r}{} & \,\, \widetilde{G_{l, K, 1}^{\alg}} \arrow{d}[swap]{\pi_{1}} \arrow{r}{}  & 
 \widetilde{G_{l, K}^{\alg}}  \arrow{d}[swap]{\pi} \arrow{r}{N} & \G_{m} 
\arrow{d}{x \mapsto x^{-n}}[swap]{} \arrow{r}{} & 1 \\ 
1 \arrow{r}{} & G_{l, K, 1}^{\alg}  \arrow{r}{} & 
G_{l, K}^{\alg} \arrow{d}[swap]{} \arrow{r}{\chi} & \G_{m} \arrow{d}[swap]{} 
\arrow{r}{} & 1\\
& & 1 & 1 \\
\end{tikzcd}
\]
\\[-0.8cm]
\caption{}
\label{diagram compatibility of tilde(GlK1alg) with GlK1alg} 
\end{figure}

The morphisms $(x \mapsto x^{-n}) \circ N$ and $\chi \circ \pi$ in the  
right square of Diagram \ref{diagram compatibility of tilde(GlK1alg) with GlK1alg} are equal 
when restricted to the dense subset $\widetilde{\rho}_l (G_K)$ of $\widetilde{G_{l, K}^{\alg}}$
because of \eqref{compatibility of chi with chi-cycl.}.
By an argument as in the proof of \cite[Chap. I, Lemma 4.1]{Ha} these morphisms are equal
on $\widetilde{G_{l, K}^{\alg}}$. Hence the right square of Diagram 
\ref{diagram compatibility of tilde(GlK1alg) with GlK1alg} commutes. The map $\pi_1$ in this diagram
is defined by $\pi$ and clearly makes the left square also commute.
\medskip

The middle vertical arrow of Diagram \ref{diagram compatibility of tilde(GlK1alg) with GlK1alg}
is an epimorphism because its image contains the dense subset $\rho_{l} (G_K)$ of 
$G_{l, K}^{\alg}$. But the image of $\pi$ is closed in $G_{l, K}^{\alg}$ by
\cite[Chap. 2, Sec. 7.4, Proposition B(b)]{Hu}.
\medskip

By the definition of $N$ and \eqref{GlKalg subset GlKalg times Gm}, the kernel of $\pi$ in
Diagram \ref{diagram compatibility of tilde(GlK1alg) with GlK1alg} is contained in ${\rm{Id}}_{V_l} \times
\G_{m}$. Hence the homomorphism $N$ restricted to kernel of $\pi$ is a monomorphism. By the commutativity 
of Diagram \ref{diagram compatibility of tilde(GlK1alg) with GlK1alg}, the kernel of $\pi$ injects into
$\mu_{n}$. Hence $\pi_{1}$ is a monomorphism. In addition $\dim \widetilde{G_{l, K}^{\alg}} = \dim G_{l, K}^{\alg}$ 
and consequently $\dim \widetilde{G_{l, K, 1}^{\alg}} = \dim G_{l, K, 1}^{\alg}$. This shows that
\begin{equation}
(\widetilde{G_{l, K, 1}^{\alg}})^{\circ} = (G_{l, K, 1}^{\alg})^{\circ}.
\label{connected component of id of tildeGlK1alg is the same as GlK1alg}
\end{equation} 
Define 
\begin{equation}
\widetilde{G_{l, K, 1}^{\alg, \, 0}} \,\, := \,\, ( \widetilde{G_{l, K}^{\alg}})^{\circ} \,\, \cap 
\,\, \widetilde{G_{l, K, 1}^{\alg}}.
\label{definition of wide tilde G l, K, 1 alg 0 }
\end{equation}

\begin{lemma} 
The following exact sequence splits:
$$
1 \,\, {\stackrel{}{\longrightarrow}} \,\, \widetilde{G_{l, K, 1}^{\alg}} (\mathbb{C}_l) \,\, {\stackrel{}{\longrightarrow}}  
\,\, \widetilde{G_{l, K}^{\alg}} (\mathbb{C}_l)\,\, {\stackrel{N}{\longrightarrow}} \,\, \G_{m} (\mathbb{C}_l){\stackrel{}{\longrightarrow}} \,\, 1.
$$
Moreover $\widetilde{G_{l, K, 1}^{\alg, \, 0}}$ is connected, i.e.
$\widetilde{G_{l, K, 1}^{\alg, \, 0}} = (\widetilde{G_{l, K, 1}^{\alg}})^{\circ}$.
\label{splitting of widetilde G l K alg mapsto Gm}
\end{lemma}

\begin{proof}
Let $\lambda \subset \mathcal{O}_K$ be a prime over $l$. Let 
$$
{\rho}_{l}^{\lambda}\colon G_{K_{\lambda}} \rightarrow \GL(V_l), \quad \quad \quad \widetilde{\rho}_{l}^{\lambda}\colon G_{K_{\lambda}} \rightarrow \GL(W_l)
$$ 
denote the restrictions of $\rho_{l}$ and $\widetilde{\rho}_{l}$ to $G_{K_{\lambda}} \subset G_K$. The representations 
$\rho_{l}^{\lambda}$ and $\widetilde{\rho}_{l}^{\lambda}$ are
Hodge--Tate. Let $G_{l, K_{\lambda}}^{\alg}$ (resp. $\widetilde{G_{l, K_{\lambda}}^{\alg}}$) be the Zariski closure of $\rho_{l} (G_{K_{\lambda}})$ in $\GL_{V_{l}}$ (resp. $\widetilde{\rho}_{l} (G_{K_{\lambda}})$ in $\GL_{W_{l}}$).
There are cocharacters:
$$
h_{\lambda}\colon \G_{m} (\mathbb{C}_l) \rightarrow \GL_{V_{l}} (\mathbb{C}_l) \quad\quad {\rm{and}} \quad\quad 
\widetilde{h}_{\lambda}\colon \G_{m} (\mathbb{C}_l) \rightarrow \GL_{W_{l}} (\mathbb{C}_l)
$$
such that for each $x \in \mathbb{C}_{l}^{\times} = \G_{m} (\mathbb{C}_l)$ the cocharacter $h_{\lambda} (x)$ (resp. 
$\widetilde{h}_{\lambda} (x))$ acts on the subspace $V(i)$ (resp. $W(i)$) of weight $i$ of $V_{\mathbb{C}_{l}}$ (resp.
$W_{\mathbb{C}_{l}}$) via multiplication by $x^i$. Because $\mathbb{Q}_l (1)$ is a $\Q_l$-rational Hodge--Tate module of
weight $1$, we obtain $\widetilde{h}_{\lambda} (x) = (h_{\lambda} (x), x)$. By \cite[p. 158--159]{Se79} the cocharacters
$h_{\lambda}, \widetilde{h}_{\lambda}$ have smaller targets as follows: 
$$
h_{\lambda}\colon \G_{m} (\mathbb{C}_l) \rightarrow  G_{l, K_{\lambda}}^{\alg} (\mathbb{C}_l), \quad \quad\quad 
\widetilde{h}_{\lambda}\colon
\G_{m} (\mathbb{C}_l) \rightarrow \widetilde{G_{l, K_{\lambda}}^{\alg}} (\mathbb{C}_l).
$$
Observe that $G_{l, K_{\lambda}}^{\alg} \subset (G_{l, K}^{\alg})_{K_{\lambda}}$ and 
$\widetilde{G_{l, K_{\lambda}}^{\alg}} \subset  (\widetilde{G_{l, K}^{\alg}})_{K_{\lambda}}$.
Enlarging targets we obtain cocharacters: 
$$
h\colon \G_{m} (\mathbb{C}_l) \rightarrow  G_{l, K}^{\alg} (\mathbb{C}_l), \quad \quad\quad 
\widetilde{h}\colon
\G_{m} (\mathbb{C}_l) \rightarrow \widetilde{G_{l, K}^{\alg}} (\mathbb{C}_l)
$$ 
such that $\widetilde{h} (x) = (h (x), x)$. By construction $N \circ \widetilde{h} (x) = x$. 
The cocharacter $\widetilde{h}$ is a splitting
of $N$.

Because $\G_m / \C_{l}$ is connected, there is the following exact sequence split by 
$\widetilde{h}$:
$$
1 \,\, {\stackrel{}{\longrightarrow}} \,\, \widetilde{G_{l, K, 1}^{\alg, \, 0}} (\mathbb{C}_l) \,\, {\stackrel{}{\longrightarrow}}  
\,\, (\widetilde{G_{l, K}^{\alg}})^{\circ} (\mathbb{C}_l)\,\, {\stackrel{N}{\longrightarrow}} \,\, \G_{m} (\mathbb{C}_l){\stackrel{}{\longrightarrow}} \,\, 1.
$$
Because $\C_l \simeq \C$ the following exact sequence splits:
$$
1 \,\, {\stackrel{}{\longrightarrow}} \,\, \widetilde{G_{l, K, 1}^{\alg, \, 0}} (\mathbb{C}) \,\, {\stackrel{}{\longrightarrow}}  
\,\, (\widetilde{G_{l, K}^{\alg}})^{\circ} (\mathbb{C})\,\, {\stackrel{N}{\longrightarrow}} \,\, \G_{m} (\mathbb{C}){\stackrel{}{\longrightarrow}} \,\, 1.
$$
By Lemma \ref{G connected implies G0 connected} the algebraic group $\widetilde{G_{l, K, 1}^{\alg, \, 0}}$ is connected. 
\end{proof}

\begin{theorem} (Serre) 
\label{pi 0 wide tilde G l, K, 1 alg cong pi 0 wide tilde G l, K alg }
There is the following isomorphism:
\begin{equation}
\widetilde{{i}}_{CC}\colon  \pi_{0} (\widetilde{G_{l, K, 1}^{\alg}}) \,\,\, {\stackrel{\simeq}{\longrightarrow}} \,\,\, \pi_{0} (\widetilde{G_{l, K}^{\alg}}).
\label{pi 0 wide tilde G l, K, 1 alg cong pi 0 wide tilde G l, K alg. isomorphism}
\end{equation}
\end{theorem}
\begin{proof} This follows from \cite[p. 113]{Se12},
but we instead recall the proof of \cite[Th. 3.3]{BK1} or \cite[Th. 4.6]{BK2}. Consider the following commutative diagram.
\begin{figure}[H]
\[
\begin{tikzcd}
& 1 \arrow{d}[swap]{} & 1 \arrow{d}[swap]{}  & 1 \arrow{d}[swap]{} \\ 
1 \arrow{r}{} &  (\widetilde{G_{l, K, 1}^{\alg}})^{\circ} \arrow{d}[swap]{} \arrow{r}{}
&  \widetilde{G_{l, K, 1}^{\alg}} \arrow{d}[swap]{} \arrow{r}{}  & \pi_{0} (\widetilde{G_{l, K, 1}^{\alg}})   
\arrow{d}[swap]{\simeq}{\widetilde{{i}}_{CC}} \arrow{r}{} &  1\\  
1 \arrow{r}{} & (\widetilde{G_{l, K}^{\alg}})^{\circ} \arrow{d}{N} \arrow{r}{} & \widetilde{G_{l, K}^{\alg}} 
\arrow{d}{N} \arrow{r}{} & \pi_{0}(\widetilde{G_{l, K}^{\alg}}) \arrow{d}[swap]{} \arrow{r}{} &  1 \\
1 \arrow{r}{} & \G_m  \arrow{d}[swap]{} \arrow{r}{=} & \G_m  \arrow{d}[swap]{} \arrow{r}{} & 1\\
& 1  & 1 \\
\end{tikzcd}  
\]
\\[-0.8cm]
\caption{}
\label{diagram to prove Serre theorem for families of l-adic representations}
\end{figure}
In Diagram \ref{diagram to prove Serre theorem for families of l-adic representations}, all rows and the middle
 column are obviously exact. The left column is exact by Lemma \ref{splitting of widetilde G l K alg mapsto Gm}.
By a chase in Diagram \ref{pi0 applied to the diagram compatibility of tilde(GlK1alg) with GlK1alg}, it follows that the right column is exact.
\end{proof}

Consider the following commutative diagram. 
\begin{figure}[H]
\[
\begin{tikzcd}
&& \,\, \pi_{0} (\widetilde{G_{l, K, 1}^{\alg}}) \arrow{d}[swap]{\overline{\pi}_{1}} \arrow{r}{\simeq}  & \pi_{0} (\widetilde{G_{l, K}^{\alg}}) \arrow{d}[swap]{\overline{\pi}} & \\ 
1 \arrow{r}{} & G_{l, K_{0}, 1}^{\alg} / (G_{l, K, 1}^{\alg})^{\circ} \arrow{r}{} & \pi_{0} (G_{l, K, 1}^{\alg})  \arrow{r}{} & 
\pi_{0} (G_{l, K}^{\alg}) \arrow{r}{} & 1\\
\end{tikzcd}
\]
\\[-0.8cm]
\caption{}
\label{pi0 applied to the diagram compatibility of tilde(GlK1alg) with GlK1alg} 
\end{figure}

In Diagram \ref{pi0 applied to the diagram compatibility of tilde(GlK1alg) with GlK1alg}, the bottom row is the exact sequence of 
Theorem \ref{L0realizing conn comp of GlKalg}. By the discussion above concerning Diagram \ref{diagram compatibility of tilde(GlK1alg) with GlK1alg}
we observe that the map $\overline{\pi}$ in Diagram \ref{pi0 applied to the diagram compatibility of tilde(GlK1alg) with GlK1alg} 
is an epimorphism and the map $\overline{\pi}_{1}$ in Diagram \ref{pi0 applied to the diagram compatibility of tilde(GlK1alg) with GlK1alg} is a monomorphism. Theorem \ref{K0 almost realizing Glk1 alg: even case} and a chase in Diagram \ref{pi0 applied to the diagram compatibility of tilde(GlK1alg) with GlK1alg} show that
\begin{equation}
\pi_{0} (G_{l, K, 1}^{\alg})  \,  = \, \pi_{0} (\widetilde{G_{l, K, 1}^{\alg}}) \,\, \cup 
\, - {\rm{Id}}_{V_{l}} (G_{l, K, 1}^{\alg})^{\circ} \,\, \pi_{0} (\widetilde{G_{l, K, 1}^{\alg}}),   
\label{decomposition of pi0GlK1alg into tilde(GlK1alg) cup -Id tilde(GlK1alg)}  
\end{equation} 
where $- {\rm{Id}}_{V_{l}} (G_{l, K, 1}^{\alg})^{\circ}$ denotes the coset of 
$- {\rm{Id}}_{V_{l}}$ in the quotient group $\pi_{0} (G_{l, K, 1}^{\alg}) = G_{l, K, 1}^{\alg} / 
(G_{l, K, 1}^{\alg})^{\circ}$. 
\medskip

Consider the following commutative diagram with exact rows. 
\begin{figure}[H]
\[
\begin{tikzcd}
1 \arrow{r}{} & (\widetilde{G_{l, K, 1}^{\alg}})^{\circ} \arrow{d}[swap]{=} \arrow{r}{} & \,\, \widetilde{G_{l, K, 1}^{\alg}} \arrow{d}[swap]{\pi_{1}} \arrow{r}{}  & 
 \pi_{0} (\widetilde{G_{l, K, 1}^{\alg}}) \arrow{d}[swap]{\overline{\pi}_{1}} \arrow{r}{} & 1\\ 
1 \arrow{r}{} & (G_{l, K, 1}^{\alg})^{\circ} \arrow{r}{} & G_{l, K, 1}^{\alg}  \arrow{r}{} & 
\pi_{0} (G_{l, K, 1}^{\alg}) \arrow{r}{} & 1\\
\end{tikzcd}
\]
\\[-0.8cm]
\caption{}
\label{the diagram compatibility of tilde(GlK1alg) with GlK1alg and pi0tilde(GlK1alg) with pi0GlK1alg} 
\end{figure}

The equality \eqref{decomposition of pi0GlK1alg into tilde(GlK1alg) cup -Id tilde(GlK1alg)} and a chase in 
Diagram \ref{the diagram compatibility of tilde(GlK1alg) with GlK1alg and pi0tilde(GlK1alg) with pi0GlK1alg}
show that 
\begin{equation}
\label{decomposition of GlK1alg into tilde(GlK1alg) cup -Id tilde(GlK1alg)}  
G_{l, K, 1}^{\alg}  \,  = \, \widetilde{G_{l, K, 1}^{\alg}} \,\, \cup 
\, - {\rm{Id}}_{V_{l}}  \,\, \widetilde{G_{l, K, 1}^{\alg}} \, ,   
\end{equation} 
so in particular $[G_{l, K, 1}^{\alg}: \, \widetilde{G_{l, K, 1}^{\alg}}] \leq 2$ and 
$\widetilde{G_{l, K, 1}^{\alg}} \, \triangleleft \, G_{l, K, 1}^{\alg}$.

\begin{remark}
Observe that the equality \eqref{decomposition of GlK1alg into tilde(GlK1alg) cup -Id tilde(GlK1alg)} continues to hold with $K$ replaced by any finite extension of $K$.
\end{remark}

\begin{corollary} \,\, $(G_{l, K, 1}^{\alg})^{\circ} = \widetilde{G_{l, K_{0}, 1}^{\alg}}$.   
\label{Connected component GlKlK1alg in form of widetilde(GlK01alg)circ}  
\end{corollary}
\begin{proof} 
It follows from Diagram \ref{diagram compatibility of tilde(GlK1alg) with GlK1alg}, with $K$ replaced by $K_{0}$, that
$\widetilde{G_{l, K_{0}, 1}^{\alg}}$ is a closed subgroup of $G_{l, K, 1}^{\alg}$. Hence the equality
\eqref{decomposition of GlK1alg into tilde(GlK1alg) cup -Id tilde(GlK1alg)} with $K_{0}$ in place of $K$ shows that 
$\widetilde{G_{l, K_{0}, 1}^{\alg}}$ is also an open subgroup of $G_{l, K, 1}^{\alg}$. Hence comparing 
the equality \eqref{decomposition of GlK1alg into tilde(GlK1alg) cup -Id tilde(GlK1alg)} with $K_{0}$ in place of $K$, the equality in Proposition \ref{L0realizing conn comp: even case}(b) with $K_{0}$ in place of $M_{0}$, and the equality \eqref{connected component of id of tildeGlK1alg is the same as GlK1alg}, we obtain $(\widetilde{G_{l, K_{0}, 1}^{\alg}})^{\circ} = \widetilde{G_{l, K_{0}, 1}^{\alg}}$. Moreover by \eqref{connected component of id of tildeGlK1alg is the same as GlK1alg} again we obtain
$(G_{l, K, 1}^{\alg})^{\circ}  \,  = \, (\widetilde{G_{l, K_{0}, 1}^{\alg}})^{\circ}$.  
\end{proof}

\begin{remark} 
\label{tildeGlK0alg connected} 
The group $\widetilde{G_{l, K_{0}}^{\alg}}$ is connected (replace $K$ with $K_{0}$ in Diagram 
\ref{pi0 applied to the diagram compatibility of tilde(GlK1alg) with GlK1alg}) and
\begin{equation}
\label{tildeGlK01alg = tildeGlK0alg cap tildeGlK1alg} 
\widetilde{G_{l, K_{0}, 1}^{\alg}} \, = \, \widetilde{G_{l, K_{0}}^{\alg}} \,\, \cap \,\, 
\widetilde{G_{l, K, 1}^{\alg}}.
\end{equation}
By Corollary \ref{Connected component GlKlK1alg in form of widetilde(GlK01alg)circ}, $\widetilde{G_{l, K_{0}, 1}^{\alg}}$ is the identity connected component of $\widetilde{G_{l, K, 1}^{\alg}}$. Hence
\begin{equation}
\widetilde{G_{l, K_{0}, 1}^{\alg}} \,\, \triangleleft \,\, \widetilde{G_{l, K, 1}^{\alg}}.
\label{tildeGlK01alg triangleleft tildeGlK1alg} 
\end{equation}
\end{remark}

\begin{lemma} 
\label{condition for GlK1alg = tilde(GlK1alg)}
When $\gpP_{\rm{S}} (V_{l}, \psi_{l})$ is trivial then 
\begin{equation}
G_{l, K, 1}^{\alg}  \,  = \, \widetilde{G_{l, K, 1}^{\alg}} \, .   
\label{GlK1alg = tilde(GlK1alg) for n odd or etale parity trivial} 
\end{equation} 
In particular, when $n$ is odd then \eqref{GlK1alg = tilde(GlK1alg) for n odd or etale parity trivial}
holds.
\end{lemma}
\begin{proof}
The equality \eqref{GlK1alg = tilde(GlK1alg) for n odd or etale parity trivial} holds by Theorem \ref{K0 almost realizing Glk1 alg: even case}(a) and the equalities \eqref{connected component of id of tildeGlK1alg is the same as GlK1alg} and \eqref{decomposition of GlK1alg into tilde(GlK1alg) cup -Id tilde(GlK1alg)}.
When $n$ is odd, the parity group $\gpP_{\rm{S}} (V_{l}, \psi_{l})$ is trivial by Proposition \ref{P(Vl, psil) trivial for n odd or GlK01 = CDl}.
\end{proof}

\begin{remark}
Observe that when $\gpP_{\rm{S}} (V_{l}, \psi_{l})$ is trivial,
the equality \eqref{GlK1alg = tilde(GlK1alg) for n odd or etale parity trivial} follows immediately from Diagrams 
\ref{pi0 applied to the diagram compatibility of tilde(GlK1alg) with GlK1alg} and
\ref{the diagram compatibility of tilde(GlK1alg) with GlK1alg and pi0tilde(GlK1alg) with pi0GlK1alg}.     
\end{remark}

\begin{remark}
\label{decomposition of Gl, K, 1 alg in to coset of widetilde Gl, K0, 1 alg}
Under the assumptions and notation of Theorem \ref{jK0K and Zar1 are isomorphisms} we have
\begin{equation}
\label{decomposition of widetilde(G(l, K, 1)alg) into not necessarily disjoint cosets of widetilde(G(l, K_0, 1)alg)}
G_{l, K, 1}^{\alg} \,\, = \,\, 
{\bigcup}_{\tilde{\sigma}_1  G_{K_0}, i = 0, 1}  \,\,\, (- {\rm{Id}}_{V_l})^{i} \rho_{l} (\tilde{\sigma}_{1})\, \widetilde{G_{l, K_0, 1}^{\alg}}.
\end{equation}
\end{remark}

\begin{theorem} 
Under the assumptions and notation of Theorem \ref{jK0K and Zar1 are isomorphisms} we have:
\begin{gather}
\widetilde{G_{l, K, 1}^{\alg}} \,\, = \,\, 
{\bigsqcup}_{\tilde{\sigma}_1  G_{K_0}}  \,\, \rho_{l} (\tilde{\sigma}_{1})\, \widetilde{G_{l, K_0, 1}^{\alg}}
\quad\quad \text{if} \,\, - {\rm{Id}}_{V_l} \in \widetilde{G_{l, K_0, 1}^{\alg}},
\label{decomposition of widetilde(G(l, K, 1)alg) into cosets of widetilde(G(l, K_0, 1)alg)} \\
\widetilde{G_{l, K, 1}^{\alg}} \,\, = \,\, 
{\bigsqcup}_{\tilde{\sigma}_1  G_{K_0}, i = 0, 1}  \,\,\, (- {\rm{Id}}_{V_l})^{i} \rho_{l} (\tilde{\sigma}_{1})\, \widetilde{G_{l, K_0, 1}^{\alg}} \quad\quad \text{if} \,\, \,\, - {\rm{Id}}_{V_l} \in \widetilde{G_{l, K, 1}^{\alg}} \, - \, \widetilde{G_{l, K_{0}, 1}^{\alg}},
\label{decomposition of widetilde(G(l, K, 1)alg) into cosets of widetilde(G(l, K_0, 1)alg) into more cosets} \\
\widetilde{G_{l, K, 1}^{\alg}} \,\, = \,\, 
{\bigsqcup}_{\tilde{\sigma}_1  G_{K_0},i}  \,\, (-{\rm{Id}}_{V_l})^i \, \rho_{l} (\tilde{\sigma}_{1})\, \widetilde{G_{l, K_0, 1}^{\alg}} \quad\quad \text{if} \,\, - {\rm{Id}}_{V_l} \notin \widetilde{G_{l, K, 1}^{\alg}},
\label{decomposition of widetilde(G(l, K, 1)alg) into cosets of widetilde(G(l, K_0, 1)alg) if 
- IdVl notin widetilde(G(l, K, 1)alg)}
\end{gather}
where $i$ runs over $\{0\}$ if $\rho_{l} (\tilde{\sigma}_{1}) \in \widetilde{G_{l, K, 1}^{\alg}}$
and over $\{0, 1\}$ if $\rho_{l} (\tilde{\sigma}_{1}) \notin \widetilde{G_{l, K, 1}^{\alg}}$.
\label{Three cases of G l, K, 1 alg expressed as widetilde G l, K, 1 alg}
\end{theorem}

\begin{proof} If  $- {\rm{Id}}_{V_l} \in \widetilde{G_{l, K_0, 1}^{\alg}}$, then by the equality \eqref{decomposition of GlK1alg into tilde(GlK1alg) cup -Id tilde(GlK1alg)}  for $K$ and $K_0$ we obtain 
$\widetilde{G_{l, K_{0}, 1}^{\alg}} = G_{l, K_{0}, 1}^{\alg}$ and  $\widetilde{G_{l, K, 1}^{\alg}} = G_{l, K, 1}^{\alg}$. Hence 
\eqref{decomposition of widetilde(G(l, K, 1)alg) into cosets of widetilde(G(l, K_0, 1)alg)} is just \eqref{jK0K and Zar1 are isomorphisms eq3}.

If $- {\rm{Id}}_{V_l} \in \widetilde{G_{l, K, 1}^{\alg}} \, - \, \widetilde{G_{l, K_{0}, 1}^{\alg}}$,
then by the equality \eqref{decomposition of GlK1alg into tilde(GlK1alg) cup -Id tilde(GlK1alg)} for $K$ and $K_0$ we obtain $\widetilde{G_{l, K, 1}^{\alg}} = G_{l, K, 1}^{\alg}$ and 
$\widetilde{G_{l, K_{0}, 1}^{\alg}} \not= G_{l, K_{0}, 1}^{\alg}$. 
In addition the cosets 
$(- {\rm{Id}}_{V_l})^{i} \, \rho_{l} (\tilde{\sigma}_{1})\, \widetilde{G_{l, K_0, 1}^{\alg}}$, 
indexed by cosets $\tilde{\sigma}_1  G_{K_0}$ (cf. \eqref{jK0K and Zar1 are isomorphisms eq3}) and numbers $i = 0, 1$ are pairwise distinct because $- {\rm{Id}}_{V_l} \in G_{l, K_{0}, 1}^{\alg}$.
Hence \eqref{decomposition of widetilde(G(l, K, 1)alg) into cosets of widetilde(G(l, K_0, 1)alg) into more cosets}
holds. 

If  $- {\rm{Id}}_{V_l} \notin \widetilde{G_{l, K, 1}^{\alg}}$ then $- {\rm{Id}}_{V_l} \notin \widetilde{G_{l, K_{0}, 1}^{\alg}}$. By equality \eqref{decomposition of GlK1alg into tilde(GlK1alg) cup -Id tilde(GlK1alg)}  for $K$ and $K_0$ we obtain $\widetilde{G_{l, K, 1}^{\alg}} \, \cap \,\, G_{l, K_{0}, 1}^{\alg} = 
\widetilde{G_{l, K_{0}, 1}^{\alg}}$. As observed above, the cosets 
$(- {\rm{Id}}_{V_l})^{i} \, \rho_{l} (\tilde{\sigma}_{1})\, \widetilde{G_{l, K_0, 1}^{\alg}}$,
indexed by cosets $\tilde{\sigma}_1  G_{K_0}$ (cf. \eqref{jK0K and Zar1 are isomorphisms eq3}) and numbers $i = 0, 1$ are pairwise distinct because $- {\rm{Id}}_{V_l} \in G_{l, K_{0}, 1}^{\alg}$.
Hence \eqref{decomposition of widetilde(G(l, K, 1)alg) into cosets of widetilde(G(l, K_0, 1)alg) if 
- IdVl notin widetilde(G(l, K, 1)alg)} holds. 
\end{proof}

\begin{lemma} 
All arrows in the following commutative diagram are isomorphisms:
\begin{figure}[H]
\[
\begin{tikzcd}[column sep=huge, row sep=large]
G_{K} /G_{K_0} \arrow{r}{}[swap]{\simeq} \arrow{d}{}[swap]{=} &
\widetilde{\rho}_l (G_K) / \widetilde{\rho}_l (G_{K_0}) \arrow{r}{\widetilde{\Zar}_{K_0/K}}[swap]{\simeq}
\arrow{d}{}[swap]{\simeq} & 
\widetilde{G_{l, K}^{\alg}} / \widetilde{G_{l, K_{0}}^{\alg}} \arrow{d}{\overline{\pi}}[swap]{\simeq}\\ 
G_{K} /G_{K_0} \arrow{r}{}[swap]{\simeq} & {\rho}_l (G_K) / {\rho}_l (G_{K_0}) \arrow{r}{{\Zar}_{K_0/K}}[swap]{\simeq}& G_{l, K}^{\alg} / G_{l, K_0}^{\alg}\\ 
\end{tikzcd}
\]
\\[-0.8cm]
\caption{}
\label{diagram of natural maps from Galois to Zariski closures}
\end{figure}
\label{lemma concerning diagram of natural maps from Galois to Zariski closures}
\end{lemma}
\begin{proof}
The right vertical arrow is an epimorphism because of Remark 
\ref{tildeGlK0alg connected}: it is the same as the right vertical arrow $\overline{\pi}$ in Diagram \ref{pi0 applied to the diagram compatibility of tilde(GlK1alg) with GlK1alg}. It follows directly from the definition of Zariski closure that the right horizontal arrows $\widetilde{\Zar}_{K_0/K}$ and ${\Zar}_{K_0/K}$ are epimorphisms. 
The left horizontal arrows are clearly epimorphisms. Hence the middle vertical arrow is also an epimorphism.
By \cite[(6.4)]{BK2} the composition of the bottom horizontal arrows is 
an isomorphism; note that \cite[(6.4)]{BK2} holds for arbitrary weight $n$. Now a chase in Diagram 
\ref{diagram of natural maps from Galois to Zariski closures} shows that all arrows are isomorphisms.
\end{proof}

Let $\tilde\sigma \in G_K$ be a lift of $\sigma \in \Gal(K_{0}/K)$. The coset
$\tilde\sigma \, G_{K_{0}}$ does not depend on the lift. 
The Zariski closure of  
$\widetilde{\rho}_{l} (\tilde\sigma \, G_{K_{0}}) = \widetilde{\rho}_{l} (\tilde\sigma) \, 
\widetilde{\rho}_{l} (G_{K_{0}})$ in $\GL_{W_l}$ is  
$\widetilde{\rho}_{l} (\tilde\sigma) \, \widetilde{G_{l, K_{0}}^{\alg}}$. It follows immediately
from Lemma \ref{lemma concerning diagram of natural maps from Galois to Zariski closures} that:
\begin{align}
\widetilde{\rho}_{l} (G_K)
 &= \bigsqcup_{\sigma \in \Gal(K_{0}/K)} \, \, 
 \widetilde{\rho}_{l} (\tilde\sigma) \, \widetilde{\rho}_{l} (G_{K_{0}}),
\label{Im rhol GK decomposed cosets for K_0} \\
\widetilde{G_{l, K}^{\alg}}
 &= \bigsqcup_{\sigma \in \Gal(K_{0}/K)} \,\, \widetilde{\rho}_{l} (\tilde\sigma) \, \widetilde{G_{l, K_{0}}^{\alg}}.
\label{decomposition of GlKalg into cosets of algebraic closures with resp. to K0}
\end{align}
Put:
$$
\widetilde{\rho}_{l} (G_K)_{1} := \widetilde{\rho}_{l} (G_K) \cap \widetilde{G_{l, K, 1}^{\alg}}. 
$$ 
Let $K \subseteq L \subset \overline{F}$ be a tower of extensions with
$L/K$ finite. Consider the following commutative diagram, in which the left and middle vertical arrows
are monomorphisms.
\begin{figure}[H]
\[
\begin{tikzcd}[column sep=huge, row sep=large]
1 \arrow{r}{} & \widetilde{G_{l, K, 1}^{\alg}} \arrow{r}{}[swap]{}  &
\widetilde{G_{l, K}^{\alg}}  \arrow{r}{N}[swap]{}
& \G_{m} \arrow{r} & 1 \\ 
1 \arrow{r}{} & \widetilde{G_{l, L, 1}^{\alg}} \arrow{r}{}[swap]{} \arrow{u}{}[swap]{} & \widetilde{G_{l, L}^{\alg}} 
\arrow{r}{N}[swap]{} \arrow{u}{}[swap]{} & \G_m  \arrow{r} \arrow{u}{}[swap]{=} & 1 \\ 
\end{tikzcd}
\]
\\[-0.8cm]
\caption{}
\label{diagram of natural maps from tilde GK alg to tilde GL alg}
\end{figure}
Observe that $\widetilde{G_{l, K, 1}^{\alg}} \, \cap \, \widetilde{G_{l, L}^{\alg}} \, = 
\,  \widetilde{G_{l, L, 1}^{\alg}}$.
If $L/K$ is Galois, then it follows from the Diagram 
\ref{diagram of natural maps from tilde GK alg to tilde GL alg} that there is a monomorphism:
\begin{equation}
\label{diagram rho GLl1 maps to rho GKll and rho GLl maps to rho GK}
\tilde{j}_{L/K}\colon \widetilde{\rho}_{l} (G_K)_1 / \widetilde{\rho}_{l} (G_L)_1 \,\,\, \hookrightarrow \,\,\, \widetilde{\rho}_{l} (G_K) /  \widetilde{\rho}_{l} (G_L).
\end{equation}

\begin{theorem} Let $v_{0}$ be a prime in $\mathcal{O}_{K_0}$ such that $v_{0} | l$. Let 
$K_{v_0}$ be the completion of $K_{0}$ at $v_{0}$. Assume that 
$\chi_{c} (G_K) = \chi_{c} (G_{{K_{v_0}}})$.
Then all arrows in the following commutative diagram are isomorphisms:
\begin{figure}[H]
\[
\begin{tikzcd}[column sep=huge, row sep=large]
\widetilde{\rho}_l (G_K) / \widetilde{\rho}_l (G_{K_0})  \arrow{r}{\widetilde{\Zar}_{K_0/K}}[swap]{\simeq} & 
\widetilde{G_{l, K}^{\alg}} / \widetilde{G_{l, K_{0}}^{\alg}} \\ 
\widetilde{\rho}_l (G_K)_1 / \widetilde{\rho}_l (G_{K_0})_1  \arrow{r}{\widetilde{\Zar}_{K_0/K,\, 1}}[swap]{\simeq} \arrow{u}{\tilde{j}_{K_0/K}}[swap]{\simeq}&  \widetilde{G_{l, K, 1}^{\alg}} / 
\widetilde{G_{l, K_{0}, 1}^{\alg}} \arrow{u}{\tilde{i}_{CC}}[swap]{\simeq}
\end{tikzcd}
\]
\caption{}
\label{diagram comparing the tilde rho (GK) with tilde GLalg for K0}
\end{figure}

Moreover each coset of $G_{K}/G_{K_0}$ has the form $\tilde{\sigma}_{1} \, G_{K_0}$ such that:
\begin{itemize}
\item[(1)]  $\widetilde{\rho}_{l} (\tilde{\sigma}_{1}) \in \widetilde{\rho}_{l} (G_K)_1$,
\item[(2)]  $\widetilde{\rho}_{l} (G_K)_1 \,\, = \,\, 
{\bigsqcup}_{\tilde{\sigma}_1  G_{K_0}}  \,\, \widetilde{\rho}_{l} (\tilde{\sigma}_{1})\, \widetilde{\rho}_{l} 
(G_{K_0})_1$,
\item[(3)]  $\widetilde{G_{l, K, 1}^{\alg}} \,\, = \,\, 
{\bigsqcup}_{\tilde{\sigma}_1  G_{K_0}}  \,\, \widetilde{\rho}_{l} (\tilde{\sigma}_{1})
\widetilde{\, G_{l, K_0, 1}^{\alg}} $.
\end{itemize} 
\label{tilde jK0K and tilde Zar1 are isomorphisms}
\end{theorem}
\begin{proof} The right vertical arrow $\tilde{i}_{CC}$ is an isomorphism by Theorem 
\ref{pi 0 wide tilde G l, K, 1 alg cong pi 0 wide tilde G l, K alg }. The top horizontal arrow
$\widetilde{\Zar}_{K_0/K}$ is an isomorphism by Lemma
\ref{lemma concerning diagram of natural maps from Galois to Zariski closures}. The left vertical arrow 
$\tilde{j}_{K_0/K}$ is a monomorphism by \eqref{diagram rho GLl1 maps to rho GKll and rho GLl maps to rho GK}. 
Take any $\sigma \in \Gal(K_{0}/K) = G_{K} / G_{K_0}$. Let $\tilde\sigma \in G_K$ be a lift of 
$\sigma$. By \eqref{Representation rho tilde of Serre}, \eqref{GlKalg subset GlKalg times Gm}, 
and the definition of $N$ we obtain the following equality for any $\tau \in G_{K}$:
\begin{equation}
\chi_{c} (\tau) = N ( \widetilde{\rho}_{l} (\tau)).
\label{chi c of tau = N tilde rho of tau}
\end{equation}
By assumption we can choose $\gamma \in G_{K_{v_0}} \subset G_{K_0} \subset G_K$ such that 
$\chi_{c} (\tilde\sigma \, \gamma^{-1}) = 1$. So $\widetilde{\rho}_{l} (\tilde{\sigma} \gamma^{-1})
\in \widetilde{\rho}_{l} (G_K)_1$ by \eqref{chi c of tau = N tilde rho of tau}. Hence we obtain
\begin{equation}
\widetilde{\rho}_{l} (\tilde{\sigma}) \widetilde{\rho}_{l} (G_{K_0}) \, = \,
\widetilde{\rho}_{l} (\tilde{\sigma} \gamma^{-1}) \widetilde{\rho}_{l} (G_{K_0}) \, = \,
\tilde{j}_{K_0/K} (\widetilde{\rho}_{l} (\tilde{\sigma} \gamma^{-1}) \widetilde{\rho}_{l} (G_{K_0})_1).
\label{cosets generated by tilde rho l tilde sigma are generated by tilde rho l tilde sigma gamma -1}
\end{equation}
Therefore $\tilde{j}_{K_0/K}$ is an epimorphism, hence an isomorphism. This proves that the bottom horizontal arrow $\widetilde{\Zar}_{K_0/K,\, 1}$ is also an isomorphism. Put $\tilde{\sigma}_{1} := \tilde{\sigma} \gamma^{-1}$.
We have $\tilde{\sigma}_{1} \, G_{K_0} = \tilde{\sigma} \, G_{K_0}$ and $\widetilde{\rho}_{l} (\tilde{\sigma}_{1}) \in \widetilde{\rho}_{l} (G_K)_1$, so claim (1) follows. Because $\tilde{j}_{K_0/K}$ and 
$\widetilde{\Zar}_{K_0/K,\, 1}$ are
 isomorphisms, claims (2) and (3) follow by applying \eqref{Im rhol GK decomposed cosets for K_0} and
\eqref{decomposition of GlKalg into cosets of algebraic closures with resp. to K0}.
\end{proof}
\begin{remark} Because $\widetilde{G_{l, K_{0}}^{\alg}}$ and $\widetilde{G_{l, K_{0}, 1}^{\alg}}$ are connected,
everywhere in \S \ref{section-computation of the identity connected component}
we can replace $K$ with a finite extension $L$ and $K_{0}$ with the corresponding
$L_{0}$. If $K \subset L \subset K_{0}$, then $K_{0} = L_{0}$. After such replacement, all results in \S \ref{section-computation of the identity connected component} hold for the base field $L$ with $K_0$ replaced by $L_{0}$.  
\label{Base change from K to L in basic theorems with tilde} 
\end{remark}

\begin{remark} Consider the following commutative diagram:
\begin{figure}[H]
\[
\begin{tikzcd}
1 \arrow{r}{} & \,\, \widetilde{\rho}_{l} (G_{K})_1 \arrow{d}[swap]{\pi_{1}} \arrow{r}{}  & 
\widetilde{\rho}_{l} (G_{K})  \arrow{d}[swap]{\pi} \arrow{r}{N} & \Z_{l}^{\times} 
\arrow{d}{x \mapsto x^{-n}}[swap]{} \\ 
1 \arrow{r}{} & \rho_{l} (G_{K})_1  \arrow{r}{} & 
\rho_{l} (G_{K})  \arrow{d}[swap]{} \arrow{r}{\chi} & \Z_{l}^{\times} \\
& & 1  \\
\end{tikzcd}
\]
\\[-0.8cm]
\caption{}
\label{diagram compatibility of tilde(rho) G K 1 with rho G K 1} 
\end{figure}
In Diagram 
\ref{diagram compatibility of tilde(rho) G K 1 with rho G K 1},
the rows and the middle column are exact, the left vertical arrow is an embedding, and the
kernel of $\pi$ injects into the kernel of the right vertical arrow, cf. the discussion following Diagram
\ref{diagram compatibility of tilde(GlK1alg) with GlK1alg}. 
\end{remark}

\begin{corollary}
Assume $n$ and $l-1$ are coprime. Then:
\begin{itemize}
\item[(a)] $\widetilde{\rho}_{l} (G_{K}) = \rho_{l} (G_{K})$,
\item[(b)] $\widetilde{\rho}_{l} (G_{K})_1 = \rho_{l} (G_{K})_1$,
\item[(c)] $\widetilde{G_{l, K}^{\alg}} = G_{l, K}^{\alg}$,
\item[(d)] $\widetilde{G_{l, K, 1}^{\alg}} = G_{l, K, 1}^{\alg}$.
\end{itemize}
\label{Relations between widetilde rho l GK1 and rho l GK1 etc}
\end{corollary}
\begin{proof}
By assumption the right vertical arrow in Diagram 
\ref{diagram compatibility of tilde(rho) G K 1 with rho G K 1} 
is an isomorphism. This implies that all vertical arrows in this diagram are isomorphisms. This proves 
(1) and (2). The equalities (3) and (4) are proven in the same way based on the discussion following Diagram
\ref{diagram compatibility of tilde(GlK1alg) with GlK1alg}.
\end{proof}

\begin{corollary} 
Assume that:
\begin{itemize}
\item[(1)] $K_0 \, \cap \, K(\mu_{\bar{l}}^{\otimes \, n}) \, = \, K$,
\item[(2)] $1 + l\Z_{l} \,\, {\rm{Id}}_{V_l}  \, \subset \, \rho_{l} (G_K)$,
\item[(3)] $- {\rm{Id}}_{V_l} \in ({G_{l, K, 1}^{\alg}})^{\circ}$, 
\item[(4)] $n$ and $l-1$ are coprime.
\end{itemize}
Then the following equality holds:
\begin{equation}
\label{overline widetilde rho l GK1 the same as widetilde GlK1 alg}
\overline{\widetilde{\rho}_{l} (G_{K})_1} \, = \, \widetilde{G_{l, K, 1}^{\alg}}.
\end{equation}
\end{corollary}
\begin{proof}
This follows from Theorem 
\ref{K giving the isom overline rho l (GK) the same as GlK1 alg: arbitrary weight} and 
Corollary \ref{Relations between widetilde rho l GK1 and rho l GK1 etc}.
\end{proof}

\section{Algebraic Sato--Tate conjecture for families of \texorpdfstring{$l$}{l}-adic representations } 
\label{algebraic ST conjecture for families of l-adic reps}

Let $A$ be an abelian variety over $K$ and $(V_A, \psi_A)$ be the Hodge structure on $H^{1} (A(\C), \Q)$
associated with a polarization of $A$.

\begin{theorem}
\label{Theorem Deligne} 
(Deligne \cite[I, Prop.\ 6.2]{D1}, Piatetski-Shapiro \cite{P-S}, Borovoi \cite{Bor}; see also 
\cite[\S  4.1]{Se77})
For any prime number $l$,
\begin{equation}
(G_{l, K}^{\alg})^{\circ} \subseteq \MT(V_A, \psi_A)_{\Q_l}.
\label{Del ineq}\end{equation}
\end{theorem}

The classical Mumford--Tate conjecture for $A/K$ states:
\begin{conjecture} (Mumford--Tate) \label{Mumford--Tate for A} 
For any prime number $l$,
\begin{equation}
(G_{l, K}^{\alg})^{\circ} = \MT(V_A, \psi_A)_{\Q_l}.
\label{MT eq for A}\end{equation}
\end{conjecture}

We can formulate the Mumford--Tate conjecture for families of $l$-adic 
representations associated with rational polarized Hodge structures 
satisfying the conditions \textbf{(D1)}, \textbf{(D2)} of \S \ref{Mumford--Tate groups of polarized Hodge structures} and conditions \textbf{(R1)}--\textbf{(R4)} of \S 
\ref{families of l-adic representations associated with Hodge structures}
 as follows.

\begin{conjecture} (Mumford--Tate) \label{Mumford--Tate} 
For any prime number $l$,
\begin{equation}
(G_{l, K}^{\alg})^{\circ} = \MT(V, \psi)_{\Q_l}.
\label{MT eq for Hodge structures associated with l-adic rep.}\end{equation}
\end{conjecture}

\begin{remark}\label{Mumford--Tate Conj and Deligne Theorem analogue} It is not known
whether the analogue of the inclusion \eqref{Del ineq} holds for more general classes of
$(V, \psi)$ beyond $(V_A, \psi_A)$. In any given case, if we have the following inclusion:
\begin{equation}
(G_{l, K}^{\alg})^{\circ} \subseteq \MT(V, \psi)_{\Q_l},
\label{Del ineq general}\end{equation}
we obtain the following commutative diagram with all horizontal arrows closed immersions
and all columns exact.

\begin{figure}[H]
\[
\begin{tikzcd}
1 \arrow{d}[swap]{} & 1 \arrow{d}[swap]{} &  1 \arrow{d}[swap]{} \\
G_{l, K_{0}, 1}^{\alg} \arrow{d}[swap]{} \arrow{r}{} &  \gpDH(V, \psi)_{\Q_l}   
\arrow{d}[swap]{} \arrow{r}{} & 
\Iso_{V_{l}, \psi_{l}} \arrow{d}[swap]{} \\ 
(G_{l, K}^{\alg})^{\circ} \arrow{d}[swap]{} \arrow{r}{} & 
\MT (V, \psi)_{\Q_l} \arrow{d}[swap]{} \arrow{r}{} & 
\GIso_{V_{l}, \psi_{l}} \arrow{d}[swap]{} \\
\G_{m} \arrow{d}[swap]{} \arrow{r}{=} & 
\G_{m} \arrow{d}[swap]{} \arrow{r}{=} & \G_{m} \arrow{d}[swap]{}\\
1 & 1 &  1\\
\end{tikzcd}
\]
\\[-0.8cm]
\caption{}
\label{under assumption that GlKalg-circ subset MT (V, psi)-Q-l} 
\end{figure}

It then follows immediately from Diagram \ref{under assumption that GlKalg-circ subset MT (V, psi)-Q-l}, 
Proposition \ref{Hodge realizing conn comp: even case}(a), (b),
and Proposition \ref{L0realizing conn comp: even case}(a), (b)
that \eqref{Del ineq general} is equivalent to  the inclusion
\begin{equation}
(G_{l, K, 1}^{\alg})^{\circ} \subseteq \gpH (V, \psi)_{\Q_l},
\label{Del ineq1}\end{equation}
and Mumford--Tate Conjecture 
\ref{MT eq for Hodge structures associated with l-adic rep.} is equivalent to the equality
\begin{equation}
(G_{l, K, 1}^{\alg})^{\circ} = \gpH(V, \psi)_{\Q_l}.
\label{H eq}
\end{equation}
\end{remark}
\medskip

In \cite[Conjectures 5.1 and 5.9]{BK2}, we formulated the algebraic Sato--Tate and Sato--Tate conjectures 
for families of strictly compatible $l$-adic 
representations of Hodge--Tate type associated with pure Hodge structures. In this paper
we state the algebraic Sato--Tate and Sato--Tate conjectures imposing conditions {\bf{(D1)}}, {\bf{(D2)}} and 
{\bf{(R1)}}--{\bf{(R4)}}.

\begin{conjecture}\label{general algebraic Sato Tate conj.} 
(Algebraic Sato--Tate conjecture; cf. \cite[Conjecture 5.1]{BK2})
\begin{itemize}
\item[(a)]
For every finite extension $K/F$ there exist a natural-in-$K$ 
reductive algebraic group $\AST_{K} (V, \psi) 
\subset \Iso_{(V, \psi)}$  over $\Q$ and a natural-in-$K$ monomorphism of group schemes for every $l$:
\begin{equation}
\gpast_{l, K}\colon  G_{l, K, 1}^{\alg} \,\, 
{\stackrel{}{\hookrightarrow}} \,\,  \AST_{K} (V, \psi)_{\Q_l}.
\label{Algebraic Sato--Tate monomorphism}
\end{equation} 
In addition the natural embedding $G_{l, K, 1}^{\alg} \subset \Iso_{(V_l, \psi_{l})}$ factors through
$\gpast_{l, K}$ and the natural embedding $\AST_{K} (V, \psi)_{\Q_l} \subset \Iso_{(V_l, \psi_{l})}$. 

\item[(b)]
The map
\eqref{Algebraic Sato--Tate monomorphism} is an isomorphism: 
\begin{equation}
\label{Algebraic Sato--Tate equality generalized}
\gpast_{l, K}\colon  G_{l, K, 1}^{\alg} \,\, 
{\stackrel{\simeq}{\longrightarrow}} \,\,  \AST_{K} (V, \psi)_{\Q_l}.
\end{equation}
\end{itemize}
\end{conjecture}

\begin{remark} The requirement that $\AST_{K} (V, \psi)$ and \eqref{Algebraic Sato--Tate monomorphism}
are natural in $K$ means that for any finite extension $L/K$, there is a natural monomorphism of 
groups schemes $\AST_{K} (V, \psi) \hookrightarrow \AST_{L} (V, \psi)$ cf. \cite[Remark 5.3]{BK2}. 
\end{remark}

\begin{remark} In \cite[Conjecture 5.1]{BK2} we had the phrase ``{\it{for every l}}''  in the wrong place. This is corrected in Conjecture \ref{general algebraic Sato Tate conj.} above.     
\end{remark}

\begin{definition}
The group $\AST_{K} (V, \psi)$ is called the \emph{algebraic Sato--Tate group} associated with the family
of representations $(\rho_l)_{l}$. Any maximal compact subgroup of 
 $\AST_{K} (V, \psi)(\C)$ is called the \emph{Sato--Tate group} and denoted $\ST_{K} (V, \psi)$; recall that any two such subgroups are conjugate \cite[\S VII.2]{Kn}.
\label{definition of ASTK(V psi) and STK(V psi)}
\end{definition}

\begin{proposition}\label{connected components iso}
Assume that Algebraic Sato--Tate Conjecture~\ref{general algebraic Sato Tate conj.}
holds for $(V, \psi)$. Then there are natural isomorphisms
\begin{equation}
\pi_{0} (G_{l, K, 1}^{\alg}) \,\,\, \simeq \,\,\, 
\pi_{0}(\AST_{K} (V, \psi)) \,\,\, \simeq \,\,\, 
\pi_{0}^{}(\ST_{K}(V, \psi)).
\label{ConnCompIsom}\end{equation}
\end{proposition} 

\begin{proof} 
Given our assumptions, this follows by an argument similar to the proof of 
\cite[Lemma~2.8]{FKRS12}.
\end{proof}

\begin{remark}
Fixing $l$ and $K$ and assuming that $\gpast_{l, K}$ is an isomorphism we can prove
\eqref{ConnCompIsom} for these $l$ and $K$, cf. the proof of \cite[Lemma~2.8]{FKRS12}. 
\label{Specification of ConnCompIsom to l and K}
\end{remark}

\begin{remark} 
Assume that Algebraic Sato--Tate Conjecture 
\ref{general algebraic Sato Tate conj.} holds for $(V, \psi)$. 
Then obviously the Sato--Tate group $\ST_{K}(V, \psi)$ is independent of $l$. 
Take a prime $v$ in $\mathcal{O}_K$
and take a Frobenius element $\Fr_{v}$ in $G_K$.
Following \cite[\S 8.3.3]{Se12} (cf.  \cite[Def.\ 2.9]{FKRS12}) one can make the following construction. Let $s_v$ be the semisimple part in $\SL_V(\C)$ of the following normalized Frobenius element:
$$
q_{v}^{-\frac{n}{2}} \rho_{l} (\Fr_{v}) \,\, \in \,\, 
 G_{l, K, 1}^{\alg} (\C) \,\, \simeq \,\, \AST_{K} (V, \psi) (\C) \,\, \subset 
 \,\, \Iso_{(V, \psi)} (\C) \,\, \subset \,\, \SL_{V} (\C);
$$ 
since the family $(\rho_{l})$ is strictly compatible, the conjugacy class ${\rm{conj}}(s_v)$ in 
$\SL_V(\C)$ is independent of $l$. By \cite[Theorem 15.3 (c) p.\ 99]{Hu}, the semisimple part
of $q_{v}^{-\frac{n}{2}} \rho_{l} (\Fr_{v})$ considered in 
$\Iso_{(V, \psi)} (\C)$ and in $\AST_{K} (V, \psi) (\C)$ is again $s_v$, but its conjugacy classes
${\rm{conj}}(s_v)$ in $\Iso_{(V, \psi)} (\C)$ and ${\rm{conj}}(s_v)$ in $\AST_{K} (V, \psi) (\C)$ might depend
on $l$. Obviously ${\rm{conj}}(s_v) \subset \AST_{K} (V, \psi) (\C)$ is independent of the 
choice of a Frobenius element $\Fr_{v}$ over $v$ and contains the
semisimple parts of all the elements of ${\rm{conj}}(q_{v}^{-\frac{n}{2}} \rho_{l} (\Fr_{v}))$
in $\AST_{K} (V, \psi) (\C)$. Moreover, the elements in ${\rm{conj}} (s_v)$ have eigenvalues of 
complex absolute value $1$ by our assumptions, so there is some conjugate of $s_v$ contained in $\ST_{K}(V, \psi)$. This allows us to make sense of the following conjecture.
\label{Sato--Tate set up}
\end{remark}

\begin{conjecture}\label{general Sato Tate conj.} 
(Sato--Tate conjecture) 
The conjugacy classes ${\rm{conj}}(s_v)$ in $\ST_{K}(V, \psi)$ are equidistributed in 
${\rm{conj}} (\ST_{K}(V, \psi))$ with respect to the measure induced by the 
Haar measure of $\ST_{K}(V, \psi)$.
\end{conjecture}

\begin{remark}
\label{correction of remark 5.8 [BK2]}
In \cite[Remark 5.8]{BK2} we claimed that $s_v$ is independent of $l$. That claim was too optimistic;
this is still not known in general. As pointed out in Remark \ref{Sato--Tate set up} above, we 
only know that 
${\rm{conj}}(s_v)$ in $\SL_V(\C)$ is independent of $l$; by contrast,
${\rm{conj}}(s_v)$ in $\Iso_{(V, \psi)} (\C)$ and ${\rm{conj}}(s_v)$ in $\AST_{K} (V, \psi) (\C)$ 
might depend on $l$.
\end{remark}

\begin{theorem} \label{connected component of ASTK} Assume Conjecture
\ref{general algebraic Sato Tate conj.}${(\rm{a})}$ holds for $(V, \psi)$. Assume in addition that for some $l$,
the group  $\gpP_{\rm{S}} (V_{l}, \psi_{l})$ is trivial and the maps $\gpast_{l, K}$ and $\gpast_{l, K_{0}}$ are isomorphisms. Let $L/K_{0}$ be a finite Galois extension. Then: 
\begin{itemize}
\item[(a)] $\AST_{K_{0}} (V, \psi) = \AST_{K} (V, \psi)^{\circ}$,
\item[(b)] $\ST_{K_{0}} (V, \psi) = \ST_{K} (V, \psi)^{\circ}$ up to conjugation
in $\AST_{K} (V, \psi) (\C)$,
\item[(c)] $\AST_{K_{0}} (V, \psi) = \AST_{L} (V, \psi)$,
\item[(d)] $\ST_{K_{0}} (V, \psi) = \ST_{L} (V, \psi)$ up to conjugation
in $\AST_{K_{0}} (V, \psi) (\C)$. 
\end{itemize}
\end{theorem}
\begin{proof}
The proof is similar to the proof of \cite[Prop. 6.4]{BK2}. However for arbitrary weight, we need to
assume that $\gpP_{\rm{S}} (V_{l}, \psi_{l})$ is trivial; we then apply Proposition \ref{connected components iso}
to relate the Sato--Tate groups to the parity group $\gpP_{\rm{S}} (V_{l}, \psi_{l})$.   
\end{proof}

\begin{proposition} 
Assume that Conjecture~\ref{general algebraic Sato Tate conj.} holds for $(V, \psi)$ and that 
the following conditions hold for some $l$:
\begin{itemize}
\item[(1)] $K_0 \, \cap \, K(\mu_{\bar{l}}^{\otimes \, n}) \, = \, K$,
\item[(2)] $1 + l\Z_{l} \,\, {\rm{Id}}_{V_l}  \, \subset \, \rho_{l} (G_K)$.
\end{itemize}
Then every subgroup of $\ST_{K} (V, \psi)$ containing $\ST_{K_{0}} (V, \psi)$ is of the 
form $\ST_{L} (V, \psi)$, for a unique field subextension $K \subset L \subset K_{0}$.
\label{Intermediate subgroup between STK (V, psi) and STK0 (V, psi) is STL (V, psi)}
\end{proposition}
\begin{proof}
By the definition of $K_{0}$ (cf. \cite[(6.4)]{BK2}), Proposition \ref{connected components iso}, Theorem 
\ref{jK0K and Zar1 are isomorphisms}, and the bottom row of Diagram
\ref{diagram of natural maps from Galois to Zariski closures}, there are natural isomorphisms:
\begin{equation}
G(K_{0} / K) \simeq G_K / G_{K_{0}} \simeq \ST_{K} (V, \psi) / ST_{K_{0}} (V, \psi).
\label{isomorphism between G(K0 / K) and STK / STK0}
\end{equation}
By the fundamental theorem 
of Galois theory, any subgroup of $\ST_{K} (V, \psi)$ containing $\ST_{K_{0}} (V, \psi)$ corresponds via 
\eqref{isomorphism between G(K0 / K) and STK / STK0} to the subgroup $G(K_{0} / L) \subseteq 
G(K_{0} / K)$ for a unique subextension $K \subset L \subset K_{0}$. Hence working with the base field $L$ instead of $K$ (see Remark 
\ref{Base change from K to L in basic theorems}) we obtain again
by Proposition \ref{connected components iso}, Theorem \ref{jK0K and Zar1 are isomorphisms},
and the bottom row of Diagram \ref{diagram of natural maps from Galois to Zariski closures} 
the following natural isomorphisms:
\begin{equation}
G(K_{0} / L) \simeq G_L / G_{K_{0}} \simeq \ST_{L} (V, \psi) / ST_{K_{0}} (V, \psi),
\label{isomorphism between G(K0 / L) and STL / STK0}
\end{equation}
naturally compatible with the isomorphisms 
\eqref{isomorphism between G(K0 / K) and STK / STK0}.
\end{proof}

\begin{proposition} \label{minimal field of connectedness independent of l}
Suppose that for some $(V, \psi)$,
Conjecture \ref{general algebraic Sato Tate conj.} holds for $K$ and $K_0$ and the group 
$\gpP_{\rm{S}} (V_{l}, \psi_{l})$ is trivial. 
Then the field $K_0$ is independent of $l$.  
\end{proposition}

\begin{proof} Cf. the proof of \cite[Prop. 6.5]{BK2}.
\end{proof}

The following Theorem is proven in \cite[Theorem 6.12]{BK2}. The assumption that 
$\gpast_{l, K_{0}}$ is an isomorphism is missing in \textit{loc. cit.} so we add it here.

\begin{theorem} \label{STK iff STK0} Assume that Conjecture \ref{general algebraic Sato Tate conj.} ${\rm{(a)}}$ holds for $(V,\psi)$ and the following conditions hold for some $l$:
\begin{itemize}
\item[(1)] $K_0 \, \cap \, K(\mu_{\bar{l}}^{\otimes \, n}) \, = \, K$;
\item[(2)] $1 + l\Z_{l} \,\, {\rm{Id}}_{V_l}  \, \subset \, \rho_{l} (G_K)$;
\item[(3)] $\gpast_{l, K}$ and $\gpast_{l, K_{0}}$ are isomorphisms.
\end{itemize}
Then: 
\begin{align}
\AST_{K} (V, \psi)_{\Q_l}  \, &= \, {\bigsqcup}_{\tilde{\sigma}_1}  \, 
\rho_{l} (\tilde{\sigma}_1) \, \AST_{K_{0}} (V, \psi)_{\Q_l},
\label{decomposition of ASTKQl into cosets over Galois representatives} \\
\ST_{K} (V, \psi)  \, &= \, {\bigsqcup}_{\tilde{\sigma}_1}  \, 
\rho_{l} (\tilde{\sigma}_1) \, \ST_{K_{0}} (V, \psi),
\label{decomposition of STK into cosets over Galois representatives}
\end{align}
where $\tilde{\sigma}_1$ runs over a set of coset representatives of $G_{K_0}$ in $G_K$ such that
${\rho}_{l} (\tilde{\sigma}_{1}) \in {\rho}_{l} (G_K)_1$.
\end{theorem}

\begin{corollary} \label{STK2 iff STK1} Consider field extensions $K \subset K_{2} \subset K_{1} \subset 
K_{0}$. Under the assumptions of Theorem \ref{STK iff STK0}, we have:
\begin{align}
\AST_{K_{2}} (V, \psi)_{\Q_l}  \, &= \, {\bigsqcup}_{\tilde{\sigma}_1}  \, 
\rho_{l} (\tilde{\sigma}_1) \, \AST_{K_{1}} (V, \psi)_{\Q_l},
\label{decomposition of ASTK2 into cosets over Galois representatives wr to ASTK1} \\
\ST_{K_{2}} (V, \psi)  \, &= \, {\bigsqcup}_{\tilde{\sigma}_1}  \, 
\rho_{l} (\tilde{\sigma}_1) \, \ST_{K_{1}} (V, \psi),
\label{decomposition of STK2 into cosets over Galois representatives wr to STK1}
\end{align}
where $\tilde{\sigma}_1$ runs over a set of coset representatives of $G_{K_1}$ in $G_{K_2}$
such that ${\rho}_{l} (\tilde{\sigma}_{1}) \in {\rho}_{l} (G_{K_2})_1$.
\end{corollary}
\begin{proof} This follows from Theorem 
\ref{STK iff STK0} and Remark \ref{Base change from K to L in basic theorems}.  
\end{proof}

To be consistent with Serre's approach, the algebraic Sato--Tate conjecture should have now the following 
form. 

\begin{conjecture}\label{general algebraic Sato Tate conj. Serre's approach} {\,}
\begin{itemize}
\item[(a)]
For every finite extension $K/F$, there exist a natural-in-$K$ 
reductive algebraic group $\widetilde{\AST_{K}} (V, \psi) 
\subset \Iso_{(V, \psi)}$ over $\Q$ and a natural-in-$K$ monomorphism  of group schemes for every $l$:
\begin{equation}
\widetilde{\gpast}_{l, K}\colon \widetilde{G_{l, K, 1}^{\alg}} \,\, 
{\stackrel{}{\hookrightarrow}} \,\, \widetilde{\AST_{K}} (V, \psi)_{\Q_l}.
\label{Algebraic Sato--Tate monomorphism. Serre's approach}
\end{equation} 
In addition the natural embedding $\widetilde{G_{l, K, 1}^{\alg}} \subset \Iso_{(V_l, \psi_{l})}$ factors through
$\widetilde{\gpast}_{l, K}$ and the natural embedding $\widetilde{\AST_{K}} (V, \psi)_{\Q_l} \subset \Iso_{(V_l, \psi_{l})}$. 

\item[(b)]
The map
\eqref{Algebraic Sato--Tate monomorphism. Serre's approach} is an isomorphism: 
\begin{equation}
\widetilde{\gpast}_{l, K}\colon \widetilde{G_{l, K, 1}^{\alg}} \,\, 
{\stackrel{\simeq}{\longrightarrow}} \,\,  \widetilde{\AST_{K}} (V, \psi)_{\Q_l}.
\label{Algebraic Sato--Tate conjecture. Serre's approach}
\end{equation} 
\end{itemize}
\end{conjecture}

\begin{definition}
The group $\widetilde{\AST_{K}} (V, \psi)$ is called the \emph{algebraic Sato--Tate group} associated with the family of representations $(\widetilde{\rho}_l)_l$. Any maximal compact subgroup of $\widetilde{\AST_{K}} (V, \psi)(\C)$ is called the \emph{Sato--Tate group} and denoted 
$\widetilde{\ST_{K}} (V, \psi)$. 
\label{definition of  wildtildeASTK(V psi) and wildtildeSTK(V psi)}
\end{definition}

\begin{proposition}
\label{connected components iso for tilde}
Assume that Algebraic Sato--Tate Conjecture~\ref{general algebraic Sato Tate conj. Serre's approach} holds for $(V, \psi)$. Then there are natural isomorphisms
\begin{equation}
\label{wide tilde ConnCompIsom}
\pi_{0} (\widetilde{G_{l, K, 1}^{\alg}}) \,\,\, \simeq \,\,\, 
\pi_{0}(\widetilde{\AST_{K}} (V, \psi)) \,\,\, \simeq \,\,\, 
\pi_{0}^{}(\widetilde{\ST_{K}}(V, \psi)).
\end{equation}
\end{proposition} 
\begin{proof} 
This follows by the same proof as Proposition \ref{connected components iso}.
\end{proof}

\begin{remark}
Fixing $l$ and $K$ and assuming that $\widetilde{\gpast}_{l, K}$ is an isomorphism, we can prove
\eqref{wide tilde ConnCompIsom} for these $l$ and $K$ cf. the proof of \cite[Lemma~2.8]{FKRS12}. 
\label{Specification of tilde ConnCompIsom to l and K}
\end{remark}

\begin{theorem} Assume Conjecture
\ref{general algebraic Sato Tate conj. Serre's approach} ${(\rm{a})}$ 
holds for $(V, \psi)$. Assume in addition that for some $l$, the maps $\gpast_{l, K}$ and $\gpast_{l, K_{0}}$ are isomorphisms. Let $L/K_{0}$ be a finite Galois extension. Then: 
\begin{itemize}
\item[(a)] $\widetilde{\AST_{K_{0}}} (V, \psi) = \widetilde{\AST_{K}} (V, \psi)^{\circ}$,
\item[(b)] $\widetilde{\ST_{K_{0}}} (V, \psi) = \widetilde{\ST_{K}} (V, \psi)^{\circ}$ up to conjugation
in $\widetilde{\AST_{K}} (V, \psi) (\C)$,
\item[(c)] $\widetilde{\AST_{K_{0}}} (V, \psi) = \widetilde{\AST_{L}} (V, \psi)$,
\item[(d)] $\widetilde{\ST_{K_{0}}} (V, \psi) = \widetilde{\ST_{L}} (V, \psi)$ up to conjugation
in $\widetilde{\AST_{K_{0}}} (V, \psi) (\C)$. 
\end{itemize}
\label{connected component of tilde AST K}
\end{theorem}
\begin{proof}
Upon Corollary \ref{Connected component GlKlK1alg in form of widetilde(GlK01alg)circ} and the assumptions we observe that $\widetilde{\AST_{K_{0}}} (V, \psi)$ is connected. Hence by applying Proposition \ref{connected components iso for tilde}, the proof is similar to the proof of \cite[Prop. 6.4]{BK2}.   
\end{proof}

\begin{proposition} 
Assume that Conjecture~\ref{general algebraic Sato Tate conj. Serre's approach} holds for $(V, \psi)$. 
For some $l$, for a prime $v_{0} \, | \, l$ in $\mathcal{O}_{K_0}$, assume that 
$\chi_{c} (G_K) = \chi_{c} (G_{{K_{v_0}}})$.  Then every subgroup of $\widetilde{\ST_{K}} (V, \psi)$ containing $\widetilde{\ST_{K_0}} (V, \psi)$ has the form $\widetilde{\ST_{L}} (V, \psi)$ for a unique field subextension $K \subset L \subset K_{0}$.
\label{Intermediate subgroup between tilde STK (V, psi) and tilde STK0 (V, psi) is STL (V, psi)}
\end{proposition}
\begin{proof}
By Lemma \ref{lemma concerning diagram of natural maps from Galois to Zariski closures},
Theorem \ref{tilde jK0K and tilde Zar1 are isomorphisms}, and Proposition \ref{connected components iso for tilde},
there are natural isomorphisms:
\begin{equation}
G(K_{0} / K) \simeq G_K / G_{K_{0}} \simeq \widetilde{\ST_{K}} (V, \psi) / \widetilde{ST_{K_{0}}} (V, \psi).
\label{isomorphism between G(K0 / K) and tilde STK / tilde STK0}
\end{equation}
By the fundamental theorem 
of Galois theory, any subgroup of $\widetilde{\ST_{K}} (V, \psi)$ containing $\widetilde{\ST_{K_{0}}} (V, \psi)$ corresponds via 
\eqref{isomorphism between G(K0 / K) and tilde STK / tilde STK0} to the subgroup $G(K_{0} / L) \subseteq 
G(K_{0} / K)$ for a unique subextension $K \subset L \subset K_{0}$. Hence working with the base field $L$ instead of $K$ (see Remark 
\ref{Base change from K to L in basic theorems with tilde}) we obtain again
by Lemma \ref{lemma concerning diagram of natural maps from Galois to Zariski closures},
Theorem \ref{tilde jK0K and tilde Zar1 are isomorphisms}, and Proposition \ref{connected components iso for tilde}
the natural isomorphisms:
\begin{equation}
G(K_{0} / L) \simeq G_L / G_{K_{0}} \simeq 
\widetilde{\ST_{L}} (V, \psi) / \widetilde{ST_{K_{0}}} (V, \psi)
\label{isomorphism between G(K0 / L) and tilde STL / tilde STK0}
\end{equation}
which are naturally compatible with the isomorphisms 
\eqref{isomorphism between G(K0 / K) and tilde STK / tilde STK0}.
\end{proof}

\begin{proposition} \label{minimal field of connectedness independent of l. tilde case}
Suppose that for some $(V, \psi)$,
Conjecture \ref{general algebraic Sato Tate conj.} holds for $K$ and $K_0$ and the group 
$\gpP_{\rm{S}} (V_{l}, \psi_{l})$ is trivial. 
Then the field $K_0$ is independent of $l$. 
\end{proposition}

\begin{proof} Cf. the proof of \cite[Prop. 6.5]{BK2}.
\end{proof}

\noindent
\begin{theorem} \label{tilde STK iff tilde STK0} Assume that 
Conjecture~\ref{general algebraic Sato Tate conj. Serre's approach}{\rm{(a)}} holds for $(V,\psi)$.
For some $l$, for a prime $v_{0} \, | \, l$ in $\mathcal{O}_{K_0}$, assume that 
$\chi_{c} (G_K) = \chi_{c} (G_{{K_{v_0}}})$ and the maps $\widetilde{\gpast}_{l, K}$ 
and $\widetilde{\gpast}_{l, K_{0}}$ are isomorphisms.
Then: 
\begin{align}
\widetilde{\AST_{K}} (V, \psi)_{\Q_l}  \, &= \, {\bigsqcup}_{\tilde{\sigma}_1}  \, 
\widetilde{\rho}_{l} (\tilde{\sigma}_1) \, \widetilde{\AST_{K_{0}}} (V, \psi)_{\Q_l},
\label{decomposition of tilde ASTKQl into cosets over Galois representatives} \\
\widetilde{\ST_{K}} (V, \psi)  \, &= \, {\bigsqcup}_{\tilde{\sigma}_1}  \, 
\widetilde{\rho}_{l} (\tilde{\sigma}_1) \, \widetilde{\ST_{K_{0}}} (V, \psi),
\label{decomposition of tilde STK into cosets over Galois representatives}
\end{align}
where $\tilde{\sigma}_1$ runs over a set of coset representatives of $G_{K_0}$ in $G_K$ such that
$\widetilde{\rho}_{l} (\tilde{\sigma}_{1}) \in \widetilde{\rho}_{l} (G_K)_1$.
\end{theorem}
\begin{proof} From Theorem \ref{tilde jK0K and tilde Zar1 are isomorphisms} we have:
\begin{equation}
\widetilde{G_{l, K, 1}^{\alg}} \,\, = \,\, 
{\bigsqcup}_{\tilde{\sigma}_1  G_{K_0}}  \,\, \widetilde{\rho}_{l} (\tilde{\sigma}_{1})
\widetilde{\, G_{l, K_0, 1}^{\alg}}. 
\label{tilde GlK 1 = bigsqcup tilde rho l tilde sigma tilde GlK0 1}
\end{equation}
Because the maps $\widetilde{\gpast}_{l, K}$ and $\widetilde{\gpast}_{l, K_{0}}$ are isomorphisms by assumption, it follows by Proposition \ref{connected components iso for tilde} and Theorem \ref{connected component of tilde AST K}
that the equality \eqref{decomposition of tilde ASTKQl into cosets over Galois representatives}
holds. Now taking $\C$-valued points in 
\eqref{decomposition of tilde ASTKQl into cosets over Galois representatives} and restricting to maximal compact subgroups proves \eqref{decomposition of tilde STK into cosets over Galois representatives}.
\end{proof}

\begin{corollary} \label{tilde STK2 iff tilde STK1} Consider field extensions $K \subset K_{2} \subset K_{1} \subset 
K_{0}$. Under the assumptions of Theorem \ref{tilde STK iff tilde STK0}, we have:
\begin{align}
\widetilde{\AST_{K_{2}}} (V, \psi)_{\Q_l}  \, &= \, {\bigsqcup}_{\tilde{\sigma}_1}  \, 
\rho_{l} (\tilde{\sigma}_1) \, \widetilde{\AST_{K_{1}}} (V, \psi)_{\Q_l},
\label{decomposition of tilde ASTK2 into cosets over Galois representatives wr to tilde ASTK1} \\
\widetilde{\ST_{K_{2}}} (V, \psi)  \, &= \, {\bigsqcup}_{\tilde{\sigma}_1}  \, 
\rho_{l} (\tilde{\sigma}_1) \, \widetilde{\ST_{K_{1}}} (V, \psi),
\label{decomposition of tilde STK2 into cosets over Galois representatives wr to tilde STK1}
\end{align}
where $\tilde{\sigma}_1$ runs over a set of coset representatives of $G_{K_1}$ in $G_{K_2}$
such that $\widetilde{\rho}_{l} (\tilde{\sigma}_{1}) \in \widetilde{\rho}_{l} (G_{K_2})_1$.
\end{corollary}
\begin{proof} This follows from Theorem 
\ref{tilde STK iff tilde STK0} and Remark \ref{Base change from K to L in basic theorems with tilde}.  
\end{proof}

\begin{lemma} For $i=1,2$, let $V_i \subset W$ be an inclusion of vector spaces over a field $L$. Let $M/L$ be a field extension. 
If $V_1 \otimes_{L} M \subset V_2 \otimes_{L} M$ in $W \otimes_{L} M$, then 
$V_1 \subset V_2$. In particular if $V_1 \otimes_{L} M = V_2 \otimes_{L} M \subset W\otimes_{L} M$, then 
$V_1 = V_2$.
\label{equality of 2 vector spaces over L follows by their equality when tensored with  M}
\end{lemma}
\begin{proof} It holds because $V_1 \otimes_{L} M + V_2 \otimes_{L} M = (V_1 + V_2) \otimes_{L} M$ in $W \otimes_{L} M$ and $L \hookrightarrow M$ is faithfully flat.
\end{proof}

\begin{corollary}
Let $X$ and $Y$ be closed subschemes of $\mathbb{A}^{n}_L$ for some field $L$. Let $M/L$ be a field extension. 
If $X \times_{{\rm{Spec}} \, L} {\rm{Spec}} \, M \,\, \subset \,\, Y \times_{{\rm{Spec}} \, L} {\rm{Spec}} \, M$,
then $X \subset Y$.
\label{containment of 2 schemes over L follows from their containment after base change to M}
\end{corollary}
\begin{proof}
Let $I(X)$ (resp. $I(Y)$) be the ideal in $L[t_1, \dots, t_n]$ defining $X \subset \mathbb{A}^{n}_L$
(resp. $Y \subset \mathbb{A}^{n}_L$). By assumption $I(Y) \otimes_{L} \, M \subset 
I(X) \otimes_{L} \, M$. Hence by Lemma \ref{equality of 2 vector spaces over L follows by their equality when tensored with  M} $I(Y) \subset I(X)$. So $X \subset Y \subset \mathbb{A}^{n}_L$.   
\end{proof}

\begin{corollary}
Let $X = {\rm{Spec}} \, A$ and $Y = {\rm{Spec}} \, B$ be affine schemes over $L$. Let $M/L$ be a field extension. 
Let $f, g\colon X \rightarrow Y$ be two morphisms such that the morphisms $f \times {\rm{Id}},\, g \times {\rm{Id}}\colon X \times_{{\rm{Spec}} \, L} {\rm{Spec}} \, M \,\, \rightarrow \,\, Y \times_{{\rm{Spec}} \, L} {\rm{Spec}} \, M$ are equal. Then $f = g$.
\label{equality of two morphisms of schemes over L follows from their equality after base change to M}
\end{corollary}
\begin{proof}
Let $f^{\sharp}, g^{\sharp}\colon B \rightarrow A$ be the $L$-algebra homomorphisms corresponding to $f$ and $g$,
respectively. By assumption, the $M$-algebra homomorphisms $f^{\sharp} \otimes_{L} \, 1, \, g^{\sharp} \otimes_{L} \, 1 \colon B \otimes_{L} \, M  \rightarrow A \otimes_{L} \, M$ are equal. We observe that 
$f^{\sharp} \otimes_{L} \, 1 - g^{\sharp} \otimes_{L} 1 = 
(f^{\sharp} - g^{\sharp}) \otimes_{L} \, 1 = 0$ as $M$-vector space homomorphisms. Let 
$W := {\rm{ker}} (f^{\sharp} - g^{\sharp})$. Then $W \otimes_{L} \, M = B \otimes_{L} \, M$.
Hence by Lemma \ref{equality of 2 vector spaces over L follows by their equality when tensored with  M} 
it follows that $W = B$, so $f^{\sharp} = g^{\sharp}$. Therefore $f = g$.
\end{proof}

\begin{corollary}
Assume Conjectures \ref{general algebraic Sato Tate conj.} and 
\ref{general algebraic Sato Tate conj. Serre's approach}. Then:
\begin{itemize}
\item[(a)] $\AST_{K} (V, \psi)^{\circ} = \AST_{K_{0}} (V, \psi)^{\circ} = \widetilde{\AST_{K_{0}}} (V, \psi)$,
\item[(b)] $\ST_{K} (V, \psi)^{\circ} = \ST_{K_{0}} (V, \psi)^{\circ} = 
\widetilde{\ST_{K_{0}}} (V, \psi)$ up to conjugation in $\AST_{K} (V, \psi) (\C)$.
\end{itemize} 
\label{connected components of AST K (V, psi), ST K (V, psi) and widetilde AST K (V, psi), widetilde ST K (V, psi)}
\end{corollary}
\begin{proof} To prove (a) it is enough to prove $\AST_{K_{0}} (V, \psi)^{\circ} = \widetilde{\AST_{K_{0}}} (V, \psi)$.
By \cite[p. 218]{Hu} the algebraic group $\AST_{K_{0}} (V, \psi)^{\circ}$ is defined over $\Q$.
Consider the $\Q$-vector spaces $V_1 := {\rm{Lie}} (\AST_{K_{0}} (V, \psi)^{\circ})$,
$V_2 := {\rm{Lie}} (\widetilde{\AST_{K_{0}}} (V, \psi))$, $W := 
{\rm{Lie}} (\Iso_{(V, \psi)})$. Obviously $V_i \subset W$ for $i = 1, 2$.  It follows by
\eqref{Algebraic Sato--Tate equality generalized}, 
\eqref{Algebraic Sato--Tate conjecture. Serre's approach}, and 
Corollary \ref{Connected component GlKlK1alg in form of widetilde(GlK01alg)circ} that
$V_1 \otimes_{\Q} \Q_{l} = V_2 \otimes_{\Q} \Q_{l} \subset W\otimes_{\Q} \Q_{l}$.
By Lemma \ref{equality of 2 vector spaces over L follows by their equality when tensored with  M}
we have $V_1 = V_2$. In the end by \cite[Theorem on p. 87]{Hu} we obtain $\AST_{K_{0}} (V, \psi)^{\circ} = \widetilde{\AST_{K_{0}}} (V, \psi)$. Hence (a) holds; this in turn implies (b).
\end{proof}

\begin{remark} Under Conjecture \ref{general algebraic Sato Tate conj. Serre's approach},
the Sato--Tate group $\widetilde{\ST_{K}}(V, \psi)$ is independent of $l$. 
The normalization of $\widetilde{\rho}_l (\Fr_{v})$ (cf. \cite[p. 113]{Se12}) gives the following normalized Frobenius 
element:
\begin{gather*}
( q_{v}^{-\frac{n}{2}} {\rm{Id}}_{V_l}, \, q_{v}^{-1} {\rm{Id}}_{\Q_{l} (1)}) \,
\widetilde{\rho}_l (\Fr_{v}) \in \widetilde{G_{l, K, 1}^{\alg}} (\C) \,\, \simeq \,\, \widetilde{\AST_{K}} (V, \psi) (\C) \\
 \subset  \,\, \Iso_{(V, \psi)} (\C) \,\, \subset \,\, \SL_{V} (\C).
\end{gather*}
The monomorphism $\pi_{1}$ in Diagrams \ref{diagram compatibility of tilde(GlK1alg) with GlK1alg} and
 \ref{the diagram compatibility of tilde(GlK1alg) with GlK1alg and pi0tilde(GlK1alg) with pi0GlK1alg}  
sends $( q_{v}^{-\frac{n}{2}} {\rm{Id}}_{V_l}, \, q_{v}^{-1} {\rm{Id}}_{\Q_{l} (1)}) \,
\widetilde{\rho}_l (\Fr_{v})$ onto $q_{v}^{-\frac{n}{2}} {\rho}_l (\Fr_{v})$. Let $\tilde{s}_v$ denote the semisimple part in $\SL_V(\C)$ of the element $( q_{v}^{-\frac{n}{2}} {\rm{Id}}_{V_l}, \, q_{v}^{-1} {\rm{Id}}_{\Q_{l} (1)}) \, \widetilde{\rho}_l (\Fr_{v})$. Observe that 
$\tilde{s}_v = (s_v, \, q_{v}^{-1} {\rm{Id}}_{\Q_{l} (1)})$. The family $(\widetilde{\rho}_l)$ is also strictly compatible. Hence the conjugacy class ${\rm{conj}}(\tilde{s}_v)$ in 
$\SL_V(\C)$ is independent of $l$. Based on \cite[Theorem 15.3 (c) p.\ 99]{Hu}, the semisimple part $\tilde{s}_v$ considered in $\Iso_{(V, \psi)} (\C)$ and in $\widetilde{\AST_{K}} (V, \psi) (\C)$ is again $\tilde{s}_v$, but its conjugacy classes ${\rm{conj}}(\tilde{s}_v)$ in $\Iso_{(V, \psi)} (\C)$ and ${\rm{conj}}(\tilde{s}_v)$ in 
$\widetilde{\AST_{K}} (V, \psi)  (\C)$ might depend on $l$. Obviously ${\rm{conj}}(\tilde{s}_v) \subset 
\widetilde{\AST_{K}} (V, \psi) (\C)$ is independent of the choice of a Frobenius element $\Fr_{v}$ over $v$ and contains the semisimple parts of all the elements of ${\rm{conj}}(( q_{v}^{-\frac{n}{2}} {\rm{Id}}_{V_l}, \, q_{v}^{-1} {\rm{Id}}_{\Q_{l} (1)}) \, \widetilde{\rho}_l (\Fr_{v}))$
in $\widetilde{\AST_{K}} (V, \psi) (\C)$. The elements in ${\rm{conj}}(\tilde{s}_v) $ have eigenvalues of 
complex absolute value $1$ by our assumptions, so there is some conjugate of $\tilde{s}_v$ contained in 
$\widetilde{\ST_{K}}(V, \psi)$. 
\label{tilde Sato--Tate set up}
\end{remark}

\begin{conjecture}\label{general Sato Tate conj. tilde} 
(Sato--Tate conjecture) 
The conjugacy classes ${\rm{conj}}(\tilde{s}_v)$ in $\widetilde{\ST_{K}}(V, \psi)$ are equidistributed in 
${\rm{conj}} (\widetilde{\ST_{K}}(V, \psi))$ with respect to the measure induced by the 
Haar measure of $\widetilde{\ST_{K}}(V, \psi)$.   
\end{conjecture}

\begin{proposition} {\,} 
\begin{itemize}
\item[(a)] Under Conjecture \ref{general algebraic Sato Tate conj.} the algebraic Sato--Tate group 
$\AST_{K} (V, \psi)$ is uniquely defined. 
\item[(b)] Under Conjecture \ref{general algebraic Sato Tate conj. Serre's approach} 
the algebraic Sato--Tate group $ \widetilde{\AST_{K}} (V, \psi)$ is uniquely defined.  
\item[(c)]  Under Conjectures \ref{general algebraic Sato Tate conj.} and 
\ref{general algebraic Sato Tate conj. Serre's approach} we have 
$\widetilde{\AST_{K}} (V, \psi) \subset \AST_{K} (V, \psi)$. Hence 
$\widetilde{\ST_{K}} (V, \psi) \subset \ST_{K} (V, \psi)$
up to conjugation in $\AST_{K} (V, \psi)(\mathbb{C})$.
\end{itemize}
\label{uniqueness of algebraic Sato--Tate groups}
\end{proposition}
\begin{proof} Claims (a) and (b) follow from Corollary \ref{containment of 2 schemes over L follows from their containment after base change to M}. Claim (c) follows from 
\eqref{decomposition of GlK1alg into tilde(GlK1alg) cup -Id tilde(GlK1alg)} and Corollary
\ref{containment of 2 schemes over L follows from their containment after base change to M}.
\end{proof}

\begin{proposition} Conjecture \ref{general algebraic Sato Tate conj. Serre's approach} implies
Conjecture \ref{general algebraic Sato Tate conj.}. 
\label{Conjectures general algebraic Sato Tate conj. implies Algebraic Sato--Tate conjecture. Serre's approach}
\end{proposition}
\begin{proof} Assume that Conjecture \ref{general algebraic Sato Tate conj. Serre's approach} holds.
Then we put 
$$
\AST_{K} (V, \psi) := \widetilde{\AST_{K}} (V, \psi) \, \cup \, - {\rm{Id}}_{V} \, 
\widetilde{\AST_{K}} (V, \psi).
$$
It is obvious that $\AST_{K} (V, \psi)$ is an algebraic group over $\Q$ and Conjecture 
\ref{general algebraic Sato Tate conj.} holds by 
\eqref{decomposition of GlK1alg into tilde(GlK1alg) cup -Id tilde(GlK1alg)}. 
\end{proof}

\begin{remark}
\label{general algebraic Sato Tate conj. K implies general algebraic Sato Tate conj. Serre's approach K0 }
Assume Conjecture \ref{general algebraic Sato Tate conj.}. For $K_{0}$ in place of $K$ we can put
$$
\widetilde{\AST_{K_{0}}} (V, \psi) \, := \, \AST_{K} (V, \psi)^{\circ}
\, = \, \AST_{K_{0}} (V, \psi)^{\circ}
$$
which is an algebraic group defined over $\Q$ cf. \cite[p. 218]{Hu}. 
Hence for the base field $K_{0}$, Conjecture \ref{general algebraic Sato Tate conj. Serre's approach} holds with 
$\widetilde{\AST_{K_0}} (V, \psi) = \AST_{K_{0}} (V, \psi) ^{\circ}$ by 
Corollary \ref{Connected component GlKlK1alg in form of widetilde(GlK01alg)circ}. 
By \eqref{decomposition of GlK1alg into tilde(GlK1alg) cup -Id tilde(GlK1alg)} with $K_{0}$ in place of $K$,
we have:
\begin{equation}
\AST_{K_{0}} (V, \psi)  = \widetilde{\AST_{K_0}} (V, \psi)  \, \cup  \, - {\rm{Id}}_{V} \, 
\widetilde{\AST_{K_0}} (V, \psi). 
\label{composition of ASTK0 V psi into cosets widetilde(ASTK0 Vpsi) and - widetilde(ASTK0 Vpsi)}
\end{equation}  
Moreover if $- {\rm{Id}}_{V_l} \notin \widetilde{G_{l, K_{0}, 1}^{\alg}}$ then
$- {\rm{Id}}_{V} \notin \widetilde{\AST_{K_{0}}} (V, \psi)$ and  
\begin{equation}
\AST_{K_{0}} (V, \psi)  = \widetilde{\AST_{K_0}} (V, \psi)  \, \sqcup  \, - {\rm{Id}}_{V} \, 
\widetilde{\AST_{K_0}} (V, \psi). 
\label{composition of ASTK0 V psi into disjoint cosets widetilde(ASTK0 Vpsi) and
- widetilde(ASTK0 Vpsi)}
\end{equation}  
\end{remark}

\begin{remark}
Assume Conjecture \ref{general algebraic Sato Tate conj.}. If
$- {\rm{Id}}_{V_l} \in \widetilde{G_{l, K, 1}^{\alg}}$ then 
$\widetilde{G_{l, K, 1}^{\alg}} = G_{l, K, 1}^{\alg}$ by the equality
\eqref{decomposition of GlK1alg into tilde(GlK1alg) cup -Id tilde(GlK1alg)}. 
Putting:
$$
\widetilde{\AST_{K}} (V, \psi) \, := \, \AST_{K} (V, \psi).
$$ 
we observe that Conjecture \ref{general algebraic Sato Tate conj. Serre's approach} holds in this case.
\medskip

If $- {\rm{Id}}_{V_l} \notin \widetilde{G_{l, K, 1}^{\alg}}$, the equality 
\eqref{decomposition of GlK1alg into tilde(GlK1alg) cup -Id tilde(GlK1alg)} has the form
\begin{equation}
G_{l, K, 1}^{\alg} = \widetilde{G_{l, K, 1}^{\alg}} \times \{\pm {\rm{Id}}_{V_l}\}.
\label{GlK1 alg = wide tilde GlK1 alg times pm IdVl}
\end{equation}
Hence $G_{l, K, 1}^{\alg} / \{\pm {\rm{Id}}_{V_l}\} = \widetilde{G_{l, K, 1}^{\alg}}$.
Since a quotient map of algebraic groups is uniquely determined by its kernel \cite[p.101]{Mi2}, it follows that 
\begin{equation}
\AST_{K} (V, \psi)_{\Q_l} / \{\pm {\rm{Id}}_{V_l}\} = 
\AST_{K} (V, \psi) / \{\pm {\rm{Id}}_{V}\} \,\otimes_{\Q} {\Q_l}.
\label{ASTK otimes Ql mod pm 1 otimes Qll is ASTK mod pm 1 otimes Ql}
\end{equation}
By \cite[Theorems 1.72, 3.34, 5.14, Corollary 5.18]{Mi2}, the group scheme
$\AST_{K} (V, \psi) / \{\pm {\rm{Id}}_{V}\}$ is affine and defined over $\Q$.
We put  
$$
\widetilde{\AST_{K}} (V, \psi) := \AST_{K} (V, \psi) / \{\pm {\rm{Id}}_{V}\}.
$$ 
Consider the following commutative diagram with exact rows. 
\begin{figure}[H]
\[
\begin{tikzcd}
 G_{l, K, 1}^{\alg} \arrow{d}{\simeq}[swap]{{\rm{ast}}_{K, l}} \arrow{r}{} & G_{l, K, 1}^{\alg} / \{\pm {\rm{Id}}_{V_l}\} \arrow{d}{\simeq}[swap]{} \arrow{r}{} & 1\\
 \AST_{K} (V, \psi)_{\Q_l} \arrow{r}{} & 
\AST_{K} (V, \psi)_{\Q_l} / \{\pm {\rm{Id}}_{V_l}\} \arrow{r}{} & 1\\ 
\end{tikzcd}
\]
\\[-0.8cm]
\caption{}
\label{the diagram defining tilde ASTK if -Id not in tilde GlK1alg} 
\end{figure}
By \eqref{ASTK otimes Ql mod pm 1 otimes Qll is ASTK mod pm 1 otimes Ql} and Diagram \ref{the diagram defining tilde ASTK if -Id not in tilde GlK1alg}, we obtain the following natural-in-$K$ isomorphism:
\begin{equation}
\widetilde{G_{l, K, 1}^{\alg}} \,\, 
{\stackrel{\simeq}{\longrightarrow}} \,\, \widetilde{\AST_{K}} (V, \psi)_{\Q_l}.
\label{The isomorphism for AST conjecture as Serre predicts}
\end{equation}
By \cite[Proposition 19.13]{Mi2} and \eqref{The isomorphism for AST conjecture as Serre predicts},
the group scheme $\widetilde{\AST_{K}} (V, \psi)$ is reductive because $\widetilde{G_{l, K, 1}^{\alg}}$ is reductive. Nevertheless we cannot fully establish Conjecture \ref{general algebraic Sato Tate conj. Serre's approach} in this case because it is not clear how to prove that $\widetilde{\AST_{K}} (V, \psi) \subset \Iso_{(V, \psi)}$ over $\Q$. 
Proposition \ref{AST conjecture for math cal is equivalent to  AST conjecture wide tilde math cal} and Proposition \ref{AST conjecture for M is equivalent to wide tilde AST conjecture } suggest that the condition $\widetilde{\AST_{K}} (V, \psi) \subset \Iso_{(V, \psi)}$ over $\Q$ is naturally expected in the setting of this paper.
\label{general alg. Sato Tate conj. K implies in part general alg. Sato Tate conj. Serre's approach K}
\end{remark}

\begin{remark} Assume Algebraic Sato--Tate Conjecture \ref{general algebraic Sato Tate conj. Serre's approach}.
By Remark \ref{tilde Sato--Tate set up} there is a one-to-one correspondence between normalized Frobenius elements 
in $\AST_{K} (V, \psi) (\C)$ and  normalized Frobenius elements in $\widetilde{\AST_{K}} (V, \psi) (\C)$. 
Similarly there is a one-to-one correspondence between conjugacy classes of normalized Frobenius elements in 
$\AST_{K} (V, \psi) (\C)$ and conjugacy classes of normalized Frobenius elements in 
$\widetilde{\AST_{K}} (V, \psi) (\C)$ because $\widetilde{\AST_{K}} (V, \psi) \, \triangleleft \, \AST_{K} (V, \psi)$ and $- {\rm{Id}}_{V}$ is in the center of $\AST_{K} (V, \psi)$. Hence if $- {\rm{Id}}_{V} \notin \, \widetilde{\AST_{K}} (V, \psi)$ then there are no conjugacy classes of normalized Frobenius elements in $- {\rm{Id}}_{V} \, \widetilde{\AST_{K}} (V, \psi) (\C)$.
\label{Conjugacy classes of normalized Frobenius elements are in wide tilde AST}
\end{remark}

\begin{corollary}
Assume Algebraic Sato--Tate Conjecture \ref{general algebraic Sato Tate conj. Serre's approach}. Then
\begin{align}
\AST_{K} (V, \psi) &= \widetilde{\AST_{K}} (V, \psi) \, \cup \, - {\rm{Id}}_{V} \, 
\widetilde{\AST_{K}} (V, \psi),
\label{AST expressed in terms of wide tilde AST} \\
\ST_{K} (V, \psi) &= \widetilde{\ST_{K}} (V, \psi) \, \cup \, - {\rm{Id}}_{V} \, 
\widetilde{\ST_{K}} (V, \psi).
\label{ST expressed in terms of wide tilde ST}
\end{align} 
Moreover the following conditions are equivalent:
\begin{itemize}
\item[(1)] $- {{\rm Id}}_{V} \in \widetilde{\AST_{K}} (V, \psi)$,
\item[(2)] $- {{\rm Id}}_{V} \in \widetilde{\ST_{K}} (V, \psi)$,
\item[(3)] $\AST_{K} (V, \psi) = \widetilde{\AST_{K}} (V, \psi)$,
\item[(4)] $\ST_{K} (V, \psi) = \widetilde{\ST_{K}} (V, \psi)$.
\end{itemize}
\label{relations between AST and wide tilde AST and between ST and wide tilde ST}
\end{corollary}

\begin{proof}
The equality \eqref{AST expressed in terms of wide tilde AST} is immediate from the proof of 
Proposition \ref{Conjectures general algebraic Sato Tate conj. implies Algebraic Sato--Tate conjecture. Serre's approach}.
The equality
\eqref{ST expressed in terms of wide tilde ST} follows from \eqref{AST expressed in terms of wide tilde AST}. Now the equivalence of conditions (1)--(4) is obvious.
\end{proof}

\begin{remark}
For a fixed base field $K$, the field $K_{0}$ depends in general on $l$ (see Remark \ref{properties of K0}).
This observation motivates the following definition.   
\end{remark}

\begin{definition}
\label{The Sato--Tate parity group}
 Under Algebraic Sato--Tate Conjecture 
\ref{general algebraic Sato Tate conj. Serre's approach} the {\em{Sato--Tate parity group}} is defined as follows:
$$
\gpP_{\rm{ST}} (V, \psi) \, :=  \, \AST_{K} (V, \psi) / \widetilde{\AST_{K}} (V, \psi) 
\, =  \, \ST_{K} (V, \psi) / \widetilde{\ST_{K}} (V, \psi).
$$
\end{definition}

\begin{remark}
By Corollary \ref{relations between AST and wide tilde AST and between ST and wide tilde ST} 
the group $\gpP_{\rm{ST}} (V, \psi)$ is either trivial or isomorphic to $\Z/2\Z$. 
It follows from Lemma \ref{condition for GlK1alg = tilde(GlK1alg)} that the natural map $\gpP(V_{l}, \psi_{l}) \rightarrow \gpP_{\rm{ST}} (V, \psi)$ is an epimorphism.
\label{on the parity Sato--Tate group}
\end{remark}

\begin{theorem} 
Assume Algebraic Sato--Tate Conjecture \ref{general algebraic Sato Tate conj. Serre's approach}. 
\begin{itemize}
\item[(a)]
If $\gpP_{\rm{ST}} (V, \psi)$ is nontrivial, i.e. $- {{\rm Id}}_{V} \notin \widetilde{\ST_{K}} (V, \psi)$, then Sato--Tate Conjecture \ref{general Sato Tate conj.} does not hold. 
\item[(b)]
If 
$\gpP_{\rm{ST}} (V, \psi)$ is trivial, i.e. $- {{\rm Id}}_{V} \in \widetilde{\ST_{K}} (V, \psi)$,
then Sato--Tate Conjectures \ref{general Sato Tate conj.}  and \ref{general Sato Tate conj. tilde}
are equivalent. 
\end{itemize}
\label{equidistribution property in ST implies equidistribution property in wide tilde ST}
\end{theorem}

\begin{proof} If $- {{\rm Id}}_{V} \notin \widetilde{\ST_{K}} (V, \psi)$ then 
$$
\ST_{K} (V, \psi) = \widetilde{\ST_{K}} (V, \psi) \, \sqcup \, - {\rm{Id}}_{V} \, 
\widetilde{\ST_{K}} (V, \psi).
$$
Let $\mu$ be the probabilistic Haar measure on $\ST_{K} (V, \psi)$. Consider the following character:
\begin{gather*}
\chi\colon \ST_{K} (V, \psi) \rightarrow \C, \\
\chi (g) \,\, = \,\,
\left\{\begin{array}{lll}
1&{\rm if}& g \in  \widetilde{\ST_{K}} (V, \psi)\\
-1&{\rm if}& g \in - {\rm{Id}}_{V} \, \widetilde{\ST_{K}} (V, \psi).\\
\end{array}
\right.
\end{gather*}
Observe that
$$
\int_{\ST_{K} (V, \psi)}\, \chi (g)\, d \mu = \int_{\widetilde{\ST_{K}} (V, \psi)}\, d \mu +
 \int_{- {\rm{Id}}_{V}\,\widetilde{\ST_{K}} (V, \psi)} -1 \, d \mu
= \frac{1}{2} + (-1) \frac{1}{2} = 0.
$$
On the other hand, let $m \in \N$ and $N_{K} \, (m) := \, \# \{v\colon N v \, \leq \, m \}$.
By Remark \ref{Conjugacy classes of normalized Frobenius elements are in wide tilde AST} we obtain
$$
\lim_{m\to\infty} \frac{1}{N_{K} \, (m)} \,  \sum_{N v \, \leq \, m} \chi(s_{v}) \, = \,
\lim_{m\to\infty} \frac{1}{N_{K} \, (m)} \,  \sum_{N v \, \leq \, m} 1 = 1.
$$
Hence the sequence ${\rm{conj}}(s_v)$ is not equidistributed in the conjugacy classes of 
$\ST_{K} (V, \psi)$, so (a) holds. 

By the same token, if Sato--Tate Conjecture \ref{general Sato Tate conj.} holds, then 
conditions (1)--(4) of Corollary \ref{relations between AST and wide tilde AST and between ST and wide tilde ST} 
hold, so in particular condition (4) holds. In this case $\mu = \tilde{\mu}$ and by Remark \ref{Conjugacy classes of normalized Frobenius elements are in wide tilde AST}, ${\rm{conj}}(\tilde{s}_v) = {\rm{conj}}(s_v)$ 
in $\widetilde{\ST_{K}} (V, \psi) = \ST_{K} (V, \psi)$. Hence ${\rm{conj}}(\tilde{s}_v)$ is equidistributed 
in $\widetilde{\ST_{K}} (V, \psi)$. Hence (b) holds.
\end{proof}

\section{Tate conjecture and Algebraic Sato--Tate conjecture for families of \texorpdfstring{$l$}{l}-adic representations}
\label{AST and Tate for families}

The framework of this section provides a wider perspective on Algebraic Sato--Tate Conjectures \ref{general algebraic Sato Tate conj.} and \ref{general algebraic Sato Tate conj. Serre's approach}. The following two Conjectures \ref{Tate conjecture for families of l-adic representations} and 
\ref{Tate conjecture for families of l-adic representations tilde} are analogues of the 
Tate conjecture for motives. For a special case of these conjectures see \cite{LP}. Throughout this section (except in Proposition \ref{uniqueness of Tate group and Tate tilde group}) we assume  only the weaker form of Conjectures 
\ref{Tate conjecture for families of l-adic representations} and 
\ref{Tate conjecture for families of l-adic representations tilde}, namely Conjectures
\ref{Tate conjecture for families of l-adic representations}(a) and 
\ref{Tate conjecture for families of l-adic representations tilde}(a).
For a wide range of examples where Conjectures 
\ref{Tate conjecture for families of l-adic representations}(a) and 
\ref{Tate conjecture for families of l-adic representations tilde}(a) are known to hold, see Remark
\ref{A Remark concerning Conjectures 8.1 (a) and 8.2 (a)}.
\medskip
 
Let $(V, \psi)$ be a rational, polarized, pure Hodge structure of weight $n$. Let 
$$
\rho_l\colon G_K \rightarrow \GIso (V_l, \psi_l)
$$
be a family of $l$-adic representations \eqref{the family of l-adic representations} satisfying the conditions 
\textbf{(D1)},
\textbf{(D2)},
\textbf{(DR1)},
\textbf{(DR2)},
\textbf{(R1)}--\textbf{(R4)}
of \S \ref{Mumford--Tate groups of polarized Hodge structures},  \S \ref{de Rham structures associated with Hodge structures},  \S \ref{families of l-adic representations associated with Hodge structures}. 

\begin{conjecture} {\,}
\label{Tate conjecture for families of l-adic representations}
\begin{itemize}
\item[(a)]
For every finite extension $K/F$, 
there exist a natural-in-$K$ reductive group scheme $\mathcal{G}_{K}^{\alg} := \mathcal{G}_{K}^{\alg} (V, \psi)
\subset \GIso_{(V, \psi)}$ over $\Q$ and a  natural-in-$K$ monomorphism of group schemes for every $l$:
\begin{equation}
{\rm{t}}_{l, K}\colon G_{l, K}^{\alg}  \,\, {\stackrel{}{\hookrightarrow}} \,\, 
{\mathcal{G}_{K}^{\alg}}_{{\Q_{l}}}.
\label{Tate conjecture for families of l-adic representations. Embedding}
\end{equation} 
In addition the natural embedding $G_{l, K}^{\alg} \subset \GIso_{(V_l, \psi_{l})}$ factors through
${\rm{t}}_{l, K}$ and the natural embedding ${\mathcal{G}_{K}^{\alg}}_{{\Q_{l}}} \subset \GIso_{(V_l, \psi_{l})}$. 

\item[(b)]
The homomorphism
\eqref{Tate conjecture for families of l-adic representations. Embedding} 
is an isomorphism:
\begin{equation} 
{\rm{t}}_{l, K}\colon G_{l, K}^{\alg}  \,\, {\stackrel{\simeq}{\longrightarrow}}  \,\, 
{\mathcal{G}_{K}^{\alg}}_{{\Q_{l}}}.
\label{Tate conjecture for families of l-adic representations. Isomorphism}
\end{equation}
\end{itemize}
\end{conjecture}

Consider the rational Hodge structure $W := V \oplus \mathbb{Q} (1)$. Then
$W_{l} = V_{l} \oplus \mathbb{Q}_{l} (1)$. Consider the family of $l$-adic 
representations $(\widetilde{\rho}_l)$ 
associated with $W$ (see \eqref{Representation rho tilde of Serre}):
$$
\widetilde{\rho}_l\colon G_K \rightarrow \GL(W_{l}).
$$

\begin{conjecture} {\,}
\label{Tate conjecture for families of l-adic representations tilde}
\begin{itemize}
\item[(a)]
For every finite extension $K/F$, 
there exists a natural-in-$K$ reductive group scheme $\widetilde{\mathcal{G}_{K}^{\alg}} := 
\widetilde{\mathcal{G}_{K}^{\alg}} (W) \subset  \mathcal{G}_{K}^{\alg} \times {\G}_m
\subset \GL_{W}$ over $\Q$ such that projections onto the factors of $\mathcal{G}_{K}^{\alg} \times {\G}_m$ give epimorphisms: 
$$
N\colon \widetilde{\mathcal{G}_{K}^{\alg}} \rightarrow \mathbb{G}_m, \quad\quad
\pi\colon \widetilde{\mathcal{G}_{K}^{\alg}} \rightarrow \mathcal{G}_{K}^{\alg}
$$
and there is a natural monomorphism:
\begin{equation}
\widetilde{{\rm{t}}}_{l, K}\colon \widetilde{G_{l, K}^{\alg}}  \,\, {\stackrel{}{\hookrightarrow}} \,\, 
\widetilde{{\mathcal{G}_{K}^{\alg}}}_{{\Q_{l}}}.
\label{Tate conjecture for families of l-adic representations tilde. Embedding}
\end{equation} 
In addition, the natural embedding $\widetilde{G_{l, K}^{\alg}} \subset \GL_{W_{l}}$ factors through
$\widetilde{{\rm{t}}}_{l, K}$ and the natural embedding $\widetilde{{\mathcal{G}_{K}^{\alg}}}_{{\Q_{l}}} \subset 
\GL_{W_{l}}$. 

\item[(b)]
The homomorphism
\eqref{Tate conjecture for families of l-adic representations tilde. Embedding} 
is an isomorphism:
\begin{equation} 
\widetilde{{\rm{t}}}_{l, K}\colon \widetilde{G_{l, K}^{\alg}}  \,\, {\stackrel{\simeq}{\longrightarrow}}  \,\, 
\widetilde{{\mathcal{G}_{K}^{\alg}}}_{{\Q_{l}}}.
\label{Tate conjecture for families of l-adic representations tilde. Isomorphism}
\end{equation}
\end{itemize}
\end{conjecture}

\begin{remark}
We recall three examples of families of $l$-adic representations associated with rational, pure, polarized Hodge structures for which Conjectures \ref{Tate conjecture for families of l-adic representations. Embedding}(a) and \ref{Tate conjecture for families of l-adic representations tilde}(a) are satisfied.

The first family comes from the category of polarized realizations which is a full subcategory of Jannsen's category of realizations. We define and describe in detail the category of polarized realizations in \S \ref{sec,category of polarized realizations}. 
See Proposition \ref{the family V l l of G K satisfies Conjectures 1 (a) and 2 (a)} for the families of 
$l$-adic representations satisfying Conjectures \ref{Tate conjecture for families of l-adic representations. Embedding}(a) and \ref{Tate conjecture for families of l-adic representations tilde}(a). 

The second and third families, described in \S \ref{computation of the identity connected component of ASTKM}, come from motives in the Deligne motivic category for absolute Hodge cycles and motives in the Andr{\' e} motivic category for motivated cycles. 
\label{A Remark concerning Conjectures 8.1 (a) and 8.2 (a)}
\end{remark}

\begin{proposition} {\,} 
\label{uniqueness of Tate group and Tate tilde group}
\begin{itemize}
\item[(a)] Under Conjecture \ref{Tate conjecture for families of l-adic representations} 
the group scheme
$\mathcal{G}_{K}^{\alg}$ is uniquely defined. 
\item[(b)] Under Conjecture \ref{Tate conjecture for families of l-adic representations tilde}
the group scheme $\widetilde{\mathcal{G}_{K}^{\alg}}$ is uniquely defined.  
\end{itemize}
\end{proposition}
\begin{proof} Claims (a) and (b) follow from Corollary \ref{containment of 2 schemes over L follows from their containment after base change to M}. 
\end{proof}

\begin{definition}
\label{chi for math cal G K alg}
Let 
\begin{equation}
\chi\colon \mathcal{G}_{K}^{\alg} \rightarrow \mathbb{G}_m
\label{morphism of group schemes over Q given by the character of polarization}
\end{equation}
denote the restriction of the character $\chi$ of the polarization $\psi$ to $\mathcal{G}_{K}^{\alg}$ (cf. \ref{def of GIso}).
\end{definition}

\begin{definition} 
\label{Definition of G K 1 alg and widetilde G K 1 alg}
Put:
\begin{align}
\mathcal{G}_{K, 1}^{\alg} &:= \mathcal{G}_{K, 1}^{\alg} (V, \psi) := {\rm{Ker}} \, \chi
\label{definition of mathcal G K 1 alg} \\
\widetilde{\mathcal{G}_{K, 1}^{\alg}} &:= \widetilde{\mathcal{G}_{K, 1}^{\alg}} (W) := {\rm{Ker}}  \, N.
\label{definition of widetilde mathcal G K 1 alg}
\end{align}
\end{definition}

\begin{lemma} 
\label{The diagram connecting wide tilde mathcal G K alg and ilde mathcal G K alg }
Under Conjectures \ref{Tate conjecture for families of l-adic representations}(a) and \ref{Tate conjecture for families of l-adic representations tilde}(a) the following Diagram~\ref{diagram compatibility of tilde(GK1alg) with GK1alg} commutes and the map $\pi_{1}$ is a monomorphism.
\begin{figure}[H]
\[
\begin{tikzcd}
1 \arrow{r}{} & \,\, \widetilde{\mathcal{G}_{K, 1}^{\alg}} \arrow{d}[swap]{\pi_{1}} \arrow{r}{}  & 
\widetilde{\mathcal{G}_{K}^{\alg}}  \arrow{d}[swap]{\pi} \arrow{r}{N} & \G_{m} 
\arrow{d}{x \mapsto x^{-n}}[swap]{} \arrow{r}{} & 1 \\ 
1 \arrow{r}{} & \mathcal{G}_{K, 1}^{\alg}  \arrow{r}{} & 
\mathcal{G}_{K}^{\alg} \arrow{d}[swap]{} \arrow{r}{\chi} & \G_{m} \arrow{d}[swap]{} 
\arrow{r}{} & 1\\
& & 1 & 1 \\
\end{tikzcd}
\]
\\[-0.8cm]
\caption{}
\label{diagram compatibility of tilde(GK1alg) with GK1alg} 
\end{figure}
\end{lemma}
\begin{proof} 
We already know that the morphisms $N$, $\pi$, and $x \mapsto x^{-n}$ are epimorphisms. We next check that the morphism $\chi$ in Diagram \ref{diagram compatibility of tilde(GK1alg) with GK1alg}
is an epimorphism. Extending base to $\Q_l$ in Diagram \ref{diagram compatibility of tilde(GK1alg) with GK1alg}, we observe that under Conjectures \ref{Tate conjecture for families of l-adic representations}(a) and
\ref{Tate conjecture for families of l-adic representations tilde} a), the right square of Diagram \ref{diagram compatibility of tilde(GlK1alg) with GlK1alg} embeds into the right square of the resulting diagram (cf. Diagram \ref{double cubic diagram} below).
Hence $\chi_{_{\Q_{l}}}\colon {\mathcal{G}_{K}^{\alg}}_{{\Q_{l}}}  \rightarrow  {\G_{m}}$ is an epimorphism.
Therefore $\chi$ in Diagram \ref{diagram compatibility of tilde(GK1alg) with GK1alg} is also an epimorphism
by \cite[p. 106, first corollary]{Wa}. 

In addition, the right square of Diagram \ref{diagram compatibility of tilde(GK1alg) with GK1alg} commutes by Corollary \ref{equality of two morphisms of schemes over L follows from their equality after base change to M} applied to the field extension $\Q_l / \Q$. The left square of  Diagram \ref{diagram compatibility of tilde(GK1alg) with GK1alg} obviously commutes with the morphism
$\pi_1$ induced by $\pi$. Because of $\widetilde{\mathcal{G}_{K}^{\alg}} \subset  \mathcal{G}_{K}^{\alg} \times {\G}_m$ and the definition of $N$, the kernel of $\pi$ is contained in ${\rm{Id}}_{V} \times \G_{m}$. Hence the homomorphism $N$ restricted to the kernel of $\pi$ is a monomorphism. By the commutativity 
of Diagram \ref{diagram compatibility of tilde(GK1alg) with GK1alg}, the kernel of $\pi$ injects into
$\mu_{n}$. Hence $\pi_{1}$ is a monomorphism.
\end{proof}

In the framework of Conjectures \ref{Tate conjecture for families of l-adic representations}(a) and \ref{Tate conjecture for families of l-adic representations tilde}(a), it is natural to make the following definition.

\begin{definition} Put:
\begin{align*}
\AST_{K} (V, \psi) &:= \mathcal{G}_{K, 1}^{\alg} (V, \psi), \\
\widetilde{\AST_{K}} (V, \psi) &:= \widetilde{\mathcal{G}_{K, 1}^{\alg} (W)}.
\end{align*}
Every maximal compact subgroup of $\AST_{K} (V, \psi) (\C)$ (resp. $\widetilde{\AST_{K}} (V, \psi) (\C))$
will be called a \emph{Sato--Tate group} and denoted $\ST_{K} (V, \psi)$ (resp. 
$\widetilde{\ST_{K}} (V, \psi)$).
\label{AST K and tilde AST K under the analog of Tate conjecture for l-adic rep.}
\end{definition}

Consider the following double cube commutative Diagram \ref{double cubic diagram}. In light of Definition 
\ref{AST K and tilde AST K under the analog of Tate conjecture for l-adic rep.} 
it is natural to denote by $\widetilde{{\rm{ast}}}_{l, K}$ and ${\rm{ast}}_{l, K}$ the horizontal arrows in the left wall of this diagram.  

\begin{figure}[H]
\[
\begin{tikzcd}[row sep=scriptsize, column sep=scriptsize]
& \widetilde{G_{l, K, 1}^{\alg}} \arrow{dl}[swap]{\widetilde{{\rm{ast}}}_{l, K}} \arrow{rr}{} \arrow{dd}[near start]{\pi_1} & & \widetilde{G_{l, K}^{\alg}} \arrow{dl}[swap]{\widetilde{{\rm{t}}}_{l, K}} \arrow{dd}[near start]{\pi} \arrow{rr}{N}  & & \G_{m} \arrow{dl}{=} \arrow{dd}{x \mapsto x^{-n}}\\
{\widetilde{\mathcal{G}_{K, 1}^{\alg}}}_{{\Q_{l}}} \arrow[crossing over]{rr}[near start]{} \arrow{dd}[swap]{\pi_1 \otimes 1} & & {\widetilde{\mathcal{G}_{K}^{\alg}}}_{{\Q_{l}}} \arrow[crossing over]{rr}[near end]{N \otimes 1} 
\arrow{dd}[near start]{\pi \otimes 1}  & & \G_{m} \arrow{dd}[near start]{x \mapsto x^{-n}}\\
& G_{l, K, 1}^{\alg} \arrow{dl}[near start, swap]{{\rm{ast}}_{l, K}}  \arrow{rr}[near start, swap]{} & & G_{l, K}^{\alg} \arrow{dl}[near start]{{\rm{t}}_{l, K}} 
\arrow{rr}[near start, swap]{\chi} & & \G_{m} \arrow{dl}{=}\\
{\mathcal{G}_{K, 1}^{\alg}}_{{\Q_{l}}} \arrow{rr}[swap]{} & & {\mathcal{G}_{K}^{\alg}}_{{\Q_{l}}}
\arrow{rr}[swap]{\chi \otimes 1} & & \G_{m} \\
\end{tikzcd}
\]
\\[-0.8cm]
\caption{}
\label{double cubic diagram} 
\end{figure}

Assume Conjectures \ref{Tate conjecture for families of l-adic representations}(a) and \ref{Tate conjecture for families of l-adic representations tilde}(a).
The back wall of Diagram \ref{double cubic diagram}  is Diagram \ref{diagram compatibility of tilde(GlK1alg) with GlK1alg}. The horizontal sequences of homomorphisms in the front wall of Diagram \ref{double cubic diagram} are exact because the kernel of a group scheme homomorphism is invariant under base change (Remark 
\ref{Kernel of group scheme morphism is invariant under base change on the morphism side} below) 
and any field extension (in particular $\Q_l / \Q$) is flat. For the same reasons, the arrows in these horizontal sequences in the left cube are monomorphisms and in the right cube are epimorphisms. 

\begin{remark}
\label{Kernel of group scheme morphism is invariant under base change on the morphism side}
Let $\phi\colon G \rightarrow G^{\prime}$ be a homomorphism of $S$-group schemes. Let $S^{\prime}$ be an $S$-scheme.
Define the $S$-homomorphism $e^{\prime}\colon S \rightarrow G^{\prime}$ to be the unit section. Recall that ${{\rm Ker}} \, \phi = G \times_{G^{\prime}} S$. For any scheme $Z$ and any $Z$-schemes $X, Y$ denote as usual $X(Y)_Z := 
{\rm{Hom}}_{Z} (Y, \, X)$.
Then for any $S^{\prime}$-group scheme $T$ in the category of group schemes over $S$:
\begin{gather*}
({{\rm Ker}} \, \phi \times_{S} S^{\prime}) (T)_{S^{\prime}} = 
((G \times_{G^{\prime}} S) \times_{S}  S^{\prime}) (T)_{S^{\prime}} = 
G(T)_S \times_{G^{\prime} (T)_S}  S^{\prime} (T)_S, \\
{{\rm Ker}} ( \phi \times_{S} {\rm{Id}}_{S^{\prime}}) (T)_{S^{\prime}} = 
((G \times_{S} S^{\prime}  \times_{G^{\prime} \times_{S} S^{\prime}}  S  \times_{S}  S^{\prime})) (T)_{S^{\prime}} = 
G(T)_S \times_{G^{\prime} (T)_S}  S^{\prime} (T)_S.
\end{gather*}
By uniqueness of the representing object we obtain:
$$
{{\rm Ker}} \, \phi \times_{S} S^{\prime} = {{\rm Ker}} ( \phi \times_{S} {\rm{Id}}_{S^{\prime}}).  
$$
\end{remark}

\begin{remark} Under Conjectures \ref{Tate conjecture for families of l-adic representations}(a) and \ref{Tate conjecture for families of l-adic representations tilde}(a), the homomorphisms
 ${\rm{ast}}_{l, K}$ and $\widetilde{{\rm{ast}}}_{l, K}$ in the left cube in Diagram 
\ref{double cubic diagram} satisfy all conditions of
 Algebraic Sato--Tate Conjectures \ref{general algebraic Sato Tate conj.}(a) and 
\ref{general algebraic Sato Tate conj. Serre's approach}(a). It is immediate from Diagram \ref{double cubic diagram} that ${\rm{ast}}_{l, K}$ (resp. $\widetilde{{\rm{ast}}}_{l, K}$) is an isomorphism iff ${\rm{t}}_{l, K}$ (resp. $\widetilde{{{\rm{t}}}}_{l, K}$) is an isomorphism. 
\label{Tate conj. is equivalent to the AST conj.}
\end{remark}
Remark \ref{Tate conj. is equivalent to the AST conj.} leads directly to the following corollary.

\begin{corollary} Assume Conjectures \ref{Tate conjecture for families of l-adic representations}(a) and \ref{Tate conjecture for families of l-adic representations tilde}(a). Then
Algebraic Sato--Tate Conjecture \ref{general algebraic Sato Tate conj.}
(resp. \ref{general algebraic Sato Tate conj. Serre's approach}) holds if and only if 
Tate Conjecture \ref{Tate conjecture for families of l-adic representations} 
(resp. \ref{Tate conjecture for families of l-adic representations tilde}) holds.
\label{Tate conjecture is equivalent to the AST conjecture}
\end{corollary}

Assume Conjectures \ref{Tate conjecture for families of l-adic representations}(a) and \ref{Tate conjecture for families of l-adic representations tilde}(a). Recall from the proof of Lemma \ref{The diagram connecting wide tilde mathcal G K alg and ilde mathcal G K alg } that in Diagram 
\ref{diagram compatibility of tilde(GK1alg) with GK1alg}, the kernel of $\pi$ injects into
$\mu_{n}$ and $\pi_{1}$ is a monomorphism. 
Hence $\dim \widetilde{\mathcal{G}_{K}^{\alg}}  = \dim \mathcal{G}_{K}^{\alg}$ and consequently $\dim \widetilde{\mathcal{G}_{K, 1}^{\alg}}  = \dim \mathcal{G}_{K, 1}^{\alg}$. This shows that
\begin{equation}
(\widetilde{\mathcal{G}_{K, 1}^{\alg}})^{\circ} = (\mathcal{G}_{K, 1}^{\alg})^{\circ}.
\label{connected component of wide tilde G K 1 alg  is the same as G K 1 alg}
\end{equation} 
Define:
\begin{align*}
\widetilde{\mathcal{G}_{K, 1}^{\alg, \, 0}} \,\, &:= \,\, (\widetilde{\mathcal{G}_{K}^{\alg}})^{\circ} \,\, 
\cap \,\,\widetilde{\mathcal{G}_{K, 1}^{\alg}}.  \\
\mathcal{G}_{K, 1}^{\alg, \, 0} \,\, &:= \,\, (\mathcal{G}_{K}^{\alg})^{\circ} \,\, 
\cap \,\, \mathcal{G}_{K, 1}^{\alg}.
\end{align*}

\begin{lemma} 
\label{splitting of wide tilde math cal G K alg maps to Gm}
The following exact sequence splits:
\begin{equation}
1 \,\, {\stackrel{}{\longrightarrow}} \,\, \widetilde{\mathcal{G}_{K, 1}^{\alg}} (\mathbb{C}_l) \,\, {\stackrel{}{\longrightarrow}}  
\,\, \widetilde{\mathcal{G}_{K}^{\alg}}(\mathbb{C}_l)\,\, {\stackrel{N}{\longrightarrow}} \,\, \G_{m} (\mathbb{C}_l){\stackrel{}{\longrightarrow}} \,\, 1.
\label{wide tilde h splits the homomorphism N from wide tilde math cal G K, 1 alg to G m}
\end{equation}
Moreover $\widetilde{\mathcal{G}_{K, 1}^{\alg, \, 0}}$ is connected, i.e.
\begin{equation}
\widetilde{\mathcal{G}_{K, 1}^{\alg, \, 0}}  = (\widetilde{\mathcal{G}_{K, 1}^{\alg}})^{\circ}.
\label{wide tilde mathcal GK1alg , 0 is connected}
\end{equation}
\end{lemma}

\begin{proof}
Consider the commutative diagram in the lid of the cube Diagram \ref{double cubic diagram}.
Enlarging the target of the cocharacter $\widetilde{h}$ from the proof of Lemma 
\ref{splitting of widetilde G l K alg mapsto Gm} we obtain the following cocharacter:
$$
\widetilde{{\rm{t}}}_{l, K} \circ \widetilde{h}\colon \G_{m} (\mathbb{C}_l) \rightarrow 
\widetilde{\mathcal{G}_{K}^{\alg}}(\mathbb{C}_l).
$$
It is clear from the lid of Diagram \ref{double cubic diagram} that 
$\widetilde{{\rm{t}}}_{l, K} \circ \widetilde{h}$ splits $N$ in the exact sequence \ref{wide tilde h splits the homomorphism N from wide tilde math cal G K, 1 alg to G m}. Because $\G_m$ is connected, the following exact sequence splits:
$$
1 \,\, {\stackrel{}{\longrightarrow}} \,\, \widetilde{\mathcal{G}_{K, 1}^{\alg, \, 0}} (\mathbb{C}_l) \,\, {\stackrel{}{\longrightarrow}}  
\,\, (\widetilde{\mathcal{G}_{K}^{\alg}})^{\circ} (\mathbb{C}_l)\,\, {\stackrel{N}{\longrightarrow}} \,\, \G_{m} (\mathbb{C}_l){\stackrel{}{\longrightarrow}} \,\, 1.
$$
Now we finish the proof in a similar way as we finished the proof of Lemma 
\ref{splitting of widetilde G l K alg mapsto Gm}.
\end{proof}

The following theorem is proven in the same way as Theorem 
\ref{pi 0 wide tilde G l, K, 1 alg cong pi 0 wide tilde G l, K alg }. 
\begin{theorem} 
\label{pi 0 wide tilde math cal G K, 1 alg cong pi 0 wide tilde math cal G K alg}
There is the following isomorphism:
\begin{equation}
\widetilde{i}_{CC}\colon \pi_{0} (\widetilde{\mathcal{G}_{K, 1}^{\alg}}) \,\,\, {\stackrel{\simeq}{\longrightarrow}} \,\,\,  
\pi_{0} (\widetilde{\mathcal{G}_{K}^{\alg}}).
\label{pi 0 wide tilde math cal G K, 1 alg cong pi 0 wide tilde math cal G K alg. isomorphism}
\end{equation}
\end{theorem}
\begin{proof} Consider the following commutative diagram.
\begin{figure}[H]
\[
\begin{tikzcd}
& 1 \arrow{d}[swap]{} & 1 \arrow{d}[swap]{}  & 1 \arrow{d}[swap]{} \\ 
1 \arrow{r}{} &  (\widetilde{\mathcal{G}_{K, 1}^{\alg}})^{\circ} \arrow{d}[swap]{} \arrow{r}{}
&  \widetilde{\mathcal{G}_{K, 1}^{\alg}} \arrow{d}[swap]{} \arrow{r}{}  & \pi_{0} 
(\widetilde{\mathcal{G}_{K, 1}^{\alg}})   \arrow{d}[swap]{\simeq}{\widetilde{i}_{CC}} \arrow{r}{} &  1\\  
1 \arrow{r}{} & (\widetilde{\mathcal{G}_{K}^{\alg}})^{\circ} \arrow{d}{N} \arrow{r}{} & 
\widetilde{\mathcal{G}_{K}^{\alg}} \arrow{d}{N} \arrow{r}{} & \pi_{0}(\widetilde{\mathcal{G}_{K}^{\alg}}) 
\arrow{d}[swap]{} \arrow{r}{} &  1 \\
1 \arrow{r}{} & \G_m  \arrow{d}[swap]{} \arrow{r}{=} & \G_m  \arrow{d}[swap]{} \arrow{r}{} & 1\\
& 1  & 1 \\
\end{tikzcd}
\]
\\[-0.8cm]
\caption{}
\label{diagram to prove analog of Serre theorem for families of l-adic representations under weak Tate 
conjectures part (a)}
\end{figure}
In Diagram
\ref{diagram to prove analog of Serre theorem for families of l-adic representations under weak Tate 
conjectures part (a)}, all rows and the middle column are obviously exact. The left column is exact by Lemma 
\ref{splitting of wide tilde math cal G K alg maps to Gm}.
Chasing in Diagram \ref{diagram to prove analog of Serre theorem for families of l-adic representations under weak Tate conjectures part (a)}, we observe that the right column is exact.
\end{proof}

\begin{lemma}
\label{splitting of math cal GK1alg,0 arrow math cal GKalg circ arrow Gm}
Assume Conjectures \ref{Tate conjecture for families of l-adic representations}(a) 
and \ref{Tate conjecture for families of l-adic representations tilde}(a). Let $n$ be odd. Then the following exact sequence splits:
\begin{equation}
1 \,\, {\stackrel{}{\longrightarrow}} \,\, \mathcal{G}_{K, 1}^{\alg, \, 0} (\mathbb{C}) \,\, {\stackrel{}{\longrightarrow}} \,\, (\mathcal{G}_{K}^{\alg})^{\circ} (\mathbb{C})\,\, {\stackrel{\chi}{\longrightarrow}} \,\, \G_{m} (\mathbb{C}){\stackrel{}{\longrightarrow}} \,\, 1.
\label{math cal s splits the homomorphism chi from  math cal (G K alg) circ to G m for n odd}
\end{equation}
Moreover $\mathcal{G}_{K, 1}^{\alg, \, 0}$ is connected, i.e.
\begin{equation}
\mathcal{G}_{K, 1}^{\alg, \, 0}  = (\mathcal{G}_{K, 1}^{\alg})^{\circ}.
\label{math cal GK1alg , 0 is connected}
\end{equation}
\end{lemma}
\begin{proof} Let $L_{0}$ be such that $(G_{l, K}^{\alg})^{\circ} = G_{l, L_{0}}^{\alg}$. 
In the proof of Theorem \ref{L0realizing conn comp} we constructed the following homomorphism 
\begin{gather*}
s\colon \G_m (\C) \rightarrow  G_{l, L_{0}}^{\alg} (\C), \\
s(z) := h(z) w(z)^{-\frac{n-1}{2}}
\end{gather*}
which splits $\chi$ in the following exact sequence:  
$$
1 \,\, {\stackrel{}{\longrightarrow}} \,\, G_{l, L_{0}, 1}^{\alg} (\C) \,\, {\stackrel{}{\longrightarrow}} \,\,  G_{l, L_{0}}^{\alg} (\C) \,\, {\stackrel{\chi}{\longrightarrow}} \,\, \C^{\times} \,\, {\stackrel{}{\longrightarrow}} \,\, 1.
$$
Define
\begin{gather*}
s_0\colon \G_m (\C) \rightarrow  (\mathcal{G}_{K}^{\alg})^{\circ} (\C), \\
s_{0} (z) := {\rm{t}}_{l, K} \circ  s(z).
\end{gather*}
It is clear that $s_0$ is a splitting of $\chi$ in the exact sequence
\eqref{math cal s splits the homomorphism chi from  math cal (G K alg) circ to G m for n odd}.
Because $(\mathcal{G}_{K}^{\alg})^{\circ}$ is connected, Lemma \ref{G connected implies G0 connected}
shows that $\mathcal{G}_{K, 1}^{\alg, \, 0}$ is connected. In particular,
\eqref{math cal GK1alg , 0 is connected} holds.
\end{proof}

\begin{proposition} \label{realizing conn comp of math cal G K 1 alg: even case}   
Assume Conjectures \ref{Tate conjecture for families of l-adic representations}(a) 
and \ref{Tate conjecture for families of l-adic representations tilde}(a). Then:
\begin{itemize}
\item[(a)] $\G_{m} {\rm{Id}}_{V} \cdot (\mathcal{G}_{K, 1}^{\alg})^{\circ} \, = \, 
\G_{m} {\rm{Id}}_{V} \cdot {\mathcal{G}_{K, 1}^{\alg, \, 0}} \, = 
\, (\mathcal{G}_{K}^{\alg})^{\circ}$;
\item[(b)] ${\mathcal{G}_{K, 1}^{\alg, \, 0}} = (\mathcal{G}_{K, 1}^{\alg})^{\circ} \, \cup \, 
- {\rm{Id}}_{V} \cdot ({\mathcal{G}_{K, 1}^{\alg}})^{\circ}$;
\item[(c)] $- {\rm{Id}}_{V} \in ({\mathcal{G}_{K, 1}^{\alg}})^{\circ}$ iff $({\mathcal{G}_{K, 1}^{\alg}})^{\circ}
= {\mathcal{G}_{K, 1}^{\alg, \, 0}}$;
\item[(d)] when $n$ is odd, $- {\rm{Id}}_{V} \, \in \, ({\mathcal{G}_{K, 1}^{\alg}})^{\circ} \,  = \,
{\mathcal{G}_{K, 1}^{\alg, \, 0}}$.  
\end{itemize} 
\end{proposition} 
\begin{proof}
The proof is similar to the proof of Proposition \ref{L0realizing conn comp: even case}. First of all observe that
$\G_{m} {\rm{Id}}_{V} \subset {{\mathcal{G}_{K}^{\alg}}}$ 
by Conjecture \ref{Tate conjecture for families of l-adic representations}(a) and Corollary 
\ref{containment of 2 schemes over L follows from their containment after base change to M}.

\noindent
(a)  Consider a coset $g ({\mathcal{G}_{K, 1}^{\alg}})^{\circ}$ in $\mathcal{G}_{K, 1}^{\alg, \, 0}$. Applying \cite[Section 7.4, Prop. B(b)]{Hu} to the homomorphism
\begin{gather*}
\G_{m} {\rm{Id}}_{V} \, \times \, ({\mathcal{G}_{K, 1}^{\alg}})^{\circ} \rightarrow 
({\mathcal{G}_{K}^{\alg}})^{\circ}, \\
(g_1, g_2) \, \mapsto \, g_1 g_2
\end{gather*}
we observe that  $\G_{m} {\rm{Id}}_{V} \cdot ({{\mathcal{G}_{K, 1}^{\alg}}})^{\circ}$ and 
$\G_{m} {\rm{Id}}_{V} \cdot g ({\mathcal{G}_{K, 1}^{\alg}})^{\circ}$ are closed in 
$({\mathcal{G}_{K}^{\alg}})^{\circ}$. 

They are also of the same dimension as ${({\mathcal{G}_{K}^{\alg}})^{\circ}}$ because of the following exact sequence:
\begin{equation}
1 \,\, {\stackrel{}{\longrightarrow}} \,\, {\mathcal{G}_{K, 1}^{\alg, \, 0}} \,\, {\stackrel{}{\longrightarrow}} \,\, ({\mathcal{G}_{K}^{\alg}})^{\circ} \,\, {\stackrel{\chi}{\longrightarrow}} \,\, \G_m \,\, {\stackrel{}{\longrightarrow}} \,\, 1.
\label{The exact sequence for math cal GK1alg 0 and GK alg circ}
\end{equation}
Because $({\mathcal{G}_{K}^{\alg}})^{\circ}$ is irreducible, we obtain 
\begin{equation}
\G_{m} {\rm{Id}}_{V} \cdot ({\mathcal{G}_{K, 1}^{\alg}})^{\circ} \, = \,\G_{m} {\rm{Id}}_{V} \cdot  
g\, ({\mathcal{G}_{K, 1}^{\alg}})^{\circ} = ({\mathcal{G}_{K}^{\alg}})^{\circ}.
\label{nonempty intersection of Gm math cal G K 1 alg with Gm g math cal G K 1 alg}
\end{equation}

\noindent
(b) From \eqref{nonempty intersection of Gm GlK1 alg with Gm gGlK1 alg} there are $\alpha, \beta \in \G_{m}$ and $g_1, g_2 \in ({\mathcal{G}_{K, 1}^{\alg}})^{\circ}$ such that:
\begin{equation}
\alpha \, {\rm{Id}}_{V} \cdot g_1 = \beta \, {\rm{Id}}_{V} \cdot g g_2. 
\label{equality of coset generators wide tilde}
\end{equation}
Applying $\chi$ to \eqref{equality of coset generators wide tilde} we obtain $\alpha^2 = \beta^2$. This implies that
$\pm {\rm{Id}}_{V} \cdot g_1 =  g g_2$. Hence
$g ({\mathcal{G}_{K, 1}^{\alg}})^{\circ} = ({\mathcal{G}_{K, 1}^{\alg}})^{\circ}$ or $g ({\mathcal{G}_{K, 1}^{\alg}})^{\circ} = - {\rm{Id}}_{V} ({\mathcal{G}_{K, 1}^{\alg}})^{\circ}$. 
\medskip

\noindent
(c) This follows immediately from (b).
\medskip

\noindent
(d) In the case of $n$ odd, the group ${\mathcal{G}_{K, 1}^{\alg, \, 0}}$  is connected by Lemma
\ref{splitting of math cal GK1alg,0 arrow math cal GKalg circ arrow Gm}. Hence (d) follows immediately from (c).
\end{proof}

Consider the following commutative diagram. 

\begin{figure}[H]
\[
\begin{tikzcd}
&& \,\, \pi_{0} (\widetilde{\mathcal{G}_{K, 1}^{\alg}}) \arrow{d}[swap]{\overline{\pi}_{1}} \arrow{r}{\simeq}  & 
\pi_{0} (\widetilde{\mathcal{G}_{K}^{\alg}}) \arrow{d}[swap]{\overline{\pi}}& \\ 
1 \arrow{r}{} &  \mathcal{G}_{K, 1}^{\alg, \, 0} / (\mathcal{G}_{K, 1}^{\alg})^{\circ} \arrow{r}{} & 
\pi_{0} (\mathcal{G}_{K, 1}^{\alg})  \arrow{r}{} & \pi_{0} (\mathcal{G}_{K}^{\alg}) \arrow{r}{} & 1\\
\end{tikzcd}
\]
\\[-0.8cm]
\caption{}
\label{pi0 applied to the diagram compatibility of wide tilde math cal G alg K 1 with math cal G alg K 1} 
\end{figure}

In Diagram \ref{pi0 applied to the diagram compatibility of wide tilde math cal G alg K 1 with math cal G alg K 1}, the top row is the isomorphism \eqref{pi 0 wide tilde math cal G K, 1 alg cong pi 0 wide tilde math cal G K alg. isomorphism}. The map $\overline{\pi}$ in Diagram \ref{pi0 applied to the diagram compatibility of wide tilde math cal G alg K 1 with math cal G alg K 1} is an epimorphism because the map $\pi$ in Conjecture
\ref{Tate conjecture for families of l-adic representations tilde}(a) is assumed to be an epimorphism.
Exactness of the bottom sequence follows from surjectivity of $\overline{\pi}$ and the definition of 
$\mathcal{G}_{K, 1}^{\alg, \, 0}$. The map $\overline{\pi}_{1}$ in Diagram 
\ref{pi0 applied to the diagram compatibility of wide tilde math cal G alg K 1 with math cal G alg K 1}  
is a monomorphism because the map $\pi_{1}$ in Diagram \ref{diagram compatibility of tilde(GK1alg) with GK1alg}  is a monomorphism and because of \eqref{connected component of wide tilde G K 1 alg  is the same as G K 1 alg}.
Proposition \ref{realizing conn comp of math cal G K 1 alg: even case}(b), equality 
\eqref{connected component of wide tilde G K 1 alg  is the same as G K 1 alg}, and a chase in Diagram 
\ref{pi0 applied to the diagram compatibility of wide tilde math cal G alg K 1 with math cal G alg K 1} show that  
\begin{equation}
{\pi_{0} (\mathcal{G}_{K, 1}^{\alg})  \,  = \, \pi_{0} (\widetilde{\mathcal{G}_{K, 1}^{\alg}}) \,\, \cup 
\, - {\rm{Id}}_{V} (\mathcal{G}_{K, 1}^{\alg})^{\circ} \, \cdot 
\pi_{0} (\widetilde{\mathcal{G}_{K, 1}^{\alg}})}, 
\label{ p0 math cal G K 1 alg = p0 wide tilde math cal G K 1 alg cup -Id p0  wide tilde math cal G K 1 alg}  
\end{equation}
where $- {\rm{Id}}_{V} (\mathcal{G}_{K, 1}^{\alg})^{\circ}$ denotes the coset of 
$- {\rm{Id}}_{V}$ in the quotient group $\pi_{0} (\mathcal{G}_{K, 1}^{\alg}) = 
\mathcal{G}_{K, 1}^{\alg} / (\mathcal{G}_{K, 1}^{\alg})^{\circ}$. 
\medskip

Consider the following commutative diagram with exact rows. 

\begin{figure}[H]
\[
\begin{tikzcd}
1 \arrow{r}{} & (\widetilde{\mathcal{G}_{K, 1}^{\alg}})^{\circ} \arrow{d}[swap]{=} \arrow{r}{} & \,\, 
\widetilde{\mathcal{G}_{K, 1}^{\alg}} \arrow{d}[swap]{\pi_{1}} \arrow{r}{}  & 
 \pi_{0} (\widetilde{\mathcal{G}_{K, 1}^{\alg}}) \arrow{d}[swap]{\overline{\pi}_{1}} \arrow{r}{} & 1\\ 
1 \arrow{r}{} & (\mathcal{G}_{K, 1}^{\alg})^{\circ} \arrow{r}{} & 
\mathcal{G}_{K, 1}^{\alg}  \arrow{r}{} & 
\pi_{0} (\mathcal{G}_{K, 1}^{\alg}) \arrow{r}{} & 1\\
\end{tikzcd}
\]
\\[-0.8cm]
\caption{}
\label{wide tilde math cal GK1alg with math cal GK1 and p0 wide tilde math cal with pi0 math cal} 
\end{figure}

Equality \eqref{ p0 math cal G K 1 alg = p0 wide tilde math cal G K 1 alg cup -Id p0  wide tilde math cal G K 1 alg} and a chase in 
Diagram \ref{wide tilde math cal GK1alg with math cal GK1 and p0 wide tilde math cal with pi0 math cal} 
show that 
\begin{equation}
\mathcal{G}_{K, 1}^{\alg}  \,  = \, \widetilde{\mathcal{G}_{K, 1}^{\alg}} \,\, \cup 
\, - {\rm{Id}}_{V}  \,\, \widetilde{\mathcal{G}_{K, 1}^{\alg}} \, .   
\label{math cal GK1alg = wide tilde math cal G1Kalg cup -Id wide tilde math cal G1Kalg}  
\end{equation} 
In particular $[\mathcal{G}_{K, 1}^{\alg} : \, \widetilde{\mathcal{G}_{K, 1}^{\alg}}] \leq 2$ and 
$\widetilde{\mathcal{G}_{K, 1}^{\alg}} \, \triangleleft \, \mathcal{G}_{K, 1}^{\alg}$.
\medskip

\begin{corollary} Assume Conjectures \ref{Tate conjecture for families of l-adic representations}(a) 
and \ref{Tate conjecture for families of l-adic representations tilde}(a). Then the following equalities hold:
\begin{itemize}
\item[(a)] $\AST_{K} (V, \psi)^{\circ} \, = \, \widetilde{\AST_{K}} (V, \psi)^{\circ} \, = \,
\widetilde{\mathcal{G}_{K, 1}^{\alg, \, 0}}$.
\item[(b)] $\AST_{K} (V, \psi) \, = \, \widetilde{\AST_{K}} (V, \psi) \, \cup \, - {\rm{Id}}_{V} \, 
\widetilde{\AST_{K}} (V, \psi)$.
\item[(c)] $\ST_{K} (V, \psi)  \, = \, \widetilde{\ST_{K}} (V, \psi) \, \cup \, - {\rm{Id}}_{V} \, \widetilde{\ST_{K}} (V, \psi)$ up to conjugation in \\ $\AST_{K} (V, \psi) (\C)$.
\end{itemize}
\label{AST V psi with resp. to AST V psi cup -1Id AST V psi}
\end{corollary}

\begin{proof}
(a) follows by \eqref{connected component of wide tilde G K 1 alg  is the same as G K 1 alg} and 
\eqref{wide tilde mathcal GK1alg , 0 is connected},
(b) follows by \eqref{math cal GK1alg = wide tilde math cal G1Kalg cup -Id wide tilde math cal G1Kalg}, and (c) 
follows by (b). 
\end{proof}

\begin{proposition} Assume Conjectures \ref{Tate conjecture for families of l-adic representations}(a) 
and \ref{Tate conjecture for families of l-adic representations tilde}(a). Then Algebraic Sato--Tate Conjecture 
\ref{general algebraic Sato Tate conj.} holds for $\AST_{K} (V, \psi)$ if and only 
Algebraic Sato--Tate Conjecture \ref{general algebraic Sato Tate conj. Serre's approach} holds for $\widetilde{\AST_{K}} (V, \psi)$.
\label{AST conjecture for math cal is equivalent to  AST conjecture wide tilde math cal}
\end{proposition}
\begin{proof} Remark \ref{Tate conj. is equivalent to the AST conj.} shows that 
Conjectures \ref{general algebraic Sato Tate conj.}(a) and 
\ref{general algebraic Sato Tate conj. Serre's approach}(a) hold. In particular 
the homomorphisms ${\rm{ast}}_{l, K}$ and $\widetilde{{\rm{ast}}}_{l, K}$ in 
Diagram \ref{double cubic diagram} are monomorphisms. Hence the proposition follows by \eqref{decomposition of GlK1alg into tilde(GlK1alg) cup -Id tilde(GlK1alg)} and Corollary \ref{AST V psi with resp. to AST V psi cup -1Id AST V psi}(b).
\end{proof} 

Because of Proposition \ref{AST conjecture for math cal is equivalent to  AST conjecture wide tilde math cal}, it is enough to discuss the Algebraic Sato--Tate conjecture (in the framework of Conjectures  \ref{Tate conjecture for families of l-adic representations}
and \ref{Tate conjecture for families of l-adic representations tilde}) only for $\AST_{K}(V, \psi)$.  

\begin{corollary} Assume Conjectures \ref{Tate conjecture for families of l-adic representations}(a) and \ref{Tate conjecture for families of l-adic representations tilde}(a). Then the following four conjectures are equivalent:
\begin{itemize}
\item{} Tate Conjecture \ref{Tate conjecture for families of l-adic representations},
\item{} Tate Conjecture \ref{Tate conjecture for families of l-adic representations tilde},
\item{} Algebraic Sato--Tate Conjecture \ref{general algebraic Sato Tate conj.},
\item{} Algebraic Sato--Tate Conjecture  \ref{general algebraic Sato Tate conj. Serre's approach}.
\end{itemize}
\label{equivalence of 4 conjectures}
\end{corollary}

\begin{proof}
It follows by Corollary \ref{Tate conjecture is equivalent to the AST conjecture} and
Proposition \ref{AST conjecture for math cal is equivalent to  AST conjecture wide tilde math cal}.
\end{proof}

\begin{remark} Under any of the four equivalent conjectures in Corollary 
\ref{equivalence of 4 conjectures}, the group schemes $\AST_{K} (V, \psi)$ and
$\widetilde{\AST_{K}} (V, \psi)$ are uniquely defined, cf. Proposition 
\ref{uniqueness of algebraic Sato--Tate groups}(a) and (b).
\label{uniqueness of AST group and AST tilde groups}
\end{remark}

\begin{proposition} Assume Conjectures \ref{Tate conjecture for families of l-adic representations}(a) 
and \ref{Tate conjecture for families of l-adic representations tilde}(a).
Then the Sato--Tate Conjecture for $\ST_{K} (V, \psi)$
(Conjecture~\ref{general Sato Tate conj.})
implies the Sato--Tate Conjecture for $\widetilde{\ST_{K}} (V, \psi)$ (Conjecture~\ref{general Sato Tate conj. tilde}).
\label{ST (V, psi) implies wide tilde ST (V, psi)}
\end{proposition}
\begin{proof}
This follows by Theorem 
\ref{equidistribution property in ST implies equidistribution property in wide tilde ST}.
\end{proof}

The Sato--Tate parity group $\gpP_{\rm{ST}} (V, \psi)$ was introduced in Definition \ref{The Sato--Tate parity group}.
Under Conjectures \ref{Tate conjecture for families of l-adic representations}(a) and \ref{Tate conjecture for families of l-adic representations tilde}(a) and their consequences, the Definition \ref{The Sato--Tate parity group} of 
$\gpP_{\rm{ST}} (V, \psi)$ is clearly valid in the framework of this section:
\begin{equation}
\gpP_{\rm{ST}} (V, \psi) \, :=  \, \AST_{K} (V, \psi) / \widetilde{\AST_{K}} (V, \psi) 
\, =  \, \ST_{K} (V, \psi) / \widetilde{\ST_{K}} (V, \psi) \, .
\label{P ST under Conjectures 8.1 and 8.2}
\end{equation}
Due to Corollary \ref{AST V psi with resp. to AST V psi cup -1Id AST V psi}, $\gpP_{\rm{ST}} (V, \psi)$ is either 
trivial or isomorphic to $\Z/2\Z$.
\medskip

Now we can extend Theorem
\ref{equidistribution property in ST implies equidistribution property in wide tilde ST} to the setup of this section.
\begin{theorem} Assume Conjectures \ref{Tate conjecture for families of l-adic representations}(a) and \ref{Tate conjecture for families of l-adic representations tilde}(a).  
\begin{itemize}
\item[(a)]
If $\gpP_{\rm{ST}} (V, \psi)$ is nontrivial, i.e. $- {{\rm Id}}_{V} \notin \widetilde{\ST_{K}} (V, \psi)$, then Sato--Tate Conjecture \ref{general Sato Tate conj.} does not hold. 
\item[(b)]
If 
$\gpP_{\rm{ST}} (V, \psi)$ is trivial, i.e. $- {{\rm Id}}_{V} \in \widetilde{\ST_{K}} (V, \psi)$,
then Sato--Tate Conjectures \ref{general Sato Tate conj.}  and \ref{general Sato Tate conj. tilde}
are equivalent. 
\end{itemize}
\label{equidistribution property in ST implies equidistribution property in wide tilde ST for 
PST under Conj. 8.1 (a) and 8.2 (a)}
\end{theorem}
\begin{proof} This follows by the proof of Theorem
\ref{equidistribution property in ST implies equidistribution property in wide tilde ST}.
\end{proof}

\begin{definition} 
\label{Sato--Tate parity under Conj. 8.1 (a) and 8.2 (a)}
Assume Conjectures \ref{Tate conjecture for families of l-adic representations}(a) and \ref{Tate conjecture for families of l-adic representations tilde}(a). Define the \emph{Serre parity group} $\gpP_{\rm{S}} (V, \psi)$:
$$
\gpP_{\rm{S}} (V, \psi) \, := \, \mathcal{G}_{K, 1}^{\alg, \, 0} / (\mathcal{G}_{K, 1}^{\alg})^{\circ}.
$$
\end{definition} 

Consider the following commutative diagram. 

\begin{figure}[H]
\[
\begin{tikzcd}
1 \arrow{r}{} &  (\widetilde{\mathcal{G}_{K}^{\alg}})^{\circ} \arrow{d}[swap]{}{\pi^{\circ}} \arrow{r}{}
&  \widetilde{\mathcal{G}_{K}^{\alg}} \arrow{d}[swap]{}{\pi} \arrow{r}{}  & \pi_{0} 
(\widetilde{\mathcal{G}_{K}^{\alg}})   \arrow{d}[swap]{}{\overline{\pi}} \arrow{r}{} &  1\\  
1 \arrow{r}{} & (\mathcal{G}_{K}^{\alg})^{\circ} \arrow{d}{} \arrow{r}{} & 
\mathcal{G}_{K}^{\alg} \arrow{d}{} \arrow{r}{} & \pi_{0}(\mathcal{G}_{K}^{\alg}) 
\arrow{d}[swap]{} \arrow{r}{} &  1 \\
&  1  & 1   & 1\\
\end{tikzcd}
\]
\\[-0.8cm]
\caption{}
\label{diagram to prove analog of Serre theorem for families of l-adic representations under weak Tate 
conjectures part (a) lll}
\end{figure}

In Diagram \ref{diagram to prove analog of Serre theorem for families of l-adic representations under weak Tate 
conjectures part (a) lll}, the map $\pi$ is an epimorphism by the setup of this section; consequently the map $\overline{\pi}$ is an epimorphism.
By Diagram \ref{diagram compatibility of tilde(GK1alg) with GK1alg} the map $\pi$ has finite kernel. Hence the algebraic groups $\widetilde{\mathcal{G}_{K}^{\alg}}$ and 
$\mathcal{G}_{K}^{\alg}$ have the same dimension and the map $\pi^{\circ}$ has finite kernel. Because the image of   
$\pi^{\circ}$ is a closed subgroup of $(\mathcal{G}_{K}^{\alg})^{\circ}$ \cite[Section 7.4, Prop. B(b)]{Hu} of the same dimension as $(\mathcal{G}_{K}^{\alg})^{\circ}$, the image is also open. Because the group $(\mathcal{G}_{K}^{\alg})^{\circ}$ is connected, the map $\pi^{\circ}$ must be an epimorphism. By the snake lemma applied to Diagram 
\ref{diagram to prove analog of Serre theorem for families of l-adic representations under weak Tate 
conjectures part (a) lll} there is the following short exact sequence:
\begin{equation}
\label{ker pi 0  Ker pi Ker overline pi}
1 \rightarrow {\rm{Ker}} \, \pi^{\circ} \rightarrow {\rm{Ker}} \, \pi \rightarrow {\rm{Ker}} \, \overline{\pi} \rightarrow 1.
\end{equation} 

\begin{proposition} 
\label{Proposition 1 rightarrow P T rightarrow P S rightarrow P ST rightarrow 1}
There is the following short exact sequence:
\begin{equation}
1 \rightarrow {\rm{Ker}} \, \overline{\pi} \rightarrow \gpP_{\rm{S}} (V, \psi) \rightarrow \gpP_{\rm{ST}} (V, \psi)
\rightarrow 1.
\label{1 rightarrow P T rightarrow P S rightarrow P ST rightarrow 1}
\end{equation}
\end{proposition}

\begin{proof} By Definition \ref{AST K and tilde AST K under the analog of Tate conjecture for l-adic rep.} and Diagram 
\ref{wide tilde math cal GK1alg with math cal GK1 and p0 wide tilde math cal with pi0 math cal} we have:
\begin{equation}
\gpP_{\rm{ST}} (V, \psi) = \pi_{0} (\mathcal{G}_{K, 1}^{\alg}) \, / \,
\pi_{0} (\widetilde{\mathcal{G}_{K, 1}^{\alg}}).
\label{P ST = p0 ( mathcal G alg K 1) over p0 (tilde mathcal G alg K 1)}
\end{equation}
Now the claim follows by the snake lemma applied to Diagram 
\ref{pi0 applied to the diagram compatibility of wide tilde math cal G alg K 1 with math cal G alg K 1}. 
\end{proof}

\section{The category of polarized realizations \texorpdfstring{$R_{K}^{\rm{p}}$}{RKp}}
\label{sec,category of polarized realizations}

Let $L \subset \C$ be a field. Let $w\colon \mathcal{C} \rightarrow {\rm{Vec}}_{L}$ be an exact, faithful, $\otimes$-functor from an abelian, $L$-linear, rigid, $\otimes$-category $\mathcal{C}$ to ${\rm{Vec}}_{L}$ (the category of finite-dimensional $L$-vector spaces).

\begin{definition} The group scheme $\Aut^{\otimes} w$ is defined as follows:
\medskip

For every $L$-algebra $R$, $\Aut^{\otimes}w \, (R)$ consists of families $(\lambda_{X})$, 
$X \in {\rm{obj}} (\mathcal{C})$ such that $\lambda_{X}$ is an $R$-linear automorphism:
$$
\lambda_{X}\colon w (X) \otimes R \,\, {\stackrel{\simeq}{\longrightarrow}} \,\,  w (X) \otimes R
$$
satisfying the following conditions:
\begin{itemize}
\item[(a)] $\lambda_{X} \otimes \lambda_{X^{\prime}} = \lambda_{X \otimes X^{\prime}}$, \quad \,\, for all
$X, X^{\prime} \, \in \, {\rm{obj}} (\mathcal{C})$,
 \item[(b)] $\lambda_{{\bf{1}}}\colon w({\bf{1}}) \otimes R \,  
{\stackrel{{\rm{id}}}{\longrightarrow}} \, w({\bf{1}}) \otimes R$, 
\item[(c)] the following Diagram \ref{definition od Aut tensor w} commutes for all $\alpha \in \Hom_{\mathcal{C}} (X, Y)$ and all
$X, Y \in {\rm{obj}} (\mathcal{C})$.
\end{itemize}
\begin{figure}[H]
\[
\begin{tikzcd}[column sep=huge, row sep=large]
w (X) \otimes R \arrow{r}{\lambda_{X}}[swap]{\simeq}
\arrow{d}{w(\alpha) \otimes 1}
& w (X) \otimes R \arrow{d}{w(\alpha) \otimes 1} \\ 
w (Y) \otimes R  \arrow{r}{\lambda_{Y}}[swap]{\simeq} 
 & w (Y) \otimes R  \\
\end{tikzcd}
\]
\\[-0.8cm]
\caption{}
\label{lambda X and lambda Y commute with w(alpha) otimes 1}
\end{figure}
\label{definition od Aut tensor w}
\end{definition}

\begin{theorem} {\bf{(Tannaka Duality)}} Let ${\rm{Rep}}_{L} G$ denote the category of finite 
dimensional, $L$-linear representations of a group scheme $G/L$. 
\begin{itemize}
\item[(a)] ${\rm{Rep}}_{L} G$ is a neutral, Tannakian category with forgetful fiber functor $w^{G}\colon {\rm{Rep}}_{L} \, G \rightarrow \, {\rm{Vec}}_{L}$ such that
$\Aut^{\otimes}w^{G} = G$.
\item[(b)] If $\mathcal{C}$ is an $L$-linear, neutral, Tannakian category 
with fiber functor $w\colon \mathcal{C} \rightarrow {\rm{Vec}}_{L}$,
then there exist a group scheme $G/L$ satisfying $\Aut^{\otimes}w = G$ and an equivalence of categories
$\mathcal{C} \,\, {\stackrel{\simeq}{\longrightarrow}} \,\, {\rm{Rep}}_{L} G$ via which $w$ corresponds to $w^G$.
\end{itemize}
\label{Tannaka duality theorem}
\end{theorem}
\begin{proof} See \cite[pp 128--138]{DM}.
\end{proof}

Below, up to Proposition \ref{the family V l l of G K satisfies Conjectures 1 (a) and 2 (a)}, $K$ can be any field of characteristic 0. Jannsen \cite{Ja90} introduced the $\Q$-linear, Tannakian category of \emph{mixed realizations} $MR_{K}$ and its $\Q$-linear Tannakian full subcategory of \emph{realizations} $R_{K}$. An object of $R_{K}$ is a direct sum of objects of the following form:
\begin{equation}
\mathcal{V} \, := \, (\, V_{_{\rm{DR}}}, \, (V_l)_{l}, \, (V_{\sigma})_{\sigma}, \, 
I_{\infty, \sigma}, \,  I_{l, \bar{\sigma}}),   
\end{equation}
where $\sigma$ and $\bar{\sigma}$ run over all embeddings 
$\sigma\colon K \hookrightarrow \C$, \, $\bar{\sigma}\colon \overline{K} \hookrightarrow \C$, such that
$\bar{\sigma}_{| K} = \sigma$ and $l$ runs over all prime numbers; $V_{_{\rm{DR}}}$ is a $K$-vector space; $V_l$ is a $\Q_l$-vector space with $G_{K}$-module structure; and $V_{\sigma}$ is a pure
$\Q$-Hodge structure of weight $n$ for any $n \in \Z$. In addition, the vector spaces $V_{_{\rm{DR}}}$, $V_l$, and $V_{\sigma}$ 
are finite dimensional and there are comparison isomorphisms:
\begin{align}
I_{\infty, \sigma}\colon& V_{\sigma} \otimes_{\Q} \C \,\,  
{\stackrel{\simeq}{\longrightarrow}} \,\, V_{_{\rm{DR}}} \otimes_{K, \sigma} \C,
\label{comparison sigma with DR} \\
I_{l, \bar{\sigma}}\colon& V_{\sigma} \otimes_{\Q} \Q_l \,\,  
{\stackrel{\simeq}{\longrightarrow}} \,\, V_{l},
\label{comparison sigma with l}
\end{align}
such that $g \circ I_{l, \bar{\sigma} \circ g} = I_{l, \bar{\sigma}}$ for all $g \in G_K$.
Notice that $V_{\sigma} = V_{\sigma \circ g}$. Such $\mathcal{V}$ is called a {\em{pure realization}}
of weight $n$. 

In particular we have the \emph{identity realization} and the \emph{Tate realization}, respectively:
\begin{align*}
{\bf{1}} \, &:= \, (\, K, \,  (\Q_{l})_{l}, \,  (\Q)_\sigma, \,
I_{\infty, \sigma}, \,  I_{\infty, \bar{\sigma}}) \\
{\bf{1}}(1) \, &:= \, (\, K (1), \,  (\Q_{l} (1))_{l}, \,  (\Q (1))_{\sigma}, \,
I_{\infty, \sigma}, \,  I_{\infty, \bar{\sigma}})
\end{align*}

In what follows, we will write $\mathcal{V}$, and sometimes $\mathcal{V}'$, to denote an arbitrary element of $R_K$.

\begin{remark} See \cite[p. 11]{Ja90} for the definition of 
$\Hom_{R_{K}} (\mathcal{V}, \, \mathcal{V}^{'})$. 
\label{Homs from mathcal V to mathcal V prime}
\end{remark}

\begin{definition}
The category $R_{K}^{\rm{p}}$ of \emph{polarized realizations} is defined to be the full subcategory of $R_K$ 
generated, for all $n \in \Z$, by objects: 
\begin{equation}
(\mathcal{V}, \psi) \, := \, (\, (V_{_{\rm{DR}}}, \psi_{_{\rm{DR}}}), \, 
(V_l, \psi_{l})_{l}, \, (V_{\sigma}, \psi_{\sigma})_{\sigma}, \, 
I_{\infty, \sigma}, \,  I_{l, \bar{\sigma}}) 
\label{polarized realization}
\end{equation}
where $(V_{\sigma}, \psi_{\sigma})_{\sigma}$ are pure, polarized, rational Hodge structures
of weight $n$, for any $n \in \Z$, and $(V_{_{\rm{DR}}}, \psi_{_{\rm{DR}}})$ is associated with 
$(V_{\sigma}, \psi_{\sigma})_{\sigma}$  (cf. \eqref{comparison VC and VDRC}, 
\eqref{comparison psi for VC and VDRC} in \S \ref{de Rham structures associated with Hodge structures}) as follows:
\begin{align}
I_{\infty, \sigma}\colon& V_{\sigma} \otimes_{\Q} \C \simeq V_{_{\rm{DR}}} \otimes_{K, \sigma} \C,
\label{comparison VC and VDRC in realizations} \\
I_{\infty, \sigma}\colon& \psi_{\sigma} \otimes_{\Q} \C \, \simeq \, \psi_{_{\rm{DR}}} \otimes_{K, \sigma} \C.
\label{comparison psi for VC and VDRC in realizations}
\end{align}
The family of $l$-adic representations $(V_l, \psi_{l})_{l}$ is associated with 
$(V_{\sigma}, \psi_{\sigma})_{\sigma}$ as follows:
\begin{align}
I_{l, \bar{\sigma}}\colon& V_{\sigma} \otimes_{\Q} \Q_l \simeq \, V_{l},
\label{comparison VC and Vl in realizations} \\
I_{l, \bar{\sigma}}\colon& \psi_{\sigma} \otimes_{\Q} \Q_l \, \simeq \, 
\psi_{l}.
\label{comparison psi for VC and Vl in realizations}
\end{align}
such that $g \circ I_{l, \bar{\sigma} \circ g} = I_{l, \bar{\sigma}}$ for all $g \in G_K$.
\label{definition of polarized realizations category} 
\end{definition} 

\begin{lemma}
The category $R_{K}^{\rm{p}}$ is semisimple.
\label{RK pol is semisimple}
\end{lemma}
\begin{proof}
Consider the subobject $(\mathcal{V}', \psi') \subset (\mathcal{V}, \psi)$ i.e.: 
\begin{gather*}
(\mathcal{V}', \psi') \, = \, (\, (V_{_{\rm{DR}}}', \psi_{_{\rm{DR}}}'), \, 
(V_{l}', \psi_{l}')_{l}, \, (V_{\sigma}', \psi_{\sigma}')_{\sigma}, \, 
I_{\infty, \sigma}', \,  I_{l, \bar{\sigma}}'), \\
V_{_{\rm{DR}}}' \subset V_{_{\rm{DR}}}, \quad\quad V_{l}' \subset V_{l}, \quad\quad V_{\sigma}' 
\subset V_{\sigma}, \\
\psi_{_{\rm{DR}}}' = {\psi_{_{\rm{DR}}}}_{| V_{_{\rm{DR}}}'}, \quad\quad \psi_{l}' = {\psi_{l}}_{| V_{l}'}, \quad\quad \psi_{\sigma}' = {\psi_{\sigma}}_{| V_{\sigma}'}
\end{gather*}
and $I_{\infty, \sigma}'$  (resp. $I_{l, \bar{\sigma}}'$) is the obvious restriction of 
${I_{\infty, \sigma}}$ (resp. $I_{l, \bar{\sigma}}$). Because 
$\psi_{\sigma}$ is nondegenerate for each $\sigma$ $\psi_{_{\rm{DR}}}$
and $\psi_{l}$ for all $l$ are nondegenerate by \eqref{comparison psi for VC and VDRC in realizations}
and \eqref{comparison psi for VC and Vl in realizations}. Let \, 
$V_{_{\rm{DR}}}'', \,\, V_{l}'', \,\, V_{\sigma}''$ be the orthogonal complements of 
$V_{_{\rm{DR}}}', \,\, V_{l}', \,\, V_{\sigma}'$ with respect to the forms
$\psi_{_{\rm{DR}}}, \,\, \psi_{l}, \,\, \psi_{\sigma}$, respectively. Hence we obtain nondegenerate forms
$$
\psi_{_{\rm{DR}}}'' = {\psi_{_{\rm{DR}}}}_{| V_{_{\rm{DR}}}''}, \quad\quad \psi_{l}'' = {\psi_{l}}_{| V_{l}''}, \quad\quad \psi_{\sigma}'' = {\psi_{\sigma}}_{| V_{\sigma}''}
$$ and a well-defined restriction
$I_{\infty, \sigma}''$  (resp. $I_{l, \bar{\sigma}}''$) of 
${I_{\infty, \sigma}}$ (resp. $I_{l, \bar{\sigma}}$). Hence:
$$
(\mathcal{V}'', \psi'') \, = \, (\, (V_{_{\rm{DR}}}'', \psi_{_{\rm{DR}}}''), \, 
(V_{l}'', \psi_{l}'')_{l}, \, (V_{\sigma}'', \psi_{\sigma}'')_{\sigma}, \, 
I_{\infty, \sigma}'', \,  I_{l, \bar{\sigma}}'')
$$
is a well-defined polarized realization and it is clear that
\begin{align*}
(\mathcal{V}, \psi) &= (\mathcal{V}', \psi') \oplus (\mathcal{V}'', \psi''). \qedhere
\end{align*}
\end{proof}

\begin{remark} Actually we can prove more than Lemma \ref{RK pol is semisimple}. Namely, if 
$\mathcal{V}' \in {\rm{obj}} (R_{K})$ and $\mathcal{V}' \subset \mathcal{V}$ 
for $(\mathcal{V}, \psi) \in {\rm{obj}} (R_{K}^{\rm{p}})$, then 
$\psi_{\sigma}' := {\psi_{\sigma}}_{| V_{\sigma}'}$ is a polarization on
$V_{\sigma}'$ because the Weil operator $C\colon V_{\sigma} \rightarrow V_{\sigma}$ restricts naturally to the Hodge structure $V_{\sigma}'$. Now defining $\psi_{_{\rm{DR}}}' = {\psi_{_{\rm{DR}}}}_{| V_{_{\rm{DR}}}'}, \,\,  \psi_{l}' = {\psi_{l}}_{| V_{l}'}$ and observing that 
$I_{\infty, \sigma}$  (resp. $I_{l, \bar{\sigma}}$) naturally restricts to the comparison isomorphism
${I_{\infty, \sigma}}'$ (resp. $I_{l, \bar{\sigma}}'$) on $\mathcal{V}' \subset \mathcal{V}$, we obtain
the polarized realization $(\mathcal{V}', \psi') \subset (\mathcal{V}, \psi)$.
\end{remark}
\medskip

The polarization $\psi\colon V \times V \rightarrow \Q (-n)$ leads to two isomorphisms 
$V(n) \,\, {\stackrel{\simeq}{\longrightarrow}} \,\,  V^{\vee}$ of $\Q$-vector spaces: 
$$
v \otimes (2 \pi \, i)^n \,\, \mapsto \,\, \psi(v, \, \cdot) \quad
{\rm{and}}
\quad
w \otimes (2 \pi \, i)^n \,\, \mapsto \,\,  \psi(\cdot, \, w).
$$ 
This gives a $\Q$-bilinear, nondegenerate pairing:
\begin{gather*}
\psi^{\vee}\colon V^{\vee} \times V^{\vee} \rightarrow \Q (n),
\\ 
\psi^{\vee} (\psi(v, \, \cdot) , \, \psi(\cdot, \, w)) :=  (2 \pi \, i)^{2n} \psi (v, w).
\end{gather*}
We define the action of Weil's operator $C\colon V \rightarrow V$ on $V^{\vee}$ as follows.
For $f \in V^{\vee}$ we put $C(f) := f \circ C$.  
Because $\psi$ is a polarization, the form 
$$
(2 \pi \, i)^{-n} \psi^{\vee} (\psi(C (v), \, \cdot) , \, \psi(\cdot, \, w)) =  
(2 \pi \, i)^{n} \psi (C(v), w)
$$ 
is positive definite on $V^{\vee} \otimes_{\Q} \R$.
Hence $\psi^{\vee}$ is a polarization. 
In the same way as above we construct natural nondegenerate pairings 
$\psi_{_{\rm{DR}}}^{\vee}$ and $\psi_{l}^{\vee}$. The comparison isomorphisms $I_{\infty, \sigma}^{\vee}$ and $I_{l, \bar{\sigma}}^{\vee}$ are the transposes of 
$I_{\infty, \sigma}$ and $I_{l, \bar{\sigma}}$ (see \cite[p. 15]{Ja90}). Hence the realization 
$$
\mathcal{V}^{\vee} \, := \, 
(\, V_{_{\rm{DR}}}^{\vee}, \, (V_{l}^{\vee})_{l}, \, (V_{\sigma}^{\vee})_{\sigma}, \, 
I_{\infty, \sigma}^{\vee}, \,  I_{l, \bar{\sigma}}^{\vee})
$$
is in the category $R_{K}^{\rm{p}}$.

\begin{remark} 
Because $R_{K}^{\rm{p}}$ is a full subcategory of $R_{K}$, we will also write
$\mathcal{V}$ instead of $(\mathcal{V}, \psi)$ for objects of $R_{K}^{\rm{p}}$ whenever
it does not lead to a conflict of notation or misunderstanding.
\end{remark}

\begin{proposition}
The category $R_{K}^{\rm{p}}$ is a $\Q$-rational, neutral, semisimple, Tannakian subcategory of $R_{K}$.
\end{proposition}
\begin{proof}
By the above computations and \cite[pp. 12--15]{Ja90} 
we observe that if $\mathcal{V}$ and $\mathcal{W}$ are objects of $R_{K}^{\rm{p}}$
then $\mathcal{V} \otimes \mathcal{W}$ and $\underline{\Hom} (\mathcal{V}, \mathcal{W}) =
\mathcal{V}^{\vee} \otimes \mathcal{W}$ are also objects of $R_{K}^{\rm{p}}$.
Moreover it is clear that ${\mathcal{V}}^{\vee\vee} = \mathcal{V}$. Applying Lemma 
\ref{RK pol is semisimple} finishes the proof.
\end{proof}

Fix $\sigma$ and put $V := V_{\sigma}$. Consider the fiber functor $H := H_{B}$:
\begin{gather*}
H\colon R_{K}^{\rm{p}} \rightarrow {\rm{Vec}}_{\Q}, \\
H (\mathcal{V}) \, = \, V.
\end{gather*}
Now fix $\mathcal{V} \in {\rm{obj}} (R_{K}^{\rm{p}})$ and put $\mathcal{W} := 
\mathcal{V} \oplus {\bf{1}}(1) \in {\rm{obj}} (R_{K}^{\rm{p}})$.
Let $R_{K} [\mathcal{V}]$ (resp. $R_{K} [\mathcal{W}]$) denote the smallest strictly full Tannakian subcategory of $R_{K}^{\rm{p}}$ containing $\mathcal{V}$ (resp. containing $\mathcal{W}$).
Consider the fiber functors $H_{\mathcal{V}} := H_{|_{R_{K} [\mathcal{V}]}}$
and $H_{\mathcal{W}} := H_{|_{R_{K} [\mathcal{W}]}}$. We put:
\begin{align*}
\mathcal{G}_{K}^{\alg} \, &:= \, \mathcal{G}_{K}^{\alg} (V, \psi) \, := \, {\rm{Aut}}^{\otimes} 
H_{\mathcal{V}} \subset \GIso_{(V, \psi)} \\
\widetilde{\mathcal{G}_{K}^{\alg}} \, &:= \, 
\widetilde{\mathcal{G}_{K}^{\alg}} (W) \, := \, {\rm{Aut}}^{\otimes} H_{\mathcal{W}}
\subset \GL_{W}.
\end{align*}

\begin{remark}
It follows directly from Definition \ref{definition od Aut tensor w} (see especially Diagram 
\ref{lambda X and lambda Y commute with w(alpha) otimes 1}) applied to the Tannakian 
category $R_{K} [\mathcal{V}]$ (resp. $R_{K} [\mathcal{W}]$) that for
every $\Q$-algebra $R$, each family $(\lambda_{\mathcal{X}})$,  $\mathcal{X} \in {\rm{obj}} 
(R_{K} [\mathcal{V}])$ (resp. $\mathcal{X} \in {\rm{obj}} (R_{K} [\mathcal{W}]$)) is determined 
uniquely by $\lambda_{\mathcal{V}}$ and $\lambda_{\mathcal{V}^{\vee}}$ (resp. by $\lambda_{\mathcal{W}}$
and $\lambda_{\mathcal{W}^{\vee}}$). 
\label{Aut in the category R K generated by mathcal V  or mathcal W}
\end{remark}

Consider the $\Q_l$-linear, Tannakian, neutral category of finite dimensional, continuous,
Galois representations ${\rm{Rep}}_{\Q_l} G_K$ with the forgetful fiber functor
$$
H_l\colon {\rm{Rep}}_{\Q_l} G_K \rightarrow {\rm{Vec}}_{\Q_l}.
$$ 
For the object
$\mathcal{V}$ (resp. for the object $\mathcal{W}$) consider the $l$-adic realization 
$V_l$ (resp. the $l$-adic realization $W_l$). Let 
${\rm{Rep}}_{\Q_l} G_K [V_l]$ (resp. ${\rm{Rep}}_{\Q_l} G_K [W_l]$) denote the smallest strictly full Tannakian subcategory of ${\rm{Rep}}_{\Q_l} G_K$ containing $V_l$ 
(resp. containing $W_l$). Let $H_{V_l}$ (resp. $H_{W_l}$) denote the fiber functor $H_l$ restricted to ${\rm{Rep}}_{\Q_l} G_K [V_l]$ (resp. ${\rm{Rep}}_{\Q_l} G_K [W_l]$). 

\begin{remark}
Again from Definition \ref{definition od Aut tensor w} applied to the Tannakian 
category ${\rm{Rep}}_{\Q_l} G_K [V_l]$ (resp. ${\rm{Rep}}_{\Q_l} G_K [W_l]$) it follows that for every $\Q_l$-algebra $R$ each family $(\lambda_{U_{l}})$,  $U_{l} \in 
{\rm{obj}} ({\rm{Rep}}_{\Q_l} G_K [V_l])$ (resp. $U_{l} \in {\rm{obj}} ({\rm{Rep}}_{\Q_l} G_K [V_l])$ is determined uniquely by $\lambda_{V_{l}}$ and $\lambda_{V_{l}^{\vee}}$ (resp. by $\lambda_{W_{l}}$
and $\lambda_{W_{l}^{\vee}}$). 
\label{Aut in the category Rep Ql GK generated by V l or W l}
\end{remark}
It is known \cite[section 7.1.3, p.70]{An1} that
$$
G_{l, K}^{\alg} (V_l) \, = \, {\rm{Aut}}^{\otimes} \, H_{V_l}, \quad\quad
\widetilde{G_{l, K}^{\alg}} (W_l) \, = \, {\rm{Aut}}^{\otimes} \, H_{W_{l}}.
$$

\begin{lemma}
\label{G l, K alg V_l naturally embeds into G K alg Q l}
There are natural embeddings:
$$
G_{l, K}^{\alg} (V_l) \, \subset \, \mathcal{G}_{K}^{\alg} (V, \psi)_{\Q_l}, 
\quad\quad \widetilde{G_{l, K}^{\alg}} (W_l) \, \subset \, 
\widetilde{\mathcal{G}_{K}^{\alg}} (W)_{\Q_l}.
$$ 
\end{lemma}
\begin{proof}
Consider $(\lambda_{U_l}) \in {\rm{Aut}}^{\otimes} \, H_{V_l} (\Q_l)$ (resp.
$(\lambda_{U_l}) \in {\rm{Aut}}^{\otimes} \, H_{W_l} (\Q_l)$) with 
$U_l \in {\rm{obj}} ({\rm{Rep}}_{\Q_l} G_K [V_l])$ (resp. 
$U_l \in {\rm{obj}} ({\rm{Rep}}_{\Q_l} G_K [W_l])$). The main constraint on the family $(\lambda_{U_l})$ 
is commutativity of the following Diagram~\ref{lambda U l and lambda U prime l commute with alpha l} for all $\alpha_l \in 
\Hom_{\Q_l [G_K]} (U_l, U^{\prime}_{l})$:

\begin{figure}[H]
\[
\begin{tikzcd}[column sep=huge, row sep=large]
U_l \arrow{r}{\lambda_{U_l}}[swap]{\simeq}
\arrow{d}{\alpha_l}
& U_l\arrow{d}{\alpha_l} \\ 
U_{l}^{\prime} \arrow{r}{\lambda_{ U^{\prime}_{l}}}[swap]{\simeq} 
 & U_{l}^{\prime} \\
\end{tikzcd}
\]
\\[-0.8cm]
\caption{}
\label{lambda U l and lambda U prime l commute with alpha l}
\end{figure}

In the family $(\lambda_{U_{l}})$ there is a
subfamily $(\lambda_{X_l})$ where $X_l = H_{V_l} (\mathcal{X})$ with 
$\mathcal{X} \in {\rm{obj}} (R_{K} [\mathcal{V}])$ (resp. $\mathcal{X} \in {\rm{obj}} 
(R_{K} [\mathcal{W}]$) such that for all $X_l = H_{V_l} (\mathcal{X})$, $Y_l = H_{V_l} (\mathcal{Y})$, $\alpha_l =  H_{V_l} (\alpha)$ with $\alpha \in \Hom_{(R_{K} [\mathcal{V}])} (\mathcal{X}, \mathcal{Y})$  (resp. $X_l = H_{W_l} (\mathcal{X})$, $Y_l = H_{W_l} (\mathcal{Y})$,
$\alpha \in \Hom_{(R_{K} [\mathcal{W}])} (\mathcal{X}, \mathcal{Y})$) the following Diagram~\ref{lambda X l and lambda Y l commute with alpha l} commutes for all $l$:

\begin{figure}[H]
\[
\begin{tikzcd}[column sep=huge, row sep=large]
X_l \arrow{r}{\lambda_{X_l}}[swap]{\simeq}
\arrow{d}{\alpha_l}
& X_l\arrow{d}{\alpha_l} \\ 
Y_{l} \arrow{r}{\lambda_{ Y_{l}}}[swap]{\simeq} 
 & Y_{l} \\
\end{tikzcd}
\]
\\[-0.8cm]
\caption{}
\label{lambda X l and lambda Y l commute with alpha l}
\end{figure}

Observe that $(I_{l, \bar{\sigma}})^{-1} (X_l) = H_{\sigma} (\mathcal{X}) \otimes_{\Q} \Q_{l}$  
and $(I_{l, \bar{\sigma}})^{-1} (\alpha_l) = H_{\sigma} (\alpha) \otimes 1$.
 Applying $(I_{l, \bar{\sigma}})^{-1}$ to the family $(\lambda_{X_l})$ and Diagram
\ref{lambda X l and lambda Y l commute with alpha l}, we obtain a family 
$(\lambda_{\mathcal X})$ with $\lambda_{\mathcal X} := (I_{l, \bar{\sigma}})^{-1} (\lambda_{X_l})$
and a commutative Diagram \ref{lambda X and lambda Y commute with H sigma of alpha} for each $l$:

\begin{figure}[H]
\[
\begin{tikzcd}[column sep=huge, row sep=large]
H_{\sigma} (\mathcal{X}) \otimes_{\Q} \Q_{l}  \arrow{r}{\lambda_{\mathcal{X}}}[swap]{\simeq}
\arrow{d}{H_{\sigma} (\alpha) \otimes 1}
& H_{\sigma} (\mathcal{X}) \otimes_{\Q}\otimes 1 \Q_l \arrow{d}{H_{\sigma} (\alpha) \otimes 1} \\ 
H_{\sigma} (\mathcal{Y}) \otimes_{\Q} \Q_l \arrow{r}{\lambda_{\mathcal{Y}}}[swap]{\simeq} 
 & H_{\sigma} (\mathcal{Y}) \otimes_{\Q} \Q_l \\
\end{tikzcd}\
\]
\\[-0.8cm]
\caption{}
\label{lambda X and lambda Y commute with H sigma of alpha}
\end{figure}

The association $(\lambda_{U_{l}}) \, \mapsto \,  (\lambda_{\mathcal{X}})$ described above
gives natural embeddings:
$$
G_{l, K}^{\alg} (V_l) (\Q_l) \,\, \hookrightarrow \,\, \mathcal{G}_{K}^{\alg} (V, \psi) (\Q_l), 
\quad\quad \widetilde{G_{l, K}^{\alg}} (W_l) (\Q_l) \hookrightarrow 
\widetilde{\mathcal{G}_{K}^{\alg}} (W) (\Q_l).
$$
This finishes the proof because: 
\begin{gather*}
\rho_{l} (G_K) \, \subset \, G_{l, K}^{\alg} (V_l) (\Q_l) \quad {\rm{and}} \quad\widetilde{\rho}_{l} (G_K) \, \subset \, \widetilde{G_{l, K}^{\alg}} (W_l) (\Q_l). \qedhere
\end{gather*}
\end{proof}

The Tannakian categories $R_{K} [\mathcal{V}]$, $R_{K} [{\bf{1}}(1)]$, $R_{K} [\mathcal{W}]$ are semisimple as full  subcategories of $R_{K}^{\rm{p}}$. It follows by 
\cite[Prop. 2.21 (a), Remark 2.29]{DM} that the following homomorphisms are epimorphisms:
\begin{align}
\pi\colon& \widetilde{\mathcal{G}_{K}^{\alg}} (W) \,\, {\stackrel{{}}{\longrightarrow}} \,\,
\mathcal{G}_{K}^{\alg} (V, \psi),
\label{G K alg W surjection onto G K alg V} \\
N\colon& \widetilde{\mathcal{G}_{K}^{\alg}} (W) \,\, {\stackrel{{}}{\longrightarrow}} \,\, 
\G_m.
\label{G K alg W surjection onto G m}
\end{align}
Consider the projections:
$\pi_{\mathcal{V}}\colon \mathcal{W} \rightarrow \mathcal{V}$, 
$\pi_{{\bf{1}}(1)}\colon \mathcal{W} \rightarrow {\bf{1}}(1)$,
$\pi_{\mathcal{V}^{\vee}}\colon \mathcal{W}^{\vee} \rightarrow \mathcal{V}^{\vee}$,
$\pi_{{\bf{1}}(-1)}\colon \mathcal{W}^{\vee} \rightarrow {\bf{1}}(-1)$. 
Then ${\rm{Id_{\mathcal{W}}}} = \pi_{\mathcal{V}} + \pi_{{\bf{1}}(1)}$
and ${\rm{Id_{\mathcal{W}^{\vee}}}} = \pi_{\mathcal{V}^{\vee}} + \pi_{{\bf{1}}(-1)}$. 
Diagram \ref{lambda X and lambda Y commute with w(alpha) otimes 1} applied to the category $R_{K} [\mathcal{W}]$ and its subcategories $R_{K} [\mathcal{V}]$ and $R_{K} [{\bf{1}}(1)]$ gives:
\begin{align*}
\lambda_{\mathcal{W}} &= 
\lambda_{\mathcal{V}} \circ \, H(\pi_{\mathcal{V}}) \otimes 1 +  
\lambda_{{\bf{1}}(1)} \circ \, H(\pi_{{\bf{1}}(1)}) \otimes 1 \\
\lambda_{\mathcal{W}^{\vee}} &= 
\lambda_{\mathcal{V}^{\vee}} \circ \, H(\pi_{\mathcal{V}^{\vee}}) \otimes 1 +  
\lambda_{{\bf{1}}(-1)} \circ \, H(\pi_{{\bf{1}}(-1)}) \otimes 1.
\end{align*}
By Remark \ref{Aut in the category R K generated by mathcal V  or mathcal W} the following homomorphism is a closed immersion:
\begin{equation}
\label{G K alg W injection into G K alg V times G m}
\pi \times N\colon \widetilde{\mathcal{G}_{K}^{\alg}} (W) \,\, {\stackrel{{}}{\longrightarrow}} 
\,\, \mathcal{G}_{K}^{\alg} (V, \psi) \times \G_m.
\end{equation}

Consider $(\mathcal{V}, \psi) \, = \, (\, (V_{_{\rm{DR}}}, \psi_{_{\rm{DR}}}), \, 
(V_l, \psi_{l})_{l}, \, (V_{\sigma}, \psi_{\sigma})_{\sigma}, \, 
I_{\infty, \sigma}, \,  I_{l, \bar{\sigma}}) \in {\rm{obj}} (R_{K}^{\rm{p}})$.
Define (cf. Definition \ref{definition of polarized realizations category} and  
\cite[Def. 2.16]{Ja90}):
$$
(\overline{\mathcal{V}}, \overline{\psi}) \, := \, 
(\mathcal{V} \otimes_{K} \overline{K},  \psi  \otimes_{K} \overline{K}) := 
$$  
$$ 
:= \, (\, (\overline{V}_{_{\rm{DR}}}, \overline{\psi}_{_{\rm{DR}}}),  \, 
(\overline{V}_l, \overline{\psi}_{l})_{l}, \, 
(\overline{V}_{\bar{\sigma}}, \overline{\psi}_{\bar{\sigma}})_{\bar{\sigma}}, \, 
\bar{I}_{\infty, \bar{\sigma}}, \, \ \overline{I}_{l, \bar{\sigma}}) \in {\rm{obj}} (R_{\overline{K}}^{\rm{p}}),
$$
where $(\overline{V}_{_{\rm{DR}}}, \overline{\psi}_{_{\rm{DR}}}) := 
(V_{_{\rm{DR}}} \otimes_{K} \overline{K}, \psi_{_{\rm{DR}}} \otimes_{K} \overline{K})$, \,
$(\overline{V}_l, \overline{\psi}_{l})_{l} := (V_l, \psi_{l})_{l}$, \, 
$\bar{I}_{l, \bar{\sigma}} := I_{l, \bar{\sigma}}$ \, and
\, $(\overline{V}_{\bar{\sigma}}, \overline{\psi}_{\bar{\sigma}})_{\bar{\sigma}} := 
(V_{\sigma}, \psi_{\sigma})_{\sigma}$, \, $\bar{I}_{\infty, \bar{\sigma}} := I_{\infty, \sigma}$, \, $g \circ I_{l, \bar{\sigma} \circ g} = I_{l, \bar{\sigma}}$
for all ${l}$, for all $g \in G_{K}$, for all $\sigma\colon K \hookrightarrow \C$, and for all $\bar{\sigma}\colon \overline{K} \hookrightarrow \C$ such that $\bar{\sigma}_{| K} = \sigma$. 
 
\begin{definition}
In the following statements, put  
$D := \End_{R_{\overline{K}}^{\rm{p}}}\, ((\overline{\mathcal{V}}, \overline{\psi}))$.
\label{The ring D for polarized realizations}
\end{definition}

\begin{proposition}
Let $K$ be a number field. Let 
$$
\mathcal{V} \, := \, (\, V_{_{\rm{DR}}}, \, (V_l)_{l}, \, (V_{\sigma})_{\sigma}, \, I_{\infty, \sigma}, \,  I_{l, \bar{\sigma}}) \in {\rm{obj}} (R_{K}^{\rm{p}})
$$
be a polarized, pure realization with components satisfying conditions 
\textbf{(D1)},
\textbf{(D2)},
\textbf{(DR1)},
\textbf{(DR2)},
\textbf{(R1)}--\textbf{(R4)}
of \S \ref{Mumford--Tate groups of polarized Hodge structures}, \S \ref{de Rham structures associated with Hodge structures}, and \S \ref{families of l-adic representations associated with Hodge structures} respectively with $D$ as in Definition \ref{The ring D for polarized realizations}. 
Then Conjectures \ref{Tate conjecture for families of l-adic representations}(a) and 
\ref{Tate conjecture for families of l-adic representations tilde}(a) hold  
for the family $(V_l)_{l}$.   
\label{the family V l l of G K satisfies Conjectures 1 (a) and 2 (a)}
\end{proposition}
\begin{proof} This follows by Lemma \ref{G l, K alg V_l naturally embeds into G K alg Q l} and properties
of morphisms \eqref{G K alg W surjection onto G K alg V}, \eqref{G K alg W surjection onto G m}, 
\eqref{G K alg W injection into G K alg V times G m}.
\end{proof}

\begin{corollary}
Let $K$ be a number field. Let 
$$
\mathcal{V} \, := \, (\, V_{_{\rm{DR}}}, \, (V_l)_{l}, \, (V_{\sigma})_{\sigma}, \, I_{\infty, \sigma}, \,  I_{l, \bar{\sigma}}) \in {\rm{obj}} (R_{K}^{\rm{p}})
$$
be a polarized, pure realization with components satisfying conditions 
\textbf{(D1)},
\textbf{(D2)},
\textbf{(DR1)},
\textbf{(DR2)},
\textbf{(R1)}--\textbf{(R4)}
of \S \ref{Mumford--Tate groups of polarized Hodge structures}, \S \ref{de Rham structures associated with Hodge structures}, and \S \ref{families of l-adic representations associated with Hodge structures} respectively with $D$ as in Definition \ref{The ring D for polarized realizations}. 
Then Conjectures \ref{Tate conjecture for families of l-adic representations},  
\ref{Tate conjecture for families of l-adic representations tilde}, 
\ref{general algebraic Sato Tate conj.}, and \ref{general algebraic Sato Tate conj. Serre's approach} are all equivalent for the family $(V_l)_{l}$.   
\label{the family V l l of G K gives equivalence of Conjectures 1 and 2 and AST and tilde AST}
\end{corollary}
\begin{proof} 
It follows from Proposition \ref{the family V l l of G K satisfies Conjectures 1 (a) and 2 (a)} and Corollary \ref{equivalence of 4 conjectures}. 
\end{proof}

\section{Remarks on equidistribution}
\label{remarks on equidistribution}

This section is logically independent of the preceding ones. In particular, we do not need any of the assumptions
\textbf{(D1)},
\textbf{(D2)},
\textbf{(DR1)},
\textbf{(DR2)},
\textbf{(R1)}--\textbf{(R4)}.
We start with some general facts about pushforward measures in measure theory.  
\medskip

\begin{remark}
Let $\phi\colon (X, \Sigma) \rightarrow (X^{\prime}, \Sigma^{\prime})$ be a measurable map with $\sigma$-algebras $\Sigma$ and $\Sigma^{\prime}$ respectively. If $\mu$ is a measure on $X$, then
the pushforward measure $\phi_{\ast} \mu$ on $X^{\prime}$ is defined as follows:
$$
\phi_{\ast} \mu (B) := \mu (\phi^{-1} (B)) \,\, \text{for all} \,\,\, B \in \Sigma^{\prime}.
$$
Hence for every measurable function $f\colon X^{\prime} \rightarrow \C$ we obtain:
\begin{equation}
\int_{B} \, f \,\, d \, \phi_{\ast}\mu =  \int_{\phi^{-1} (B)} \, f \circ \phi \,\,\,\, d \mu.
\label{integrating against pushforward measure}
\end{equation}  
If $\mu$ is a probabilistic measure, then the pushforward measure $\phi_{\ast} \mu$ is also probabilistic.
If $\phi\colon X \rightarrow X^{\prime}$ is a continuous map between topological spaces $X$ and 
$X^{\prime}$ and the $\sigma$-algebras $\Sigma$, $\Sigma^{\prime}$ are Borel, then $\phi$ is a measurable map.
Important special cases of this are discussed in the following Remark \ref{Inducing sigma algebra from X to X'}.
\label{pushforward measure}
\end{remark}

\begin{remark}
Let $\pi\colon X \rightarrow X^{\prime}$ be a surjective map of sets. If $X$ is a topological space (resp. $(X, \Sigma)$
is a measurable space with $\sigma$-algebra $\Sigma$), we can naturally induce on $X^{\prime}$ 
a quotient topology so that $\pi$ becomes continuous (resp. we can naturally induce a $\sigma$-algebra 
$\Sigma^{\prime}$ on $X^{\prime}$ so that $\pi\colon (X, \Sigma) \rightarrow (X^{\prime}, \Sigma^{\prime})$ becomes a measurable map) because $\pi^{-1}$ commutes with complement and arbitrary unions and intersections of sets.
In this case, $f\colon X^{\prime} \rightarrow \C$ is continuous (resp. measurable) iff $f \circ \pi$ is continuous 
(resp. measurable).  

In particular, if $G$ is a compact group and $U \triangleleft G$ is a normal open subgroup, thne the quotient $G / U$ is a finite group and the induced topology on $G / U$ via the quotient map 
$\pi\colon G \rightarrow G / U$ is the discrete topology. If $\Sigma_{G}$ is the 
Borel $\sigma$-algebra on $G$ and $\mu$ is the Haar measure on $G$, then the induced Borel $\sigma$-algebra 
$\Sigma_{G / U}$ on $G / U$ is $2^{G / U}$ and the pushforward measure 
$\pi_{\ast} \mu$ on $G / U$ is the atomic (or discrete) measure. 
\label{Inducing sigma algebra from X to X'}
\end{remark}

\begin{remark}
For any epimorphism $\pi\colon G \rightarrow G^{\prime}$ of compact Lie groups,
we have 
$\pi_{\ast} \mu = \mu^{\prime}$ where $\mu$ (resp. $\mu^{\prime}$) is the probabilistic Haar measure on $G$
(resp. on $G^{'}$) \cite[Theorem C, Chap. XII]{Hal}. 
\label{pushforward measure via an epimorphism of compact Lie groups}
\end{remark}

\begin{remark}
\label{Integrating on G and on X(G)}
Let $G$ be a compact group with Haar measure $\mu$ and let $X (G)$ be the set of conjugacy classes of $G$.
Consider the natural map $\pi\colon G \rightarrow X(G)$. If $f\colon G \rightarrow \C$ is a class function, we will also
denote by $f$ the naturally induced function $X(G) \rightarrow \C$. 
By \eqref{integrating against pushforward measure}, for a measurable class function $f\colon G \rightarrow \C$ we have:
$$
\int_{X(G)} \, f \, d \, \pi_{\ast} \mu = \int_{G} \, f \, d\mu.
$$
\end{remark}

Let $G$ be a compact group with probabilistic (i.e., normalized) Haar measure
$\mu$. Let $G_0 \triangleleft G$ be an open subgroup. Let $\mu_{0}$
denote the probabilistic Haar measure of $G_0$, so that 
$\mu_{0} = [G:\, G_0] \, \mu_{| G_0}$. The measure induced by $\mu$ on $X(G)$ (resp. $\mu_{0}$ on $X(G_{0})$) will also be denoted $\mu$ (resp. $\mu_{0}$) cf. Remark \ref{Integrating on G and on X(G)}. The group $G$ can be decomposed into a finite number of cosets with 
respect to $G_{0}$:
\begin{equation}
G = \bigsqcup_{i = 0}^{t} \, g_i \, G_{0}.
\label{coset decomposition}
\end{equation} 

\begin{remark} \label{Peter-Weyl for equidistribution}
Let $(x_n)$ be a sequence of elements of $X(G)$.
By definition, the sequence $(x_n)$ is equidistributed in $X(G)$ if and only if for each continuous class function $f$ of $G$,
\[
\lim_{n \to \infty} \frac{1}{n} \sum_{i=1}^n f(x_i) = \int_G f\,d\mu.
\]
By the Peter-Weyl theorem (see \cite[Appendix to Chapter I]{Se1}), we may further say that the sequence $(x_n)$ is equidistributed in $X(G)$ if and only if for each nontrivial irreducible character $\chi$ of $G$,
\begin{equation} \label{Peter-Weyl reduction equation for equidistribution}
\lim_{n \to \infty} \frac{1}{n} \sum_{i=1}^n \chi(x_i) = \int_G \chi\,d\mu.
\end{equation}
Let $\rho\colon G \rightarrow GL(V)$ be the irreducible representation with character $\chi$, where $V$ is a complex vector space of dimension $d$.
Because $G$ is compact, for each $g \in G$, $\chi (g)$ is a sum of $d$ complex numbers of absolute value 1, and hence $|\chi (g)| \leq d$.
 \end{remark}

\medskip

Consider the inclusion homomorphism $j\colon G_0 \to G$. The map induced by $j$ on the sets of 
conjugacy classes of $G_{0}$ and $G$ will be denoted as follows:
\begin{equation}
X(j)\colon X(G_0) \to X(G).
\label{The map X(j)}
\end{equation}

\begin{remark} 
The map $X(j)$ is not necessarily an inclusion (that is, there may be \emph{fusion} of conjugacy classes from $G_0$ 
to $G$), but is finite-to-one. In any case, given a sequence of elements of $X(G)$ belonging to the image of $X(G_0)$, we may still ask whether this sequence is equidistributed for the image (pushforward) of $\mu_0$.
\end{remark}

The following proposition is of independent interest, but its proof presents a strategy 
we will apply later in this section.

\begin{proposition}\label{equidistr. of xv implies equidistr. of xnj} Let $(x_i)$ be a sequence of elements of $X(G)$ that is equidistributed with respect to 
$\mu$. Put $k_n := \sharp \{x_i \in \mathrm{image} (X(j))\colon 1 \leq i \leq n\}$. Assume that:
\begin{equation}
\lim_{n \rightarrow \infty}\, \frac{k_n}{n} \, = \, \frac{1}{[G:\, G_{0}]}.
\label{limit condition}
\end{equation} 
Then the subsequence $(x_{n_k})$ of elements of $(x_i)$ that are contained in the image of $X(j)$ is equidistributed 
with respect to the pushforward of $\mu_{0}$.
\end{proposition}
\begin{proof} 
We equate equidistribution in $X(G)$ with a statement about continuous class functions on $G$
as per Remark~\ref{Peter-Weyl for equidistribution};
similarly, we equate equidistribution in the image of $X(j)$ with a statement about continuous class functions on $G_0$ which are constant on fibers of the map $X(j)$.
Let $f\colon G \rightarrow \mathbb{C}$ be any continuous class function. Observe that 
$f_{0} := f \circ j\colon G_{0} \rightarrow \mathbb{C}$ is a continuous class function on $G_{0}$.
Let $f_{1}\colon G \rightarrow \mathbb{C}$ be the extension of $f_{0}$ by 
$0$ beyond $G_{0}$. The continuous function $f_1$ is a class function on $G$ because $G_{0} \triangleleft G$.
By assumption,
\begin{equation}
\lim_{n \rightarrow \infty}\, \frac{1}{n}\, \sum_{i = 1}^{n} f_{1} (x_i)  = \, \int_{G} \, f_{1} d\mu.
\label{equidistribution for G}
\end{equation} 
By the construction of $f_{1}$, the equation \eqref{equidistribution for G} assumes the following form:
\begin{equation}
\lim_{n \rightarrow \infty}\, \frac{1}{n}\, \sum_{k = 1}^{k_n} f_{0} (x_{n_k})  = \, \int_{G_0} \, f_0 \, d\mu_{| G_0}.
\label{equidistribution for G0 1}
\end{equation}
Hence:
\begin{equation}
\lim_{n \rightarrow \infty}\, \frac{k_n}{n}\, \frac{1}{k_n}\, \sum_{k = 1}^{k_n} f(x_{n_k})  = 
\, \int_{G_0} \, f_0 \frac{1}{[G:\, G_0]} \, d\mu_{0}.
\label{equidistribution for G0 2}
\end{equation}
Since
$$
\frac{1}{k_n}\, \sum_{k = 1}^{k_n} f(x_{n_k}) = \frac{\frac{1}{n}\, \sum_{i = 1}^{n} f_{1} (x_i)}{\frac{k_n}{n}}
$$
and the limits of the numerator and denominator on the right exist by assumptions
(see \eqref{limit condition} and \eqref{equidistribution for G}), the limit on the left also exists and 
by \eqref{integrating against pushforward measure}, \eqref{limit condition}, and \eqref{equidistribution for G0 2} we get:
\begin{equation}
\lim_{n \rightarrow \infty}\, \frac{1}{k_n}\, \sum_{k = 1}^{k_n} f(x_{n_k})  = 
\, \int_{G_0} \, f \circ j \, d\mu_{0} = \int_{G} \, f  \, d \, j_{\ast} \mu_{0}
\label{equidistribution for G0 3}
\end{equation}
as desired.
\end{proof}

For the rest of the section, we study the problem of inverting Proposition~\ref{equidistr. of xv implies equidistr. of xnj}, which is to say deducing an equidistribution statement in $X(G)$ from equidistribution statements in $X(G_0)$
and $X(G/G_0)$.
Our main tool for this will be induced characters.

\begin{lemma}
\label{relation between integral of chi and chi0}
 Let $\chi_{0}$ be a character of $G_{0}$ and set $\chi := \Ind_{G_{0}}^{G} \, (\chi_{0})$. Then
\begin{equation}
\int_{G} \, \chi \, d \mu = \int_{G_{0}}\, {\chi_{0}}\, d \mu_{0}.
\end{equation}
\end{lemma}
\begin{proof}
By \cite[Theorem 12]{Se4} and \cite[p. 34]{Se4}:
\begin{equation}
\chi (g) =  \sum_{g_{i}^{-1} g g_{i} \in G_{0}} \, \chi_{0} (g_{i}^{-1} g g_{i}).
\label{integral of induced character}
\end{equation}
Observe that $g_{i}^{-1} g g_{i} \in G_{0}$ if and only if $g \in g_{i} G_{0} g_{i}^{-1} = G_{0}$ 
because $G_0 \triangleleft G$. In the following computations, we treat $\chi_{0}$ as a function on $G$ by extending it by $0$ outside $G_{0}$. With this convention in mind, by applying \eqref{integral of induced character} we obtain:
\begin{equation}
\int_{G} \, {\chi (g)} \, d \mu =  \int_{G} \, \sum_{g_{i}^{-1} g g_{i} \in G_{0}} \, \chi_{0} (g_{i}^{-1} g g_{i}) \, 
d \mu = {\sum_{i=0}^{t}  \int_{G} \, \chi_{0} (g_{i}^{-1} g g_{i})} \, d \mu.
\label{proof of integral of induced character}
\end{equation}
Making the substitutions $g \mapsto g_{i} g g_{i}^{-1}$ for each $i$ under the integrals on the right of 
\eqref{proof of integral of induced character}, then applying the right and left translation
invariance of $\mu$, we obtain: 
$$
\int_{G} \, {\chi (g)} \, d \mu =  
{\sum_{i=0}^{t}  \int_{G_0} \, \chi_{0} (g)} \, d \mu_{| G_0}  = \int_{G_{0}} \, \chi_{0} (g) \, d \mu_{0}
$$
because the function $g \mapsto \chi_{0} (g_{i} g g_{i}^{-1})$ has support in $G_{0}$ and
$\mu_{0} = [G:\, G_0] \, \mu_{| G_0}$.
\end{proof}

\begin{lemma} Let $\chi_{0}$ be a character of $G_{0}$ and $\chi := \Ind_{G_{0}}^{G} \, (\chi_{0})$.
Then
\begin{equation}
\int_{G} \, \chi \, d \, j_{\ast} \mu_{0} = [G:\, G_0] \, \int_{G_{0}}\, {\chi_{0}}\, d \mu_{0},
\end{equation}
where $j\colon G_{0} \rightarrow G$ is the inclusion homomorphism. 
\label{relation between integral of chi against pushforward measure and integral of chi0}
\end{lemma}
\begin{proof} By \eqref{integrating against pushforward measure} we obtain:
$$
\int_{G} \, \chi \, d \, j_{\ast} \mu_{0} \, = \, \int_{G_{0}} \, \chi \circ j \, d \, \mu_{0} \, = \, 
{\sum_{i=0}^{t}  \int_{G_{0}} \, \chi_{0} (g_{i}^{-1} g g_{i})} \, d \mu_{0}
$$ 
$$
= \, [G:\, G_0] \,{\sum_{i=0}^{t}  \int_{G_{0}} \, \chi_{0} (g_{i}^{-1} g g_{i})} \, d\, \mu_{| G_0}.
$$
Again treating $\chi_{0}$ as a function on $G$ by extending it by $0$ outside $G_{0}$ and  
then applying the right and left translation invariance of $\mu$ we get: 
(cf. the proof of Lemma \ref{relation between integral of chi and chi0})
$$
\int_{G} \, \chi \, d \, j_{\ast} \mu_{0} \, = \, 
[G:\, G_0] \,{\sum_{i=0}^{t}  \int_{G} \, \chi_{0} (g_{i}^{-1} g g_{i})} \, d \, \mu \, = \, 
[G:\, G_0] \,{\sum_{i=0}^{t}  \int_{G} \, \chi_{0} (g)} \, d \, \mu
$$ 
$$
= \, [G:\, G_0]^2 \, \int_{G_{0}} \, \chi_{0} (g) \, d \, \mu_{| G_0} \, = \, 
[G:\, G_0] \, \int_{G_{0}} \, \chi_{0} (g) \, d \, \mu_{0}.
$$
\end{proof}

\begin{remark} \label{equidistribution via induced characters}
We will use Lemmas~\ref{relation between integral of chi and chi0} and 
\ref{relation between integral of chi against pushforward measure and integral of chi0} in the following fashion. 
In \eqref{Peter-Weyl reduction equation for equidistribution}, suppose that $\chi$ is the character associated to the representation $\rho$; assume also that \eqref{limit condition} holds.
If $\rho$ is induced from a representation $\rho_0$ of $G_0$, we may  equate the desired equality with the corresponding equality for the character of $\rho_0$, using \eqref{limit condition} to compare the left-hand sides and Lemmas~\ref{relation between integral of chi and chi0} and 
\ref{relation between integral of chi against pushforward measure and integral of chi0} to compare the right-hand sides.
\end{remark}

\begin{lemma}
Let $(x_n)$ be a sequence of elements of $X(G)$. Suppose that:
\begin{enumerate}
\item[(1)] 
The image of the sequence $(x_n)$ in $X(G/[G,G])$ is equidistributed for the probabilistic Haar measure on $G/[G,G]$.
\item[(2)]
Every irreducible character $\chi$ of $G$ of dimension $>1$ 
is induced from an irreducible character $\chi_{0}$ of $G_{0}$.
\item[(3)]  The equation \eqref{limit condition} holds.
\item[(4)] The subsequence of elements of $(x_n)$ that are contained in the image of $X(j)$ is equidistributed with respect to the pushforward of $\mu_0$.
\end{enumerate}
Then $(x_n)$ is equidistributed in $X(G)$ with respect to $\mu$.
\label{reduction for equidistribution easy case}
\end{lemma}
\begin{proof}
For $\chi$ of dimension 1, $\chi$ is abelian and thus factors through the epimorphism 
$G \rightarrow G/[G,G]$. Hence this case is covered by condition (1) and
Remark \ref{pushforward measure via an epimorphism of compact Lie groups}.

For $\chi$ of dimension $>1$, let $(x_{n_k})$ be the subsequence of elements of $(x_n)$ that are contained in the image of
$X(j)$. By assumption (4) we have:
\begin{equation}
\lim_{n \rightarrow \infty}\, \frac{1}{k_n} \sum_{k = 1}^{k_n} \chi (x_{n_k}) =
\int_{G} \, \chi \, d \, j_{\ast} \mu_{0}.
\label{equidistribution of sequence coming from G0 wr to pushforward measure}
\end{equation}
By (2), (3), (4) and Lemmas \ref{relation between integral of chi and chi0} and 
\ref{relation between integral of chi against pushforward measure and integral of chi0}, the equality \eqref{equidistribution of sequence coming from G0 wr to pushforward measure} is equivalent to:
\begin{equation}
\lim_{n \rightarrow \infty}\, \frac{1}{n} \sum_{i = 1}^{n} \chi (x_{i}) =
\int_{G} \, \chi d \, \mu.
\label{equidistribution of the sequence coming from G wr to the measure mu}
\end{equation}
 because $\chi (x_i) = 0$ for $x_i$ not in the subsequence $(x_{n_k})$. 
\end{proof}

\begin{proposition} Let $(x_n)$ be a sequence of elements of $X(G)$. Let $(x_{n_k})$ be the subsequence of elements of 
$(x_n)$ in the image of $X(j)$.
\begin{itemize}
\item[(a)] Each fiber of $x_{n_k}$ of the map $X(j)$ equals 
$\{g_i y_{n_k} g_{i}^{-1}\colon 0 \leq i \leq t\}$, where $y_{n_k} \in X(G_0)$ is an element of this fiber.
\item[(b)] If assumptions (1)--(3) of Lemma \ref{reduction for equidistribution easy case} hold and the sequence 
$(g_i y_{n_k} g_{i}^{-1})_{i, k}$ is equidistributed in $X(G_0)$ with respect to $\mu_0$,
then $(x_n)$ is equidistributed in $X(G)$ with respect to $\mu$.
\end{itemize}
\label{proposition about the reduction of equidistribution to G0}
\end{proposition}
\begin{proof} The claim (a) is obvious. To prove (b), 
assumption (1) of Lemma \ref{reduction for equidistribution easy case} again handles the case
where $\chi$ is a character of dimension 1, so we may assume that $\chi$ is of dimension $>1$.
The number of elements in the sequence 
$(g_i y_{n_k} g_{i}^{-1})_{i, j}$, for $1 \leq k \leq k_n$ and $0 \leq i \leq t$, asymptotically equals 
$[G:\, G_0] \, k_n$ by assumption (3) of Lemma \ref{reduction for equidistribution easy case}. 
Hence by \eqref{integral of induced character} 
and equidistribution of $(g_i y_{n_k} g_{i}^{-1})_{i, j}$ we have: 
\begin{equation}
\lim_{n \rightarrow \infty}\, \frac{1}{[G:\, G_0] \, k_n} \,\, \sum_{k = 1}^{k_n} \sum_{i = 0}^{t} \chi_{0} (g_i y_{n_k} g_{i}^{-1}) =
\int_{G_0} \, \chi_{0} \, d \, \mu_{0}.
\label{equidistribution of sequence gi xnj gi -1}
\end{equation}  
Multiplying \eqref{equidistribution of sequence gi xnj gi -1} by $[G:\, G_0]$ and applying \eqref{integral of induced character} and Lemma \ref{relation between integral of chi against pushforward measure and integral of chi0},
we observe that \eqref{equidistribution of sequence gi xnj gi -1} is equivalent to  
\eqref{equidistribution of sequence coming from G0 wr to pushforward measure}. Hence the Proposition follows by Lemma \ref{reduction for equidistribution easy case}.
\end{proof}

\begin{remark}
A basic example to which
Lemma~\ref{reduction for equidistribution easy case}
applies is where $G_0 = \SO(2)$ and $G$ is the normalizer of $\SO(2)$ in $\SU(2)$.
\end{remark}

When Lemma~\ref{reduction for equidistribution easy case} does not apply, we can still make a nontrivial reduction, loosely inspired by \cite{Jo}.

Because $G$ and $G_{0}$ are compact, $G_{0}$ is open in $G$, and $G/G_{0}$ is finite, then each subgroup $H$ such that
$G_{0} \subset H \subset G$ is automatically compact and open in $G$. Let  $\mu_{H}$ be the probabilistic
Haar measure of $H$, which satisfies $\mu_{H} = [G:\, H] \, \mu_{| H}$. We can generalize Lemma 
\ref{relation between integral of chi and chi0} as follows.

\begin{lemma}
\label{relation between integral of chi prime and chi}
 Let $\chi$ be a character of $H$. Then
\begin{equation}
\int_{G}\ {{\rm{Ind}}_{H}^{G}  (\chi)} d \mu = \int_{H}\, {\chi}\, d \mu_{H}.
\end{equation}
\end{lemma}
\begin{proof}
By \cite[Theorem 12]{Se4} and \cite[p. 34]{Se4}:
\begin{equation}
\label{integral of induced character from H}
{\rm{Ind}}_{H}^{G} (\chi) (g) =  \sum_{g_{i}^{-1} g g_{i} \in H} \, \chi (g_{i}^{-1} g g_{i}).
\end{equation}
Observe that $g_{i}^{-1} g g_{i} \in H$ if and only if $g \in g_{i} H g_{i}^{-1}$. 
In the following computations, we may treat the function 
$g \mapsto \chi (g_{i}^{-1} g g_{i})$ as a function on $G$ with support in $g_{i} H g_{i}^{-1}$.
Hence by applying \eqref{integral of induced character from H}, making the substitution 
$g \mapsto g_{i} g g_{i}^{-1}$ for each $i$, and applying right and left shift invariance of $\mu$, we obtain: 
\begin{align*}
\int_{G} {\rm{Ind}}_{H}^{G} (\chi) (g) \, d\mu(g) &=  \int_{G} \, {\sum_{g_{i}^{-1} g g_{i} \in H} \chi (g_{i}^{-1} g g_{i})} 
\, d \mu (g) \\
&= {\sum_{i} \int_{G} \, \chi (g_{i}^{-1} g g_{i})} \, d \mu (g) \\
&= {\sum_{i} \int_{H} \, \chi (h)} \, d \mu_{| H} (h) \\
&= [G: \, H]  \int_{H} \, \chi (h) \frac{1}{[G: \, H]}\, d \mu_{H} (h) \\
&= \int_{H} \, \chi (h) \, d \mu_{H} (h). \qedhere
\end{align*}
\end{proof}

\begin{definition} Let $H$ be a subgroup of $G$ such that $G_{0} \subset H \subset G$.
Let $\lambda$ be a character of $H/G_{0}$. Let $\widetilde{\lambda}$ denote its lift to
the character of $H$, so that for each $h \in H$:
\begin{equation}
\label{lift of a character from H/G0 to H}
\widetilde{\lambda} (h) := \lambda (h G_{0}).
\end{equation}
\end{definition}

\begin{lemma} 
\label{lift of induction is the induction of the lift}
Let $H$ be a subgroup of $G$ such that $G_{0} \subset H \subset G$. Then
\begin{equation}
\widetilde{{{\rm{Ind}}}_{H /G_{0}}^{G/G_{0}}} \lambda = {{\rm{Ind}}}_{H}^{G} \widetilde{\lambda}.
\end{equation}
\end{lemma} 
\begin{proof} Observe that $G = \bigsqcup_{i=1}^{m} g_{i} H$ if and only if $G /G_{0} = \bigsqcup_{i=1}^{m} g_{i} G_{0}  H/G_{0}$. Hence:
\begin{align*}
\widetilde{{{\rm{Ind}}}_{H/G_{0}}^{G/G_{0}}} \, \lambda (g) &= {{\rm{Ind}}}_{H/G_{0}}^{G/G_{0}} \, \lambda (g G_{0})\\
&= \sum_{g_{i}^{-1} G_{0} \, g G_{0} \, g_{i} G_{0} \in H/G_{0}} \lambda (g_{i}^{-1} G_{0} \, g G_{0} \, 
g_{i} G_{0}) \\
&= \sum_{g_{i}^{-1} g g_{i} G_{0} \in H/G_{0}} \lambda (g_{i}^{-1} g g_{i} G_{0}) \\
&= \sum_{g_{i}^{-1} g g_{i} \in H} \widetilde{\lambda} (g_{i}^{-1} g g_{i}) \\
&= {{\rm{Ind}}}_{H}^{G} \, \widetilde{\lambda} (g). \qedhere
\end{align*}
\end{proof}

\begin{remark} Let $H \subset G$ and let $\chi$ (resp. $\lambda$) be a character of $G$ (resp. $H$). Then 
there is the following formula:
\begin{equation}
{{\rm{Ind}}}_{H}^{G} (\lambda) \cdot \chi  = {{\rm{Ind}}}_{H}^{G} (\lambda \cdot {{\rm{Re}}}_{H}^{G} (\chi)). 
\label{induction-restriction formula}
\end{equation}
\end{remark} 

\begin{lemma} \label{Artin Clifford lemma}
Every character of $G$ is a $\Q$-linear combination of characters of the form $\Ind^G_{G_1} \chi_1$ where $G_1$ is the preimage in $G$ of a cyclic subgroup of $G/G_0$
and $\chi_1$ is an irreducible character of $G_1$, not induced from any proper subgroup of 
$G_1$ containing $G_{0}$, whose restriction to $G_0$ is irreducible.  
\end{lemma}
\begin{proof}
Let $1_{G/G_{0}}$ denote the trivial character of $G/G_{0}$. Naturally $\widetilde{1_{G/G_{0}}} = 1_{G}$. By Artin's theorem on induced characters  \cite[Chapter IX, Theorem 17]{Se4}, there exist $a_{i} \in \Q$, cyclic subgroups $H_{i}/G_{0}$ of $G/G_{0}$, and characters $\lambda_i$ of $H_i/G_0$ such that
\begin{equation}
1_{G/G_{0}} = \sum_{i=1}^{s} \,  a_i \, {{\rm{Ind}}}_{H_{i}/G_{0}}^{G/G_{0}} \lambda_{i}.
\label{Brauer induction for 1 G G0}
\end{equation}
Lifting the equation \eqref{Brauer induction for 1 G G0} to $G$ and applying Lemma 
\ref{lift of induction is the induction of the lift}, we obtain:
\begin{equation}
1_{G} = \sum_{i=1}^{s} \,  a_i \, {{\rm{Ind}}}_{H_{i}}^{G} \widetilde{\lambda_{i}}.
\label{lifting Brauer induction for 1 G}
\end{equation}
By \eqref{induction-restriction formula}, for any character $\chi$ of $G$, we obtain:
\begin{equation}
\chi =  1_{G} \cdot \chi = \sum_{i=1}^{s} \,  a_i \, ({{\rm{Ind}}}_{H_{i}}^{G} \widetilde{\lambda_{i}} \cdot \chi ) =
 \sum_{i=1}^{s} \,  a_i \, {{\rm{Ind}}}_{H_{i}}^{G} (\widetilde{\lambda_{i}} \cdot {{\rm{Re}}}_{H_{i}}^{G} \chi ).
\label{application of induction-restriction formula for character chi presentation} 
\end{equation}

We can write $\widetilde{\lambda_{i}} \cdot {{\rm{Re}}}_{H_{i}}^{G} \chi = \sum_{j=1}^{t_i} \chi_{i,j}$where $\chi_{i,j}$ are irreducible characters of $H_{i}$. Hence by \eqref{application of induction-restriction formula for character chi presentation}
we obtain:
\begin{equation}
\chi = \sum_{i=1}^{s} \sum_{j =1}^{t_i} a_i \, {{\rm{Ind}}}_{H_{i}}^{G} \chi_{i,j}.
\label{chi as sum of inductions of chi i j}
\end{equation}
where $\chi_{i,j}$ are irreducible. We can also assume that $\chi_{i,j}$ is not induced from any proper subgroup of 
$H_i$ containing $G_{0}$ because induction is transitive.
This presents $\chi$ as a $\Q$-linear combination of characters of the form $\Ind^G_{G_1} \chi_1$ where $G_1$
is the preimage in $G$ of a cyclic subgroup of $G/G_0$ and $\chi_{1}$ is irreducible, not induced from any proper subgroup of 
$G_1$ containing $G_{0}$.

To complete the argument, we may assume that $G/G_0$ is itself cyclic and that $\chi$ is an irreducible character not induced from any proper subgroup of $G$ containing $G_0$.
Let $\rho$ be the associated representation.
By \cite[Prop. 24]{Se4}, $\rho |_{G_0}$ is isotypic;
let $\tau$ be an irreducible subrepresentation of $\rho |_{G_0}$.
By the second theorem of Clifford \cite[Theorem 11.20]{C-R} cf. \cite[p. 493]{Ko},
the associated projective representation $\overline{\rho}$ of $\rho$
factors as a tensor product $\overline{\rho}_1 \otimes \overline{\rho}_2$ in which $\overline{\rho}_1 |_{G_0}$
is isomorphic to $\overline{\tau}$.
Moreover, $\ker(\overline{\rho}_1)$ is the centralizer of $G_0/(G_0 \cap \ker(\rho))$ in $G$
while $\ker(\overline{\rho}_2)$ contains $G_0$.  
Since $\rho$ is irreducible and $G/G_0$ is cyclic, $\overline{\rho}_2$ must be the trivial representation,
and so $\rho|_{G_0}$ is irreducible.
\end{proof}

\begin{lemma}
Let $(x_n)$ be a sequence of elements of $X(G)$. Suppose that for every subgroup $G_1$ of $G$ containing $G_0$ such that $G_1/G_0$ is cyclic,
the subsequence of elements of $(x_n)$ that are contained in the image of $X(G_1) \to X(G)$ is equidistributed with respect to the pushforward of the probabilistic Haar measure of $G_1$.
Then $(x_n)$ is equidistributed in $X(G)$ with respect to $\mu$.
\label{reduction for equidistribution}
\end{lemma}
\begin{proof}
This follows immediately from Remark~\ref{equidistribution via induced characters} and 
Lemma~\ref{Artin Clifford lemma}.
\end{proof}

In the remainder of this section, we explore the obstruction to deducing equidistribution on $G$ from equidistribution on $G_0$.

\begin{theorem}
Assume that $G_0$ is a connected Lie group, the simple factors of the derived group of $G_0$ are pairwise distinct, and none of these factors admits a nontrivial diagram automorphism.
Let $(x_n)$ be a sequence of elements of $X(G)$. Suppose that for every subgroup $G_1$ of $G$ containing $G_0$ such that $G_1/G_0$ is cyclic and acts trivially on $G_0$,
the subsequence of elements of $(x_n)$ that are contained in the image of $X(G_1) \to X(G)$ is equidistributed with respect to the pushforward of the probabilistic Haar measure of $G_1$.
Then $(x_n)$ is equidistributed in $X(G)$ with respect to $\mu$.
\label{reduction for equidistribution2}
\end{theorem}
\begin{proof}
We check \eqref{Peter-Weyl reduction equation for equidistribution} for a particular irreducible representation $\rho$ with character $\chi$. By Lemma~\ref{Artin Clifford lemma},
we may assume that $G/G_0$ is cyclic, $\rho$ is injective, and $\rho|_{G_0}$ is irreducible.
Let $Z$ be the center of $G_0$ and let $H$ be the derived group of $G_0$;
then $H$ is semisimple and $HZ = G_0$ (see \cite[Theorem~5.22]{Sep}).
Consequently, $\rho|_H$ is irreducible and $H$ is characteristic in $G$; in particular,
we have a restriction homomorphism $\Aut(G_0) \to \Aut(H)$ which induces a homomorphism 
$\Out(G_0) \to \Out(H)$.

Using highest weight vectors, we obtain a canonical bijection between the isomorphism classes of irreducible representations of $H$ and the quotient of the weight lattice $L$ of $H$ by the action of the Weyl group $W$. The outer automorphism group $\Out(H)$ acts faithfully on $L/W$.
The conjugation action $G \to \Aut(H)$ induces a homomorphism $G/G_0 \to \Out(H)$ and hence an action of $G/G_0$ on $L/W$; this action must fix
the vector corresponding to $\rho|_H$. Moreover, the images of $G/G_0$ in $\Out(G_0)$ and $\Out(H)$ are isomorphic because $Z$ is central in $G$.

Now our condition that $H$ has pairwise distinct simple factors, none of which admits a nontrivial diagram automorphism, implies that $\Out(H)$ is trivial. By the previous paragraph, this means that $G/G_0$ is generated by an element that centralizes $G_0$. Hence this case is covered by our input hypothesis.
\end{proof}

\begin{remark}
In Theorem~\ref{reduction for equidistribution2}, the restrictions on the derived group cannot be lifted.
We give two illustrative examples of representations that cannot be handled when the restrictions are lifted.
\begin{itemize}
\item
Take $H := \SU(3)$, $G_0 = ZH$ for $Z$ either a finite cyclic group or $U(1)$, $G := G_0 \rtimes \Z/2\Z$ for the nontrivial action on $H$ (fixing $Z$), and $\rho$ to be the adjoint representation. The induced homomorphism $G/G_0 \to \Out(G_0)$ is nontrivial,
and the highest weight vector of $\rho$ is fixed by the action of $G/G_0$.

\item
Take $H = G_0 := \SU(2) \times \SU(2)$, $G := \Aut(G_0) \simeq G_0 \rtimes \Z/2\Z$, and $\rho$ to be the external product of the standard representations of the two factors. This type of example can be regarded as a \emph{tensor induction} as in \cite{Ko}; however, this does not lead to a reduction for the equidistribution problem.
\end{itemize}
\end{remark}

\section{Application to the Sato--Tate conjecture} 
\label{application to the Sato--Tate conjecture}

We next specialize the previous discussion to the setting of Sato--Tate groups. First, we make some general remarks about the relationship between equidistribution and $L$-functions, again following \cite[Appendix to Chapter I]{Se1}.

\begin{definition}
Let $\rho\colon G \to \GL(V)$ be a continuous representation of a compact topological group on a finite-dimensional $\C$-vector space $V$.
Let $L/K$ be a finite extension of number fields and let $\Pi := \Pi_{L}$ denote the set of all prime ideals in the ring of integers $\mathcal{O}_{L}$. Let $(x_v)$ be a sequence of elements of $X(G)$ indexed 
by the prime ideals $v$ of $\mathcal{O}_{L}$. Define the \emph{$L$-function}:
$$
L(s, \, \rho) := \prod_{v} \det (1 - \rho (x_v) (N v)^{-s})^{-1}.
$$
This product is absolutely convergent for $\mathrm{Re}(s) > 1$, and so defines a nowhere-vanishing holomorphic function in this region.
For any subset $T$ of $\Pi$, we also set:
$$
L_{T} (s, \, \rho) := \prod_{v \in T} \det (1 - \rho (x_v) (N v)^{-s})^{-1}.
$$
\end{definition}

\begin{lemma}
Let $S$ denote any set of ideals in $\Pi$ whose symmetric difference with the set of prime ideals of $L$ of degree $1$ over $K$ is finite.
Then $L_{\Pi - S} (s, \, \rho)$ defines a nowhere-vanishing analytic function on the region 
${\rm{Re}} (s) > \frac{1}{2}$. 
\label{equidistribution of primes that split completely}\end{lemma} 
\begin{proof} 
\medskip

For every $v \in \Pi$ we have:
$$
\chi (x_v) = {\rm{Tr}} \, \rho (x_v) = \sum_{j=1}^{d} \, \xi_{v,j},
$$
where the $\xi_{v,j}$ are eigenvalues of $\rho (x_v)$ and are complex numbers of absolute value 1. 
For almost all $v \in \Pi - S$, we have $N v = p_{v}^{r_v}$ with $r_v \geq 2$ 
where $v \, | \, p_v$ and $p_v$ is a prime number.
Hence:
$$
L_{\Pi - S} (s, \, \rho) = \prod_{v \in \Pi - S} \prod_{j=1}^d (1 - \xi_{v,j} p_{v}^{ - s r_v})^{-1}.
$$
This infinite product is absolutely convergent for ${\rm{Re}} (s) > \frac{1}{2}$ 
because the series
$$
\sum_{v \in \Pi - S} \sum_{j=1}^d p_{v}^{ - s r_v} \, = \, \sum_{v \in \Pi - S} d \, p_{v}^{ - s r_v}
$$ 
is absolutely convergent for ${\rm{Re}} (s) > \frac{1}{2}$. 
\end{proof}

In the notation of Lemma \ref{equidistribution of primes that split completely} we obtain:

\begin{equation}
L_{S} (s, \, \rho) = \frac{L (s, \, \rho)}{L_{\Pi - S} (s, \, \rho)}.
\label{LS(s, rho)}
\end{equation}

\begin{lemma}
For $S$ as in Lemma \ref{equidistribution of primes that split completely},
the sequence $(x_v)$ is equidistributed with respect to the measure $\mu$ on $X(G)$ if and only if
the sequence $(x_v)_S$ is equidistributed with respect to the measure $\mu$
on $X(G)$.
\label{Equidistribution with respect to all primes and primes in S}
\end{lemma} 
\begin{proof} 
By \eqref{Peter-Weyl for equidistribution}, the sequence  $(x_v)$ is equidistributed with respect to the measure $\mu$ on $X(G)$ if and only if for each nontrivial irreducible representation $\rho$,
$$
\lim_{n \to \infty} \frac{1}{\#\{v \in \Pi\colon Nv \leq n\}} \sum_{v \in \Pi\colon N v \leq n} \chi_\rho(x_v) = \int_G \chi_\rho d\mu.
$$
Similarly, the sequence $(x_v)_S$ is equidistributed with respect to the measure $\mu$ on $X(G)$ if and only if for each $\rho$,
$$
\lim_{n \to \infty} \frac{1}{\#\{v \in S\colon Nv \leq n\}} \sum_{v \in S\colon N v \leq n} \chi_\rho(x_v) = \int_G \chi_\rho d\mu.
$$
To equate these, note that by Lemma~\ref{equidistribution of primes that split completely} applied to the trivial representation (evaluating at $s=1$), the sum
\[
\sum_{v \in \Pi - S} d p_v^{-r_v}
\]
is absolutely convergent. This implies firstly that
$$
\lim_{n \to \infty} \frac{\#\{v \in \Pi - S\colon Nv \leq n\}}{\#\{v \in \Pi\colon Nv \leq n\}} = 0
$$
and secondly that for any $\rho$,
$$
\lim_{n \to \infty} \frac{1}{\#\{v \in \Pi\colon Nv \leq n\}} \sum_{v \in \Pi - S\colon N_v \leq n} \chi_\rho(x_v) = 0.
$$
This yields the desired equivalence.
 \end{proof} 

\begin{lemma} Let $\mathcal{G}$ be an algebraic group over a field of characteristic $0$. Then:
\begin{itemize}
\item[(a)] Every unipotent subgroup $U \subset \mathcal{G}$ is connected.
\item[(b)] For every unipotent element $u \in \mathcal{G}$ there is a smallest unipotent subgroup of $\mathcal{G}$ containing $u$.
\item[(c)] The set of unipotent elements $U(\mathcal{G}) \subset \mathcal{G}$ is Zariski closed in $\mathcal{G}$.
\item[(d)] $U(\mathcal{G}) \subset \mathcal{G}^{0}$.
\end{itemize} 
\label{unipotent elements in an algebraic group}
\end{lemma}

\begin{proof} (a), (b), and (c) are well known. 
Containment (d) follows from (a) and (b).
\end{proof}

For the rest of this section we assume Conjectures \ref{general algebraic Sato Tate conj.} and \ref{general algebraic Sato Tate conj. Serre's approach}. Therefore we can work under the assumptions 
of Theorem \ref{STK iff STK0} and Theorem \ref{tilde STK iff tilde STK0}; consequently, we return to all of the assumptions of sections \ref{Mumford--Tate groups of polarized Hodge structures}--\ref{AST and Tate for families}. Let 
\begin{gather*}
\mathcal{G} := \AST_{K} (V, \psi)(\C), \quad \mathcal{G}_{0} := \AST_{K_{0}} (V, \psi)(\C) \quad \emph{or} \\
\mathcal{G} := \widetilde{\AST_{K}} (V, \psi)(\C), \quad \mathcal{G}_{0} := \widetilde{\AST_{K_{0}}} (V, \psi)(\C).
\end{gather*}
Let 
\begin{gather*}
G := \ST_{K} (V, \psi), \quad  G_{0} := \ST_{K_{0}} (V, \psi) \quad \emph{or} \\
G := \widetilde{\ST_{K}} (V, \psi), \quad G_0 := \widetilde{\ST_{K_{0}}} (V, \psi).
\end{gather*}
By \eqref{isomorphism between G(K0 / L) and STL / STK0} and
\eqref{isomorphism between G(K0 / L) and tilde STL / tilde STK0} and Propositions \ref{ConnCompIsom} and \ref{connected components iso for tilde}, applied for the base fields $K$ and $K_{0}$,
we obtain the following natural isomorphisms of groups: 
\begin{equation}
\Gal(K_0/K) \,\, \simeq \,\, \mathcal{G}/\mathcal{G}_{0} \,\, \simeq \,\, G/G_0.
\label{Gal(K0 K) simeq mathcal G mathcal G  0 simeq G G 0}
\end{equation}
Hence we obtain the following equality:
\begin{equation}
\mathcal{G}_{0} \cap G \,\, = \,\, G_0. 
\label{mathcal G 0  cap G  simeq G 0}
\end{equation}

Observe that (cf. Corollary 
\ref{connected components of AST K (V, psi), ST K (V, psi) and widetilde AST K (V, psi), widetilde ST K (V, psi)}):
\begin{equation}
\mathcal{G}^{\circ} \subset \mathcal{G}_{0} \subset \mathcal{G}.
\label{Containment of Alg. Sato--Tate groups} 
\end{equation}

By \eqref{decomposition of STK into cosets over Galois representatives} \emph{or} 
\eqref{decomposition of tilde STK into cosets over Galois representatives}, $G_{0}$ is open in $G$. 
We may thus set notation as in \S\ref{remarks on equidistribution};
in particular, let $\mu$ be the probabilistic Haar measure on $G$ and $\mu_{0} := [K_{0} :\, K] \, \mu_{| G_{0}}$ be the corresponding probabilistic Haar measure on $G_{0}$.  

Let $S_{\mathrm{fin}}$ be the set of primes $v \subset \mathcal{O}_K$ that are ramified in the representation $\rho_{l}$ 
of $G_K$ \eqref{the family of l-adic representations} together with the primes over $l$. For the prime $v \notin S_{\mathrm{fin}}$ 
(resp. $w  \subset \mathcal{O}_{K_{0}}$ not over a prime in $S_{\mathrm{fin}}$), choose a prime $\overline{v} \subset 
\mathcal{O}_{\overline{K}}$ dividing $v$ (resp. $\overline{w} \subset \mathcal{O}_{\overline{K}}$ dividing $w$).
Let $f_v \in \mathcal{G}$ (resp. $f_{w} \in \mathcal{G}_{0}$) denote the normalized Frobenius element for $\overline{v} 
\,|\, v$ (resp. the normalized Frobenius element for $\overline{w} \, | \,w$), see Remarks \ref{Sato--Tate set up} and \ref{tilde Sato--Tate set up}.
We have $f_v = s_v u_v$ where $s_v  \in \mathcal{G}$ (resp. $u_v \in \mathcal{G}$)
is the semisimple part of $f_v$ (resp. is the unipotent part of $f_v$). In the same way, we obtain a factorization 
$f_w = s_w u_w$ with $s_w  \in \mathcal{G}_{0}$ semisimple and $u_w \in \mathcal{G}_{0}$ unipotent.
By Lemma \ref{unipotent elements in an algebraic group}, the containment \eqref{Containment of Alg. Sato--Tate groups}, and the 
equality \eqref{mathcal G 0  cap G  simeq G 0}, we observe that:
\begin{equation}
\label{fv in mathcal G0 iff sv in mathcal G0 iff fv in G0}
f_v \, \in \, \mathcal{G}_{0} \quad \iff \quad s_v \, \in \, \mathcal{G}_{0} \quad \iff s_v  \, \in \, G_{0}.
\end{equation}

\begin{proposition} \label{translation into Chebotarev}
The image of the sequence $(f_v)$ in $X(G/G_0)$ is equidistributed for the probabilistic Haar measure on $G/G_0$.
\end{proposition}
\begin{proof}
Via the isomorphism $\Gal(K_0/K) \simeq G/G_0$ \eqref{Gal(K0 K) simeq mathcal G mathcal G  0 simeq G G 0}, 
the claim translates into the Chebotarev density theorem.
\end{proof}

\medskip

By equalities \eqref{decomposition of ASTKQl into cosets over Galois representatives} and 
\eqref{decomposition of STK into cosets over Galois representatives} of Theorem \ref{STK iff STK0} \emph{or} 
by equalities \eqref{decomposition of tilde ASTKQl into cosets over Galois representatives} and 
\eqref{decomposition of tilde STK into cosets over Galois representatives} of
Theorem \ref{tilde STK iff tilde STK0} we obtain:
\begin{align*}
\mathcal{G} \, &= \, \bigsqcup_{i =0}^{t} \, g_i \, \mathcal{G}_{0}, \\
G \, &= \, \bigsqcup_{i =0}^{t} \, g_i \, G_{0},
\end{align*}
with $g_0 = e$ for fixed coset representatives $g_{0}, \dots, g_{t}$. These representatives may be taken to be images of elements of the Galois group $G_{K}$ via the $l$-adic representation. This choice of ${g_{i}}$'s is possible by 
\textit{loc. cit.}
 
Let $\chi_{0}$ be a character of $G_{0}$ and let $\chi := {{\rm{Ind}}_{G_{0}}^{G} \, \chi_{0}}$. 
Recall the formula \eqref{integral of induced character}: 
$$
\chi (u) := \sum_{{{0 \leq i \leq t}, \, {g_{i}^{-1} u g_{i} \in G_{0}}}}  \,\, \chi_{0} (g_{i}^{-1} u g_{i}).
$$
Because $G_{0} \triangleleft G$ we obtain:
\begin{equation}
g_{i}^{-1} u g_{i} \in G_{0}  \quad \iff \quad u \in G_{0} \quad\quad \forall \,\, 0 \leq i \leq t.
\label{gi-1 u gi in G0 iff u in G0}
\end{equation}

It is clear due to \eqref{fv in mathcal G0 iff sv in mathcal G0 iff fv in G0}  that for each $0 \leq i \leq t$:
\begin{equation}
f_v \, \in  g_i \, \mathcal{G}_{0} \quad \iff s_v  \, \in g_i \, G_{0}.
\label{fv in gi mathcal G0 iff fv in g_i G0}
\end{equation}
Hence by \eqref{fv in gi mathcal G0 iff fv in g_i G0} the elements of the subsequence $({f_{v_{i}}})_{v_i}$  of the sequence 
$(f_v)$ such that $f_{v_{i}} \in g_{i} \, \mathcal{G}_{0}$ are in one-to-one correspondence with the elements of the subsequence
$({s_{v_{i}}})_{v_i}$  of the sequence $(s_v)$ such that $s_{v_{i}} \in g_{i} \, G_{0}$ for each $i = 0, \dots, t$.
By \eqref{integral of induced character}, \eqref{fv in gi mathcal G0 iff fv in g_i G0}, and 
\eqref{gi-1 u gi in G0 iff u in G0} we obtain: 
$$
\chi(s_{v_{i}}) = 0 \quad {\text{for}} \,\,\,\, \forall \, i > 0 \,\,\,\,  {\text{and}} \, \,\,\, \forall \, 
v_i \notin S_{\mathrm{fin}}.
$$ 
On the other hand
\begin{equation}
\chi(s_{v_{0}}) = \sum_{j = 0}^{t}  \chi_{0} (g_{j}^{-1} s_{v_{0}} g_{j}) =  
\sum_{j = 0}^{t} \chi_{0} (s_{w_{0, j}}),
\end{equation}
where $s_{w_{0, j}} = g_{j}^{-1} \, s_{v_{0}} \, g_{j} \in G_{0}$ for
each prime ideal $w_{0, j}$ over $v_0$, by the choice of ${g_{j}}$'s.
The primes $v_{0}$ are the only primes in the sequence $(v)$ outside $S_{\mathrm{fin}}$
that split completely in $\mathcal{O}_{K_{0}}$ because $s_{v_{0}} \in G_{0}$ due to 
\eqref{fv in mathcal G0 iff sv in mathcal G0 iff fv in G0}. 
\medskip

\noindent
Going towards the equidistribution questions, let $m \in \N$ and consider the sum:
\begin{equation}
\label{Equidistribution. Going from K0 to K}
\sum_{i=0}^{t} \, \sum_{N v_{i} \, \leq \, m} \chi(s_{v_{i}}) = \sum_{N v_{0} \, \leq \, m} \, \chi(s_{v_{0}}) =
\sum_{j=0}^t \, \sum_{N w_{0, j} \, \leq \, m} \, \chi_{0} (s_{w_{0, j}}).
\end{equation}

Put
$N_{K} \, (m) := \, \# \{v_{i}\colon N v_{i} \, \leq \, m , \, 0 \leq i \leq t \}$. Observe that $N_{K} \, (m)$ 
is asymptotically equal to
$$
[K_{0}, \, K] \, \# \{v_{0}\colon N v_{0} \, \leq \, m \}
$$ 
and this is equal to 
$$
\# \{w_{0, j}\colon N w_{0, j} \, \leq \, m , \, 0 \leq j \leq t \}.
$$
\medskip
By Lemma \ref{relation between integral of chi and chi0} and 
\eqref{Equidistribution. Going from K0 to K} the following equalities 
\eqref{equidistribution of fv} and \eqref{equidistribution of fw} are equivalent: 
\begin{align}
\lim_{m\to\infty} \frac{1}{N_{K} \, (m)} \, \sum_{i=0}^{t} \, \sum_{N v_{i} \, \leq \, m} \chi(s_{v_{i}}) \, &= \,
 \int_{G}\, {\chi}\, d \mu
\label{equidistribution of fv} \\
\lim_{m\to\infty} \frac{1}{N_{K} \, (m)} \, \sum_{j=0}^t \, \sum_{N w_{0, j} \, \leq \, m} \, 
\chi_{0} (s_{w_{0, j}}) \, &= \, \int_{G_{0}}\, {\chi_{0}}\, d \mu_{0}.
\label{equidistribution of fw}
\end{align}

Let $(s_v)$ (resp. $(s_w))$ denote the sequence in $X(G)$ of conjugacy classes of elements 
$s_v$ (resp. the sequence in $X(G_{0})$ of conjugacy classes of elements $s_w$) by a slight abuse of notation.

\begin{lemma}
If $(s_v)$ is equidistributed in $X(G)$ with respect to $\mu$, then $(s_{w})$ is equidistributed in $X(G_{0})$ with respect to $\mu_{0}$. 
\label{equidistribution of fv implies equidistribution of fw}
\end{lemma}

\begin{proof} 
Let $(s_{v_{0}})$ be the subsequence of $(s_v)$ with $s_{v_{0}} \in X(G_{0})$. The primes 
$v_{0}$ are precisely those that split completely in $K_{0}/K$. The sequence $(s_{w_{0, j}})_{0, j}$ is 
obtained from $(s_{v_{0}})$ by replacing each $s_{v_{0}}$ with $t+1$ conjugates
$s_{w_{0, j}} = g_{j}^{-1} \, s_{v_{0}} \, g_{j}$.  Let $\chi_{0}$ be an irreducible character of $G_{0}$ and 
$\chi := {{\rm{Ind}}_{G_{0}}^{G} \, \chi_{0}}$. By assumption and the definition of equidistribution,
\eqref{equidistribution of fv} holds. Hence the sequence $(s_{w_{0, j}})_{0, j}$ 
is equidistributed in $X(G_{0})$ with respect to $\mu_{0}$ because of \eqref{equidistribution of fw}.
By Lemma \ref{Equidistribution with respect to all primes and primes in S}, the sequence $(s_{w})_{w}$ is equidistributed in $X(G_{0})$ with respect to $\mu_{0}$.
\end{proof}

\begin{theorem}
Assume that every nontrivial, irreducible character of $G$ with nonabelian image 
is induced from an irreducible character of $G_{0}$.
Then the sequence $(s_v)$ is equidistributed in $X(G)$ with respect to $\mu$ if and only if 
the sequence $(s_{w})$ is equidistributed in $X(G_{0})$ with respect to $\mu_{0}$.
\label{equidistribution of fw implies equidistribution of fv if chi is induced from and irred. charac.}
\end{theorem}

\begin{proof} Equidistribution of $(s_v)$ implies equidistribution of $(s_w)$ by Lemma
\ref{equidistribution of fv implies equidistribution of fw}. 
We deduce the converse implication applying Lemma \ref{reduction for equidistribution easy case}:
condition (i) follows from the analytic continuation of abelian $L$-functions associated to Hecke characters (see \cite[Corollary~15]{Jo}); conditions (ii) and (iii) hold by hypothesis; condition (iv) holds by the Chebotarev density theorem.
\end{proof}

\medskip

Our next goal is the strengthening of Theorem \ref{equidistribution of fw implies equidistribution of fv if chi is induced from and irred. charac.}. Let $H$ be a subgroup of $G$, not necessarily normal, such that $G_{0} \subset H \subset G$.
Via isomorphism $\Gal(K_0/K) \simeq G/G_0$ we have $\Gal(K_0/L) \simeq H / G_{0}$ for some subfield $L$ of $K_0$ containing $K$. Let $f_v \in \mathcal{G}$, as before, be the normalized Frobenius for $\overline{v} \, | \, v$ and let $f_{w_{1}} \in 
\mathcal{H}$ be the normalized Frobenius for $\overline{w_1} \, | \, w_1$ for a prime $w_1 \subset \mathcal{O}_{L}$,
outside the finite set of primes over $S_{\mathrm{fin}}$.
Recall that by \eqref{isomorphism between G(K0 / L) and STL / STK0}, we have a natural isomorphism
\begin{equation}
G(K_{0} / L) \simeq H / G_{0}.
\label{G(K0 / L) cong H / G0}
\end{equation}
Similarly to \eqref{coset decomposition} we can write 
\begin{equation}
\label{The left coset representatives of H in G}
G \, = \, \bigsqcup_{i =0}^{t} \, g_i \, H,
\end{equation}
for fixed left coset representatives $g_{0}, \dots, g_{t} \in G$ with $g_0 = e$. 
Again these representatives may be taken to be images of elements of the Galois group $G_{K}$ via the $l$-adic representation. This choice of ${g_{i}}$'s is possible by equality \eqref{decomposition of STK2 into cosets over Galois representatives wr to STK1} of Corollary
\ref{STK2 iff STK1} \emph{or} equality 
\eqref{decomposition of tilde STK2 into cosets over Galois representatives wr to tilde STK1}
of Corollary \ref{tilde STK2 iff tilde STK1}. 

Let: 
\begin{gather*}
\mathcal{H} := \AST_{L} (V, \psi)(\C) \quad \emph{or} \quad \mathcal{H} := \widetilde{\AST_L} (V, \psi)(\C), \\
H := \ST_{L} (V, \psi)(\C) \quad \emph{or} \quad H := \widetilde{\ST_L} (V, \psi)(\C).
\end{gather*}

By Propositions \ref{ConnCompIsom} and \ref{connected components iso for tilde}, applied for base fields $K$ and $L$
we obtain the following, natural bijection of coset spaces:  
\begin{equation}
H \backslash G \,\,\, \simeq  \,\,\, \mathcal{H} \backslash \mathcal{G}. 
\label{Bijections of cosets for AST and ST} 
\end{equation}
Hence we have:
\begin{equation}
H \, = \, G \, \cap \, \mathcal{H}. 
\label{G cap mathcal H equals H} 
\end{equation}
Therefore 
$$
\mathcal{G} \, = \, \bigsqcup_{i =0}^{t} \, g_i \, \mathcal{H},
$$
for the left coset representatives $g_{0}, \dots, g_{t}$ with $g_0 = e$ (see \eqref{The left coset representatives of H in G}). 

\begin{lemma} There is a one-to-one correspondence between the set of elements $g_i \in 
\mathcal{G}$ such that $g_{i}^{-1} f_{v} g_{i}\in \mathcal{H}$ and the set of primes $w \, | \, v$ of 
$\mathcal{O}_{L}$ that split completely over $v$, for all $v \notin S_{\mathrm{fin}}$.
\label{primes of L splitting completely over  v0}
\end{lemma} 
\begin{proof} Let $\overline{v}$ be a fixed prime in $\mathcal{O}_{\overline{K}}$ over $v$ and let
$G_{\overline{v}}$ be the decomposition group of $\overline{v}$. We observe that the double coset space
$G_{L} \backslash G_{K} / G_{\overline{v}}$ is in bijection with the set of primes $w|v$ of $\mathcal{O}_L$ via the map
$G_{L} \gamma \, G_{\overline{v}} \mapsto \gamma (\overline{v}) \cap L$. We have the following bijections:
\begin{equation}
\label{double coset comparison}
G_{L} \backslash G_{K} / G_{\overline{v}} \,\,\, \cong \,\,\, G(K_{0}/L) \backslash G(K_{0} / K) / \langle f_{v} \rangle 
\,\,\, \cong \,\,\, \mathcal{H} \backslash \mathcal{G} / \langle f_{v} \rangle
\end{equation}
The second bijection follows by Diagram \ref{diagram of natural maps from Galois to Zariski closures} in Lemma
\ref{lemma concerning diagram of natural maps from Galois to Zariski closures} and by Diagram 
\ref{diagram comparing the rho (GK) with GLalg for K0} in Theorem \ref{jK0K and Zar1 are isomorphisms},  
Diagram \ref{diagram comparing the tilde rho (GK) with tilde GLalg for K0} in Theorem 
\ref{tilde jK0K and tilde Zar1 are isomorphisms}
and Conjectures \ref{general algebraic Sato Tate conj.} and \ref{general algebraic Sato Tate conj. Serre's approach} 
that are assumed in this section. We would like to point out that applying Diagrams 
\ref{diagram of natural maps from Galois to Zariski closures},  
\ref{diagram comparing the rho (GK) with GLalg for K0} and 
\ref{diagram comparing the tilde rho (GK) with tilde GLalg for K0} we should consider all quotient groups 
in these diagrams in the form of right cosets. 

By a slight abuse of notation,
let $\langle f_{v} \rangle$ denote the subgroup generated by the Frobenius element $f_v$ in $G(K_{0} / K)$ in the middle term of \eqref{double coset comparison} and let $\langle f_{v} \rangle$ denote the subgroup generated by the normalized Frobenius $f_v$ in $\mathcal{G}$ 
(denoted in the same way as Frobenius) in the right term of \eqref{double coset comparison}. 
Then the double coset space $\mathcal{H} \backslash \mathcal{G} / \langle f_{v} \rangle$ is in bijection with the set of primes $w|v$ of $\mathcal{O}_L$ via the map $\mathcal{H} \,\gamma \, \langle f_{v} \rangle \mapsto \gamma (\overline{v}) \cap L$. Under this bijection, the decomposition group of $w$ over $v$ is the group $\gamma \langle f_{v} \rangle \gamma^{-1} / (\gamma \langle f_{v} \rangle \gamma^{-1} \cap \mathcal{H})$, so the inertia degree of $w$ over $v$ equals the index
$[\gamma \langle f_{v} \rangle \gamma^{-1} : \gamma \langle f_{v} \rangle \gamma^{-1} \cap \mathcal{H}]$.
Consequently, $w$ splits completely over $v$ if and only if $\gamma \langle f_{v} \rangle \gamma^{-1} \subseteq \mathcal{H}$, or equivalently if and only if $\gamma f_v \gamma^{-1} \in \mathcal{H}$. 

To summarize, at this point we have a one-to-one correspondence between the double cosets
$\mathcal{H} \gamma \langle f_{v} \rangle$ for which $\gamma f_v \gamma^{-1} \in \mathcal{H}$ and the set of primes $w|v$ of $\mathcal{O}_{L}$ that split completely over $v$. Now note that each such double coset consists of \emph{one} right coset of $\mathcal{H}$:
\[
\mathcal{H} \gamma \langle f_{v} \rangle = \bigcup_{i \in \Z} \mathcal{H} g f_v^i = \bigcup_{i \in \Z} 
\mathcal{H} (\gamma f_v^i \gamma^{-1}) \gamma = \mathcal{H} \gamma.
\]
We may thus replace the double cosets in the correspondence by the right cosets
$\mathcal{H} \gamma$ for which $\gamma f_v \gamma^{-1} \in \mathcal{H}$. 
Since $\mathcal{G} = \bigsqcup_{i=0}^t \mathcal{H} g_i^{-1}$ by \eqref{The left coset representatives of H in G} and
\eqref{Bijections of cosets for AST and ST}, this yields the desired result.
\end{proof}

We can write $f_v = s_v \, u_v$ where $s_v \in \mathcal{G}$ is the semisimple part of $f_v$ and 
$u_v  \in \mathcal{G}$ is the unipotent part of $f_v$. Recall that $s_v \, u_v = u_v \, s_v$. 

\begin{lemma}
\label{Lemma about bijections of double cosets for AST and ST}
By the isomorphism \eqref{Bijections of cosets for AST and ST} and Lemma \ref{unipotent elements in an algebraic group},
there are natural bijections of double coset spaces: 
\begin{equation}
\label{Bijections of double cosets for AST and ST} 
H \backslash G / \langle s_{v} \rangle \cong \,\,\, 
\mathcal{H} \backslash \mathcal{G} / \langle s_{v} \rangle \cong \,\,\, 
\mathcal{H} \backslash \mathcal{G} / \langle f_{v} \rangle.
\end{equation}
\end{lemma}
\begin{proof}
 Observe that (cf. the proof of Corollary 
\ref{connected components of AST K (V, psi), ST K (V, psi) and widetilde AST K (V, psi), widetilde ST K (V, psi)})
\begin{equation}
\mathcal{G}^{\circ} \subset \mathcal{G}_{0} \subset \mathcal{H} \subset \mathcal{G}.
\label{Containment of some Alg. Sato--Tate groups} 
\end{equation}

Observe that $\forall g \in \mathcal{G}$ and $\forall i$ we obtain 
$g u_{v}^{i} g^{-1} \in U(\mathcal{G})$. Hence by Lemma \ref{unipotent elements in an algebraic group} 
and \eqref{Containment of some Alg. Sato--Tate groups} we have:
\begin{equation}
\mathcal{H} g f_{v}^{i} = \mathcal{H} g s_{v}^{i} u_{v}^{i} = \mathcal{H} g u_{v}^{i} s_{v}^{i} =
\mathcal{H} g u_{v}^{i} g^{-1} g s_{v}^{i} = \mathcal{H} g s_{v}^{i}.
\label{replacing f v with s v}
\end{equation}
Hence:
\begin{equation}
\mathcal{H} g \langle f_{v} \rangle  = \mathcal{H} g \langle s_{v} \rangle.
\label{replacing l f v r with l s v r>}
\end{equation}
Therefore:
\begin{equation}
\mathcal{H} \backslash \mathcal{G} / \langle s_{v} \rangle \, \simeq \,\,\, 
\mathcal{H} \backslash \mathcal{G} / \langle f_{v} \rangle. 
\label{replacing g with G to obtain double cosets equality}
\end{equation}

Recall that $\langle s_{v} \rangle \subset G$. Assume that for $g_1, g_2 \in G$
the double cosets $H g_1 \langle s_{v} \rangle$ and $H g_2 \langle s_{v} \rangle$ 
in $H \backslash G / \langle s_{v} \rangle$
map to the same double coset $\mathcal{H} g_1 \langle s_{v} \rangle = \mathcal{H} g_2 \langle s_{v} \rangle$ in 
$\mathcal{H} \backslash \mathcal{G} / \langle s_{v} \rangle$. Hence 
$g_2 \langle s_{v} \rangle {\langle s_{v} \rangle}^{-1} g_{1}^{-1} \subset \mathcal{H}$. Therefore 
$H g_1 \langle s_{v} \rangle = H g_2 \langle s_{v} \rangle$ by \eqref{G cap mathcal H equals H}.
It follows that:
\begin{equation}
H \backslash G / \langle s_{v} \rangle \simeq \,\,\, 
\mathcal{H} \backslash \mathcal{G} / \langle s_{v} \rangle.
\label{Replacing g_1 and g_2 with G to get double coset equality}
\end{equation}
Now the Lemma follows by \eqref{replacing g with G to obtain double cosets equality} and 
\eqref{Replacing g_1 and g_2 with G to get double coset equality}.  
\end{proof}

\begin{corollary} There is a one-to-one correspondence between the set of elements $g_i \in 
G$ such that $g_{i}^{-1} s_{v} g_{i} \in H$ and the set of primes $w \, | \, v$ of 
$\mathcal{O}_{L}$ that split completely over $v$, for all $v \notin S_{\mathrm{fin}}$.
\label{primes of L splitting completely over v0 correspond to conjugations of sv}
\end{corollary}
\begin{proof} 
By lemma \ref{primes of L splitting completely over  v0} there is a one-to-one correspondence between the set of elements $g_i \in 
G$ such that $g_{i}^{-1} f_{v} g_{i} \in \mathcal{H}$ and the set of primes $w \, | \, v$ of 
$\mathcal{O}_{L}$ that split completely over $v$. On the other hand by 
Lemma \ref{unipotent elements in an algebraic group}, containment \eqref{Containment of some Alg. Sato--Tate groups}, 
equality \eqref{G cap mathcal H equals H} and Lemma \ref{Lemma about bijections of double cosets for AST and ST}
we have:  
\begin{gather*}
g_{i}^{-1} f_{v} g_{i} \in \mathcal{H} \iff \mathcal{H} g_{i}^{-1} \langle f_{v} \rangle = \mathcal{H} g_{i}^{-1}
\iff \mathcal{H} g_{i}^{-1} \langle s_{v} \rangle = \mathcal{H} g_{i}^{-1} \iff \\
\iff H g_{i}^{-1} \langle s_{v} \rangle = H g_{i}^{-1} \iff g_{i}^{-1} s_{v} g_{i} \in H.
\qedhere
\end{gather*}
\end{proof}

For  $m \in \N$ put $N_{K} (m) :=  \# \{v \in {\text{Spec}} \, \mathcal{O}_K\colon N v  \leq m  \}$.
Define $N_{L} (m)$ in the same way. E. Landau proved that $N_{K} (m)$ and $N_{L} (m)$ are asymptotically equal. Let $(w_1)$ be the sequence of primes in ${\text{Spec}} \, \mathcal{O}_L$ that split completely 
over $K$ and let $S_{L} (m) :=  \# \{w_1 \in {\text{Spec}} \, \mathcal{O}_K\colon N w_1 \leq m  \}$. We observe that
$S_{L} (m)$ and $N_{L} (m)$ are asymptotically equal (cf. the proof of Lemma
\ref{Equidistribution with respect to all primes and primes in S}). 
\medskip

Let $\chi = {{\rm{Ind}}}_{H}^{G} \chi_{1}$. By Lemma \ref{primes of L splitting completely over v0} and \eqref{integral of induced character from H} we obtain:
\begin{equation}
\sum_{N v \, \leq \, m} \chi (s_{v}) = \sum_{N w_{1} \, \leq \, m} \, \chi_{1} (s_{w_{1}}).
\label{Equidistribution. Going from L to K}
\end{equation}
By Lemma \ref{relation between integral of chi prime and chi} and 
\eqref{Equidistribution. Going from L to K}, the following equalities 
\eqref{equidistribution of fv 2} and \eqref{equidistribution of fw 2} are equivalent: 

\begin{align}
\lim_{m\to\infty} \frac{1}{N_{K} \, (m)} \, \sum_{N v \, \leq \, m} \chi (s_{v})\, &= \,
 \int_{G}\, {\chi}\, d \mu,
\label{equidistribution of fv 2} \\
\lim_{m\to\infty} \frac{1}{N_{L} \, (m)} \sum_{N w_{1} \, \leq \, m} \, \chi_{1} (s_{w_{1}}) \, &= \, \int_{H}\, {\chi_{1}}\, d \mu_{H}.
\label{equidistribution of fw 2}
\end{align}

Let $(s_v)$ (resp. $(s_{w_1}))$ denote the sequence in $X(G)$ of conjugacy classes of elements 
$s_v$ (resp. sequence in $X(H)$ of conjugacy classes of elements $s_{w_1}$) 
by slight abuse of notation.

\begin{remark} 
Assume that the character $\chi$ in formula \eqref{chi as sum of inductions of chi i j} is irreducible. 
Because $\chi$ and $\chi_{i,j}$ are irreducible, $\int_{G} \, {\chi} \, d \mu = 0$ and 
$\int_{H_{i}} \, {\chi_{i,j}} \, d \mu_{H_{i}} = 0$ for all $\chi_{i,j} \not= 1$. 
Integrating \eqref{chi as sum of inductions of chi i j} over $G$ and applying Lemma 
\ref{relation between integral of chi prime and chi}, we obtain:
\begin{equation}
\int_{G}\, {\chi}\, d \mu = \sum_{i=1}^{s} \sum_{j =1}^{t_i} \, a_{i} \int_{H_{i}}\, {\chi_{i,j}}\, d \mu_{H_i}.
\label{integral of chi as a sum of integrals of chi i, j}
\end{equation}
\end{remark}
\medskip

\begin{theorem} Under the assumptions of Theorem \ref{STK iff STK0} 
(resp. \ref{tilde STK iff tilde STK0} ), Sato--Tate Conjecture \ref{general Sato Tate conj.} (resp. Sato--Tate Conjecture \ref{general Sato Tate conj. tilde}) holds for the representation $\rho_l$ (resp. $\widetilde{\rho}_l$) with respect to $\ST_{K} (V, \psi)$ (resp. $\widetilde{\ST_{K}}(V, \psi)$)
if and only if for each subextension $K_1$ of $K_0/K$ for which $K_0/K_1$ is cyclic,
the same conjecture holds for the representation $\rho_{l} \, | \, G_{K_1}$ (resp.  
$\widetilde{\rho}_l \, | \, G_{K_1}$) with respect to $\ST_{K_1} (V, \psi)$ 
(resp. $\widetilde{\ST_{K_1}}(V, \psi)$).
\label{Sato--Tate conjecture STK iff STK1}
\end{theorem}
\begin{proof}
The ``only if'' direction is proven in the same way as Lemma \ref{equidistribution of fv implies equidistribution of fw}. The "if" direction is proven as follows. 
Let $H_i$ be the groups that appear in equations \eqref{Brauer induction for 1 G G0} and 
\eqref{lifting Brauer induction for 1 G}. By Proposition 
\ref{Intermediate subgroup between STK (V, psi) and STK0 (V, psi) is STL (V, psi)} 
(resp. Proposition \ref{Intermediate subgroup between tilde STK (V, psi) and tilde STK0 (V, psi) is STL (V, psi)}) there is a field subextension $K \subset L_{i} \subset K_{0}$, for each $1 \leq i \leq s$,  such that 
$H_{i} = \ST_{L_i} (V, \psi)$ (resp. $H_i = \widetilde{\ST_{L_i}}(V, \psi)$). By hypothesis, 
the Sato--Tate conjecture holds for each group $H_{i}$.  Then for each $i, j$ we can apply \eqref{equidistribution of fw 2} to get:
\begin{equation}
\lim_{m\to\infty} \frac{1}{N_{L_{i}} \, (m)} \sum_{N w_{1,i} \, \leq \, m} \, \chi_{i,j} (s_{w_{1,i}}) \, = 
\, \int_{H_i}\, {\chi_{i,j}}\, d \mu_{H_{i}}.
\label{equidistribution of fw i, j}
\end{equation}
Let $\chi \in X(G)$ be 
irreducible and consider the formula \eqref{chi as sum of inductions of chi i j}. 
By \eqref{Equidistribution. Going from L to K}, \eqref{integral of chi as a sum of integrals of chi i, j},
and \eqref{equidistribution of fw i, j}:
\begin{align*} 
\lim_{m\to\infty} \frac{1}{N_{K} \, (m)} \, \sum_{N v \, \leq \, m} \chi(s_{v})\, &= \,
\sum_{i=1}^{s} \sum_{j =1}^{t_i} a_i  \, \lim_{m\to\infty} \frac{1}{N_{K} \, (m)} \, \sum_{N v \, \leq \, m}
\, {{\rm{Ind}}}_{H_{i}}^{G} \chi_{i,j} (s_v) \\
&= \sum_{i=1}^{s} \sum_{j =1}^{t_i} a_i  \, \lim_{m\to\infty} \frac{1}{N_{L_i} \, (m)} \, \sum_{N w_{1, i} \, \leq \, m} \, 
\,  \chi_{i,j} (s_{w_{1, i}}) \\
&= \sum_{i=1}^{s} \sum_{j =1}^{t_i} a_i  \, \int_{H_i}\, {\chi_{i,j}}\, d \mu_{H_{i}} = \int_{G}\, {\chi}\, d \mu.
\end{align*}
Consequently $(s_v)$ is equidistributed in $X(G)$.
\end{proof}

\begin{remark}
A stronger form of Theorem \ref{Sato--Tate conjecture STK iff STK1},
in which the Sato--Tate conjecture for $\ST_{K} (V, \psi)$ and $\ST_{K_0} (V, \psi)$ are equated, was stated in
\cite[Theorem 6.12]{BK2} and was left without proof. We do not know how to prove such a statement in general.
However, Theorem~\ref{reduction for equidistribution2} implies that this holds when the derived group of $G_0$ has pairwise distinct simple factors, none of which admits a nontrivial diagonal automorphism
(this condition rules out factors of types $A_n, D_4, E_6$).
\label{Sato--Tate conjecture refinement}
\end{remark}

\section{Relations among motivic categories}
\label{relations among motivic categories}
Let $K$ be a field contained in $\C$. Let ${{\rm SmProj}_K}$ denote the category of smooth projective
varieties over $K$ and let $\mathcal{V}$ be a full subcategory of ${{\rm SmProj}_K}$. 
Let $\mathcal{M}_{{\rm{\sim}}} := \mathcal{M}_{K, {{\rm{_{\sim}}}}}$ denote the motivic category corresponding to an equivalence relation $\sim$. In our applications, $\sim$ will be one of the following properties
defining the corresponding motivic category:

\begin{itemize}
\item[$\sim \, = {\rm{rat}}$]: \quad rational equivalence, 
\item[$\sim \, = {\rm{alg}}$]: \quad algebraic equivalence, 
\item[$\sim \, = {\rm{\otimes}}$]: \quad smash nilpotent equivalence,
\item[$\sim \, = {\rm{hom}}$]: \quad homological equivalence,
\item[$\sim \, = {\rm{num}}$]: \quad numerical equivalence, 
\item[$\sim \, = {\rm{mot}}$]: \quad motivated cycles of Andr{\' e},
\item[$\sim \, = {\rm{ahc}}$]: \quad absolute Hodge cycles of Deligne.
\end{itemize}

We will also denote $\mathcal{M}_{\rm{hs}} :=  \mathcal{M}_{\Q, {\rm{hs}}}$ the category of $\Q$-rational Hodge structures. Y. Andr{\' e} \cite[Chap. 9]{An1} constructs the category of motivated cycles $\mathcal{M}_{{\rm{mot}}, \, \mathcal{V}} := \mathcal{M}_{K, {{\rm{mot}}}}$ modeled on $\mathcal{V}$. We will denote 
$\mathcal{M}_{{\rm{mot}}} := \mathcal{M}_{{\rm{mot}}, \, {{\rm SmProj}_K}}$.
\medskip 

Diagram \ref{diagram of motivic categories} indicates the functors relating these motivic categories
 (cf. \cite[p. 30]{MNP}; see also \cite{An1} concerning the category 
$\mathcal{M}_{{\rm{mot}}})$. 

\begin{figure}[H]
\[
\begin{tikzcd}
& \mathcal{M}_{{\rm{hs}}}\arrow{rr}{=}&& \mathcal{M}_{{\rm{hs}}} & \\
& \mathcal{M}_{{\rm{mot}}} \arrow{u}\arrow{rr}&& \mathcal{M}_{{\rm{ahc}}} \arrow{u} & \\
\mathcal{M}_{{\rm{rat}}} \arrow{r} & \mathcal{M}_{{\rm{alg}}}  \arrow{r} \arrow{u}&\mathcal{M}_{{\otimes}}\arrow{r}& \mathcal{M}_{{\rm{hom}}} \arrow{u}\arrow{r} & \mathcal{M}_{{\rm{num}}}\\ 
{{\rm SmProj}_K} \arrow{u}[swap]{ch}&&&& \\
\end{tikzcd}
\]
\\[-1cm]
\caption{}
\label{diagram of motivic categories}
\end{figure}

In the construction of $\mathcal{M}_{\sim \,}$ (see for example  \cite[p.\ 200]{DM}, \cite[p. 31]{An1}, \cite[Chap. 2]{MNP} etc.)  one starts with the category ${{\rm SmProj}_K}$ and constructs the new category with objects 
$h_{{\sim}}(X)$ and morphisms between them:
\begin{equation}
{\rm{Hom}}_{\mathcal{M}_{\sim}} (h_{{\sim}}(X), \, h_{{\sim}}(Y)) :=  C_{\sim}^d (X, Y) 
\label{definition of Hom in MK}
\end{equation}
for $X$ and $Y$ smooth projective over $K$ and $X$ of pure dimension $d$. 
The $\Q$-vector space $C_{\sim}^d (X, Y)$ is defined as follows: 
\begin{equation} 
C_{\sim}^{d} (X, Y) \,\, = \,\,
\left\{
\begin{array}{lll}
Z^d (X \times Y)_{\Q}  / Z_{\sim}^d (X \times Y)_{\Q} &\rm{if}& {\sim} = {\rm{rat}}, {\rm{alg}}, 
{\rm{\otimes}}, {\rm{hom}}, {\rm{num}},\\
\CH_{\AH}^d (X \times Y) &\rm{if}& {\sim} = {\rm{ahc}},\\
\CH_{{\rm{mot}}}^d (X \times Y) &\rm{if}& {\sim} = {\rm{mot}}
\end{array}
\right\}
\end{equation}
where $\CH_{\AH}^d (X \times Y)$ is defined in \cite[p. 34]{D1} and $\CH_{{\rm{mot}}}^d (X \times Y)$ is the image of the natural map $Z^{d}_{{\rm{mot}}} (X \times Y) \rightarrow H^{2d} (X \times Y) (d)$ (see \cite[pp 95--96]{An1}). For any smooth, projective $X$ over $K$,
the motivated cycles $Z^{\ast}_{{\rm{mot}}} (X)$ do not depend on the choice of a Weil cohomology 
$H^{\ast}$ up to natural isomorphism \cite[Lemme 9.2.1.2]{An1}. 

Based on this, one performs the classical construction due to Grothendieck to obtain corresponding categories of motives $\mathcal{M}_{{\sim}}$ (see for example \cite[p.31--35]{An1}, \cite[pp. 200--205]{DM}, \cite[p. 25--29]{MNP}). 

\section{Assumptions on \texorpdfstring{$\mathcal{M}_{{\sim}}$}{Msim}}
\label{assumptions on Msim}

From now on in this paper, we will work under the following classical assumptions on the category $\mathcal{M}_{{\sim}}$.
\medskip

\noindent
{\bf{ASSUMPTION 1.}} \,
The Chow--K{\" u}nneth decomposition holds in $\mathcal{M}_{{\sim}}$: there are idempotents 
$p_{r} (X) \in {\rm{End}}_{\mathcal{M}_{\sim}} (h_{{\sim}}(X))$ orthogonal to each other lifting the 
K{\" u}nneth components in  Betti, {\' e}tale, and de Rham realizations such that:
\begin{equation} 
h_{{\sim}}(X) \, = \, \bigoplus_{r \geq 0} \,\, h_{{\sim}}^{r} (X),
\label{Chow- Kunneth decomposition of hX in to dir sum of hiM}
\end{equation} 
where  $h_{{\sim}}^{r} (X) = (h_{{\sim}}(X), p_{r} (X), 0)$.

\begin{remark} 
Observe that the Chow--K{\" u}nneth decomposition holds 
in $\mathcal{M}_{{\rm{ahc}}}$ (see \cite[p. 202]{DM}) and in $\mathcal{M}_{{\rm{mot}}}$
(see \cite[Prop. 2.2]{An2}; cf. \cite[Lemme 9.2.1.3]{An1}).
\end{remark}

\begin{remark}
The Chow--K{\" u}nneth conjecture for rational equivalence (see \cite[Chapter 6, pp. 67--69]{MNP})
states that there is a Chow--K{\" u}nneth decomposition:
\begin{equation} 
h_{{\rm{rat}}}(X) \, = \, \bigoplus_{r \geq 0} \,\, h_{{\rm{rat}}}^{r} (X) 
\end{equation}
The Chow--K{\" u}nneth conjecture directly implies the Chow--K{\" u}nneth 
decomposition for all equivalence relations ${\sim}$ listed above except 
${\sim} = {\rm{num}}$. 
Note that Assumption 1 is not well formulated for ${\sim} = {\rm{num}}$ because the realization functors are not available
(see \cite[\S 2.6]{MNP}). If we assume Standard Conjecture D (equivalence between homological and numerical equivalence relations), then Assumption 1 for numerical equivalence becomes well formulated and equivalent to Assumption 1 for homological equivalence.
\end{remark}

\begin{definition}
Let $h_{{\sim}}(X) (m)$ be the twist of 
$h_{{\sim}}(X)$ by the $m$-th power of the Lefschetz motive $\LL := h_{{\sim}}^{2} (\PP^1)$ for some $m \in \Z$.
A direct summand of a motive of the form $h_{{\sim}}^{r} (X) (m)$ will be called a {\it{homogeneous motive}} of weight $n := 2 m - r$.

\label{homogeneous motive definition}
\end{definition}
\medskip

Recall that ${\bf{1}} := h_{{\sim}} ({\rm{Spec}} \, K)$. Hence $\LL = {\bf{1}} (-1)$ so
$\LL^{\otimes \, r} = {\bf{1}} (-r)$.
\medskip

\begin{remark}
Let $L/K$ be a field extension such that $K \subset L \subset \overline{K}$.  If $M \in \mathcal{M}_{{\sim}}$ is a direct summand of $h_{{\sim}} (X) (m)$, 
then $M \otimes_{K} L$ is a direct summand of $h_{{\sim}} (X \otimes_{K} L) (m)$ in $\mathcal{M}_{L, {\sim}}
\,\, $. 
\label{base field extension and direct summands of motives}
\end{remark}
\medskip

The properties of the Chow--K{\" u}nneth decomposition (cf. \cite[Def. 6.1.2 and Example 6.1.5. (3)]{MNP}), the properties of Weil cohomology, and the cycle map \cite[p. 10]{Kl}
imply the existence of ({\em{Motivic Poincar\'e duality}})
$$
P\colon h_{{\sim}}^{r} (X) \, \otimes \, h_{{\sim}}^{2d-r} (X) (d-r) \,\, {\stackrel{}{\longrightarrow}}\,\,  
h_{{\sim}}^{2d} (X)(d-r) \,\, {\stackrel{Tr}{\longrightarrow}} \,\, {\bf{1}} (-r)
$$ 
whose Betti, {\' e}tale, and de Rham realizations give corresponding 
Poincar\'e dualities in Betti, {\' e}tale, and de Rham cohomologies.
\medskip

\noindent
{\bf{ASSUMPTION 2.}} \,
The category $\mathcal{M}_{{\rm{\sim}}}$ is Tannakian and semisimple over $\Q$.
\medskip

Let $H_{B}$ denote now the corresponding fiber functor:
\begin{equation}
H_{B}\colon \mathcal{M}_{{\rm{\sim}}} \rightarrow \mathcal{M}_{{\rm{hs}}}.
\label{fiber functor}
\end{equation}

\begin{remark}
\begin{itemize}
\item[(1)] The category $\mathcal{M}_{{\rm{ahc}}}$ is Tannakian and semisimple over $\Q$ \cite[Prop. 6.5]{DM}. 
\item[(2)] For every motive $M$ in $\mathcal{M}_{{\rm{mot}}}$,
for any ``suitably large'' (depending on $M$) family of absolute cycles $\mathcal{V}$, the tensor subcategory of $\mathcal{M}_{{\rm{mot}}, \mathcal{V}}$ generated by $M$ is Tannakian and semisimple
\cite[Th{\' e}or{\` e}me 9.2.3.1]{An1}.
Consequently, $\mathcal{M}_{{\rm{mot}}}$ is Tannakian and semisimple.
\item[(3)] For the category $\mathcal{M}_{{\rm{hom}}}$ the Betti realization $H_B$ is a fiber functor. 
For $\mathcal{M}_{{\rm{num}}}$ the Betti realization $H_B$ is not well defined in general.
Under Standard Conjecture D (equivalence of $\mathcal{M}_{{\rm{num}}}$ with $\mathcal{M}_{{\rm{hom}}}$),
$H_B$ is well defined on $\mathcal{M}_{{\rm{num}}}$. Moreover Standard Conjecture D implies Standard 
Conjecture C (i.e. $p_{r} (X) \in {\rm{End}}_{\mathcal{M}_{{\rm{hom}}}} (h_{{\rm{hom}}}(X))$), 
see \cite[Th{\' e}or{\` e}me 5.4.2.1]{An1}.
\item[(4)] Jannsen proved \cite{Ja92} that $\mathcal{M}_{{\rm{num}}}$ is a semisimple abelian category. 
He also proved \cite{Ja92} that under Conjecture C, the category $\mathcal{M}_{{\rm{num}}}$ 
is Tannakian and semisimple. Hence under Standard Conjecture D, both categories $\mathcal{M}_{{\rm{hom}}}$ and 
$\mathcal{M}_{{\rm{num}}}$ are Tannakian and semisimple.
\end{itemize}
\end{remark}
 
Next we want an assumption that provides a polarization on the homogeneous motives $h^{r}_{{\sim}}(X) (m)$ in $\mathcal{M}_{{\sim}}$. 
\medskip

\noindent
{\bf{ASSUMPTION 3.}} There is a natural morphism ({\em{Motivic star operator}}) for every smooth projective $X$:
$$
\ast\colon  h_{{\sim}}^{r} (X) \, \rightarrow \, h_{{\sim}}^{2d-r} (X),
$$ 
sent via Betti, {\' e}tale, and de Rham realizations onto the morphism  $H^r (X) \rightarrow H^{2d-r} (X)$, 
which is defined by an absolute Hodge cycle 
(see \cite[Proposition 6.1 and discussion on p. 199]{DM}).

\begin{remark}
Assumption 3 leads to the following natural morphism in $\mathcal{M}_{{\sim}}$:
\begin{equation}
\psi_{\sim}^{r}\colon   h_{{\sim}}^{r} (X) \otimes h_{{\sim}}^{r} (X) 
\,\, {\stackrel{{\rm{id}} \, \otimes \, \ast}{\longrightarrow}} \,\, h_{{\sim}}^{r} (X) \otimes 
h_{{\sim}}^{2d-r} (X) (d-r) \,\, 
 {\stackrel{P}{\longrightarrow}} \,\, {\bf{1}}(-r). 
\label{motivic Hodge polarization}
\end{equation}
Under Assumption 3 the functor $\mathcal{M}_{{\sim}} \rightarrow \mathcal{M}_{{{\rm{ahc}}}}$ sends 
$\psi_{\sim}^{r}$ to $\psi_{}^{r}$ \cite[p. 199]{DM}:
\begin{equation}
\psi_{}^{r}\colon   H^{r} (X) \otimes H^{r} (X) 
\,\, {\stackrel{{\rm{id}} \, \otimes \, \ast}{\longrightarrow}} \,\, H^{r} (X) \otimes 
H^{2d-r} (X) (d-r) \,\, 
 {\stackrel{P}{\longrightarrow}} \,\, {\Q}(-r). 
\label{Polarized Hodge structure coming from motives}
\end{equation}
\end{remark}

Assumption 3 holds for $\mathcal{M}_{{{\rm{ahc}}}}$ (see \textit{loc. cit.}).
Hence for $h_{{\sim}}^{r}(X)$ in $\mathcal{M}_{{\sim}}$, we obtain a polarization of the real Hodge structure  associated with the rational Hodge structure on the Betti realization $H_B (h_{{\sim}}^{r}(X)) \, = 
\, H^{r} (X(\C), \, \Q)$,  because $H_B$ factors through $\mathcal{M}_{{{\rm{ahc}}}}$ and $\psi_{}^{r}$ gives the polarization on $\mathcal{M}_{{{\rm{ahc}}}}$  \cite[pp.\ 197--199]{DM} (cf. \cite[p.\ 478--480]{Pan}, \cite[pp.\ 2--4]{Ja90}). By taking a Tate twist in \eqref{motivic Hodge polarization} and applying the Betti realization, we 
also obtain a polarization on $h_{{\sim}}^{r}(X)(m)$. Therefore, for any homogeneous motive $M \in \mathcal{M}_{{{\sim}}}$ of weight $n$, there is the following map in $\mathcal{M}_{{\sim}}$:
\begin{equation}
\psi_{{\sim}}\colon  M \otimes M \,\, {\stackrel{}{\longrightarrow}} \,\, {\bf{1}}(-n). 
\label{motivic Hodge polarization on M}
\end{equation}
By \cite[Cor. 2.12, p. 40]{PS} we obtain the polarization of the real Hodge structure associated with rational Hodge structure on the Betti realization $V := H_B (M)$:
\begin{equation}
\psi\colon V  \otimes V \,\, {\stackrel{}{\longrightarrow}} \,\, {\Q}(-n). 
\label{Hodge polarization on M}
\end{equation}

Concerning the polarization for $\mathcal{M}_{{\rm{mot}}}$, it follows from \cite[Prop. 2.2, p. 16]{An2}, 
(cf. \cite[Remarque 9.2.2.1]{An1})  that for  $M =  h^r(X) \in \mathcal{M}_{{\rm{mot}}}$, the 
element $\psi^r \in \CH^{2d-r}_{\AH} (X \times X)$ discussed above comes from a motivated cycle. 
Hence as above, for the case of a homogeneous motive $M$ in
$\mathcal{M}_{{\rm{ahc}}}$, we can associate with $M$ a polarization of the real Hodge structure 
associated with the rational Hodge structure on the Betti realization $V = H_B (M)$.
Consequently, Assumption 3 also holds for $\mathcal{M}_{{\rm{mot}}}$.
\medskip

Concerning polarizations for all other $\mathcal{M}_{\sim}$, we observe that if 
$\psi^r \in \CH^{2d-r}_{\AH} (X \times X)$ comes from an algebraic cycle, then we would
have the natural Hodge structure as in $\mathcal{M}_{{\rm{ahc}}}$ and 
$\mathcal{M}_{{\rm{mot}}}$.  
\medskip

Let $\overline{\mathcal{M}}_{\sim} := {\mathcal{M}_{\overline{K}, \sim}}$ and let 
\begin{equation}
D := D (M) := {{\rm End}}_{\mathcal{M}_{\overline{K}, \sim}} (\overline{M}).
\label{definition of D(M)}
\end{equation} 
Because $\overline{\mathcal{M}}_{\sim}$ is semisimple, 
$D$ is a finite dimensional  $\Q$-vector space. Because $\overline{\mathcal{M}}_{\sim}$ is Tannakian and 
$V = H_B (M)$, $D \subset \End_{\Q} (V)$. Hence in this case we obtain corresponding representation 
$\rho_e$, $K_{e} := \overline{K}^{\Ker \rho_e}$ and the 
twisted decomposable algebraic Lefschetz group $\gpDL_{K}(V, \psi, D)$ 
(see \eqref{rho e representation}, \eqref{decomposition into twisted Lefschetz for fixed elements}).
\medskip

\noindent
{\bf{ASSUMPTION 4.}} For every two objects $M, N$ of 
$\mathcal{M}_{{\rm{\sim}}}$ one has:
\begin{equation}
{\rm{Hom}}_{\overline{\mathcal{M}}_{{\rm{\sim}}}} (\overline{M}, \, \overline{N}) \subset 
{\underline{\rm{Hom}}}_{\mathcal{M}_{{\rm{\sim}}}} (M, \, N).
\label{key containment assumption}
\end{equation} 

\begin{remark}
The property \eqref{key containment assumption} holds in $\mathcal{M}_{{\rm{ahc}}}$ by 
\cite[p. 215]{DM}, \cite[p. 53]{Ja90} and in $\mathcal{M}_{{\rm{mot}}}$ 
by \cite[Scolie 2.5]{An2}, \cite[Sec. 4]{An2}. 
\label{The key containment assumption in ahc and mot}
\end{remark}

\begin{remark}
The Artin motive of principal interest in this paper is
${\rm{Hom}}_{\overline{\mathcal{M}}_{{\rm{\sim}}}} (\overline{M}, \, \overline{N})$. 
The subcategory of Artin motives in every $\mathcal{M}_{{\rm{\sim}}}$ is the same 
regardless of what $\sim$ is (cf. \cite[p. 35]{An1}). So the category of Artin motives will be always 
denoted $\mathcal{M}_{K}^{0}$. 
\label{Artin motives as subcategory of M sim}
\end{remark}

Let $M$ be a homogeneous motive which is a direct summand of $h_{{\sim}}^{r}(X)(m)$.

Let $H^{r} := H^r (X(\C),\, \Q (m))$ denote the Betti realization of $h_{{\sim}}^{r}(X)(m)$.
The polarization of the rational Hodge structure on $H^{r}$ gives the rational, polarized Hodge structure 
$(V, \psi)$, where $V = H_{B} (M)$ (see \eqref{Hodge polarization on M}).

Let $D$ be the $\Q$-algebra defined in \eqref{definition of D(M)}.

\begin{lemma} The Hodge structure $(V, \psi)$ and the $\Q$-algebra $D$
associated with the motive $M$ satisfy conditions {\bf{(D1)}}, {\bf{(D2)}} of \S \ref{Mumford--Tate groups of polarized Hodge structures}. 
\label{conditions D1, D2 are satisfied by Hodge str. associated with M}
\end{lemma}
\begin{proof}
By the assumption the category $\mathcal{M}_{{\rm{\sim}}} $ is Tannakian, hence the Betti realization 
is a fiber functor. This functor factors through $\mathcal{M}_{{\rm{ahc}}}$. Hence the condition {\bf{(D1)}} holds by \cite[Prop. 6.1]{DM}. For the condition {\bf{(D2)}} cf. \cite[Prop. 2.9]{D1}.
\end{proof}

Let $H^{r}_{_{\rm{DR}}} := H^{r}_{_{\rm{DR}}} (X) (m)$ denote the De Rham realization of 
$h_{{\sim}}^{r}(X)(m)$ and let $V_{_{\rm{DR}}}$ be the De Rham realization of $M$.
The De Rham realization of the map \eqref{motivic Hodge polarization on M} gives
the space $(V_{_{\rm{DR}}}, \psi_{_{\rm{DR}}})$.

\begin{lemma} 
The space $(V_{_{\rm{DR}}}, \psi_{_{\rm{DR}}})$ and the 
$\Q$-algebra $D$ associated with the motive $M$ satisfy conditions {\bf{(DR1)}}, {\bf{(DR2)}} of \S \ref{de Rham structures associated with Hodge structures}. 
\label{conditions DR1, DR2 are satisfied by (VDR, psiDR) associated with M}
\end{lemma}
\begin{proof}
{\bf{(DR1)}} follows because the De Rham realization of $\mathcal{M}_{{\rm{\sim}}}$ is also a fiber functor by comparison
with the Betti realization. {\bf{(DR2)}} follows by compatibility of the Hodge decompositions of $V_{_{\rm{DR}}} \otimes_{K} \C$ and
$V \otimes_{\Q} \C$ and by Lemma \ref{conditions D1, D2 are satisfied by Hodge str. associated with M}.
\end{proof}

Let $H^{r}_l := H^{r}_{et}({\overline X},\, \Q_l (m))$ be the $l$-adic realization of 
$h_{{\sim}}^{r}(X)(m)$ and let $V_l$ denote the $l$-adic realization of $M$. Consider the natural representations 
$\rho_{H^{r}_l}\colon G_{K} \rightarrow \GIso_{(H^{r}_l, \psi_l)}(\Q_l)$ and $\rho_{l}\colon 
G_K \rightarrow \GIso_{(V_l, \psi_l)}(\Q_l)$.

\begin{remark}
It is not known in general that the family of $l$-adic representations attached to a motive in the 
motivic category $\mathcal{M}_{{\rm{\sim}}}$ is strictly compatible in the sense of Serre. 
In particular it is the case in $\mathcal{M}_{{\rm{ahc}}}$ \cite[p. 475]{Pan} and in $\mathcal{M}_{{\rm{mot}}}$ 
\cite[p. 96]{An1}. However for the motive  $h^{r}_{{\sim}} (X)(m)$ the corresponding family of $l$-adic representations
is strictly compatible. For a homogeneous motive $M$ the $l$-adic realizations would be a strictly compatible family 
if, for example, the idempotents which cut $M$ out of $h^{r}_{{\sim}} (X)(m)$ come from algebraic cycles.
\label{strictly compatible family questions}
\end{remark}

\begin{lemma}
If the Hodge structure $(V, \psi)$, the $\Q$-algebra $D$, and the representation $\rho_{l}$ 
for the motive $M$ are such that the family $(\rho_{l})$ is strictly compatible, 
then conditions {\bf{(R1)}}--{\bf{(R4)}} of \S \ref{families of l-adic representations associated with Hodge structures} are satisfied. 
\label{conditions R1-R4 are satisfied by Hodge str. and associated l-adic rep. for M}
\end{lemma}
\begin{proof}
The condition {\bf{(R1)}} holds because for every prime ideal $\lambda \, | \, l$ in $\mathcal{O}_{K}$
the category of $G_{K_{\lambda}}$-Hodge--Tate modules is a Tannakian tensor 
category over $\Q_l$ \cite[p. 157]{Se79}. To check condition {\bf{(R2)}}, take a big enough
set $S$ of primes of $\mathcal{O}_K$, including primes over $l$, such that $X$ has a proper and
regular model $\mathcal{X}_{S}$ over $\mathcal{O}_{K, S}$. It follows from the smooth proper base 
change theorem \cite[Chap. VI, Cor. 4.2]{Mi1} that for every prime $v \in {\rm{Spec}} \, \mathcal{O}_{K, S}$, writing $\mathcal{X}_v$ for the special fiber over the residue field $k_v$, there is a canonical isomorphism of 
$G(\overline{k_v}/ k_v)$-modules:
$$
H^i_{et}({\overline X},\, \Q_l (m)) \simeq H^i_{et}({\overline{\mathcal{X}_v}},\, \Q_l (m)).
$$
Hence, by Deligne's proof of the Weil conjectures \cite{D3}, the Frobenius action on $W_l$ has eigenvalues of 
$\C$-absolute value $q_{v}^{-\frac{n}{2}}$. Hence the same follows for the action of Frobenius on
$V_l$. Condition {\bf{(R3)}} holds for the motive $h^{i}_{{\sim}} (X)(m)$
by Remark \ref{Hodge and Hodge--Tate representations in etale cohomology} and for the motive $M$ by construction of a motivic category, 
its realizations and filtration preserving the comparison isomorphisms between Betti, de Rham, and {\' e}tale cohomologies.  Condition {\bf{(R4)}} is obvious.
\end{proof}

\section{Motivic Galois group and motivic Serre group}
\label{motivic Galois group and motivic Serre group}

From now on we will use the following notation (cf. \cite[Chapter 9]{BK2}):
\begin{itemize}
\item[$\bullet$] $\mathcal{M}_{{\sim}} (M)$ denotes
the smallest strictly full Tannakian subcategory of  $\mathcal{M}_{{\sim}}$ containing the object $M$ of  $\mathcal{M}_{{\rm{\sim}}}$,
\item[$\bullet$] $G_{\mathcal{M}_{{\sim}}} := {\rm Aut}^{\otimes} (H_{B})$
is the motivic Galois group for  $\mathcal{M}_{{\sim}}$,
\item[$\bullet$] $G_{\mathcal{M}_{{\sim}}(M)} := {\rm Aut}^{\otimes} (H_{B} | \mathcal{M}_{{\sim}} (M))$ is the motivic Galois group for $\mathcal{M}_{{\sim}}(M)$,
\item[$\bullet$] $h^{0}(D)$ denotes the Artin motive corresponding to $D$, 
\item[$\bullet$] $\mathcal{M}_{K}^{0} (D)$ is the smallest Tannakian subcategory of 
$\mathcal{M}_{K}^{0}$ containing $h^{0}(D)$,
\item[$\bullet$] $G_{\mathcal{M}_{K}^{0} (D)} := {\rm Aut}^{\otimes} (H_{B} | \mathcal{M}_{K}^{0} (D))$ is the motivic Galois group for $\mathcal{M}_{K}^{0} (D)$.
\end{itemize}

It is well known that $\mathcal{M}_{K}^{0}$ is a semisimple Tannakian category. Hence
$\mathcal{M}_{K}^{0} (D)$, being a strictly full Tannakian subcategory, is also semisimple.
Similarly, by assumption, $\mathcal{M}_{{\sim}}$ is semisimple, so the strictly full Tannakian subcategory $\mathcal{M}_{{\sim}} (M)$ is also semisimple.
Hence the algebraic groups $G_{\mathcal{M}_{\sim} (M)}$ are reductive
but not necessarily connected (see also \cite[Prop.\ 2.23, p.\ 141]{DM}, 
\cite[Prop.\ 6.23, p.\ 214]{DM}, \cite[p.\ 379]{Se94}). 

By \cite[Prop.\ 6.1 (e), p.\ 197]{DM} the Hodge decomposition of 
$V \otimes_{\Q} \C$ is $D = D(M)$-equivariant for $M = h_{{\sim}}^r(X)$. 
Hence for any homogeneous motive $M \in \mathcal{M}_{{\sim}}$, the Betti realization $V := H_B (M)$ 
admits a Hodge decomposition of $V \otimes_{\Q} \C$ which is $D = D(M)$-equivariant (cf.
\cite[Cor. 2.12, p. 40]{PS}).  

In the same way as in \cite{BK2}, we have the following properties of the motivic categories
$\mathcal{M}_{{\sim}}$, $\mathcal{M}_{{\sim}}(M)$, and $\mathcal{M}_{K}^{0} (D)$.
By the assumption \eqref{key containment assumption}: 
\begin{equation}
h^{0}(D) \, \subset \, \underline{{\rm End}}_{\mathcal{M}_{\sim} (M)} (M).
\label{key containment assumption corollary}
\end{equation}
Hence the motive $h^{0}(D)$ splits off of $\underline{{\rm End}}_{\mathcal{M}_{\sim} (M)} (M) = 
M^{\vee} \otimes M$  in $\mathcal{M}_{\sim} (M)$. Since 
$G_{\mathcal{M}_{K}^{0}} \simeq G_K$, we observe that
$$
G_{\mathcal{M}_{K}^{0} (D)} \simeq \Gal(K_{e}/K).
$$
Hence the top horizontal and left vertical maps in the following 
Diagram \ref{surjection of motivic Galois groups} are faithfully flat
(see \cite[(2.29)]{DM}):
\begin{figure}[H]
\[
\begin{tikzcd}
G_{\mathcal{M}_{\sim}} \arrow[two heads]{d}[swap]{} \arrow[two heads]{r}{}  & G_{K} \arrow[two heads]{d}[swap]{} \\
G_{\mathcal{M}_{\sim} (M)}  \arrow[two heads]{r}{}  & \Gal(K_e/K)\\
\end{tikzcd}
\]
\\[-0.8cm]
\caption{}
\label{surjection of motivic Galois groups}
\end{figure}
In particular all homomorphisms in Diagram \ref{surjection of motivic Galois groups} are surjective. 

By \eqref{motivic Hodge polarization on M}, \eqref{Hodge polarization on M}, and the definition and properties of ${\rm Aut}^{\otimes} (H_{B} | \mathcal{M}_{\sim} (M))$ cf.\ \cite[p.\ 128--130]{DM}, we obtain:
\begin{equation}
G_{\mathcal{M}_{\sim}(M)} \subset \GIso_{(V, \psi)}. 
\label{GMKA subset of GSpV}
\end{equation}

\begin{definition}
Define the following algebraic groups:
\begin{align*}
G_{\mathcal{M}_{\sim}(M), 1}  &:= G_{\mathcal{M}_{\sim}(M)} \cap \Iso_{(V, \psi)}, \\
G_{\mathcal{M}_{\sim}(M), 1}^{0} &:= (G_{\mathcal{M}_{\sim}(M)})^{\circ} \cap 
\Iso_{(V, \psi)}.
\end{align*}
The algebraic group $G_{\mathcal{M}_{\sim}(M), 1}$ will be called 
the \emph{motivic Serre group}.
\label{motivic Galois group, motivic Serre group}
\end{definition}

\begin{definition}
\label{Definition of GMKA1tau}
For any $\tau \in \Gal(K_{e}/K)$, put
\begin{align}
G_{\mathcal{M}_{\sim}(M)}^{\tau} &:= G_{\mathcal{M}_{\sim}(M)} \, \cap \, \GIso_{(V, \psi)}^{\tau},
\label{def of GMKAtau} \\
G_{\mathcal{M}_{\sim}(M), 1}^{\tau} &:= G_{\mathcal{M}_{\sim}(M), 1} \, \cap \, \GIso_{(V, \psi)}^{\tau}.
\label{def of GMKA1tau} 
\end{align}
\end{definition}

\begin{remark}
The bottom horizontal arrow in the diagram \eqref{surjection of motivic Galois groups} is
\begin{equation}
G_{\mathcal{M}_{\sim} (M)} \rightarrow G_{\mathcal{M}_{K}^{0} (D)} \simeq \Gal(K_{e}/K).
\label{map from GMKA to Gal}\end{equation}
Let $g \in G_{\mathcal{M}_{\sim} (M)}$ and let $\tau := \tau (g)$ be the image of $g$ via 
the map \eqref{map from GMKA to Gal}. Hence for any element $\beta \in D$ considered as an 
endomorphism of $V$ we have:
\begin{equation}
g \beta g^{-1} = \rho_{e}(\tau)(\beta).
\label{GMKA1 acts on V via Galois}
\end{equation}
\end{remark}

It follows from \eqref{def of GMKAtau}, the surjectivity of \eqref{map from GMKA to Gal}, and \eqref{GMKA1 acts on V via Galois} that:
\begin{align}
G_{\mathcal{M}_{\sim}(M)} \,\, &= \bigsqcup_{\tau \in \Gal(K_{e}/K)} \, G_{\mathcal{M}_{\sim}(M)}^{\tau}
\label{GMKA decomposition into GMKAtau} \\
G_{\mathcal{M}_{\sim}(M)} / G_{\mathcal{M}_{\sim}(M)}^{{\text{id}}} \,\, &= \,\, \Gal(K_{e}/K).
\label{GMot over GMotid = GKeK}
\end{align}
Because
\begin{equation}
G_{\mathcal{M}_{\sim}(M), 1}^{\tau} = G_{\mathcal{M}_{\sim}(M), 1} \cap G_{\mathcal{M}_{\sim}(M)}^{\tau},
\label{Gmot1tau in different way}
\end{equation}
we have:
\begin{equation}
G_{\mathcal{M}_{\sim}(M), 1} = \bigsqcup_{\tau \in \Gal(K_{e}/K)} \, G_{\mathcal{M}_{\sim}(M), 1}^{\tau}.
\label{GMKA1 decomposition into GMKA1tau}
\end{equation}

The map \eqref{map from GMKA to Gal} gives the following natural map: 
\begin{equation}
G_{\mathcal{M}_{\sim} (M), 1} \rightarrow  \Gal(K_{e}/K).
\label{map from GMKA1 to Gal}
\end{equation}
Let $\tau \in \Gal(K_{e}/K)$. By \eqref{decomposable twisted Lefschetz for fixed element}, 
\eqref{decomposition into twisted Lefschetz for fixed elements}, \eqref{def of GMKA1tau}, 
\eqref{GMKA1 decomposition into GMKA1tau}, and Definition \ref{motivic Galois group, motivic Serre group} 
we obtain:
\begin{align}
G_{\mathcal{M}_{\sim}(M), 1}^{\tau} &\subset DL_{K}^{\tau}(V,\, \psi, D),
\label{GMKAtau subset DLKAtau} \\
G_{\mathcal{M}_{\sim}(M), 1} &\subset DL_{K}(V,\, \psi, D). 
\label{GMKA1 subset of DLKA}
\end{align}

\section{Motivic Mumford--Tate group and motivic Serre group}
\label{motivic Mumford--Tate group and motivic Serre group}

Because $X/K$ is smooth projective, Remarks    
\ref{Hodge--Tate representations in etale cohomology} and 
\ref{properties of Hodge--Tate representations } show that 
the $l$-adic realization $V_l := H^r (\overline{X}, \, \Q_l (m))$ of 
the motive $h_{{\sim}}^r(X) (m)$ is of Hodge--Tate type. Hence Bogomolov's theorem applies, so the image of 
the representation $\rho_l$ contains an open subgroup of homotheties of the 
group $\GL (V_l)$ \cite[Prop.\ 2.8]{Su}, provided $H^r (\overline{X}, \, \Q_l (m))$ has nonzero weight. The nonzero weight assumption is essential because for $X$ of dimension 
$d$, the $G_K$-module $H^{2d} (\overline{X}, \, \Q_l (d)) \simeq \Q_l$ is trivial, and hence
 the image of the
Galois representation does not contain nontrivial homotheties.

From now on, let $M \in \mathcal{M}_{{\sim}}$ be a homogeneous motive, that is, a direct 
summand of a motive $h^r_{{\sim}} (X)(m)$ for some $r$ and $m$. 
We assume that the $l$-adic realization of $h^{r}_{{\sim}} (X)(m)$ has nonzero weights 
with respect to the $G_K$-action. The $l$-adic realization of $\overline{M}$ is a 
$\Q_l[G_F]$-direct summand of the $l$-adic realization of $h^{r}_{{\sim}} (X)(m)$. 
Hence the $l$-adic representation corresponding to $V_l := H_l (\overline{M})$ has image 
that contains an open subgroup of homotheties. 

\begin{remark}
The $l$-adic representation
\begin{equation}
\rho_l\colon G_K \rightarrow \GL(V_l)
\label{l-adic representation associated with M}
\end{equation} 
associated with $M$ factors through $G_{\mathcal{M}_{\sim}(M)} (\Q_l)$
by the definition of $G_{\mathcal{M}_{\sim}(M)}$ and the comparison isomorphism of Betti and
{\' e}tale realizations $H_{B} (M) \otimes_{\mathbb{Q}} \, \mathbb{Q}_l \, \simeq \, H_{et}(\overline{M})$ 
(cf. \cite[p.\ 386]{Se94}).
Hence
\begin{equation}
G_{l, K}^{\alg} \subset {G_{\mathcal{M}_{{\sim}}(M)}}_{\Q_l} 
\label{GlKalg subset GMKAQl}\end{equation}
where ${G_{\mathcal{M}_{\sim}(M)}}_{\Q_l} := {G_{\mathcal{M}_{\sim}(M)}} \otimes_{\Q} \Q_l$.
\label{Remark on GlKalg subset GMKAQl}
\end{remark}

The following commutative Diagram \ref{GlK1 subset GMKA1Ql} follows from the definitions of the corresponding group schemes and 
\eqref{GlKalg subset GMKAQl}. All horizontal arrows are closed immersions and the columns are exact.

\begin{figure}[H]
\[
\begin{tikzcd}
1 \arrow{d}[swap]{} &  1 \arrow{d}[swap]{} &  1 \arrow{d}[swap]{} \\
G_{l, K, 1}^{\alg} \arrow{d}[swap]{} \arrow{r}{} & 
{G_{\mathcal{M}_{{\sim}}(M), 1}}_{\Q_l} \arrow{d}[swap]{} \arrow{r}{} & 
\Iso_{(V_l, \psi_l)} \arrow{d}[swap]{}\\
G_{l, K}^{\alg} \arrow{d}[swap]{} \arrow{r}{} & 
{G_{\mathcal{M}_{{\sim}}(M)}}_{\Q_l} \arrow{d}[swap]{} \arrow{r}{} & 
\GIso_{(V_l, \psi_l)} \arrow{d}[swap]{}\\
\G_{m} \arrow{d}[swap]{} \arrow{r}{=} & \G_{m} \arrow{d}[swap]{} \arrow{r}{=} &  
\G_{m} \arrow{d}[swap]{}\\
1  & 1  & 1\\
\end{tikzcd}
\]
\\[-0.8cm]
\caption{}
\label{GlK1 subset GMKA1Ql}
\end{figure}
In particular:
\begin{equation}
G_{l, K, 1}^{\alg} \subset {G_{\mathcal{M}_{{\sim}}(M), 1}}_{\Q_l}. 
\label{GlK1alg subset GMKA1Ql}
\end{equation}

The following theorem extends \cite[Theorem 10.2]{BK2}. 

\begin{theorem}\label{equality of conn comp for GM and GM1} The group scheme
$G_{\mathcal{M}_{{\sim}}(M), 1}$ is reductive. There is a short exact sequence of finite groups:
\begin{equation}
1 \rightarrow G^{0}_{\mathcal{M}_{{\sim}}(M), 1}/(G_{\mathcal{M}_{{\sim}}(M), 1})^{\circ} \rightarrow \pi_{0} (G_{\mathcal{M}_{{\sim}}(M), 1}) \,\, {\stackrel{i_{M}}{\longrightarrow}}  \,\, \pi_{0} (G_{\mathcal{M}_{{\sim}}(M)}) \rightarrow 1. 
\label{properties of iM}\end{equation}
In particular $G_{\mathcal{M}_{{\sim}}(M), 1}^{0}$ is 
connected if and only if $i_M$ is an isomorphism.
\end{theorem}
\begin{proof}  
We will write $\mathcal{M}(M)$ for $\mathcal{M}_{{\sim}} (M)$ in the following commutative Diagram \ref{motivic diagram in Serre theorem} to make notation simpler. 

\begin{figure}[H]
\[
\begin{tikzcd}
&&&1 \arrow{d}[swap]{} \\ 
& 1 \arrow{d}[swap]{} & 1 \arrow{d}[swap]{}  & G^{0}_{\mathcal{M}(M), 1}/(G_{\mathcal{M}(M), 1})^{\circ} 
\arrow{d}[swap]{} \\ 
1 \arrow{r}{} &  (G_{\mathcal{M}(M), 1})^{\circ} \arrow{d}[swap]{} \arrow{r}{}
&  G_{\mathcal{M}(M), 1} \arrow{d}[swap]{} \arrow{r}{}  & \pi_{0} (G_{\mathcal{M}(M), 1})   
\arrow{d}{i_{M}} \arrow{r}{} &  1\\  
1 \arrow{r}{} & (G_{\mathcal{M}(M)})^{\circ} \arrow{d}[swap]{} \arrow{r}{} & G_{\mathcal{M}(M)} 
\arrow{d}[swap]{} \arrow{r}{} & \pi_{0}(G_{\mathcal{M}(M)}) \arrow{d}[swap]{} \arrow{r}{} &  1 \\
1 \arrow{r}{} & \G_m  \arrow{d}[swap]{} \arrow{r}{=} & \G_m  \arrow{d}[swap]{} \arrow{r}{} & 1\\
& 1  & 1 \\
\end{tikzcd}
\]
\\[-0.8cm]
\caption{}
\label{motivic diagram in Serre theorem}
\end{figure}

By our assumptions $\mathcal{M}_{{\sim}}$ is semisimple. The category $\mathcal{M}_{{\sim}}(M)$ is a strictly full subcategory of $\mathcal{M}_{{\sim}}$ because by Tannaka duality 
(Theorem~\ref{Tannaka duality theorem}), these categories are equivalent to the corresponding categories of representations of $G_{\mathcal{M}_{{\sim}}}$ and $G_{\mathcal{M}_{{\sim}}(M)}$. 
Hence $\mathcal{M}_{{\sim}}(M)$ is semisimple and consequently $G_{\mathcal{M}_{{\sim}}(M)}$ is reductive. 
The exactness of the middle vertical column of Diagram \ref{motivic diagram in Serre theorem}
follows by the definition of $G_{\mathcal{M}_{{\sim}}(M), 1}$ and the exactness of the middle column in 
Diagram \ref{GlK1 subset GMKA1Ql}. This shows that $G_{\mathcal{M}_{{\sim}}(M), 1}$ is also reductive. Observe that the left vertical column is not \emph{a priori}
exact at the term $(G_{\mathcal{M}(M)})^{\circ}$.
By definition the rows of  Diagram \ref{motivic diagram in Serre theorem} are exact. 
Hence the map $i_M$ is surjective. By the definition of 
$G_{\mathcal{M}_{{\sim}}(M), 1}^{0}$ (see Definition
\ref{motivic Galois group, motivic Serre group}) the right vertical column is also exact. Hence 
\eqref{properties of iM} is exact. Because $G_{\mathcal{M}_{{\sim}}(M), 1}^{0}$ has the 
same dimension as $G_{\mathcal{M}_{{\sim}}(M), 1}$, we obtain by \eqref{properties of iM} that 
$G_{\mathcal{M}_{{\sim}}(M), 1}^{0}$ is connected if and only if
$G_{\mathcal{M}_{{\sim}}(M), 1}^{0} = (G_{\mathcal{M}_{{\sim}}(M), 1})^{\circ}$ if and only if 
$i_M$ is an isomorphism if and only if the left column of  Diagram \ref{motivic diagram in Serre theorem} 
is exact.
\end{proof}

\begin{corollary}\label{equality of quotients concerning for GM and GM1} There are natural isomorphisms:
\begin{align}
G_{\mathcal{M}_{{\sim}}(M), 1}/\,  G_{\mathcal{M}_{{\sim}}(M), 1}^{\id} \,\, 
&{\stackrel{\simeq}{\longrightarrow}} \,\, 
G_{\mathcal{M}_{{\sim}}(M)} /\, G_{\mathcal{M}_{{\sim}}(M)}^{\id}
\label{GMKA1 mod GMKA1Id equals GMKA mod GMKAId} \\
G_{\mathcal{M}_{{\sim}}(M), 1}/\, G_{\mathcal{M}_{{\sim}}(M), 1}^{\id} \,\, 
&{\stackrel{\simeq}{\longrightarrow}} \,\, \gpDL_{{\sim}} (V, \psi, D) /\, \gpDL_{{\sim}}^{\id}(V, \psi, D)  
\,\, {\stackrel{\simeq}{\longrightarrow}} \,\, \Gal(K_e/K).
\label{GMKA1 mod GMKA1Id equals DLKA mod DLKAId equals GLeK} 
\end{align}
In particular the natural map \eqref{map from GMKA1 to Gal} is surjective.
\end{corollary}
\begin{proof} The isomorphism \eqref{GMKA1 mod GMKA1Id equals GMKA mod GMKAId} follows from the
surjectivity of $i_M$ in \eqref{properties of iM} of Theorem 
\ref{equality of conn comp for GM and GM1} and from \eqref{Gmot1tau in different way}.
The isomorphism \eqref{GMKA1 mod GMKA1Id equals DLKA mod DLKAId equals GLeK} follows from
\eqref{Twisted decomposable embeds into GKeK}, \eqref{GMot over GMotid = GKeK}, \eqref{GMKAtau subset DLKAtau}, 
\eqref{GMKA1 subset of DLKA}, and \eqref{GMKA1 mod GMKA1Id equals GMKA mod GMKAId}.
\end{proof}

\begin{corollary}\label{equality of quotients concerning for GMid and GM1id} 
The group $G_{\mathcal{M}_{{\sim}}(M), 1}^{0}$ is connected if and only if
\begin{equation}
G_{\mathcal{M}_{{\sim}}(M), 1}^{\id}/\, (G_{\mathcal{M}_{{\sim}}(M), 1})^{\circ} \,\,
{\stackrel{\simeq}{\longrightarrow}} \,\,
G_{\mathcal{M}_{{\sim}}(M)}^{\id}/\, (G_{\mathcal{M}_{{\sim}}(M)})^{\circ}. 
\label{GMKA1Id mod GMKA1circ equals GMKAId mod GMKAcirc}
\end{equation}
\end{corollary}
\begin{proof} This follows from Theorem 
\ref{equality of conn comp for GM and GM1}
and \eqref{GMKA1 mod GMKA1Id equals GMKA mod GMKAId}.
\end{proof}

\begin{lemma} We have the following inclusions:
\label{MT subset GMot, H subset GMot1}
\begin{align}
\MT(V, \psi) &\subset (G_{\mathcal{M}_{\sim} (M)})^{\circ},
\label{MT(V) subset of Gmot} \\
\gpDH(V, \psi) &\subset \, {G^{0}_{\mathcal{M}_{{\sim}}(M), 1}},
\label{DH(V) subset of Gmot,0,1} \\
\gpH(V, \psi) &\subset (G_{\mathcal{M}_{\sim} (M), 1})^{\circ}.
\label{H(V) subset of Gmot,1}
\end{align}
\end{lemma}
\begin{proof} 
Algebraic cycles and motivated cycles are 
absolute Hodge cycles in cohomological realization \cite[Example 2.1 (a)]{D1}, \cite[Remarque 9.2.2.1]{An1}. 
Hence these three types of 
cycles are Hodge cycles in cohomological realization. Moreover both groups in 
\eqref{MT(V) subset of Gmot} are reductive. Now applying the ``tensor invariance criterion" for reductive groups, we obtain 
\eqref{MT(V) subset of Gmot}. Moreover by \eqref{MT(V) subset of Gmot} we obtain the following commutative 
Diagram \ref{DH(V, psi) subset Gmot,1} in which all horizontal arrows are closed immersions, the columns are exact,
and the top commutative
squares are cartesian.

\begin{figure}[H]
\[
\begin{tikzcd}
1 \arrow{d}[swap]{} &  1 \arrow{d}[swap]{} &  1 \arrow{d}[swap]{} \\
\gpDH(V, \psi)  \arrow{d}[swap]{} \arrow{r}{} & 
{G^{0}_{\mathcal{M}_{{\sim}}(M), 1}}  \arrow{d}[swap]{} \arrow{r}{} & 
\Iso_{(V, \psi)} \arrow{d}[swap]{}\\
\MT(V, \psi) \arrow{d}{\chi} \arrow{r}{} & 
({G_{\mathcal{M}_{{\sim}}(M)}})^{\circ} \arrow{d}{\chi} \arrow{r}{} & 
\GIso_{(V, \psi)} \arrow{d}{\chi}\\
\G_{m} \arrow{d}[swap]{} \arrow{r}{=} & \G_{m} \arrow{d}[swap]{} \arrow{r}{=} &  
\G_{m} \arrow{d}[swap]{}\\
1  & 1  & 1\\
\end{tikzcd}
\]
\\[-0.8cm]
\caption{}
\label{DH(V, psi) subset Gmot,1}
\end{figure}
The property \eqref{H(V) subset of Gmot,1} follows immediately from Diagram \ref{DH(V, psi) subset Gmot,1}.
\end{proof}
 
The following Proposition is the motivic analogue of Proposition 
\ref{Hodge realizing conn comp: even case}. 

\begin{proposition} \label{Motivic L0realizing conn comp: even case} 
The following statements hold:
\begin{itemize}
\item[(a)] $\G_{m} {\rm{Id}}_{V} \cdot (G_{\mathcal{M}_{{\sim}}(M), 1})^{\circ} \, = \, 
\G_{m} {\rm{Id}}_{V} \cdot G_{\mathcal{M}_{{\sim}}(M), 1}^{0} \, = 
\, (G_{\mathcal{M}_{{\sim}}(M)})^{\circ}$.
\item[(b)] $G_{\mathcal{M}_{{\sim}}(M), 1}^{0} = (G_{\mathcal{M}_{{\sim}}(M), 1})^{\circ}  \, \cup \, 
- {\rm{Id}}_{V} \cdot (G_{\mathcal{M}_{{\sim}}(M), 1})^{\circ} $.
\item[(c)] $- {\rm{Id}}_{V} \in (G_{\mathcal{M}_{{\sim}}(M), 1})^{\circ}$ iff 
$(G_{\mathcal{M}_{{\sim}}(M), 1})^{\circ} = G_{\mathcal{M}_{{\sim}}(M), 1}^{0} $.
\item[(d)] When $n$ is odd then $- {\rm{Id}}_{V} \, \in (G_{\mathcal{M}_{{\sim}}(M), 1})^{\circ} = G_{\mathcal{M}_{{\sim}}(M), 1}^{0}$.  
\end{itemize} 
\end{proposition} 
\begin{proof}

(a)  Consider a coset $g (G_{\mathcal{M}_{{\sim}}(M), 1})^{\circ}$ in 
$G_{\mathcal{M}_{{\sim}}(M), 1}^{0}$. Applying \cite[Section 7.4, Prop. B(b)]{Hu} to the homomorphism
\begin{gather*}
\G_{m} {\rm{Id}}_{V} \, \times \, (G_{\mathcal{M}_{{\sim}}(M), 1})^{\circ} \rightarrow 
(G_{\mathcal{M}_{{\sim}}(M)})^{\circ}, \\
(g_1, g_2) \, \mapsto \, g_1 g_2
\end{gather*}
we observe that  $\G_{m} {\rm{Id}}_{V} \cdot (G_{\mathcal{M}_{{\sim}}(M), 1})^{\circ}$ and $\G_{m} {\rm{Id}}_{V} \cdot g (G_{\mathcal{M}_{{\sim}}(M), 1})^{\circ}$ are closed in $(G_{\mathcal{M}_{{\sim}}(M)})^{\circ}$. They are also of the same dimension as $(G_{\mathcal{M}_{{\sim}}(M)})^{\circ}$ because of the exact sequence
\begin{equation}
1 \,\, {\stackrel{}{\longrightarrow}} \,\, G_{\mathcal{M}_{{\sim}}(M), 1}^{0} \,\, {\stackrel{}{\longrightarrow}} \,\, 
(G_{\mathcal{M}_{{\sim}}(M)})^{\circ} \,\, {\stackrel{\chi}{\longrightarrow}} \,\, \G_m \,\, {\stackrel{}{\longrightarrow}} \,\, 1.
\label{The exact sequence for GM0-sim1 and (GM-sim)0}
\end{equation}
and the finiteness of $G^{0}_{\mathcal{M}_{{\sim}}(M), 1}/(G_{\mathcal{M}_{{\sim}}(M), 1})^{\circ}$.

Because $(G_{\mathcal{M}_{{\sim}}(M)})^{\circ}$ is an irreducible algebraic group and $\G_{m} {\rm{Id}}_{V} \cdot (G_{\mathcal{M}_{{\sim}}(M), 1})^{\circ}$ is a closed subgroup of the same dimension, we must have 
\begin{equation}
\G_{m} {\rm{Id}}_{V} \cdot (G_{\mathcal{M}_{{\sim}}(M), 1})^{\circ} \, = 
\,\G_{m} {\rm{Id}}_{V} \cdot  g\, (G_{\mathcal{M}_{{\sim}}(M), 1})^{\circ} = 
(G_{\mathcal{M}_{{\sim}}(M)})^{\circ}.
\label{equality of Gm IdV . GMsimM1circ and  GMsimM1circ}
\end{equation}

\noindent
(b) From \eqref{equality of Gm IdV . GMsimM1circ and  GMsimM1circ} there are $\alpha, \beta \in \G_{m}$ and $g_1, g_2 \in (G_{\mathcal{M}_{{\sim}}(M), 1})^{\circ}$ such that
\begin{equation}
\alpha \, {\rm{Id}}_{V} \cdot g_1 = \beta \, {\rm{Id}}_{V} \cdot g g_2. 
\label{equality of coset generators for motivic Galois groups}
\end{equation}
Applying $\chi$ to \eqref{equality of coset generators for motivic Galois groups} we obtain 
$\alpha^2 = \beta^2$. This implies that $\pm {\rm{Id}}_{V} \cdot g_1 =  g g_2$. Hence
$g (G_{\mathcal{M}_{{\sim}}(M), 1})^{\circ} = (G_{\mathcal{M}_{{\sim}}(M), 1})^{\circ}$ or 
$g (G_{\mathcal{M}_{{\sim}}(M), 1})^{\circ} = 
- {\rm{Id}}_{V} (G_{\mathcal{M}_{{\sim}}(M), 1})^{\circ}$. 
\medskip

\noindent
(c) This follows immediately from (b).
\medskip

\noindent
(d) This follows immediately from (c) and the following Lemma \ref{for n odd G0M1 = GM10}.
\end{proof}

\begin{lemma}
\label{for n odd G0M1 = GM10}
 For odd weight $n$ we have:
$$
G_{\mathcal{M}_{{\sim}}(M), 1}^{0} = (G_{\mathcal{M}_{{\sim}}(M), 1})^{\circ}.
$$
\end{lemma}
\begin{proof}
The proof is basically the same as the proof of Lemma \ref{for n odd DH = H}. We explain 
the modification. Extending the codomain of $\mu_{\infty, V}$ 
and $w$ of the proof of Lemma \ref{for n odd DH = H} using \eqref{MT(V) subset of Gmot}, we obtain the following corresponding cocharacters:  
\begin{gather*}
\mu_{\infty, V}\colon \G_m (\C) \rightarrow (G_{\mathcal{M}_{{\sim}}(M)})^{\circ}, \\
w\colon \G_m (\C) \rightarrow (G_{\mathcal{M}_{{\sim}}(M)})^{\circ}.
\end{gather*}
These cocharacters have precisely the same properties as $\mu_{\infty, V}$ 
and $w$ of Lemma \ref{for n odd DH = H}, as is clear from the bottom squares of Diagram
\ref{DH(V, psi) subset Gmot,1}. This leads directly to the cocharacter
\begin{gather*}
s\colon \G_m (\C) \rightarrow  (G_{\mathcal{M}_{{\sim}}(M)})^{\circ}, \\
s(z) := \mu_{\infty, V}(z) \, w(z)^{-\frac{n-1}{2}},
\end{gather*}
which splits $\chi$ in the following exact sequence:
$$
1 \,\, {\stackrel{}{\longrightarrow}} \,\, G_{\mathcal{M}_{{\sim}}(M), 1}^{0} (\C) \,\, {\stackrel{}{\longrightarrow}} \,\,  
(G_{\mathcal{M}_{{\sim}}(M)})^{\circ}(\C) \,\, {\stackrel{\chi}{\longrightarrow}} \,\, \G_m (\C) \,\, {\stackrel{}{\longrightarrow}} \,\, 1.
$$
The group $(G_{\mathcal{M}_{{\sim}}(M)})^{\circ}$ is obviously connected. 
It follows by Lemma \ref{G connected implies G0 connected} that $G_{\mathcal{M}_{{\sim}}(M), 1}^{0}$ is connected. Hence 
$G_{\mathcal{M}_{{\sim}}(M), 1}^{0} = (G_{\mathcal{M}_{{\sim}}(M), 1})^{\circ}$
because $(G^{0}_{\mathcal{M}_{{\sim}}(M), 1})^{\circ} = 
(G_{\mathcal{M}_{{\sim}}(M), 1})^{\circ}$ by Theorem \ref{equality of conn comp for GM and GM1}.
\end{proof}

\begin{definition}
\label{def of motivic Mumford Tate and Sato Tate groups}
The algebraic groups:
\begin{align*}
\MMT_{{\sim}} (M) & := G_{\mathcal{M}_{{\sim}}(M)} \\
\MS_{{\sim}} (M) & := G_{\mathcal{M}_{{\sim}}(M), 1}
\end{align*}
will be called the \emph{motivic Mumford--Tate group} and (as before) the 
\emph{motivic Serre group} for $M$ respectively.
\end{definition} 

Serre made the following conjecture \cite[sec.\ 3.4]{Se94}.
\begin{conjecture} {\text{(Serre)}}
\label{Conj. of Serre MT = MMT0}
\begin{equation} 
\MT (V, \psi) = \MMT_{{\sim}} (M)^{\circ}.
\label{Conjecture of Serre MT = MMT0}
\end{equation}
\end{conjecture}
\begin{remark}
\label{equivalent statements of the Serre Conj. MT = MMT0}
By Proposition \ref{Hodge realizing conn comp: even case}(b), Lemma \ref{MT subset GMot, H subset GMot1},
and Proposition \ref{Motivic L0realizing conn comp: even case}(b) and the diagram above it,
 Conjecture \ref{Conj. of Serre MT = MMT0} holds for $M$ if and only if 
$$
\gpDH(V, \psi) = G_{\mathcal{M}_{{\sim}}(M), 1}^{0},
$$
if and only if 
$$
\gpH(V, \psi) = \MS_{{\sim}} (M)^{\circ}.
$$
\end{remark}

\begin{definition} A motive $M \in \mathcal{M}_{{\text{ahc}}}$ will be called an \emph{AHC motive}
if every Hodge cycle on any object of $\mathcal{M}_{{\text{ahc}}}(M)$ is an absolute Hodge cycle
(cf. \cite[p.\ 29]{D1}, \cite[p.\ 473]{Pan}).
\label{AHC motives}
\end{definition}

\begin{remark}
\label{A Serre conjecture about connected component of motivic Mumford--Tate}
By \cite{DM} Conjecture \ref{Conj. of Serre MT = MMT0} holds 
for motives $h^{1}_{{\text{ahc}}}(A)$
where $A/K$ is an abelian variety and 
for AHC motives $M$ in $\mathcal{M}_{{\text{ahc}}}$ (cf. \cite[Corollary p.\ 474]{Pan}).
\end{remark}

Similarly to \cite[Chapter 10]{BK2} we make the following definitions and state the 
following conjectures in this more general context of motivic categories.

\begin{conjecture} (Motivic Mumford--Tate) \label{Motivic Mumford--Tate} 
For any prime number $l$,
\begin{equation}
G_{l, K}^{\alg} = {\MMT_{{\sim}} (M)}_{\Q_l}.
\label{MMT eq}
\end{equation}
\end{conjecture}

By the diagram above Theorem \ref{equality of conn comp for GM and GM1}, 
Conjecture \ref{Motivic Mumford--Tate} is equivalent to the following.
\begin{conjecture} (Motivic Sato--Tate) \label{Motivic Sato--Tate} 
For any prime number $l$,
\begin{equation}
G_{l, K, 1}^{\alg} = \MS_{{\sim}}(M)_{\Q_l} .
\label{MST eq}\end{equation}
\end{conjecture}

\begin{remark}
Conjecture \ref{Motivic Mumford--Tate} is equivalent to the conjunction of the following  
equalities:
\begin{align}
(G_{l, K}^{\alg})^{\circ} &= ({\MMT_{{\sim}} (M)}_{\Q_l})^{\circ}, 
\label{MMT0 eq} \\
\pi_{0} (G_{l, K}^{\alg}) &= \pi_{0} ({\MMT_{{\sim}} (M)}_{\Q_l}).
\label{GMT0C eq}
\end{align}
Similarly, Conjecture \ref{Motivic Sato--Tate} is equivalent to the conjunction of the following equalities:
\begin{align}
(G_{l, K, 1}^{\alg})^{\circ} &= ({\MS_{{\sim}} (M)}_{\Q_l})^{\circ}, 
\label{MST0 eq} \\
\pi_{0} (G_{l, K, 1}^{\alg}) &= \pi_{0} ({\MS_{{\sim}} (M)}_{\Q_l}).
\label{MST0C eq}
\end{align}
\end{remark}

\begin{definition} Let $M$ be a motive in $\mathcal{M}_{{\sim}}$. The {\bf{Serre motivic parity group}}:  
\begin{equation}
\gpP_{S} (M, \psi_{{\sim}}) :=  G_{\mathcal{M}_{{\sim}}(M), 1}^{0} / (G_{\mathcal{M}_{{\sim}}(M), 1})^{\circ}
\end{equation}
is the component group of $G_{\mathcal{M}_{{\sim}}(M), 1}^{0}$.
\label{The motivic parity group}
\end{definition}

\begin{proposition}
If $n$ is odd or if $(G_{\mathcal{M}_{{\sim}}(M), 1})^{\circ} = C_{D} (\Iso_{(V, \psi)})$ 
(in particular if  $\gpH(V, \psi) = C_{D} (\Iso_{(V, \psi)})$) then $\gpP_{S} (M, \psi_{{\sim}})$ 
is trivial. 
\label{P(V, psi-sim) trivial for n odd or H = CD}
\end{proposition}
\begin{proof}
This follows by Lemma \ref{for n odd G0M1 = GM10} and the following computation (apply 
\eqref{GMKAtau subset DLKAtau}  and Lemma \ref{MT subset GMot, H subset GMot1}): 
\begin{gather*}
\gpH(V, \psi) \subset (G_{\mathcal{M}_{{\sim}}(M), 1})^{\circ} \subset
G_{\mathcal{M}_{{\sim}}(M), 1}^{0} \subset G_{\mathcal{M}_{{\sim}}(M), 1}^{\id}
\subset C_{D} (\Iso_{(V, \psi)}). 
\qedhere
\end{gather*}
\end{proof}

\begin{proposition}
\label{P(M, psi-sim) nontrivial for n even and dim V odd}
If $n$ is even and ${\rm dim}_{\Q}\, V$ is odd, then
$\gpP_{S} (M, \psi_{{\sim}})$ is nontrivial. 
\end{proposition}
\begin{proof} Again because $n$ is even, $\Iso_{(V, \psi)} = \gpO_{(V, \psi)}$. Hence (cf. Definition 
\ref{motivic Galois group, motivic Serre group}) we obtain
$(G_{\mathcal{M}_{{\sim}}(M), 1})^{\circ} \subset \Iso_{(V, \psi)} \subset {\pm {\rm{Id}}_{V}} \SL_{V}$.  
It follows that $- {\rm{Id}}_{V} \notin (G_{\mathcal{M}_{{\sim}}(M), 1})^{\circ}$ because 
${\rm{det}} (- {\rm{Id}}_{V}) = (-1)^{{\rm{dim}}_{\Q}\, V} = -1$ and $(G_{\mathcal{M}_{{\sim}}(M), 1})^{\circ}$ 
is connected.
\end{proof}

\begin{example} Let $A$ be an abelian threefold over $K$. The motive $h_{{\sim}}^{2}(A)$ has Betti realization 
$V :=  H^{2} (A(\C), \, \Q) = \bigwedge^{2} \, H^{1} (A(\C), \, \Q)$. 
Hence by Example \ref{An example of nontrivial Betti parity group} and 
Proposition \ref{P(M, psi-sim) nontrivial for n even and dim V odd}, the group 
$\gpP_{S} (h_{{\sim}}^{2}(A), \, \psi_{{\sim}})$ is nontrivial. 
\label{An example of nontrivial motivic parity group}
\end{example}

\begin{example} Let $A$ be an elliptic curve over $K$. The motive 
$\Sym^{2} \, h_{{\sim}}^{1}(A)$ has Betti realization $V := \Sym^{2} \, H^{1} (A(\C), \, \Q)$. 
Hence by Example \ref{Another example of nontrivial Betti parity group} and Proposition
\ref{P(M, psi-sim) nontrivial for n even and dim V odd}, the group  
$\gpP_{S} (\Sym^{2} \, h_{{\sim}}^{1}(A), \, \psi_{{\sim}})$ is nontrivial.  
\label{Another example of nontrivial motivic parity group}
\end{example}

Let $M$ be a homogeneous motive in $\mathcal{M}_{{\sim}}$.
By \eqref{DH(V) subset of Gmot,0,1} and \eqref{H(V) subset of Gmot,1} there is a natural epimorphism:
\begin{equation}
\gpP(V, \psi) \rightarrow \gpP_{S} (M, \psi_{{\sim}}).
\label{morphism between Betti and motivic parity groups}
\end{equation}
By Corollary \ref{Corollary-epimorphism between DR parity and Betti parity groups} 
we have a natural epimorphism:
$$
\gpP(V_{_{\rm{DR}}}, \psi_{_{\rm{DR}}}) \rightarrow \gpP(V, \psi).
$$
Hence by \eqref{morphism between Betti and motivic parity groups} we have a natural epimorphism:
\begin{equation}
\gpP(V_{_{\rm{DR}}}, \psi_{_{\rm{DR}}}) \rightarrow \gpP_{S} (M, \psi_{{\sim}}).
\label{epimorphism between DR parity and motivic parity groups}
\end{equation}
By \eqref{GlKalg subset GMKAQl} we have:
$$
(G_{l, K}^{\alg})^{\circ} \, \subset \, ({G_{\mathcal{M}_{{\sim}}(M)}}_{\Q_l})^{\circ}.
$$
Hence by \eqref{GlKalg subset GMKAQl} we obtain:
\begin{gather*}
G_{l, K, 1}^{\alg, \, 0} = (G_{l, K}^{\alg})^{\circ}  \cap G_{l, K, 1}^{\alg}  \subset 
({G_{\mathcal{M}_{{\sim}}(M)}}_{\Q_l})^{\circ} \cap \Iso_{(V_{l}, \psi_{l})} =  \\
 (({G_{\mathcal{M}_{{\sim}}(M)}})^{\circ})_{\Q_l} \cap {\Iso_{(V, \psi)}}_{\Q_l} = 
{G_{\mathcal{M}_{{\sim}}(M), 1}^{0}}_{\Q_l}. 
\end{gather*}
It follows by \eqref{GlK1alg subset GMKA1Ql} that there is a natural epimorphism:
\begin{equation}
\gpP_{\rm{S}} (V_l, \psi_l) \rightarrow \gpP_{S} (M, \psi_{{\sim}}).
\label{morphism between l-adic and motivic parity groups}
\end{equation}

\begin{remark} Recall that each parity group has order at most 2. 
These groups are generated by corresponding cosets of $- {\rm{Id}}$.  Moreover each homomorphism between parity groups
carries the coset of $- {\rm{Id}}$ in the source into the coset of $- {\rm{Id}}$ in the target. Hence it is clear that   
each homomorphism between parity groups is an epimorphism.
\label{Each hom. between parity gr. carries the coset of - Id in the source into the coset of - Id in the target.}
\end{remark}

\begin{corollary}
If $\gpP_{S} (M, \psi_{{\sim}})$ is nontrivial then it is naturally isomorphic to each of the following parity 
groups: $\gpP(V, \psi)$, $\gpP(V_{_{\rm{DR}}}, \psi_{_{\rm{DR}}})$, and $\gpP_{\rm{S}} (V_l, \psi_l)$.
 \label{If mot. parity gr. nontrivial then all other parity gr. nontrivial}
\end{corollary}
\begin{proof}
Every parity group is either trivial or isomorphic to $\{\pm 1\}$. Hence the corollary follows by epimorphisms 
\eqref{morphism between Betti and motivic parity groups}, 
\eqref{epimorphism between DR parity and Betti parity groups}, 
\eqref{epimorphism between DR parity and motivic parity groups}, \eqref{morphism between l-adic and motivic parity groups}.
\end{proof}

\begin{proposition} 
Consider parity groups:
$\gpP(V, \psi) , \gpP(V_{_{\rm{DR}}}, \psi_{_{\rm{DR}}}), \gpP_{\rm{S}} (V_l, \psi_l)$ and
$\gpP_{S} (M, \psi_{{\sim}})$.
\begin{itemize}
\item[(a)] If $n$ is odd, then all these groups are trivial.
\item[(b)] If $n$ is even and $\dim_{\Q}\, V$ is odd, then all these groups are nontrivial.
\end{itemize}
\label{general conditions for all parity groups to be trivial and nontrivial}
\end{proposition}
\begin{proof}

(a) This follows from Propositions
\ref{P(V, psi) trivial for n odd or H = CD}, \ref{P(V-DR, psi) trivial for n odd or H-DR = CD},
\ref{P(Vl, psil) trivial for n odd or GlK01 = CDl}, \ref{P(V, psi-sim) trivial for n odd or H = CD}.

(b) This follows from Propositions
\ref{P(V, psi) nontrivial for n even and dim V odd}, \ref{P(Vl, psil) nontrivial for n even and dim V odd}, \ref{P(M, psi-sim) nontrivial for n even and dim V odd}, and Corollary 
\ref{gpP(VDR, psiDR) nontrivial for n even and dim V odd}.
\end{proof}

\section{The algebraic Sato--Tate group for motives}
\label{algebraic Sato--Tate group for motives}

As in the previous section,
we work with motives $M$ which are direct summands of motives of the form 
$h^r (X)(m)$;
in this section, we propose a candidate for the algebraic Sato--Tate group 
for such motives. We prove, under the assumption in Definition
\ref{GMKA1 as AST group}, that our candidate for 
the algebraic Sato--Tate group is the expected one.  
In particular the assumption of Definition \ref{GMKA1 as AST group} holds
if $M$ is an AHC motive (see Definition \ref{AHC motives} and Remark 
\ref{A Serre conjecture about connected component of motivic Mumford--Tate}). 
 
Observe that $C_D (\GIso_{(V, \psi)}) = \GIso_{(V, \psi)}^{{\text{id}}}$ and 
$C_D (\Iso_{(V, \psi)}) = \gpDL_{{\sim}}^{\id} (V, \psi, D)$. It follows by Lemma 
\ref{MT subset GMot, H subset GMot1}, \eqref{def of GMKAtau}, and 
\eqref{GMKAtau subset DLKAtau}
that:
\begin{gather}
\MT (V, \psi) \subseteq (G_{\mathcal{M}_{{\sim}}(M)})^{\circ} \subseteq G_{\mathcal{M}_{{\sim}}(M)}^{\id} \subseteq 
C_D (\GIso_{(V, \psi)}), 
\label{MTA subset GMKA0 subset CDGSp} \\
\gpH (V, \psi) \subseteq  (G_{\mathcal{M}_{{\sim}}(M), 1})^{\circ} \subseteq 
G_{\mathcal{M}_{{\sim}}(M), 1}^{\id} \subseteq C_D (\Iso_{(V, \psi)}). 
\label{HA subset GMKA10 subset CDSp}
\end{gather}

\begin{remark}
Assume that $G_{\mathcal{M}_{{\sim}}(M), 1}^{0}$ is connected. By \eqref{CD GIso = Gm CD Iso},
Corollary \ref{equality of quotients concerning for GMid and GM1id}, Proposition
\ref{Motivic L0realizing conn comp: even case}(a), and 
Remark \ref{equivalent statements of the Serre Conj. MT = MMT0} we observe that the equality
\begin{equation}
\MT (V, \psi) = C_D (\GIso_{(V, \psi)})
\label{MTA = CD GSp for BGK classes}
\end{equation}
is equivalent to the equality:
\begin{equation}
\gpH (V, \psi) = C_D (\Iso_{(V, \psi)}).
\label{HA = CD Sp for BGK classes}
\end{equation}
\end{remark}

\begin{remark} In \cite[p.\ 380]{Se94} there are examples of the computation of  
$\MMT_{{\sim}} (M) = G_{\mathcal{M}_{{\sim}}(M)}$. 
In \cite[Theorems 7.3, 7.4]{BK1}, we computed $\MMT_{{\text{ahc}}} (M)$ for abelian varieties
of dimension $\leq 3$ and families of abelian varieties of type I, II, and III
in the Albert classification. 
\label{Examples of computation of motivic Mumford--Tate}
\end{remark}

\begin{remark}
If Serre's Conjecture \ref{Conj. of Serre MT = MMT0} holds for $M$,
then by \eqref{GlKalg subset GMKAQl}, \eqref{GlK1alg subset GMKA1Ql}, and Remark
\ref{equivalent statements of the Serre Conj. MT = MMT0} we would have:
\begin{align}
(G_{l, K}^{\alg})^{\circ} &\subset \MT (V, \psi)_{\Q_l},
\label{Assumption analogues to Deligne Theorem 1} \\
(G_{l, K, 1}^{\alg})^{\circ} &\subset \gpH (V, \psi)_{\Q_l}.
\label{Assumption analogues to Deligne Theorem 2}
\end{align}
In particular \eqref{Assumption analogues to Deligne Theorem 1} and 
\eqref{Assumption analogues to Deligne Theorem 2} hold for AHC motives in 
$\mathcal{M}_{{\text{ahc}}}$ (cf.\ Remark 
\ref{A Serre conjecture about connected component of motivic Mumford--Tate}).
\label{Glalg subset of MT}
\end{remark}

Following Serre  \cite[sec. 13, p. 396--397]{Se94} we make the following 
definition.
\begin{definition} 
The \emph{algebraic Sato--Tate group} $\AST_{K} (M)$ is defined as follows:
\begin{equation}
\AST_{K} (M) := \MS_{{\sim}} (M). 
\label{AST group}
\end{equation}
Every maximal compact subgroup of $\AST_{K} (M)(\C)$ will be called
a \emph{Sato--Tate group} associated with $M$ and denoted $\ST_K(M)$.
\label{GMKA1 as AST group}
\end{definition}

\begin{remark}
In \cite[Def. 11.7]{BK2} we defined the algebraic Sato--Tate group for motives in 
$\mathcal{M}_{{\text{ahc}}}$ for which the Serre Conjecture  \ref{Conj. of Serre MT = MMT0} 
holds. In this paper we decided not to assume  Conjecture  \ref{Conj. of Serre MT = MMT0} in Definition  
\ref{GMKA1 as AST group}. 
\label{The difference between the Def. 11.7 BK2 and the present definition}
\end{remark}

\begin{remark} 
Theorem \ref{General properties of AST} and Corollaries \ref{The natural candidate for AST group example 1}
and \ref{cor alg Sato--Tate for MT explained by endo} below show that $\AST_{K} (M)$ is a natural candidate for the algebraic Sato--Tate group for the homogeneous motive $M$. 
\end{remark}

\begin{corollary}\label{First properties of AST group}
The following sequence is exact:
$$
1 \, \rightarrow \, \pi_{0} (\AST_{K} (M) \cap \MMT_{{\sim}} (M)^{\circ}) \, 
\rightarrow \,  \pi_{0} (\AST_{K} (M)) \, {\stackrel{i_M}{\longrightarrow}} \, \pi_{0} ( \MMT_{{\sim}} (M)) \, \rightarrow \, 1.
$$
Moreover:
\begin{gather}
\pi_{0} (\AST_{K} (M) \cap \MMT_{{\sim}} (M)^{\circ}) \subseteq \Z/2\Z,
\label{Ker iM subset Z mod 2} \\
\pi_{0} (\AST_{K} (M) \cap \MMT_{{\sim}} (M)^{\circ}) = \{1\} \,\,\,\, \text{iff} \,\,\,
- {\rm{Id}}_{V} \in \AST_{K} (M)^{\circ}.
\label{Ker iM isom Z mod 2 iff iM isom}
\end{gather}
\end{corollary}
\begin{proof} The sequence in this corollary is just the exact sequence \eqref{properties of iM} with
$$
\pi_{0} (\AST_{K} (M)) \cap \MMT_{{\sim}} (M)^{\circ}) = G^{0}_{\mathcal{M}_{{\sim}}(M), 1}/(G_{\mathcal{M}_{{\sim}}(M), 1})^{\circ}.
$$
Hence \eqref{Ker iM subset Z mod 2} (resp. \eqref{Ker iM isom Z mod 2 iff iM isom}) follows from
Proposition \ref{Motivic L0realizing conn comp: even case}(b) (resp. 
Proposition \ref{Motivic L0realizing conn comp: even case}(c)).
\end{proof}

\begin{theorem}\label{General properties of AST} $\AST_{K} (M)$ is reductive and has the following properties: 
\begin{align} 
\AST_{K} (M) & \subset \gpDL_{K}(V, \psi, D),
\label{ASTKA subset DLKA}\\
\gpH(V, \psi) &\subset \AST_{K} (M)^{\circ},
\label{(AST)circ} \\
\pi_{0} (\AST_{K} (M)) & \simeq \pi_{0} ( \MMT_{{\sim}} (M)) \,\,\,\, \text{iff} \,\,\,
- {\rm{Id}}_{V} \in \AST_{K} (M)^{\circ},
\label{pi0 AST iff Id in conn. comp} \\
\pi_{0} (\AST_{K} (M)) &= \pi_{0} (\ST_{K} (M)),
\label{pi0 AST C  is pi0 ST C}\\
G_{l, K, 1}^{\alg}  &\subset \AST_{K}(M)_{\Q_l}, \,\,\, i.e. \,\,Conjecture \,\,\, \ref{general algebraic Sato Tate conj.} \,
{\rm{(a)}} \,\, holds \,\,  for \,\, M.
\label{Conjecture: general algebraic Sato Tate conj. for M} 
\end{align}
\end{theorem}

\begin{proof}
The group $\AST_{K} (M)$ is reductive by Theorem \ref{equality of conn comp for GM and GM1}. 
Moreover  \eqref{ASTKA subset DLKA} is \eqref{GMKA1 subset of DLKA} and \eqref{(AST)circ}
is just \eqref{H(V) subset of Gmot,1}. Property \eqref{pi0 AST iff Id in conn. comp} follows from Corollary
\ref{First properties of AST group}. 
Because $\AST_{K} (M)^{\circ} (\C)$ is a connected complex Lie group and any maximal compact subgroup
of a connected complex Lie group is a connected real Lie group, the equality \eqref{pi0 AST C  is pi0 ST C} follows.
Property \eqref{Conjecture: general algebraic Sato Tate conj. for M} follows by \eqref{GlK1alg subset GMKA1Ql}.
\end{proof}

\begin{remark} Recall (see Remark \ref{equivalent statements of the Serre Conj. MT = MMT0}) that
$\AST_{K} (M)^{\circ} = \gpH(V, \psi)$ is equivalent to Serre's Conjecture
\ref{Conjecture of Serre MT = MMT0}.
\end{remark}

\begin{remark} By Proposition \ref{Motivic L0realizing conn comp: even case}(d) for the odd weight $n$:
\begin{gather*}
\pi_{0} (\, \AST_{K} (M) \cap \MMT_{{\sim}} (M)^{\circ} \, ) = \{1\} \\
\pi_{0} (\AST_{K} (M)) = \pi_{0} ( \MMT_{{\sim}} (M)).
\end{gather*}
\end{remark}

Now we discuss conditions under which the Algebraic Sato--Tate is amenable to computations and 
the Algebraic Sato--Tate conjecture holds.

Consider the following commutative diagrams with exact rows. Exactness of the rows of Diagram \ref{diagram compatibility of AST GMKA1 with DLKA 2} follows by
Corollary \ref{equality of quotients concerning for GM and GM1}. The left and middle vertical 
arrows of Diagram \ref{diagram compatibility of AST GMKA1 with DLKA 1} are monomorphisms by 
\eqref{ASTKA subset DLKA}. 

\begin{figure}[H]
\[
\begin{tikzcd}
1 \arrow{r}{} & \AST_{K} (M)^{\circ} \arrow{d}[swap]{} \arrow{r}{} & 
\AST_{K} (M) \arrow{d}[swap]{} \arrow{r}{} & \pi_{0}( \AST_{K} (M)) 
\arrow{d}[swap]{} \arrow{r}{} &  1\\ 
1 \arrow{r}{} & \gpL(V, \psi, D) \arrow{r}{} & \gpDL_{K} (V, \psi, D) \arrow{r}{} & 
\pi_{0} ( \gpDL_{K} (V, \psi, D) ) \arrow{r}{}   \,\, & 1\\
\end{tikzcd}
\]
\\[-0.8cm]
\caption{}
\label{diagram compatibility of AST GMKA1 with DLKA 1} 
\end{figure}

\begin{figure}[H]
\[
\begin{tikzcd}
1 \arrow{r}{} & \pi_{0} (G_{\mathcal{M}_{{\sim}}(M), 1}^{\id}) \arrow{d}[swap]{} \arrow{r}{} & 
\pi_{0} ( \AST_{K} (M) ) \arrow{d}[swap]{} \arrow{r}{} & \Gal(K_e / K)  
\arrow{d}{=} \arrow{r}{} & 1\\ 
1 \arrow{r}{} & \pi_{0} (\gpDL_{K}^{\id} (V, \psi, D)) \arrow{r}{} & 
\pi_{0}( \gpDL_{K} (V, \psi, D)) \arrow{r}{} & \Gal(K_e / K) \arrow{r}{} &  1\\
\end{tikzcd}
\]
\\[-0.8cm]
\caption{}
\label{diagram compatibility of AST GMKA1 with DLKA 2} 
\end{figure}

\medskip

\begin{corollary}\label{The natural candidate for AST group example 1}
Assume that $\gpH(V, \psi) = C_{D} (\Iso_{(V, \psi)})$. Then
\begin{align}
\AST_{K}(M)^{\circ} &= \gpL(V, \psi, D)
\label{AST0 = DL0} \\
\pi_{0}(\AST_{K}(M)) &=  \Gal(K_e / K)
\label{pi0 = G(Ke/K)} \\
\AST_{K}(M) &= \gpDL_{K} (V, \psi, D).
\label{AST = DL}
\end{align}
\end{corollary}
\begin{proof} By the assumption and \eqref{HA subset GMKA10 subset CDSp} we obtain:
$$
\pi_{0}( G_{\mathcal{M}_{{\sim}}(M), 1}^{\id}) = \pi_{0}(\gpDL_{K}^{\id} (V, \psi, D)) = 1.
$$
Hence \eqref{pi0 = G(Ke/K)} follows immediately from Diagram 
\ref{diagram compatibility of AST GMKA1 with DLKA 2}. Moreover, the middle vertical arrow in Diagram 
\ref{diagram compatibility of AST GMKA1 with DLKA 2}, which is the right
vertical arrow in Diagram \ref{diagram compatibility of AST GMKA1 with DLKA 1}, 
is an isomorphism. Because $\gpL(V, \psi, D) = (C_{D} \Iso_{(V, \psi)})^{\circ}$, the assumption
and the connectedness of $\gpH(V, \psi)$ yield $\gpH(V, \psi) = \gpL(V, \psi, D)$. 
Hence by \eqref{(AST)circ}, the left vertical arrow in Diagram \ref{diagram compatibility of AST GMKA1 with DLKA 1} 
is  an isomorphism, and so the middle vertical arrow in Diagram 
\ref{diagram compatibility of AST GMKA1 with DLKA 1} is an isomorphism. 
\end{proof} 

Consider the following commutative diagrams with exact rows. Exactness of the rows of Diagram 
\ref{diagram compatibility of GlK1alg with DLKA 2} follows by
Corollary \ref{Corollary-GK1alg mod GK1algID with resp. to DL mod DLId}. The left and middle vertical 
arrows of Diagram \ref{diagram compatibility of GlK1alg with DLKA 1} are monomorphisms by 
\eqref{GlL1alg subset DLLA}. 

\begin{figure}[H]
\[
\begin{tikzcd}
1 \arrow{r}{} & \,\, (G_{l, K, 1}^{\alg})^{\circ} \arrow{d}[swap]{} \arrow{r}{}  & 
G_{l, K, 1}^{\alg} \arrow{d}[swap]{} \arrow{r}{} & \pi_{0}(G_{l, K, 1}^{\alg}) 
\arrow{d}[swap]{} \arrow{r}{} & 1 \\ 
1 \arrow{r}{} & \gpL(V, \psi, D)_{\Q_l} \arrow{r}{} & 
\gpDL_{K} (V, \psi, D)_{\Q_l} \arrow{r}{} & \pi_{0} ( \gpDL_{K} (V, \psi, D)_{\Q_l}) 
\arrow{r}{} & 1\\
\end{tikzcd}
\]
\\[-0.8cm]
\caption{}
\label{diagram compatibility of GlK1alg with DLKA 1} 
\end{figure}
\begin{figure}[H]
\[
\begin{tikzcd}
1  \arrow{r}{}  & \pi_{0} ((G_{l, K, 1}^{\alg})^{\id})  \arrow{d}[swap]{} \arrow{r}{} & 
\pi_{0} (G_{l, K, 1}^{\alg})  \arrow{d}[swap]{} \arrow{r}{}  & \Gal(K_e / K)  
 \arrow{d}{=} \arrow{r}{} & 1\\ 
1 \arrow{r}{} & \pi_{0} (\gpDL_{K}^{\id} (V, \psi, D)_{\Q_l} ) \arrow{r}{}  & 
\pi_{0}( \gpDL_{K} (V, \psi, D)_{\Q_l} ) \arrow{r}{}  & \Gal(K_e / K) \arrow{r}{}  & 1 \\
\end{tikzcd}
\]
\\[-0.8cm]
\caption{}
\label{diagram compatibility of GlK1alg with DLKA 2} 
\end{figure}

\begin{corollary} \label{cor alg Sato--Tate for MT explained by endo}
Assume that $\gpH(V, \psi) = C_{D} (\Iso_{(V, \psi)})$ and the Mumford--Tate conjecture holds for $M$.
Then:
\begin{align}
(G_{l, K, 1}^{\alg})^{\circ} &= \gpL(V, \psi, D)_{\Q_l},
\label{AST0l = DL0l} \\
\pi_{0}(G_{l, K, 1}^{\alg}) &=  \Gal(K_e / K).
\label{pi0l = G(Ke/K)}
\end{align}
Moreover, the algebraic Sato--Tate conjecture holds:
\begin{equation}
G_{l, K, 1}^{\alg}  = \AST_{K}(M)_{\Q_l} = \gpDL_{K} (V, \psi, D)_{\Q_l}.
\label{alg Sato--Tate for MT explained by endo}
\end{equation}
\end{corollary}
\begin{proof} 
By \eqref{GlK1alg subset GMKA1Ql} and Corollary \ref{The natural candidate for AST group example 1}: 
\begin{equation}
G_{l, K, 1}^{\alg} \subset \AST_{K}(M)_{\Q_l} = \gpDL_{K} (V, \psi, D)_{\Q_l}.
\label{AST conj approach 1}
\end{equation}

By Remark \ref{Mumford--Tate Conj and Deligne Theorem analogue}, the Mumford--Tate conjecture 
\eqref{MT eq for Hodge structures associated with l-adic rep.} holds if and only if the equality 
\eqref{H eq} holds, namely:
$$
(G_{l, K, 1}^{\alg})^{\circ} = \gpH(V, \psi)_{\Q_l}.
$$
Hence by the assumptions and connectedness of $\gpH(V, \psi)$, the left vertical arrow 
in Diagram \ref{diagram compatibility of GlK1alg with DLKA 1} is an isomorphism and
so \eqref{AST0l = DL0l} follows.
Moreover by \eqref{GlK1algtau subset DLKtauVpsiDQl} and the assumption, we observe that 
$(G_{l, K, 1}^{\alg})^{\circ} = (G_{l, K, 1}^{\alg})^{\id}$. Hence 
$$
\pi_{0}((G_{l, K, 1}^{\alg})^{\id}) = \pi_{0}(\gpDL_{K}^{\id} (V, \psi, D)) = 1.
$$
Hence \eqref{pi0l = G(Ke/K)} follows from Diagram 
\ref{diagram compatibility of GlK1alg with DLKA 2} and all vertical arrows
in Diagrams \ref{diagram compatibility of GlK1alg with DLKA 1} and
\ref{diagram compatibility of GlK1alg with DLKA 2} are isomorphisms. 
Applying also \eqref{AST conj approach 1}, the equality \eqref{alg Sato--Tate for MT explained by endo} follows.   
\end{proof}

\begin{remark} V. Cantoral-Farf{\' a}n and J. Commelin have proved
\cite{CF-C} that for abelian motives in $\mathcal{M}_{{\rm{mot}}}$, the Mumford-Tate conjecture implies the Algebraic Sato-Tate conjecture.
\end{remark}

\section{Computation of the identity connected component of \texorpdfstring{$\AST_{K}(M)$}{ASTKM}}
\label{computation of the identity connected component of ASTKM}

Consider the Tate motive $\TT := \LL^{-1} = h^2(\PP^{1})^{-1} = \Q(1)$
and the motivic Galois groups $G_{\mathcal{M}_{{\sim}}(M)}$, 
$G_{\mathcal{M}_{{\sim}}(M \oplus \TT)}$, and $G_{\mathcal{M}_{{\sim}}(\TT)} = \G_m$.
By the definition of the motivic Galois group, there is a natural monomorphism:
\begin{equation}
G_{\mathcal{M}_{{\sim}}(M \oplus \TT)} \hookrightarrow G_{\mathcal{M}_{{\sim}}(M)} \times {\G_m}.
\label{embedding of GML into GM times Gm}
\end{equation} 
The projections 
\begin{equation}
N\colon G_{\mathcal{M}_{{\sim}}(M \oplus \TT)} \rightarrow  {\G_m} \quad {\rm{and}} \quad  
\pi\colon G_{\mathcal{M}_{{\sim}}(M \oplus \TT)} \rightarrow G_{\mathcal{M}_{{\sim}}(M)}
\label{natural epimorphisms GMsim M oplus TT to Gm and GMsim M}
\end{equation} 
are epimorphisms because categories 
$\mathcal{M}_{{\sim}}(M)$, $\mathcal{M}_{{\sim}}(M \oplus \TT)$ and 
$\mathcal{M}_{{\sim}}(\TT)$ are strictly full and fully faithful subcategories of
the semisimple category $\mathcal{M}_{{\sim}}$. Hence they are also semisimple categories 
and the corresponding functors among them are also fully faithful (see \cite[Prop. 2.21, p. 139]{DM},
cf. \cite[p. 472]{Pan}). 

\begin{definition}
\label{GM 1 and GM circ 1}
Define the following groups:
\begin{align*}
G_{\mathcal{M}_{{\sim}}(M \oplus \TT), 1} &:= {\rm{Ker}} \, N \\
G^{0}_{\mathcal{M}_{{\sim}}(M \oplus \TT), 1} &:= 
G_{\mathcal{M}_{{\sim}}(M \oplus \TT), 1} \cap 
(G_{\mathcal{M}_{{\sim}}(M \oplus \TT)})^{\circ}.
\end{align*}
\end{definition}

Recall that we denote $V = H_{B} (M)$. The same proof as the proof of \eqref{MT(V) subset of Gmot} of Lemma \ref{MT subset GMot, H subset GMot1} 
gives:
\begin{equation}
\widetilde{\MT }(V, \psi) \subset (G_{{\mathcal{M}_{{\sim}} (M \oplus \TT)}})^{\circ}.
\label{wide tilde MT subset (G M o plus T) circ}
\end{equation}
By \eqref{embedding of tilde MT into MT times Gm}, \eqref{embedding of GML into GM times Gm}, and the definitions of $\widetilde{\MT }(V, \psi)$ and $G_{{\mathcal{M}_{{\sim}} (M \oplus \TT)}}$, 
the following diagram with exact columns commutes, with \eqref{wide tilde MT subset (G M o plus T) circ} as the left arrow in the middle row:

\begin{figure}[H]
\[
\begin{tikzcd} 
1 \arrow{d}[swap]{} &  1 \arrow{d}[swap]{} &  1 \arrow{d}[swap]{}\\
{\gpH}(V, \psi) \arrow{d}[swap]{} \arrow[hook]{r}{} & G^{0}_{\mathcal{M}_{{\sim}}(M \oplus \TT), 1} 
\arrow{d}[swap]{} \arrow[hook]{r}{} &  G_{\mathcal{M}_{{\sim}}(M \oplus \TT), 1} \arrow{d}[swap]{}\\ 
\widetilde{\MT }(V, \psi) \arrow{d}[swap]{N} \arrow[hook]{r}{} & (G_{\mathcal{M}_{{\sim}}(M \oplus \TT)})^{\circ} \arrow{d}[swap]{N} \arrow[hook]{r}{} & G_{\mathcal{M}_{{\sim}}(M \oplus \TT)} \arrow{d}[swap]{N}\\ 
\G_{m} \arrow{d}[swap]{} \arrow{r}{=} & \G_{m} \arrow{d}[swap]{} \arrow{r}{=} & \G_{m} \arrow{d}[swap]{}\\
1  & 1 & 1\\
\end{tikzcd}
\]
\\[-0.8cm]
\caption{}
\label{diagram compatibility of N of wide tilde MT with N of GM TT} 
\end{figure}

From Diagram \ref{diagram compatibility of N of wide tilde MT with N of GM TT} we obtain 
natural isomorphism:  
\begin{equation}
\widetilde{\MT }(V, \psi) / {\gpH}(V, \psi)  
\,\,\, {\stackrel{\simeq}{\longrightarrow}}  \,\,\, 
G_{\mathcal{M}_{{\sim}}(M \oplus \TT)} / G_{\mathcal{M}_{{\sim}}(M \oplus \TT), 1}.
\label{widetlde MT mod H iso with GM mod GM 1}
\end{equation}

\begin{theorem} 
\label{proposition: i M oplus T is an isomorphism}
The following natural map is an isomorphism:
\begin{equation}
i_{M \oplus \TT}\colon \pi_{0} (G_{\mathcal{M}_{{\sim}}(M \oplus \TT), 1}) \,\,\, {\stackrel{\simeq}{\longrightarrow}}  \,\,\,  \pi_{0} (G_{\mathcal{M}_{{\sim}}(M \oplus \TT)}). 
\label{i M oplus T is an isomorphism}
\end{equation}
In addition there is the following equality:
\begin{equation}
G^{0}_{\mathcal{M}_{{\sim}}(M \oplus \TT), 1} = (G_{\mathcal{M}_{{\sim}}(M \oplus \TT), 1})^{\circ}.
\label{Connected component of identity of GM plus T 1}
\end{equation}
\end{theorem}
\begin{proof}
Consider the following exact sequence: 
\begin{equation}
1 \rightarrow G^{0}_{\mathcal{M}_{{\sim}}(M \oplus \TT), 1} \rightarrow (G_{\mathcal{M}_{{\sim}}(M \oplus \TT)})^{\circ} \,\, {\stackrel{N}{\longrightarrow}}  \,\, \G_{m} \rightarrow 1.
\label{exact sequence: G circ M 1 arrow G M circ arrow Gm}
\end{equation}
Because of Diagram \ref{diagram compatibility of N of wide tilde MT with N of GM TT} the cocharacter:
$$
\G_{m} (\C) \,\,\, {\stackrel{{\widetilde{\mu_{\infty, V}}}}{\longrightarrow}} \,\,\, \widetilde{\MT }(V, \psi) (\C)
\rightarrow (G_{\mathcal{M}_{{\sim}}(M \oplus \TT)})^{\circ} (\C)
$$
splits $N$ in the exact sequence:
\begin{equation}
\label{exact sequence: G circ M 1 arrow G M circ arrow Gm evaluated at C}
1 \rightarrow G^{0}_{\mathcal{M}_{{\sim}}(M \oplus \TT), 1} (\C) \rightarrow (G_{\mathcal{M}_{{\sim}}(M \oplus \TT)})^{\circ} (\C) \,\, {\stackrel{N}{\longrightarrow}}  \,\, \G_{m} (\C) \rightarrow 1  
\end{equation}
Hence by Lemma \ref{G connected implies G0 connected},
$$
G^{0}_{\mathcal{M}_{{\sim}}(M \oplus \TT), 1} = (G_{\mathcal{M}_{{\sim}}(M \oplus \TT), 1})^{\circ}.
$$
Observe that the commutative Diagram \ref{motivic diagram in Serre theorem} holds for the motive 
$M \oplus \TT$ so we obtain the isomorphism \eqref{i M oplus T is an isomorphism} from the right
exact column of this diagram for $M \oplus \TT$.
\end{proof}

\begin{lemma}
There is the following commutative diagram with exact rows and columns. The map $\pi_1$ is induced by $\pi$.
\begin{figure}[H]
\[
\begin{tikzcd}
1 \arrow{r}{} & \,\, G_{\mathcal{M}_{{\sim}}(M \oplus \TT), 1}  \arrow{d}[swap]{\pi_1} \arrow{r}{}  & 
G_{\mathcal{M}_{{\sim}}(M \oplus \TT)}  \arrow{d}[swap]{\pi} \arrow{r}{N} & \G_{m} 
\arrow{d}{x \mapsto x^{-n}}[swap]{} \arrow{r}{} & 1 \\ 
1 \arrow{r}{} & G_{\mathcal{M}_{{\sim}}(M), 1}  \arrow{r}{} & 
G_{\mathcal{M}_{{\sim}}(M)} \arrow{d}[swap]{} \arrow{r}{\chi} & \G_{m} \arrow{d}[swap]{} 
\arrow{r}{} & 1\\
& & 1 & 1 \\
\end{tikzcd}
\]
\\[-0.8cm]
\caption{}
\label{diagram compatibility of GM plus L with GM} 
\end{figure}
\label{commutativity of diagram compatibility of GM plus L with GM}
\end{lemma}
\begin{proof} Observe that the right square of the Diagram \ref{diagram compatibility of GM plus L with GM}  
is the front face of the following cube Diagram \ref{proposition: i M oplus T is an isomorphism}. 
\begin{figure}[H]
\[
\begin{tikzcd}[row sep=scriptsize, column sep=scriptsize]
& {\widetilde{\MT}}(V, \psi) \arrow{dl}{} \arrow{rr}{N} \arrow{dd}[near end]{\pi} & & \G_{m} \arrow{dl}{=} 
\arrow{dd}{x \mapsto x^{-n}} \\
G_{\mathcal{M}_{{\sim}}(M \oplus \TT)} \arrow[crossing over]{rr}[near start]{N} \arrow{dd}{\pi} & & \G_{m}
\\
& {\MT}(V, \psi) \arrow[dl] \arrow{rr}[near start, swap]{\chi} & & \G_{m} \arrow{dl}{=} \\
G_{\mathcal{M}_{{\sim}}(M)} \arrow{rr}[swap]{\chi} & &  \G_{m} 
\arrow[crossing over, leftarrow]{uu}[near end, swap]{x \mapsto x^{-n}} 
\\
\end{tikzcd}
\]
\\[-0.8cm]
\caption{}
\label{the cube diagram} 
\end{figure}
Observe that the 5 remaining faces of Diagram \ref{the cube diagram} commute 
(the rear face commutes by Lemma \ref{commutativity of diagram compatibility of tilde(MT (V, psi)) with MT (V, psi)}). Hence the front face also commutes  
by \eqref{widetlde MT mod H iso with GM mod GM 1} and a diagram chase.
\end{proof} 

By the definition of $N$ and \eqref{embedding of GML into GM times Gm}, the kernel of $\pi$ in Diagram  
\ref{diagram compatibility of GM plus L with GM} is contained in ${\rm{Id}}_{V} \times
\G_{m}$. Hence the restriction of $N$ to the kernel of $\pi$ is a monomorphism. By the commutativity 
of Diagram \ref{diagram compatibility of GM plus L with GM}, the kernel of $\pi$ injects into
$\mu_{n}$. Hence $\pi_{1}$ is a monomorphism. In addition $\dim G_{\mathcal{M}_{{\sim}}(M \oplus \TT)} = 
\dim G_{\mathcal{M}_{{\sim}}(M)}$ 
and consequently $\dim G_{\mathcal{M}_{{\sim}}(M \oplus \TT), 1} = 
\dim G_{\mathcal{M}_{{\sim}}(M), 1}$. This shows that
\begin{equation}
(G_{\mathcal{M}_{{\sim}}(M \oplus \TT), 1})^{\circ} = (G_{\mathcal{M}_{{\sim}}(M), 1})^{\circ}.
\label{connected component of id of GM plus T 1 is the same as GM 1}
\end{equation} 

Consider the following commutative diagram. 
\begin{figure}[H]
\[
\begin{tikzcd}
&& \,\, \pi_{0} (G_{\mathcal{M}_{{\sim}}(M \oplus \TT), 1}) \arrow{d}[swap]{\overline{\pi}_{1}} \arrow{r}{\simeq}  & \pi_{0} (G_{\mathcal{M}_{{\sim}}(M \oplus \TT)}) \arrow{d}[swap]{\overline{\pi}}& \\ 
1 \arrow{r}{} &  G^{0}_{\mathcal{M}_{{\sim}}(M), 1}/(G_{\mathcal{M}_{{\sim}}(M), 1})^{\circ} \arrow{r}{} & 
\pi_{0} (G_{\mathcal{M}_{{\sim}}(M), 1})  \arrow{r}{} & \pi_{0} (G_{\mathcal{M}_{{\sim}}(M)}) \arrow{r}{} & 1\\
\end{tikzcd}
\]
\\[-0.8cm]
\caption{}
\label{pi0 applied to the diagram compatibility of GM 1 with GM} 
\end{figure}

The top row is the isomorphism of Theorem \ref{proposition: i M oplus T is an isomorphism}. The bottom row is the exact sequence of Theorem \ref{equality of conn comp for GM and GM1}. 
The map $\overline{\pi}$ in Diagram \ref{pi0 applied to the diagram compatibility of GM 1 with GM}  
is an epimorphism because the map $\pi$ in Diagram \ref{diagram compatibility of GM plus L with GM} is an epimorphism. The map $\overline{\pi}_{1}$ in Diagram \ref{pi0 applied to the diagram compatibility of GM 1 with GM}  is a monomorphism because the map $\pi_{1}$ in Diagram \ref{diagram compatibility of GM plus L with GM} is a monomorphism and because of \eqref{connected component of id of GM plus T 1 is the same as GM 1}.
Proposition \ref{Motivic L0realizing conn comp: even case}(b)  and a chase in Diagram 
\ref{pi0 applied to the diagram compatibility of GM 1 with GM} show that  
\begin{equation}
{\pi_{0} (G_{\mathcal{M}_{{\sim}}(M), 1})  \,  = \, \pi_{0} (G_{\mathcal{M}_{{\sim}}(M \oplus \TT), 1}) \,\, \cup 
\, - {\rm{Id}}_{V} (G_{\mathcal{M}_{{\sim}}(M), 1})^{\circ} \,\, 
\pi_{0} (G_{\mathcal{M}_{{\sim}}(M \oplus \TT), 1}).}   
\label{decomposition of p0 GM 1 into p0 GM plus T 1 cup -Id p0 GM plus T 1}  
\end{equation}
where $- {\rm{Id}}_{V} (G_{\mathcal{M}_{{\sim}}(M), 1})^{\circ}$ denotes the coset of 
$- {\rm{Id}}_{V}$ in the quotient group $\pi_{0} (G_{\mathcal{M}_{{\sim}}(M), 1}) = G_{\mathcal{M}_{{\sim}}(M), 1} / (G_{\mathcal{M}_{{\sim}}(M), 1})^{\circ}$. 
\medskip

Consider the following commutative diagram with exact rows. 
\begin{figure}[H]
\[
\begin{tikzcd}
1 \arrow{r}{} & (G_{\mathcal{M}_{{\sim}}(M \oplus \TT), 1})^{\circ} \arrow{d}[swap]{=} \arrow{r}{} & \,\, 
G_{\mathcal{M}_{{\sim}}(M \oplus \TT), 1} \arrow{d}[swap]{\pi_{1}} \arrow{r}{}  & 
 \pi_{0} (G_{\mathcal{M}_{{\sim}}(M \oplus \TT), 1}) \arrow{d}[swap]{\overline{\pi}_{1}} \arrow{r}{} & 1\\ 
1 \arrow{r}{} & (G_{\mathcal{M}_{{\sim}}(M), 1})^{\circ} \arrow{r}{} & 
G_{\mathcal{M}_{{\sim}}(M), 1}  \arrow{r}{} & 
\pi_{0} (G_{\mathcal{M}_{{\sim}}(M), 1}) \arrow{r}{} & 1\\
\end{tikzcd}
\]
\\[-0.8cm]
\caption{}
\label{the diagram compatibility of GM plus T 1 with GM 1 and pi0 GM plus 1 with pi0 GM 1} 
\end{figure}

The equality \eqref{decomposition of p0 GM 1 into p0 GM plus T 1 cup -Id p0 GM plus T 1} and a chase in 
Diagram \ref{the diagram compatibility of GM plus T 1 with GM 1 and pi0 GM plus 1 with pi0 GM 1} 
show that 
\begin{equation}
G_{\mathcal{M}_{{\sim}}(M), 1}  \,  = \, G_{\mathcal{M}_{{\sim}}(M \oplus \TT), 1} \,\, \cup 
\, - {\rm{Id}}_{V}  \,\, G_{\mathcal{M}_{{\sim}}(M \oplus \TT), 1} \, .   
\label{decomposition of GM 1 into GM plus T 1 cup -Id GM plus T 1}  
\end{equation} 
In particular $[G_{\mathcal{M}_{{\sim}}(M), 1}: \, G_{\mathcal{M}_{{\sim}}(M \oplus \TT), 1}] \leq 2$ and 
$G_{\mathcal{M}_{{\sim}}(M \oplus \TT), 1} \, \triangleleft \, G_{\mathcal{M}_{{\sim}}(M), 1}$.
\medskip

To be consistent with Serre's approach, as in \S \ref{section-computation of the identity connected component} and \ref{algebraic ST conjecture for families of l-adic reps}, we make the following definition. 

\begin{definition} 
The \emph{algebraic Sato--Tate group} $\widetilde{\AST_{K}} (M)$ is defined as follows:
\begin{equation}
\widetilde{\AST_{K}} (M) := \MS_{{\sim}} (M \oplus \TT) = G_{\mathcal{M}_{{\sim}}(M \oplus \TT), 1}. 
\label{Serre AST group}
\end{equation}
Every maximal compact subgroup of $\widetilde{\AST_{K}} (M)(\C)$ will be called
a \emph{Sato--Tate group} associated with $M$ and denoted $\widetilde{\ST_{K}} (M)$.
\label{GMLLKA1 as AST group}
\end{definition}

\begin{corollary} The following equalities hold:
\begin{itemize}
\item[(a)] $\AST_{K}(M)^{\circ} = G^{0}_{\mathcal{M}_{{\sim}}(M \oplus \TT), 1}$.
\item[(b)] $\AST_{K}(M) = \widetilde{\AST_{K}} (M) \, \cup \, - {\rm{Id}}_{V} \widetilde{\AST_{K}} (M)$.
\item[(c)] $\ST_{K}(M) = \widetilde{\ST_{K}} (M) \, \cup \, - {\rm{Id}}_{V} \widetilde{\ST_{K}} (M)$
up to conjugation in \\ $\AST_{K}(M) (\C)$.
\end{itemize}
\label{properties of AST M psi with resp. to AST M}
\end{corollary}
\begin{proof}
(a) follows by \eqref{Connected component of identity of GM plus T 1},
(b) follows by \eqref{decomposition of GM 1 into GM plus T 1 cup -Id GM plus T 1}, and (c) 
follows by (b). 
\end{proof}

Consider the following commutative Diagram~\ref{wide tilde GlK 1 subset GMK M 1 Ql}
 of group schemes with exact columns. 
The middle horizontal arrow is a closed immersion by Remark \ref{Remark on GlKalg subset GMKAQl}. This 
remark obviously holds for all motives in $\mathcal{M}_{{\sim}}$, not only polarized motives.
The middle horizontal arrow is a closed immersion. Hence the top horizontal arrow is also a closed immersion. 
\begin{figure}[H]
\[
\begin{tikzcd}
1 \arrow{d}[swap]{} &  1 \arrow{d}[swap]{} \\
\widetilde{G_{l, K, 1}^{\alg}} \arrow{d}[swap]{} \arrow[hook]{r}{} & 
{G_{\mathcal{M}_{{\sim}}(M \oplus \TT), 1}}_{\Q_l} \arrow{d}[swap]{}\\
\widetilde{G_{l, K}^{\alg}} \arrow{d}{N} \arrow[hook]{r}{} & 
{G_{\mathcal{M}_{{\sim}}(M \oplus \TT)}}_{\Q_l} \arrow{d}{N}\\
\G_{m} \arrow{d}[swap]{} \arrow{r}{=} & \G_{m} \arrow{d}[swap]{}\\
1  & 1 \\
\end{tikzcd}
\]
\\[-0.8cm]
\caption{}
\label{wide tilde GlK 1 subset GMK M 1 Ql}
\end{figure}

By Definition \ref{GMLLKA1 as AST group} and Diagram \ref{wide tilde GlK 1 subset GMK M 1 Ql},
we obtain:
\begin{equation}
\widetilde{G_{l, K, 1}^{\alg}}  \subset \widetilde{\AST_{K}} (M)_{\Q_l}.
\label{wide tilde GlKalg subset AST M Ql}
\end{equation} 
Hence Algebraic Sato--Tate Conjecture \ref{general algebraic Sato Tate conj. Serre's approach}(a)
holds for $\widetilde{\AST_{K}} (M)$ with respect to the representation
$$
\widetilde{\rho}_l\colon G_K \rightarrow \GL (V_{l} \oplus \Q_l(1)).
$$ 
Algebraic Sato--Tate Conjecture \ref{general algebraic Sato Tate conj. Serre's approach}(b)
for $\widetilde{\AST_{K}} (M)$ states:
$$
\widetilde{G_{l, K, 1}^{\alg}} \cong  \widetilde{\AST_{K}} (M)_{\Q_{l}}.
$$

\begin{proposition}
The Algebraic Sato--Tate conjecture holds for 
$\AST_{K}(M)$ (Conjecture \ref{general algebraic Sato Tate conj.}) if and only if the Algebraic Sato--Tate conjecture holds for $\widetilde{\AST_{K}} (M)$ (Conjecture \ref{general algebraic Sato Tate conj. Serre's approach}).
\label{AST conjecture for M is equivalent to wide tilde AST conjecture }
\end{proposition}
\begin{proof} 
This follows by \eqref{Conjecture: general algebraic Sato Tate conj. for M}, \eqref{wide tilde GlKalg subset AST M Ql}, \eqref{decomposition of GlK1alg into tilde(GlK1alg) cup -Id tilde(GlK1alg)}
and Corollary \ref{properties of AST M psi with resp. to AST M}(b).
\end{proof}

Because of Proposition \ref{AST conjecture for M is equivalent to wide tilde AST conjecture }, it is enough to discuss in the following sections the Algebraic Sato--Tate conjecture only for $\AST_{K}(M)$.  

\begin{proposition}
The Sato--Tate conjecture for $\ST_{K}(M)$ (Conjecture \ref{general Sato Tate conj.})
implies the Sato--Tate conjecture for $\widetilde{\ST_{K}} (M)$ (Conjecture \ref{general Sato Tate conj. tilde}).
 \label{ST implies tilde ST}
\end{proposition}
\begin{proof}
This follows by Theorem 
\ref{equidistribution property in ST implies equidistribution property in wide tilde ST}.
\end{proof}

Recall that $D := {{\rm End}}_{\mathcal{M}_{\overline{K}, \sim}} (\overline{M})$ (see (\ref{definition of D(M)})). 

\begin{proposition} Put $\widetilde{\mathcal{G}_{K}^{\alg}} := 
G_{\mathcal{M}_{{\sim}}(M \oplus \TT)}$ and $\mathcal{G}_{K}^{\alg} := G_{\mathcal{M}_{{\sim}}(M)}$.
Assume that the family $(V_l)_{l}$ of $l$-adic realizations of $M$ is a strictly compatible
family of $l$-adic representations $(\rho_{l})_l$ in the sense of Serre.
Then Conjectures \ref{Tate conjecture for families of l-adic representations}(a) and 
\ref{Tate conjecture for families of l-adic representations tilde}(a) hold  
for the family $(\rho_{l})_l$.   
\label{the family V l l of G K for pure motive M satisfies Conjectures 1 (a) and 2 (a)}
\end{proposition}
\begin{proof} By Lemmas \ref{conditions D1, D2 are satisfied by Hodge str. associated with M},
\ref{conditions DR1, DR2 are satisfied by (VDR, psiDR) associated with M},
\ref{conditions R1-R4 are satisfied by Hodge str. and associated l-adic rep. for M}, and the 
assumption of strict compatibility, the conditions 
\textbf{(D1)},
\textbf{(D2)},
\textbf{(DR1)},
\textbf{(DR2)},
\textbf{(R1)}--\textbf{(R4)}
of \S \ref{Mumford--Tate groups of polarized Hodge structures}, \S \ref{de Rham structures associated with Hodge structures}, and \S \ref{families of l-adic representations associated with Hodge structures} hold for the Betti, de Rham and $l$-adic realizations of $M$. 
Observe that the middle left horizontal arrow in Diagram \ref{GlK1 subset GMKA1Ql}
and the middle horizontal arrow in Diagram \ref{wide tilde GlK 1 subset GMK M 1 Ql}
are closed immersions. Hence the Proposition follows by the properties of the morphisms
\eqref{embedding of GML into GM times Gm} and
\eqref{natural epimorphisms GMsim M oplus TT to Gm and GMsim M}.
\end{proof}

\begin{corollary} Under the notation and assumptions of Proposition 
\ref{the family V l l of G K for pure motive M satisfies Conjectures 1 (a) and 2 (a)},
Conjectures \ref{Tate conjecture for families of l-adic representations},  
\ref{Tate conjecture for families of l-adic representations tilde}, 
\ref{general algebraic Sato Tate conj.}, and \ref{general algebraic Sato Tate conj. Serre's approach} are equivalent for the family of $l$-adic representations $(\rho_l)_{l}$.   
\label{the family V l l of G K coming from a pure motive gives equivalence of Conjectures 1 and 2 and AST and tilde AST}
\end{corollary}
\begin{proof} 
This follows from Proposition \ref{the family V l l of G K for pure motive M satisfies Conjectures 1 (a) and 2 (a)} and Corollary \ref{equivalence of 4 conjectures}. 
\end{proof}

\begin{remark} Recall the category of polarized realizations $R_{K}^{\rm{p}}$ from \S \ref{sec,category of polarized realizations}. Deligne's category $\mathcal{M}_{{\text{ahc}}}$ of motives for absolute Hodge cycles is a full subcategory of $R_{K}^{\rm{p}}$ via the Betti, de Rham and $l$-adic realizations.
\end{remark}

\begin{remark}
It is not known in general that the family of $l$-adic representations coming from $l$-adic realizations of a motive in $\mathcal{M}_{{\text{ahc}}}$ or $\mathcal{M}_{{\text{mot}}}$ is strictly compatible in the sense of Serre.
\end{remark}

 We would like to propose the following {\bf{Geometricity Conjecture}}. 
\begin{conjecture} Objects of $R_{K}^{\rm{p}}$ are realizations of motives from $\mathcal{M}_{{\text{ahc}}}$. More precisely, let 
$\mathcal{V} \, := \, (\, V_{_{\rm{DR}}}, \, (V_l)_{l}, \, (V_{\sigma})_{\sigma}, \, 
I_{\infty, \sigma}, \,  I_{l, \bar{\sigma}}) \in {\rm{obj}} (R_{K}^{\rm{p}})$ be a pure, polarized realization with components satisfying conditions 
\textbf{(D1)},
\textbf{(D2)},
\textbf{(DR1)},
\textbf{(DR2)},
\textbf{(R1)}--\textbf{(R4)}
of \S \ref{Mumford--Tate groups of polarized Hodge structures}, \S \ref{de Rham structures associated with Hodge structures}, and \S \ref{families of l-adic representations associated with Hodge structures} respectively with $D$ as in Definition \ref{The ring D for polarized realizations}. Then $\mathcal{V}$ comes from a pure motive of the Deligne motivic category for absolute Hodge cycles $\mathcal{M}_{K, {{\rm{{{\rm{ahc}}}}}}}$ via the Betti, de Rham and $l$-adic realizations.
\label{Geometricity conjecture}
\end{conjecture}
 
\begin{remark}
Conjecture \ref{Geometricity conjecture} might be formulated for other motivic categories.
Nevertheless the Betti, de Rham and $l$-adic realizations functor 
$\mathcal{M}_{K, {_{\sim}}}(M) \rightarrow R_{K}^{\rm{p}}$
factors through $\mathcal{M}_{K, {{\rm{{{\rm{ahc}}}}}}}$.
Hence the Geometricity Conjecture is stated in the weakest possible form.
\end{remark}

Recall that there is also a natural epimorphism (see Remark \ref{on the parity Sato--Tate group}): 
$$
\gpP(V_{l}, \psi_{l}) \rightarrow \gpP_{\rm{ST}} (V, \psi).
$$

\begin{definition}
\label{The definition of the Sato--Tate parity group for motives}
The Sato--Tate parity group for a motive $M$ is defined as follows 
(cf. Corollary \ref{properties of AST M psi with resp. to AST M}): 
$$
\gpP_{\rm{ST}} (M, \psi_{{\sim}}) \, :=  \, \AST_{K} (M) / \widetilde{\AST_{K}} (M) 
\, =  \, \ST_{K} (M) / \widetilde{\ST_{K}} (M).
$$
\end{definition}

\begin{remark}  
 While $\gpP_{\rm{ST}} (V, \psi)$ (see \eqref{P ST under Conjectures 8.1 and 8.2})
is defined assuming Conjectures \ref{Tate conjecture for families of l-adic representations}(a) and 
\ref{Tate conjecture for families of l-adic representations tilde}(a) for the family of
$l$-adic representations $(\rho_{l})_l$, the group $\gpP_{\rm{ST}} (M, \psi_{{\sim}})$ is defined 
unconditionally for motives $M$ for which $l$-adic realizations give a family of
$l$-adic representations $(\rho_{l})_l$ which is strictly compatible in the sense of Serre (see Lemmas \ref{conditions D1, D2 are satisfied by Hodge str. associated with M}, \ref{conditions DR1, DR2 are satisfied by (VDR, psiDR) associated with M}, and \ref{conditions R1-R4 are satisfied by Hodge str. and associated l-adic rep. for M}
and Proposition \ref{the family V l l of G K for pure motive M satisfies Conjectures 1 (a) and 2 (a)}).
\label{Sato--Tate parity under Conj. 8.1 (a) and 8.2 (a) for motives}
\end{remark} 

Now Theorem \ref{equidistribution property in ST implies equidistribution property in wide tilde ST for 
PST under Conj. 8.1 (a) and 8.2 (a)} has a counterpart for motives.

\begin{theorem} 
Assume that $l$-adic realizations of the motive $M$ give a family of
$l$-adic representations $(\rho_{l})_l$ which is strictly compatible in the sense of Serre.
\begin{itemize}
\item[(a)]
If $\gpP_{\rm{ST}} (M, \psi_{{\sim}})$ is nontrivial, i.e. $- {{\rm Id}}_{V} \notin \widetilde{\ST_{K}} (M)$, then Sato--Tate Conjecture \ref{general Sato Tate conj.} does not hold.
\item[(b)]
If 
$\gpP_{\rm{ST}} (M, \psi_{{\sim}})$ is trivial, i.e. $- {{\rm Id}}_{V} \in \widetilde{\ST_{K}} (M)$,
then Sato--Tate Conjectures \ref{general Sato Tate conj.}  and \ref{general Sato Tate conj. tilde}
are equivalent. 
\end{itemize}
\label{equidistribution property in ST implies equidistribution property in wide tilde ST for 
PST of motives}
\end{theorem}
\begin{proof} The proof is the same as for Theorem
\ref{equidistribution property in ST implies equidistribution property in wide tilde ST}.
\end{proof}

\begin{remark} 
Theorem \ref{equidistribution property in ST implies equidistribution property in wide tilde ST for 
PST of motives}(a) can be illustrated using Example \ref
{Another example of nontrivial Betti parity group}, 
as the Sato--Tate conjecture is known for $A$ when $K$ is a totally real field \cite{HSBT}, \cite{BLGG} or a CM field \cite{Ten}.
In these cases, Conjecture~\ref{general Sato Tate conj. tilde} holds but Conjecture~\ref{general Sato Tate conj.} does not.
\label{equidistribution in ST contra equidistribution in wide tilde ST for motives} 
\end{remark}

By Definitions \ref{The motivic parity group}, 
\ref{The definition of the Sato--Tate parity group for motives}, and equality  
\eqref{connected component of id of GM plus T 1 is the same as GM 1} there is a natural homomorphism:
\begin{equation}
\gpP_{S} (M, \psi_{{\sim}}) \rightarrow \gpP_{\rm{ST}} (M, \psi_{{\sim}}).
\label{Homomorphism between P M psi tilde (M) and P S T (M)}
\end{equation}

The groups $\gpP_{S} (M, \psi_{{\sim}})$ and $\gpP_{\rm{ST}} (M, \psi_{{\sim}})$ have order at most 2. 
These groups are generated by the corresponding cosets of $- {\rm{Id}}_V$.
Moreover the homomorphism \eqref{Homomorphism between P M psi tilde (M) and P S T (M)}
carries the coset of $- {\rm{Id}}_V$ in the source into the coset of $- {\rm{Id}}_V$ in the target. Hence   
the homomorphism \eqref{Homomorphism between P M psi tilde (M) and P S T (M)} is an epimorphism, cf. Remark
\ref{Each hom. between parity gr. carries the coset of - Id in the source into the coset of - Id in the target.}.

Consider the following commutative diagram. 

\begin{figure}[H]
\[
\begin{tikzcd}
1 \arrow{r}{} & (G_{\mathcal{M}_{{\sim}}(M \oplus \TT)})^{\circ} \arrow{d}[swap]{}{\pi^{\circ}} \arrow{r}{}
&  G_{\mathcal{M}_{{\sim}}(M \oplus \TT)}  \arrow{d}[swap]{}{\pi} \arrow{r}{}  & \pi_{0} 
(G_{\mathcal{M}_{{\sim}}(M \oplus \TT)})   \arrow{d}[swap]{}{\overline{\pi}} \arrow{r}{} &  1\\  
1 \arrow{r}{} & (G_{\mathcal{M}_{{\sim}}(M)})^{\circ} \arrow{d}{} \arrow{r}{} & 
G_{\mathcal{M}_{{\sim}}(M)} \arrow{d}{} \arrow{r}{} & \pi_{0}(G_{\mathcal{M}_{{\sim}}(M)}) 
\arrow{d}[swap]{} \arrow{r}{} &  1 \\
&  1  & 1   & 1\\
\end{tikzcd}
\]
\\[-0.8cm]
\caption{}
\label{Diagram between row of tilde AST and row of AST }
\end{figure}

The map $\pi$ is an epimorphism cf. \eqref{natural epimorphisms GMsim M oplus TT to Gm and GMsim M}. Consequently 
$\overline{\pi}$ is an epimorphism.
By Diagram \ref{diagram compatibility of GM plus L with GM} of Lemma
\ref{commutativity of diagram compatibility of GM plus L with GM} and the comments below the proof of 
this lemma, the map $\pi$ has finite kernel. Hence the algebraic groups $G_{\mathcal{M}_{{\sim}}(M \oplus \TT)}$
and $G_{\mathcal{M}_{{\sim}}(M)}$ have the same dimension and the map $\pi^{\circ}$ has finite kernel. Because the image of   
$\pi^{\circ}$ is a closed subgroup of $(G_{\mathcal{M}_{{\sim}}(M)})^{\circ}$ \cite[Section 7.4, Prop. B(b)]{Hu} of the same dimension as $(G_{\mathcal{M}_{{\sim}}(M \oplus \TT)})^{\circ}$, the image is also open. Because the group 
$(G_{\mathcal{M}_{{\sim}}(M)})^{\circ}$ is connected, the map $\pi^{\circ}$ must be an epimorphism. 
By the snake lemma applied to Diagram 
\ref{Diagram between row of tilde AST and row of AST }, there is the following short exact sequence:
\begin{equation}
1 \rightarrow {\rm{Ker}} \, \pi^{\circ} \rightarrow {\rm{Ker}} \, \pi \rightarrow {\rm{Ker}} \, \overline{\pi} \rightarrow 1.
\label{ker pi 0  Ker pi Ker overline pi for motives}
\end{equation} 

\begin{proposition} There is the following short exact sequence:
\label{Proposition 1 rightarrow P T rightarrow P S rightarrow P ST rightarrow 1 for motives}
\begin{equation}
1 \rightarrow {\rm{Ker}} \, \overline{\pi} \rightarrow \gpP_{S} (M, \psi_{{\sim}}) \rightarrow 
\gpP_{\rm{ST}} (M, \psi_{{\sim}}) \rightarrow 1 \,.
\label{1 rightarrow P T rightarrow P S rightarrow P ST rightarrow 1 for motives}
\end{equation}
\end{proposition}

\begin{proof} 
By Diagram \ref{the diagram compatibility of GM plus T 1 with GM 1 and pi0 GM plus 1 with pi0 GM 1}
and Definitions \ref{def of motivic Mumford Tate and Sato Tate groups}, \ref{GMKA1 as AST group}, and
\ref{GMLLKA1 as AST group} we have:
\begin{equation}
\gpP_{\rm{ST}} (M, \psi_{{\sim}}) = \pi_{0} (G_{\mathcal{M}_{{\sim}}(M), 1}) \, / \,
\pi_{0} (G_{\mathcal{M}_{{\sim}}(M \oplus \TT), 1}).
\label{P ST = p0 ( mathcal G alg K 1) over p0 (tilde mathcal G alg K 1) motivic}
\end{equation}
Now the claim follows by the snake lemma applied to Diagram 
\ref{pi0 applied to the diagram compatibility of GM 1 with GM}. 
\end{proof}

\section{Equidistribution of Frobenii in \texorpdfstring{$l$}{l}-adic realization of motives}
\label{equidistribution of Frobenii in l-adic realization of motives}
\medskip

As before, let $M$ be a direct factor of $h_{{\sim}}^{r}(X)(m)$. By definition, $M$ is a homogeneous motive
of weight  $n = 2 m - r$. Recall that $H^{r}_{l}$ denotes the $l$-adic realization of $h_{{\sim}}^{r}(X)(m)$ 
(see \S 13). Let $\rho_{l}$ be the $l$-adic representation coming from the 
$l$-adic realizations of $M$. Consider the finite Galois extension  $K_{0} / K$ associated with
$\rho_{l}$ according to \eqref{epsilon} and \eqref{epsilon and tilde epsilon}.

\begin{theorem} Assume that the family $(\rho_{l})$ 
is strictly compatible. Assume that $n \not= 0$ and there is $c \in \N$ such 
that $(\Z_{l}^{\times})^c \,\, {\rm{Id}}_{H^{r}_{l}}  \, \subset \, \rho_{H^{r}_{l}} (G_K)$ for all $l$. 
Moreover assume that for some $l$ coprime to $c$:
\begin{itemize}
\item[(1)] $K_0 \, \cap \, K(\mu_{\bar{l}}^{\otimes \, n}) \, = \, K$,
\item[(2)] $\gpast_{l, K}$ is an isomorphism with respect to $\rho_l$.
\end{itemize}
Then the Sato--Tate Conjecture holds for the representation $\rho_l \colon G_K \rightarrow \GIso_{(V_l, \psi_l)}(\Q_l)$ 
(resp. $\widetilde{\rho}_l\colon G_K \rightarrow \GL (V_{l} \oplus \Q_l(1))$)
with respect to $\ST_{K}(M)$ (resp. $\widetilde{\ST_{K}} (M)$) if and only if it holds for $\rho_l\colon G_{K_1} \rightarrow \GIso_{(V_l, \psi_l)}(\Q_l)$ (resp. $\widetilde{\rho}_l\colon G_{K_1} \rightarrow \GL (V_{l} \oplus \Q_l(1))$)  
with respect to $\ST_{K_1} (M)$ (resp. $\widetilde{\ST_{K_{1}}} (M)$) for all subextensions $K_1$ of $K_0/K$ such that 
$K_{0} / K_{1}$ is cyclic (cf. Theorem~\ref{Sato--Tate conjecture STK iff STK1}).
\label{ST with resp. to STK the same as with resp. to STK1}
\end{theorem}   

\begin{proof} Lemmas \ref{conditions D1, D2 are satisfied by Hodge str. associated with M}, 
\ref{conditions DR1, DR2 are satisfied by (VDR, psiDR) associated with M} and 
\ref{conditions R1-R4 are satisfied by Hodge str. and associated l-adic rep. for M} show that
$\rho_l$ satisfies conditions {\bf{(D1)}}, {\bf{(D2)}}, {\bf{(DR1)}}, {\bf{(DR2)}} and {\bf{(R1)}}--{\bf{(R4)}}. The category 
$\mathcal{M}_{{\rm{\sim}}} $ is abelian semisimple by assumption, hence $V_l$ is a direct factor of $H_l$ as a 
$\Q_l[G_K]$-module. Hence  $(\Z_{l}^{\times})^c \,\, {\rm{Id}}_{V_l}  \, \subset \, \rho_{l} (G_K)$ for all $l$. Since $l$ is coprime to $c$ then $1 + l\, \Z_l \subset (\Z_{l}^{\times})^c$. Therefore we can apply directly Theorem  
\ref{Sato--Tate conjecture STK iff STK1}. 
\end{proof}

\begin{remark} Let $A/K$ be an abelian variety with $\dim A \leq 4$. Then by
\cite[p. 5 Corollary 1]{W} the Lang conjecture holds, i.e. $\rho_l(G_K)$ contains the group of all 
homotheties in $\GL_{T_{l}(A)} (\Z_{l})$ for all $l \gg 0$.
\label{Lang conjecture for dim A < 5}
\end{remark}

\begin{remark}
If $n$ is odd, then by Propositions \ref{P(Vl, psil) trivial for n odd or GlK01 = CDl} and 
\ref{minimal field of connectedness independent of l} the field $K_{0}$ is independent of $l$. 
Hence for $l$ big enough $K_0 \, \cap \, K(\mu_{\bar{l}}^{\otimes \, n}) \, = \, K$.
\label{conditions for K0 cap K(mu-barl-otimes n = K)}
\end{remark}

\begin{remark} Let $A/K$ be an abelian variety with $\dim A \leq 3$. Then by
\cite[Theorem 6.11]{BK1} the homomorphism $\gpast_{l, K}$ is an isomorphism with respect to $\rho_l$. 
\label{AST conjecture for dim A leq 3}
\end{remark}

\begin{remark}
Assume that Algebraic Sato--Tate Conjecture~\ref{general algebraic Sato Tate conj.}
holds for $M$. By Proposition \ref{connected components iso}, $\ST_{K_0}(M)$ is connected if and only if 
$\gpP(V_l, \psi_l)$ is trivial (cf. Definition \ref{The l-adic parity group}). In particular $\ST_{K_0}(M)$ is 
connected when $n$ is odd (see Theorem \ref{conn comp equal for for GlK1 alg and GlK alg}).
Hence in some cases Theorems \ref{reduction for equidistribution2} and 
\ref{ST with resp. to STK the same as with resp. to STK1} provide an opportunity to 
check the Sato--Tate conjecture on the identity connected component of the Sato--Tate group. 
\label{connectedness of ST(K0)(M)}
\end{remark}

\begin{example} Let $A/K$ be an abelian surface over a number field. Consider the motive $h_{{\rm{ahc}}}^{1} (A)$. 
Put $\ST_{K}(A) := \ST_{K}(h_{{\rm{ahc}}}^{1} (A))$. Observe that $\ST_{K}(A)^{\circ} = \ST_{K_{0}}(A)$ by Remark
\ref{connectedness of ST(K0)(M)}. It was shown in \cite[Table 8, p. 1434]{FKRS12} that there are 52 possibilities for $\ST_{K}(A)$ and each quotient $\ST_{K}(A) / \ST_{K_{0}}(A)$ is solvable. 
Taking into account Remarks \ref{Lang conjecture for dim A < 5}, \ref{conditions for K0 cap K(mu-barl-otimes n = K)}, 
\ref{AST conjecture for dim A leq 3} we observe that all assumptions of Theorem \ref{ST with resp. to STK the same as with resp. to STK1} hold in this case. The identity connected component $\ST_{K_{0}}(A)$ in every of these 52 cases is isomorphic to one of the following 6 groups: $\gpU(1)$, $\SU(2)$, $\gpU(1) \times \gpU(1)$, $\gpU(1) \times \SU(2)$, 
$\SU(2) \times \SU(2)$, $\USp(4)$ (see \cite[Table 1, p. 1402]{FKRS12}).
\label{example of ST for abelian surfaces} 
\end{example}

\begin{remark} The Sato--Tate conjecture for abelian surfaces was proven in many cases by Christian Johansson
\cite{Jo}.
In the cases where $\ST_{K_{0}}(A)$ is one of $\gpU(1)$, $\gpU(1) \times \gpU(1)$, $\gpU(1) \times \SU(2)$, $\USp(4)$, by Theorem  \ref{reduction for equidistribution2}
it is only necessary to check the Sato--Tate conjecture for $\ST_{K_{0}}(A)$ itself.
In the case where $\ST_{K_{0}}(A) = \SU(2) \times \SU(2)$, one must separately treat the case
where $\ST_{K}(A) = N(\SU(2) \times \SU(2))$.
In the case where $\ST_{K_{0}}(A) = \SU(2)$, one must separately treat several additional cases.
\label{Sato--Tate for abelian surfaces by Christian Johannson}
\end{remark}

\begin{remark}
Let $A/K$ be an abelian variety such that over some finite extension of $K$, $A$ becomes isogenous to a product of abelian varieties with complex multiplication. Then the Sato--Tate conjecture is known for $A$; see for example \cite[Proposition~16]{Jo}. This includes some examples where $\ST_{K}(A) / \ST_{K_{0}}(A)$ is not solvable; for example, in \cite{FKS} one finds an example of an abelian threefold for which $\ST_{K}(A) / \ST_{K_{0}}(A)$ is simple of order 168. 
One can make higher-dimensional examples using the method of \cite{GK}.
\label{ST for abelian 3-folds}
\end{remark}

\end{document}